\documentclass[10pt, a4paper, twoside, reqno]{amsart}

\usepackage{geometry}

\usepackage{xcolor,soul,lipsum}
\usepackage[bookmarksopen=true]{hyperref}
\usepackage{url}
\usepackage{latexsym}
\usepackage{latexsym, amsmath, amstext, amssymb, amsfonts, amscd, bm, array, multirow, amsbsy, mathrsfs, mathtools}
\usepackage{amsthm}
\usepackage{t1enc}
\usepackage[mathscr]{eucal}

\usepackage{enumerate}
\usepackage[all]{xy} \CompileMatrices
\usepackage{amscd}
\hypersetup{colorlinks=false,pdfborderstyle={/S/U/W 0}}
\usepackage{slashed} %% for the Dirac operator
\usepackage{tikz-cd}
\usetikzlibrary{matrix}
\usepackage[normalem]{ulem}
\SelectTips{cm}{12} % computer modern fonts of size 12 for arrowheads
\numberwithin{equation}{section}

\theoremstyle{plain}
\newtheorem{theorem}[equation]{Theorem}
\newtheorem{lemma}[equation]{Lemma}
\newtheorem{corollary}[equation]{Corollary}
\newtheorem{proposition}[equation]{Proposition}

\theoremstyle{definition}
\newtheorem{definition}[equation]{Definition}
\newtheorem{definition-theorem}[equation]{Definition-Theorem}
\newtheorem{example}[equation]{Example}
\newtheorem{construction}[equation]{Construction}
\theoremstyle{remark}
\newtheorem{remark}[equation]{Remark}
\newtheorem{notation}[equation]{Notation}
\newtheorem{problem}[equation]{Problem}

\usepackage{geometry}
\usepackage{enumitem}

\usepackage[utf8]{inputenc}
\usepackage{color}

\usepackage{cite}

\newcommand{\tr}{\operatorname{tr}}
\newcommand{\id}{{\operatorname{id}}}

\newcommand{\Aut}{{\operatorname{Aut}}}
\newcommand{\CC}{{\mathbb C}}

\newcommand{\RR}{{\mathbb R}}
\newcommand{\NN}{{\mathbb N}}
\newcommand{\ZZ}{{\mathbb Z}}

\newcommand{\vol}{\operatorname{vol}}

\newcommand{\surj}{\to\kern-1.8ex\to}

\newcommand{\cG}{\mathcal{G}}

\newcommand{\cA}{\mathcal{A}}
\newcommand{\cL}{\mathcal{L}}

\newcommand{\cU}{\mathcal{U}}

\newcommand{\cV}{\mathcal{V}}

\newcommand{\cW}{\mathcal{W}}

\newcommand{\Diff}{\mathrm{Diff}}

\newcommand{\Conf}{\mathrm{Conf}}
\newcommand{\Fix}{{\mathrm{Fix}}}

\newcommand{\scCon}{{\mathbb{C}\hspace{-0.02cm}\mathrm{on}}}
\newcommand{\Sol}{\mathrm{Sol}}

\newcommand{\scSol}{{\mathbb{S}\mathrm{ol}}}
\newcommand{\Met}{\mathrm{Met}}

\newcommand{\G}{\mathrm{G}}
\newcommand{\bbH}{\mathbb{H}}
\newcommand{\bbX}{{\mathbb{X}}}
\newcommand{\bbA}{{\mathbb{A}}}
\newcommand{\cE}{\mathcal{E}}

\newcommand{\cI}{\mathcal{I}}
\newcommand{\dd}{\mathrm{d}}

\newcommand{\scC}{\mathscr{C}}

\newcommand{\scS}{\mathscr{S}}
\newcommand{\scD}{\mathscr{D}}
\newcommand{\scA}{{\mathscr{A}}}
\newcommand{\scB}{{\mathscr{B}}}
\newcommand{\scP}{\mathscr{P}}
\newcommand{\Equivar}{{\operatorname{Equivar}}}
\newcommand{\dslash}{{/\hspace{-0.1cm}/}}
\newcommand{\Gpd}{{\mathscr{G}\mathrm{pd}}}
\newcommand{\Cat}{{\mathscr{C}\mathrm{at}}}
\newcommand{\opp}{{\mathrm{op}}}
\newcommand{\rev}{{\mathrm{rev}}}
\newcommand{\hol}{{\operatorname{hol}}}

\newcommand{\curv}{\mathrm{curv}}
\newcommand{\cl}{{\mathrm{cl}}}

\newcommand{\rmfam}{{\mathrm{fam}}}

\newcommand{\HLB}{{\mathrm{HLB}}}
\newcommand{\Grb}{{\mathrm{Grb}}}

\newcommand{\bbGrb}{{\mathbb{G}\hspace{-0.01cm}\mathrm{rb}}}
\newcommand{\scG}{{\mathscr{G}}}

\newcommand{\pr}{{\mathrm{pr}}}
\newcommand{\U}{\mathrm{U}}
\newcommand{\Con}{\operatorname{Con}}

\newcommand{\Mdl}{\operatorname{Mdl}}

\newcommand{\scMdl}{{\mathbb{M}\hspace{-0.015cm}\mathrm{dl}}}

\newcommand{\rmH}{\mathrm{H}}

\newcommand{\rmB}{{\mathrm{B}}}

\newcommand{\rmb}{{\mathrm{b}}}
\newcommand{\Mfd}{{\mathscr{M}\mathrm{fd}}}
\newcommand{\Cart}{{\mathscr{C}\mathrm{art}}}
\newcommand{\Cartfam}{{\mathscr{C}\mathrm{art}_{\mathrm{fam}}}}
\newcommand{\Ch}{{\mathscr{C}\mathrm{h}}}
\newcommand{\Ab}{{\mathscr{A}\mathrm{b}}}
\newcommand{\sSet}{{\mathscr{S}\mathrm{et}_\Delta}}
\newcommand{\Kan}{{\mathscr{K}\mathrm{an}}}
\newcommand{\Set}{{\mathscr{S}\mathrm{et}}}
\newcommand{\Fun}{{\operatorname{Fun}}}
\newcommand{\scFun}{{\mathscr{F}\mathrm{un}}}
\newcommand{\scH}{{\mathscr{H}}}
\newcommand{\scK}{{\mathscr{K}}}
\newcommand{\scX}{{\mathscr{X}}}
\newcommand{\Ho}{\operatorname{Ho}}
\newcommand{\holim}{{\operatorname{holim}}}
\newcommand{\hocolim}{{\operatorname{hocolim}}}
\newcommand{\colim}{{\operatorname{colim}}}
\renewcommand{\lim}{{\operatorname{lim}}}
\newcommand{\Lan}{{\operatorname{Lan}}}
\newcommand{\hoLan}{{\operatorname{hoLan}}}
\newcommand{\Tot}{\operatorname{Tot}}
\newcommand{\sfc}{\operatorname{\mathsf{c}}}
\newcommand{\sfe}{{\mathsf{e}}}

\newcommand{\GCov}{{\operatorname{GCov}}}

\newcommand{\scSym}{{\mathbb{S}\hspace{-0.015cm}\mathrm{ym}}}
\newcommand{\SYM}{{\operatorname{SYM}}}
\newcommand{\scAut}{{\mathbb{A}\hspace{-0.015cm}\mathrm{ut}}}

\newcommand{\AUT}{{\operatorname{AUT}}}
\newcommand{\Grp}{{\operatorname{Grp}}}
\newcommand{\res}{{\mathrm{res}}}
\newcommand{\bbS}{{\mathbb{S}}}

\newcommand{\bbString}{{\mathbb{S}\mathrm{tring}}}
\newcommand{\ECA}{{\mathrm{ECA}}}
\newcommand{\GMet}{{\mathrm{GMet}}}
\newcommand{\CalECA}{{\mathcal{E}\hspace{-0.01cm}\mathcal{C}\hspace{-0.1cm}\mathcal{A}}}
\newcommand{\AtCA}{{\mathrm{AtCA}}}
\newcommand{\DD}{{\operatorname{DD}}}
\newcommand{\SC}{{\check{\mathrm{S}}\mathrm{C}}}
\newcommand{\textint}{\textstyle{\int}}
\newcommand{\coll}{\operatorname{coll}}

\newcommand{\bbGamma}{{\mathbbe{\Gamma}}}
\newcommand{\wtG}{{\widetilde{G}}}
\newcommand{\wtF}{{\widetilde{F}}}

\newcommand{\bas}{{\mathit{bas}}}
\newcommand{\Ric}{\operatorname{Ric}}
\newcommand{\bbG}{{\mathbb{G}}}

\newcommand{\bbP}{{\mathbb{P}}}
\newcommand{\bbF}{{\mathbb{F}}}
\newcommand{\SmGrp}{{\mathscr{S}\mathrm{m}\mathscr{G}\mathrm{rp}}}

\newcommand{\qqandqq}{\qquad \text{and} \qquad}
\newcommand{\qandq}{\quad \text{and} \quad}

\renewcommand{\ul}[1]{\underline{#1}}
\newcommand{\qen}{\hfill$\triangleleft$}

\newcommand{\wequiv}{\overset{\sim}{\longrightarrow}}
\newcommand{\<}{\langle}
\renewcommand{\>}{\rangle}

\DeclareMathAlphabet{\mathbbe}{U}{bbold}{m}{n}

\newcommand{\bbDelta}{{\mathbbe{\Delta}}}
\mathtoolsset{showonlyrefs,showmanualtags}

\usepackage{enumitem}
\setlist[itemize]{leftmargin=*}

\newcolumntype{P}[1]{>{\centering\arraybackslash}p{#1}}

\usepackage[math]{anttor}
\usepackage[T1]{fontenc}

\begin{document}

\title[Higher geometric structures on manifolds]{Higher Geometric Structures on Manifolds\\and the Gauge Theory of Deligne Cohomology}
 
\author[S. Bunk]{S. Bunk} \address{Mathematical Institute, University of Oxford, United Kingdom}
\email{severin.bunk@maths.ox.ac.uk}

\author[C. S. Shahbazi]{C. S. Shahbazi} \address{Departamento de Matem\'aticas, Universidad UNED - Madrid, Reino de Espa\~na}
\email{cshahbazi@mat.uned.es} 
\address{Fakult\"at f\"ur Mathematik, Universit\"at Hamburg, Bundesrepublik Deutschland.}
\email{carlos.shahbazi@uni-hamburg.de}

\thanks{2020 MSC. Primary: 53C08. Secondary: 18N50, 18N60, 53Z05, 58D27.}
\keywords{Higher geometry; gerbes; higher gauge theory; moduli stacks; Courant algebroids; supergravity}

\begin{abstract}
We study smooth higher symmetry groups and moduli $\infty$-stacks of generic higher geometric structures on manifolds.
Symmetries are automorphisms which cover non-trivial diffeomorphisms of the base manifold.
We construct the smooth higher symmetry group of any geometric structure on $M$ and show that this completely classifies---via a universal property---equivariant structures on the higher geometry.
We construct moduli stacks of higher geometric data as $\infty$-categorical quotients by the action of the higher symmetries, extract information about the homotopy types of these moduli $\infty$-stacks, and prove a helpful sufficient criterion for when two such higher moduli stacks are equivalent.

In the second part of the paper we study higher $\U(1)$-connections.
First, we observe that higher connections come organised into higher groupoids, which further carry affine actions by Baez-Crans-type higher vector spaces.
We compute a presentation of the higher gauge actions for $n$-gerbes with $k$-connection, comment on the relation to higher-form symmetries, and present a new String group model.
We construct smooth moduli $\infty$-stacks of higher Maxwell and Einstein-Maxwell solutions, correcting previous such considerations in the literature, and compute the homotopy groups of several moduli $\infty$-stacks of higher $\U(1)$-connections.
Finally, we show that a discrepancy between two approaches to the differential geometry of NSNS supergravity (via generalised and higher geometry, respectively) vanishes at the level of moduli $\infty$-stacks of NSNS supergravity solutions.
\end{abstract}

\maketitle

\vspace{-0.5cm}
\tableofcontents

%%%%%%%%%%%%%%%%%%%%%%%%%%%%%%%%%%%%%%%%%%%%%%%%%%%%%%%%%%%%%%%%%%%%%%%%%%%%

\section{Introduction}
\label{sec:intro}

%%%%%%%%%%%%%%%%%%%%%%%%%%%%%%%%%%%%%%%%%%%%%%%%%%%%%%%%%%%%%%%%%%%%%%%%%%%%

In~\cite{Gajer:Geo_of_Deligne_coho} Gajer presented simultaneously a geometric description and a categorification of the Deligne cohomology groups $\rmH^n(M; \scD(n))$ of any manifold $M$, for each $n \in \NN$.
The seminal insight was that these groups arise as isomorphism classes of $\rmB^{n-1} \U(1)$-bundles with $n$-connection~\cite[Thm.~C]{Gajer:Geo_of_Deligne_coho}.
For $k \in \NN_0$, a $k$-connection on a $\rmB^{n-1} \U(1)$-bundle consists of a tower of local connection forms of degrees $1,2, \ldots, k$ on the base manifold%
\footnote{Note that we have adopted a slightly different terminology from~\cite{Gajer:Geo_of_Deligne_coho}.}.
It was soon realised that this structure can be resolved further:
there are Kan complexes (in fact symmetric monoidal $n$-groupoids) of $\rmB^{n-1} \U(1)$-bundles with $k$-connection on any manifold $M$.
They can be obtained via the Dold-Kan correspondence from \v{C}ech resolutions of the Deligne complex $\scD(n)$---see, for instance,~\cite{FSS:Cech_for_diff_classes}---or modelled in terms of differential function spectra~\cite{HS:Quadratic_functions} (which is more in line with the Cheeger-Simons, rather than Deligne, description of differential cohomology).
In the modern literature, a $\rmB^n \U(1)$-bundle usually goes by the term \textit{$n$-gerbe}.
Shifting indices accordingly, we can describe $\rmH^{n+2}(M; \scD(n{+}2))$ as the connected components of a simplicial abelian group of $n$-gerbes with $(n{+}1)$-connection on the manifold $M$.

Connections on (higher) gerbes play important roles in string theory and the various supergravity theories arising as its low-energy limits~\cite{Freed:Dirac_charge_quantisation, DFM:Spin_strs_and_superstrings, FMS:Uncertainty_of_fluxes, Szabo:Quant_of_Higher_Ab_GT, ABEHSN:SymTFTs_from_String_thy, Schreiber:DCCT_v2, FMS:Heisenberg_and_NC_fluxes}.
Nevertheless, despite the success of connections on classical principal bundles in geometry, topology and physics, very little is currently known about the mathematical gauge theory of higher connections and, in particular, their moduli theory (for first steps, see, for instance,~\cite{MR:YM_for_BGrbs, FRS:Higher_gerbe_connections, FSS:Cech_for_diff_classes, Szabo:Quant_of_Higher_Ab_GT}).
In part, this is due to the fact that treating moduli of higher geometric structures genuinely requires higher-categorical technology, whose development is only a recent---and still ongoing---achievement in mathematics.

In the present paper we move beyond the pure structural theory of higher bundles with connections (such as how they organise into higher groupoids, their local-to-global properties, classification, parallel transport, etc.) and initiate the study of their actual \textit{gauge} theory.
This requires us, in the first place, to develop a general formalism for constructing smooth higher symmetry groups and moduli $\infty$-(pre)stacks of generic higher geometric structures on manifolds%
\footnote{We emphasise that the focus of the present paper lies on \textit{finite}, or \textit{global}, aspects of higher gauge theoretic moduli stacks, rather than infinitesimal moduli in the sense of deformation theory; the latter will be explored in later works.}.
Importantly, the action of symmetries includes the action of diffeomorphisms on the higher geometric data.
We show that these smooth higher symmetry groups completely classify equivariant structures of higher geometric structures via a universal property, extract information about the homotopy types of the spaces underlying our higher moduli stacks, and prove a helpful sufficient criterion for when two moduli stacks are equivalent.

This theory is applicable, for instance, to higher bundles with connection in full generality, but in the second part of this paper, we focus on the geometry of higher $\U(1)$-connections, and thus Deligne differential cocycles.
This produces numerous insights into the gauge theory of higher connections, including that higher connections come organised into higher groupoids, rather than sets, and carry an affine action by higher Baez-Crans vector spaces.
We compute an explicit presentation of the higher gauge actions for $n$-gerbes with $k$-connection arising from the formalism in Part~\ref{part:Higher Geometry}, comment on the relation of these smooth higher gauge groups to higher-form symmetries, and show that a particular higher symmetry group of a $1$-gerbe with $1$-connection produces a new String group model (in the sense of~\cite{Killingback:Anomalies_and_loop_geometry, Stolz:Conj_on_pos_Ric, Bunk:Pr_ooBdls_and_String}).
The proof of the latter uses our theorem on equivalences of moduli stacks together with further results on the topology of moduli stacks.
We present and analyse smooth higher stacks of higher $\U(1)$-connections, including moduli stacks of higher Maxwell and Einstein-Maxwell theory, on fixed background data in a fully homotopical manner, correcting previous such considerations in the literature.
Finally, we address two approaches to the geometry of NSNS supergravity, originating in higher and generalised geometry, respectively.
We show that a discrepancy between two approaches to the B-field in the differential geometry of NSNS supergravity vanishes at the level of higher moduli stacks.
Resolving this discrepancy was one of the main original motivations for this paper.

%%%%%%%%%%%%%%%%%%%%%%%%%%%%%%%%%%%%%%%%%%%%%%%%%%%%%%%%%%%%%%%%%%%%%%%%%%%%

\section{Overview of main results}
\label{sec:main results}

%%%%%%%%%%%%%%%%%%%%%%%%%%%%%%%%%%%%%%%%%%%%%%%%%%%%%%%%%%%%%%%%%%%%%%%%%%%%

%%%%%%%%%%%%%%%%%%%%%%%%%%%%%%%%%%%%%%%%%%%%%%%%%%%%%%%%%%%%%%%%%%%%%%%%%%%%

\subsection{Moduli $\infty$-(pre-)stacks of higher geometric structures on manifolds}
\label{sec:main results: Part I}

%%%%%%%%%%%%%%%%%%%%%%%%%%%%%%%%%%%%%%%%%%%%%%%%%%%%%%%%%%%%%%%%%%%%%%%%%%%%

This paper is split into two main parts, plus a third part of appendices.
In Part~\ref{part:Higher Geometry} we treat higher geometric structures on manifolds, as well as their smooth higher symmetries and moduli at a general level.
In Part~\ref{part:Higher U(1) gauge fields} we analyse in detail various geometric and gauge theoretic problems which feature higher $\U(1)$-connections.

%%%%%%%%%%%%%%%%%%%%%%%%%%%%%%%%%%%%%%%%%%%%%%%%%%%%%%%%%%%%%%%%%%%%%%%%%%%%

\subsubsection{Higher geometric data: smooth families and diffeomorphism actions}

%%%%%%%%%%%%%%%%%%%%%%%%%%%%%%%%%%%%%%%%%%%%%%%%%%%%%%%%%%%%%%%%%%%%%%%%%%%%

We begin by developing techniques which produce global higher geometric structures on manifolds from local data and incorporate both smooth families of geometric structures and smooth diffeomorphism actions.
We briefly sketch the key constructions here, before describing our main results in the following sections.
At the most general level, higher geometric structures are modelled as morphisms of $\infty$-presheaves or $\infty$-sheaves~\cite{Lurie:HTT, Schreiber:DCCT_v2}.
In our case, these are functors out of nerves of 1-categories, so that we can describe them in terms of model categories of simplicial presheaves.
Having this explicit description available is beneficial for computational purposes.

Concretely, let $\Cart$ denote the category of cartesian spaces%
\footnote{A \textit{cartesian} space is a manifold $c$ which is diffeomorphic to $\RR^n$ for some $n \in \NN_0$.} and smooth maps; all smooth families of geometric data in this paper are parameterised over objects of $\Cart$.
A local model for \textit{smooth families} of a higher geometric structure is a simplicial homotopy sheaf on a certain category $\Cartfam$ (Definition~\ref{def:Cartfam}), whose objects are smooth families of cartesian spaces%
\footnote{This approach enhances the formalism from, for instance,~\cite{Schreiber:DCCT_v2, FSS:Cech_for_diff_classes} to describe smooth families of higher geometric data on a given manifold.}.
In order to also treat the smooth action of the diffeomorphism group $\Diff(M)$ on smooth families of geometric data, we introduce the following category $\rmB\scD$:
its objects are cartesian spaces, and its morphisms are commutative squares of smooth maps
\begin{equation}
\label{eq:intro:BD square}
\begin{tikzcd}
	c_0 \times M \ar[r, "\hat{\varphi}"] \ar[d, "\pr"']
	& c_1 \times M \ar[d, "\pr"]
	\\
	c_0 \ar[r, "f"']
	& c_1
\end{tikzcd}
\end{equation}
where $\hat{\varphi}$ restricts to a diffeomorphism on all fibres and satisfies that its fibrewise inverse is also smooth.
We understand functors $\rmB\scD^\opp \to \sSet$ as describing smooth families of higher geometric data on $M$, subject to the smooth action of $\Diff(M)$ via pullback.
We present a functor
\begin{equation}
\label{eq:intro: (-)^scD}
	\Fun(\Cartfam^\opp, \sSet) \longrightarrow \Fun(\rmB\scD^\opp, \sSet)\,,
\end{equation}
which takes a local model for higher geometric structures and produces simplicial presheaves on $\rmB\scD$ in a particularly well-behaved way (Theorem~\ref{st:fctrs from H_fam to sPSh(rmBscD)}).
As in the non-family case, the idea is to resolve $M$ by a good open covering by cartesian spaces and then evaluate a simplicial presheaf on the \v{C}ech nerve of this covering (compare, for instance,~\cite{FSS:Cech_for_diff_classes}).
The difficulty in the enhanced formalism here lies in combining this idea with smooth parameterisations of coverings and geometric data, as well as the action of diffeomorphisms of $M$:
we resolve this by passing through an intermediate category of families of coverings of $M$, which includes families of diffeomorphisms in its morphisms, and then average over all good open coverings by means of a left Kan extension (Proposition~\ref{st:Lan_varpi}).
Our construction can also be used to produce explicit concretifications in the sense of~\cite{Schreiber:DCCT_v2, BSS:Stack_of_YM_fields, FRS:Higher_gerbe_connections} (see, in particular, Section~\ref{sec:smooth fams of higher U(1)-conns}).

We devote a section to showing that this averaging construction has a deeper homotopical meaning:
on the left fibrations associated to our simplicial presheaves, the above left Kan extension amounts precisely to an $\infty$-categorical localisation at all refinements of open coverings (Theorem~\ref{st:Lan_(varpi^D) wtG as a localisation}).
We are thus able to produce, in a computable fashion, smooth families of geometric structures on $M$ from purely local data, and these families carry a natural smooth action of the diffeomorphism group of $M$.

%%%%%%%%%%%%%%%%%%%%%%%%%%%%%%%%%%%%%%%%%%%%%%%%%%%%%%%%%%%%%%%%%%%%%%%%%%%%

\subsubsection{Smooth higher groups of automorphisms and symmetries}

%%%%%%%%%%%%%%%%%%%%%%%%%%%%%%%%%%%%%%%%%%%%%%%%%%%%%%%%%%%%%%%%%%%%%%%%%%%%

With this construction set up we move on to our main goal of investigating the smooth higher symmetry groups and moduli stacks of higher geometric data on $M$.
First, we study their symmetry groups:
let $G^\scD \colon \rmB\scD^\opp \to \sSet$ be a simplicial presheaf.
We fix a vertex $\cG \in G^\scD(\RR^0)$; for instance, if $G$ classifies higher bundles on cartesian spaces, then $\cG$ is the datum of a higher bundle on $M$.
We let $\Diff_{[\cG]}(M) \subset \Diff(M)$ denote the subgroup of all those diffeomorphisms whose action on $M$ admits a lift to $\cG$; that is, it consists of those diffeomorphisms $\varphi$ for which there is an equivalence $\cG \simeq \varphi^* \cG$ between $\cG$ and the pullback of $\cG$ along the action of $\varphi$.
Let $\rmB\scD[\cG] \subset \rmB\scD$ denote the category objects are cartesian spaces and whose morphisms are commutative squares~\eqref{eq:intro:BD square} where $\hat{\varphi}$ restricts to an element of $\Diff_{[\cG]}(M)$ on each fibre.

A smooth $\infty$-group is a group object $\bbGamma$ in the $\infty$-category $\scP(N\Cart)$ of presheaves of spaces%
\footnote{In this introduction we will freely pass between the $\infty$-category $\scP(N\scC)$ of presheaves of spaces on the nerve $N\scC$ of a 1-category $\scC$ and its presentations by the projective model category of simplicial presheaves on $\scC$ and the covariant model category $\sSet_{/N\scC^\opp}$.
In the main text, we employ the technology in Appendix~\ref{app:rectification} to transition between these models.}
on $\Cart$; smooth $\infty$-groups form an $\infty$-category $\SmGrp$.
Any smooth group action $\Phi \colon \bbGamma \to \Diff(M)$ on $M$ for which there exists an equivariant structure on $\cG$ must factor through $\Diff_{[\cG]}(M)$.
For any such action, we construct a smooth $\infty$-group $\scSym_\Phi(\cG)$, consisting of all possible ways to lift the action of smooth families in $\bbGamma$ from $M$ to $\cG$.
In particular, if $\bbGamma = *$ is the trivial group, this produces the smooth $\infty$-group $\scAut(\cG)$ of automorphisms of $\cG$.
For instance, if $\cG$ is a higher principal bundle on $M$, then $\scAut(\cG)$ is its smooth higher gauge group.
We show (Corollaries~\ref{st:Aut-Sym-Diff[cG] extension} and~\ref{st:Aut-Sym_phi-Gamma extension}):

\begin{theorem}
\label{st:intro:Aut-Sym SES}
Given any section $\cG \in G^\scD(\RR^0)$ and any smooth $\infty$-group action $\Phi \colon \bbGamma \to \Diff_{[\cG]}(M)$, there is an extension of group objects in $\scP(N\Cart)$
\begin{equation}
\label{eq:intro:Aut-Sym SES}
\begin{tikzcd}
	\scAut(\cG) \ar[r]
	& \scSym_\Phi(\cG) \ar[r]
	& \bbGamma\,.
\end{tikzcd}
\end{equation}
\end{theorem}

We derive this fact using that smooth $\infty$-groups can be presented as left fibrations whose fibres are reduced simplicial sets~\cite{NSS:Pr_ooBdls_I, Cisinski:HiC_and_HoA}.
The significance of the smooth $\infty$-group $\scSym_\Phi(\cG)$ lies in the fact that it completely controls equivariant structures on the geometric object $\cG$.
To see this, we define the following functor
\begin{equation}
	\Equivar(\cG) \colon (\SmGrp_{/\Diff(M)} )^\opp \longrightarrow \scS\,.
\end{equation}
Its domain is the $\infty$-category of smooth $\infty$-group actions on $M$, i.e.~of morphisms of smooth $\infty$-groups $\Phi \colon \bbGamma \to \Diff(M)$.
Its target is the $\infty$-category of spaces.
The value of $\Equivar(\cG)$ on a pair $(\bbGamma, \Phi)$ is the space of equivariant structures on $\cG$ for the action $(\bbGamma, \Phi)$ on $M$.
We prove a significant generalisation of the main results in~\cite{BMS:2-Grp_Ext} (Theorem~\ref{st:scSym represents Equivar} and Corollary~\ref{st:char of equivar structures}):

\begin{theorem}
\label{st:intro:Equivar presentation}
The functor $\Equivar(\cG)$ is representable.
It is represented by the morphism
\begin{equation}
	\scSym_\iota(\cG) \longrightarrow \Diff_{[\cG]}(M)
\end{equation}
from~\eqref{eq:intro:Aut-Sym SES}, where $\bbGamma = \Diff_{[\cG]}(M)$ and $\iota \colon \Diff_{[\cG]}(M) \hookrightarrow \Diff(M)$ is the canonical inclusion.
Moreover, the are natural equivalences between the following spaces:
\begin{enumerate}
\item the space $\Equivar_\Phi(\cG)$ of $(\bbGamma, \Phi)$-equivariant structures on $\cG$,

\item the space of lifts of $\Phi \colon \bbGamma \to \Diff(M)$ through the morphism $\scSym_\iota(\cG) \longrightarrow \Diff_{[\cG]}(M)$,

\item the space of splittings of the smooth $\infty$-group extension $\scAut(\cG) \longrightarrow \scSym_\Phi(\cG) \longrightarrow \bbGamma$, i.e.~of sections of the morphism $\scSym_\Phi(\cG) \longrightarrow \bbGamma$ in $\SmGrp$.
\end{enumerate}
\end{theorem}

%%%%%%%%%%%%%%%%%%%%%%%%%%%%%%%%%%%%%%%%%%%%%%%%%%%%%%%%%%%%%%%%%%%%%%%%%%%%

\subsubsection{Moduli $\infty$-(pre)stacks of higher geometric structures}

%%%%%%%%%%%%%%%%%%%%%%%%%%%%%%%%%%%%%%%%%%%%%%%%%%%%%%%%%%%%%%%%%%%%%%%%%%%%

The smooth higher automorphism and symmetry groups of higher geometric structures are key ingredients in the study of moduli:
they encode the redundancies we wish to divide out in our systems.
We put geometric moduli problems on the following general footing, which is motivated by physical field theories.
This formalism lies at the heart of the present paper.
Again, to maintain maximum availability of computational tools, we begin our presentation using simplicial presheaves, and only then pass to left fibrations and $\infty$-categorical language in order to formulate and prove our main results.

Consider a pair of morphisms of simplicial presheaves on $\rmB\scD$,
\begin{equation}
\label{eq:intro:Sol --> Conf --> Fix}
\begin{tikzcd}
	\Sol^\scD \ar[r, hookrightarrow]
	& \Conf^\scD \ar[r]
	& \Fix^\scD\,.
\end{tikzcd}
\end{equation}
We think of $\Conf^\scD$ as describing all possible configurations of the higher geometric data we are interested in.
For example, this could consist of smooth families of higher bundles with connection on $M$, such as the higher $\U(1)$-bundles with connection that underlie the geometry of Deligne cohomology.
The morphism $\Sol^\scD \hookrightarrow \Conf^\scD$ is an inclusion of a full simplicial subpresheaf, which we think of as describing smooth families of only those configurations whose moduli we wish to study.
For instance, these could be families of solutions to a set of field equations (see Section~\ref{sec:main results: Part II} for more on this perspective), such as higher bundles with connections on $M$ solving higher Yang-Mills equations.
Finally, the morphism $\Conf^\scD \to \Fix^\scD$ is an objectwise Kan fibration; we think of this as forgetting part of the geometric data, leaving only the part which we wish to keep fixed in the moduli problem.
An important example is where $\Conf^\scD$ describes smooth families of higher bundles with connection on $M$, $\Fix^\scD$ encodes smooth families of higher bundles without connection on $M$, and the morphism $\Conf^\scD \to \Fix^\scD$ forgets the connections.

For each section $\cG \in \Fix^\scD(\RR^0)$ we obtain an $\infty$-presheaf on $N\Cart$ which encodes smooth families of solutions on $\cG$ (such as connections on a fixed higher bundle which solve a certain set of field equations), together with the morphisms of such solutions which are intrinsic to the geometric set-up; we denote this $\infty$-presheaf by $\scSol(\cG)$.
Recall that, given a smooth higher group action $\Phi \colon \bbGamma \to \Diff_{[\cG]}(M)$, we obtain an associated smooth higher symmetry group $\scSym_\Phi(\cG)$.
We show that the $\infty$-group $\scSym_\Phi(\cG)$ naturally acts on $\scSol(\cG)$.

\begin{definition}
The moduli $\infty$-prestack of solutions on $\cG$ modulo the action of $\scSym_\Phi(\cG)$ is the $\infty$-categorical quotient
\begin{equation}
	\scMdl_\Phi(\cG) \coloneqq \scSol(\cG) \dslash \scSym_\Phi(\cG)
	\quad \in \scP(N\Cart)\,.
\end{equation}
\end{definition}

We construct this quotient explicitly as an $\infty$-categorical left Kan extension and provide a presentation as a left fibration over $N\Cart^\opp$ (Definition~\ref{def:scMdl_phi(cG)}).
Further, we show (Theorem~\ref{st:descent for spl pshs on tint BH}):

\begin{theorem}
Whenever $\bbGamma$ is a sheaf of ordinary groups on $\Cart$ (satisfying descent with respect to good open coverings) and the solution subpresheaf satisfies a certain descent condition, then $\scMdl_\Phi(\cG) \in \scP(N\Cart)$ is an $\infty$-stack (or $\infty$-sheaf), i.e.~it satisfies descent with respect to good open covers.
\end{theorem}

A very important situation is the case where $(\bbGamma, \Phi) = (*, e)$ is the trivial group action on $M$, and thus $\scSym_e(\cG) = \scAut(\cG)$ is the smooth $\infty$-group of automorphisms of $\cG$ (if $\cG$ is a higher principal bundle on $M$, then $\scAut(\cG)$ is its smooth higher gauge group).
The following statement shows that this special case already encodes a large amount of information about the moduli stack for \textit{any} choice of action $(\bbGamma, \Phi)$ (Theorem~\ref{st:Mod-Mdl-Diff prBun}):

\begin{theorem}
For each smooth $\infty$-group action $\Phi \colon \bbGamma \to \Diff_{[\cG]}(M)$ on $M$, there is a principal $\infty$-bundle
\begin{equation}
	\scMdl_e(\cG) \longrightarrow \scMdl_\Phi(\cG)
\end{equation}
in $\scP(N\Cart)$ with structure group $\bbGamma$.
\end{theorem}

Let $S \colon \scP(N\Cart) \to \scS$ denote the functor which takes the \textit{underlying space} of an $\infty$-presheaf on $\Cart$~\cite{Bunk:R-loc_HoThy, Bunk:Pr_ooBdls_and_String}.
It also goes by the name of \textit{smooth singular complex}, or \textit{concordance space}~\cite{BEBdBP:Classifying_spaces_of_oo-sheaves} and is equivalent to the $\infty$-categorical colimit functor.
Since $S$ preserves group objects and principal $\infty$-bundles~\cite{Bunk:Pr_ooBdls_and_String}, we obtain insights into the topology of our moduli stacks (Corollaries~\ref{st:moduli prBun in spaces} and~\ref{st:LES for S(Mdl_phi)}):

\begin{corollary}
For each smooth $\infty$-group action $\Phi \colon \bbGamma \to \Diff_{[\cG]}(M)$ on $M$, there is a fibre sequence
\begin{equation}
	S\bbGamma
	\longrightarrow S \big( \scMdl_e(\cG) \big)
	\longrightarrow S \big( \scMdl_\Phi(\cG) \big)
\end{equation}
of spaces.
In particular, this induces a long exact sequence of homotopy groups.
\end{corollary}

One of the main results in the first part of this paper is a simple sufficient criterion for when two higher moduli stacks are equivalent.
It applies to situations where we have an augmentation of the sequence~\eqref{eq:intro:Sol --> Conf --> Fix} by a further morphism, i.e.~to a sequence
\begin{equation}
\label{eq:intro:Sol --> Conf --> Fix_0 --> Fix_1}
\begin{tikzcd}
	\Sol^\scD \ar[r, hookrightarrow]
	& \Conf^\scD \ar[r]
	& \Fix_0^\scD \ar[r, "p"]
	& \Fix_1^\scD
\end{tikzcd}
\end{equation}
of simplicial presheaves on $\rmB\scD$.
We say that $p$ is an \textit{0-connected} if, for each object $c \in \rmB\scD$, the morphism $p_{|c} \colon \Fix_0^\scD(c) \longrightarrow \Fix_1^\scD(c)$ induces a bijection of connected components.
Let $\cG_0 \in \Fix_0^\scD(\RR^0)$, and let $\cG_1 \coloneqq p(\cG_0)$ be its image in $\Fix_1^\scD(\RR^0)$.
In this case, the presheaves $\scSol^\scD(\cG_0)$ and $\scSol^\scD(\cG_1)$ of solutions with respect to the fixed data $\cG_0$ and $\cG_1$, respectively, are not equivalent, and neither are the associated smooth higher symmetry groups whose action we divide out.
A crucial example is where $\Fix_0^\scD$ describes $n$-gerbes on $M$ with $k$-connection, $\Fix_1^\scD$ describes $n$-gerbes with $l$-connection, for $l < k$, and the morphism $p$ forgets part of the connection data (see Section~\ref{sec:main results: Part II} below for more detail on this example).
We show (Theorem~\ref{st:equiv result for Mdl oo-stacks}):

\begin{theorem}
\label{st:intro:equiv result for Mdl oo-stacks}
In the situation~\eqref{eq:intro:Sol --> Conf --> Fix_0 --> Fix_1}, suppose that $p$ is 0-connected.
Then, for each smooth $\infty$-group action $\Phi \colon \bbGamma \to \Diff_{[\cG]}(M)$, there is a canonical equivalence in $\scP(N\Cart)$ between moduli $\infty$-prestacks,
\begin{equation}
	\scMdl_\Phi(\cG_0) \simeq \scMdl_\Phi(\cG_1)\,.
\end{equation}
\end{theorem}

This insight allows us to compare and reconcile different moduli problems which are a priori different (such as in NSNS supergravity; see below).
At the same time it allows us to treat certain moduli problems in several different, but equivalent ways, some of which may be particularly well adapted to the problems we wish to solve.

%%%%%%%%%%%%%%%%%%%%%%%%%%%%%%%%%%%%%%%%%%%%%%%%%%%%%%%%%%%%%%%%%%%%%%%%%%%%

\subsection{Applications to higher $\U(1)$-connections and supergravity}
\label{sec:main results: Part II}

%%%%%%%%%%%%%%%%%%%%%%%%%%%%%%%%%%%%%%%%%%%%%%%%%%%%%%%%%%%%%%%%%%%%%%%%%%%%

In Part~\ref{part:Higher U(1) gauge fields} of this paper we apply the general formalism and results from Part~\ref{part:Higher Geometry} to the geometry of higher $\U(1)$-connections on higher gerbes, and thus of Deligne differential cocycles.

%%%%%%%%%%%%%%%%%%%%%%%%%%%%%%%%%%%%%%%%%%%%%%%%%%%%%%%%%%%%%%%%%%%%%%%%%%%%

\subsubsection{Properties of higher $\U(1)$-connections}

%%%%%%%%%%%%%%%%%%%%%%%%%%%%%%%%%%%%%%%%%%%%%%%%%%%%%%%%%%%%%%%%%%%%%%%%%%%%

We enhance the well-known construction of classifying $\infty$-stacks of $n$-gerbes with $k$-connection (see~\cite{FSS:Cech_for_diff_classes} for a review) to obtain simplicial presheaves
\begin{equation}
	\Grb^{n,\scD}_{\nabla|k} \colon \rmB\scD^\opp \longrightarrow \sSet\,.
\end{equation}
These simplicial presheaves describe smooth families of $n$-gerbes with $k$-connections on a fixed manifold $M$, which are further acted on by the smooth diffeomorphism group of $M$.
They arise by means of the formalism in Part~\ref{part:Higher Geometry} and the Dold-Kan correspondence, applied to a family-version of the Deligne complex.
Given an $n$-gerbe on $M$ without connection, $\cG \in \Grb^{n,M}(\RR^0)$, a $k$-connection on $\cG$ is a tuple $\cA^{(k)} = (A_1, \ldots, A_k)$ of locally defined forms of degrees $1, 2, \ldots, k$, satisfying certain compatibility conditions.
On the level of $\infty$-sheaves on $N\Cart$, similar objects were constructed already by Schreiber in~\cite{Schreiber:DCCT_v2} by means of the abstract formalism of \textit{(differential) concretification} (see, for instance,~\cite[Def.~5.2.8, Def.~5.2.105, Rmk.~5.2.106, Def.~6.4.118]{Schreiber:DCCT_v2} and~\cite[Def.~3.3, Prop.~3.4]{BSS:Stack_of_YM_fields}).
Here we employ explicit constructions of the $\infty$-sheaves describing smooth families of $n$-gerbes with $k$-connection%
\footnote{Note that our use of the term \textit{moduli $\infty$-stacks} differs from that in~\cite{Schreiber:DCCT_v2}:
there, it refers to the higher stacks describing smooth families of higher geometric structures on manifolds (or on other $\infty$-stacks), whereas here we reserve the term to describe quotients of solution stacks by the action of smooth higher symmetry groups.}.

We make the following observations (Theorem~\ref{st:forgetting conns is a principal map} and Section~\ref{sec:Con_k on fixed n-gerbe}):

\begin{theorem}
\label{st:intro:k-conns on n-gerbes}
The $k$-connections on any $n$-gerbe on $M$ form a smooth $(k{-}1)$-groupoid.
Moreover, this $(k{-}1)$-groupoid carries a smooth affine action by a smooth simplicial vector space.
\end{theorem}

These are genuine features of \textit{higher} connections.
The first result in particular is in stark contrast to classical gauge theory, where the connections on a given principal bundle simply form a \textit{set} and do not come with an intrinsic notion of morphisms between them.
The second statement categorifies the fact that connections on principal bundles form affine spaces over the vector space of 1-forms valued in the adjoint bundle.

%%%%%%%%%%%%%%%%%%%%%%%%%%%%%%%%%%%%%%%%%%%%%%%%%%%%%%%%%%%%%%%%%%%%%%%%%%%%

\subsubsection{Symmetries of higher $\U(1)$-bundles and connections}

%%%%%%%%%%%%%%%%%%%%%%%%%%%%%%%%%%%%%%%%%%%%%%%%%%%%%%%%%%%%%%%%%%%%%%%%%%%%

Let $(\cG, \cA^{(k)})$ be a fixed $n$-gerbe with fixed $k$-connection $\cA^{(k)}$ on $M$, for $k \leq n$; we can think of this as a principal $\infty$-bundle for the smooth structure group $\rmB^{n-k}\rmB_\nabla^k \U(1)$, where $\rmB$ denotes the delooping, or classifying space, functor, and where $\rmB_\nabla^k \U(1)$ classifies principal $\infty$-bundles with connection for the group $\rmB^{k-1}\U(1)$.
We are interested in the higher groupoids of connections on $(\cG, \cA^{(k)})$, i.e.~the ways of extending $\cA^{(k)}$ to a full connection on $\cG$, modulo the actions of $\scAut(\cG, \cA^{(k)})$ and $\scSym_\Phi(\cG, \cA^{(k)})$ (for any smooth action $\Phi \colon \bbGamma \to \Diff_{[\cG]}(M)$).

So far, we have written the action of the smooth higher automorphism group $\scAut(\cG, \cA^{(k)})$ only in terms of an $\infty$-functor, or a left fibration.
We devote a section to unravelling, or strictifying, this higher group action in terms of an action of a presheaf of simplicial groups.
We show that this smooth $\infty$-group action is presented precisely by a higher form of the standard action of $\U(1)$-gauge transformations, adapted to the structure group $\rmB^{n-k}\rmB_\nabla^k \U(1)$.
The strictification heavily relies on the Dold-Kan correspondence and is proven in Theorem~\ref{st:Cech pres of gauge action}.
The smooth higher gauge actions on higher groupoids of connections has not been considered before (though truncated versions appear in, for instance,~\cite{Freed:Dirac_charge_quantisation, FMS:Heisenberg_and_NC_fluxes, FMS:Uncertainty_of_fluxes, Szabo:Quant_of_Higher_Ab_GT}); here we derive its full form from first principles.

There is a canonical equivalence
\begin{equation}
	\scAut(\cG, \cA^{(k)})
	\simeq \Grb^{n-1,M}_{\nabla,k}
\end{equation}
of smooth $\infty$-groups, where the right-hand side is the smooth higher group of $(n{-}1)$-gerbes with $k$-connection on $M$ and their canonical (abelian) tensor product.
In particular, these smooth higher groups categorify the $\infty$-dimensional Lie groups considered as the gauge groups in higher abelian gauge theories in~\cite{FMS:Heisenberg_and_NC_fluxes, FMS:Uncertainty_of_fluxes, Szabo:Quant_of_Higher_Ab_GT}, where only the zero-truncations of $\scAut(\cG, \cA^{(k)})$ were considered.

The smooth higher gauge and symmetry groups have interesting relations to higher-form symmetries in physics~\cite{GKSW:Gen_global_syms, FMT:Top_sym_in_QFT}.
In particular, they describe and formalise the higher symmetries in QFTs investigated in~\cite{Sharpe:Gen_global_syms}.
We point out that, for each fixed $n$-gerbe with $n$-connection on $M$ and each smooth map $\sigma \colon \Sigma^{(n-p)} \to M$ from a closed $(n{-}p)$-manifold $\Sigma^{(n-p)}$, higher gerbe holonomy provides a morphism of smooth groups (Remark~\ref{rmk:higher form syms}, see also Remark~\ref{rmk:hol and Mdl(G,A^1)})
\begin{equation}
	U_{(-)}(\Sigma^{(n-p)}, \sigma) \colon \pi_p \scAut(\cG, \cA^{(n)}) \longrightarrow \U(1)\,,
\end{equation}
for $p \geq 0$, where the homotopy groups are formed in $\scP(N\Cart)$.
These morphisms only depend on the homotopy class of $\sigma$, for $p > 0$, and further descend to moduli stacks of connections on $(\cG, \cA^{(n)})$ modulo gauge transformations (Remark~\ref{rmk:hol and Mdl(G,A^1)}).
If we consider instead an $(n{+}1)$-connection $\cA$ on $\cG$, then homotopy-invariance also holds for $p = 0$.

Finally, we use the theory developed thus far to provide a new smooth model for the String group (as introduced in~\cite{Stolz:Conj_on_pos_Ric, Killingback:Anomalies_and_loop_geometry}).
Let $H$ be a compact, simple, simply connected Lie group, and let $\Phi = L \colon H \to H$ be the action of $H$ on itself via left multiplication.
Further, let $\cG_\bas$ be a basic gerbe on $H$ (i.e.~one whose class in $\rmH^3(H;\ZZ)$ is a generator).
This admits a distinguished connection $\cA_\bas$~\cite{Meinrenken:The_basic_gerbe}.
It follows from~\cite{FRS:Higher_gerbe_connections} that the smooth 2-group $\scSym_L(\cG_\bas, \cA_\bas)$ of symmetries of $(\cG_\bas, \cA_\bas)$ which lift the action $L$ is a smooth model for the String group $\bbString(H)$ of $H$ (note that $\bbString(H)$ is defined only up to weak homotopy equivalence).
Further, it was shown in~\cite{BMS:2-Grp_Ext, Bunk:Pr_ooBdls_and_String} that also $\scSym_L(\cG_\bas)$ is a smooth higher group model for $\bbString(H)$.
Here, we prove the following intermediate case (Theorem~\ref{st:String^1(H)}):

\begin{theorem}
Let $\cA^{(1)}$ be any 1-connection on the basic gerbe $\cG_\bas$ over a compact, simple, simply connected Lie group $H$.
Then, the extension of smooth $\infty$-groups (compare Theorem~\ref{st:intro:Aut-Sym SES})
\begin{equation}
	\scAut(\cG_\bas, \cA^{(1)}) \longrightarrow \scSym_L(\cG_\bas, \cA^{(1)}) \longrightarrow H
\end{equation}
is a smooth String group extension.
\end{theorem}

%%%%%%%%%%%%%%%%%%%%%%%%%%%%%%%%%%%%%%%%%%%%%%%%%%%%%%%%%%%%%%%%%%%%%%%%%%%%

\subsubsection{Gauge theory of higher $\U(1)$-connections}

%%%%%%%%%%%%%%%%%%%%%%%%%%%%%%%%%%%%%%%%%%%%%%%%%%%%%%%%%%%%%%%%%%%%%%%%%%%%

We present various important examples of higher gauge theoretic moduli stacks for higher $\U(1)$-connections, including higher Maxwell and Einstein-Maxwell theory, higher self-dual connections and higher $\U(1)$-BF theory, as recently explored in the context of Turaev-Viro models in~\cite{HMT:Gen_Ab_TV_and_U(1)-BF} and references therein.
In particular, we extend arguments from~\cite{MR:YM_for_BGrbs} to show (Theorem~\ref{st:MW existence (vacuum)}):

\begin{theorem}
For each Riemannian metric $g$ on a closed manifold $M$, and each $n$-gerbe with $k$-connection $(\cG, \cA^{(k)})$ on $M$ with $k \leq n$, there exists an extension of the $k$-connection $\cA^{(k)}$ on $\cG$ to a higher (vacuum) Maxwell solution with respect to $g$.
\end{theorem}

We then turn to the relation between different moduli stacks of connections on $n$-gerbes:
for $l, k \in \NN_0$ with $l < k$, let
\begin{equation}
\label{eq:intro:forgetting connections}
	p^k_l \colon \Grb^{n,\scD}_{\nabla|k} \to \Grb^{n,\scD}_{\nabla|l}
\end{equation}
denote the morphism which forgets the highest form parts $(A_{l+1}, \ldots, A_k)$ of a $k$-connection on $\cG$.
Using the various properties of \v{C}ech-Deligne cohomology, we show the following technical result (Corollary~\ref{st:Grb conn comps}), which allows us to apply Theorem~\ref{st:intro:equiv result for Mdl oo-stacks} to higher $\U(1)$-connections:

\begin{proposition}
\label{st:intro:forgetting connections}
The morphism~\eqref{eq:intro:forgetting connections} is a Kan fibration and 0-connected whenever $k \leq n$.
\end{proposition}

It follows that $\Diff_{[\cG, \cA^{(k)}]}(M) = \Diff_{[\cG]}(M)$ depends only on the equivalence class of $\cG$ whenever $k \leq n$.
As a direct application of Theorem~\ref{st:intro:equiv result for Mdl oo-stacks} and Proposition~\ref{st:intro:forgetting connections}, we obtain that various pairs of moduli stacks of higher $\U(1)$-connections are equivalent.
For instance, we obtain:

\begin{theorem}
Let $(\cG, \cA^{(k)})$ be an $n$-gerbe with $k$-connection on $M$, and let $\cA^{(l)}$ be the $l$-connection on $\cG$ obtained by forgetting part of the connection data (by means of the morphism~\eqref{eq:intro:forgetting connections}).
For each smooth higher group action $\Phi \colon \bbGamma \to \Diff_{[\cG]}(M)$, there is a canonical equivalence of moduli $\infty$-(pre)stacks of Maxwell solutions on $(\cG, \cA^{(k)})$, and $(\cG, \cA^{(l)})$, respectively,
\begin{equation}
	\scMdl_{MW, \Phi}(\cG, \cA^{(k)}) \simeq \scMdl_{MW, \Phi}(\cG, \cA^{(l)})\,.
\end{equation}
\end{theorem}

In other words, we can fix an arbitrary part of a higher $\U(1)$-connection on an $n$-gerbe in any moduli problem, as long as we use the correct notion of gauge transformation in each case.
We clarify which truncation of the full moduli stack of higher Maxwell theory for $n=1$ was investigated in~\cite{MR:YM_for_BGrbs} and compute the homotopy type of the higher stacks of $n$-gerbes with $k$-connection (Theorem~\ref{st:spaces of n-gerbes}) and of higher Maxwell solutions on $M$ (Theorem~\ref{st:HoType of Sol_Mw}).

%%%%%%%%%%%%%%%%%%%%%%%%%%%%%%%%%%%%%%%%%%%%%%%%%%%%%%%%%%%%%%%%%%%%%%%%%%%%

\subsubsection{Moduli stacks of NSNS supergravity solutions}

%%%%%%%%%%%%%%%%%%%%%%%%%%%%%%%%%%%%%%%%%%%%%%%%%%%%%%%%%%%%%%%%%%%%%%%%%%%%

One motivation for the present paper is to achieve a comparison of two higher moduli stacks of solutions to NSNS supergravity on a fixed manifold $M$, as we now explain.
Through Gajer's insight, together with the relevance of differential cohomology theories in string theory~\cite{Freed:Dirac_charge_quantisation, DFM:Spin_strs_and_superstrings, FMS:Uncertainty_of_fluxes, Szabo:Quant_of_Higher_Ab_GT, ABEHSN:SymTFTs_from_String_thy, Schreiber:DCCT_v2, FMS:Heisenberg_and_NC_fluxes}, (higher) gerbes have come to play important roles in string theory and the various supergravity theories arising as its low-energy limits.
In particular, it has been argued that the B-field in string theory is described by a 2-connection on a 1-gerbe~\cite{Kapustin:D-branes_in_non-triv_B-fields}.
The investigation of solutions to supergravity equations and their moduli has long been an open problem in differential geometry (see, for instance, the introduction of~\cite{GFGM:Futaki_invars_and_Yaus_conjecture} for a nice overview and further references), but only little is known in situations where the 1-gerbe underlying the B-field has non-trivial topology (i.e.~the class $[\cG] \in \rmH^3(M;\ZZ)$ is non-zero).

Recently, the introduction of generalised geometry has led to significant progress both in differential geometry in general and mathematical supergravity in particular~\cite{Hitchin:Generalised_CY_manifolds, Gualtieri:PhD_thesis} (see also~\cite{GFS:Gen_Ricci_flow} for a textbook introduction).
In this approach, the B-field (in NSNS supergravity) is described as part of a generalised metric on an exact Courant algebroid on $M$.
Such a generalised metric can always be decomposed into a Riemannian metric on $M$ and an isotropic splitting of the exact Courant algebroid; the latter then corresponds to the B-field in NSNS supergravity.

The question which thus arises is how these two models for the geometry of NSNS supergravity are related and, even more importantly, how the solutions to NSNS supergravity in the two approaches compare.
As we explain below, there is no equivalence between the (higher) stacks of field configurations which the two models predict for NSNS supergravity.
It is therefore a highly important question whether and, if so, how the discrepancy between these two approaches to the B-field (Problem~\ref{prob:mouli of NSNS B-fields}) and to moduli of NSNS supergravity solutions on $M$ (Problem~\ref{prob:moduli of NSNS solutions}) can be reconciled.

To answer these questions, we introduce models for exact Courant algebroids on $M$ (without and with isotropic splittings, respectively) in our formalism, i.e.~as simplicial-presheaves,
\begin{equation}
	\ECA^\scD,\, \ECA_\nabla^\scD \colon \rmB\scD^\opp \longrightarrow \sSet\,.
\end{equation}
Our description has the following pleasant feature:
\v{S}evera's classification of exact Courant algebroids produces a canonical bijection between elements of $\rmH^3(M;\RR)$ and isomorphism classes of exact Courant algebroids on $M$.
Thus, the groupoid of exact Courant algebroids categorifies the degree-three de Rham cohomology of $M$.
A natural question is whether this also categorifies the abelian group structure on $\rmH^3(M;\RR)$.
For generic exact Courant algebroids, this appears to be unknown, but our model provides a groupoid which is canonically equivalent to that of exact Courant algebroids on $M$ and has a canonical symmetric monoidal structure which achieves this refined categorification (Theorem~\ref{st:AtCA_nabla and categorification of AbGrp H^3}).

We enhance Hitchin's construction of the generalised tangent bundle to a morphism of simplicial presheaves
\begin{equation}
	\AtCA_\nabla \colon \Grb^{1,\scD}_{\nabla|1} \longrightarrow \ECA^\scD
\end{equation}
on $\rmB\scD$.
It associates to a smooth family of gerbes with connective structure on $M$ a smooth family of exact Courant algebroids on $M$.
Then, as already observed by Hitchin~\cite{Hitchin:Brackets_forms_and_functionals}, for each smooth family of 1-gerbes $(\cG, \cA^{(1)})$ with 1-connection on $M$, there is a canonical bijection between curvings on $(\cG, \cA^{(1)})$ (i.e.~extensions of $\cA^{(1)}$ to a full connection on $\cG$) and smooth families of isotropic splittings of $\AtCA(\cG, \cA^{(1)})$.

In supergravity, (semi-classical) charge quantisation~\cite{DFM:Spin_strs_and_superstrings, FMS:Heisenberg_and_NC_fluxes, FMS:Uncertainty_of_fluxes, Szabo:Quant_of_Higher_Ab_GT} forces the \v{S}evera class of the exact Courant algebroid describing the B-field to lie in the image of the rationalisation map $\rmH^3(M;\ZZ) \to \rmH^3(M;\RR)$, and thus the exact Courant algebroid to be a generalised tangent bundle of a 1-gerbe with connective structure $(\cG, \cA^{(1)})$ on $M$.
Viewing the latter as a principal 2-bundle on $M$ with structure 2-group $\rmB_\nabla \U(1)$, the associated exact Courant algebroid $\AtCA_\nabla(\cG, \cA^{(1)})$ is its higher Atiyah algebroid~\cite{Collier:Inf_Syms_of_DD_Gerbes}.
From this perspective, the correct notion of automorphisms acting on the B-field configurations are those of $(\cG, \cA^{(1)})$.
This is in analogy to how one can describe connections on a principal $G$-bundle $P$ as splittings of its Atiyah sequence, but the gauge transformations acting on the connections remain the automorphisms of $P$.
Describing connections on $P$ modulo automorphisms of the Atiyah algebroid of $P$ instead is a different interesting question.

We then still arrive at the following puzzle:
if the NSNS B-field is modelled as a connection on a fixed 1-gerbe $\cG$, its configurations form a smooth groupoid (see Theorem~\ref{st:intro:k-conns on n-gerbes} above), and the relevant smooth symmetry 2-group is $\scSym_\Phi(\cG)$.
In contrast, suppose the NSNS B-field is modelled as an isotropic splitting of an exact Courant algebroid (or as part of a generalised metric), which---by charge quantisation---we may assume to be of the form $\AtCA_\nabla(\cG, \cA^{(1)})$ for some 1-connection $\cA^{(1)}$ on $\cG$.
In that case, the configurations of the B-field form a smooth set, and the relevant smooth symmetry 2-group is $\scSym_\Phi(\cG, \cA^{(1)})$.
Neither the two objects describing configurations, nor the two symmetry 2-groups are equivalent.
Nevertheless, combining Theorem~\ref{st:intro:equiv result for Mdl oo-stacks} and Proposition~\ref{st:intro:forgetting connections} allows us to prove (Theorems~\ref{st:equiv of B-field moduli} and~\ref{st:equiv of NSNS SuGra moduli}):

\begin{theorem}
Let $\cG$ be a 1-gerbe on $M$, and let $\cA^{(1)}$ be any 1-connection on $\cG$.
Further, let $\Phi \colon \bbGamma \to \Diff_{[\cG]}(M)$ be a smooth higher group action on $M$.
\begin{enumerate}
\item There is a canonical equivalence of moduli $\infty$-stacks,
\begin{equation}
	\scMdl_{B, \Phi}(\cG)
	\simeq \scMdl_{B, \Phi}(\cG, \cA^{(1)})
\end{equation}
of NSNS B-fields described by $\cG$ and $\AtCA_\nabla(\cG, \cA^{(1)})$, respectively.

\item There is a canonical equivalence of moduli $\infty$-prestacks of NSNS supergravity solutions  
\begin{equation}
	\scMdl_{NS, \Phi}(\cG)
	\simeq \scMdl_{NS, \Phi}(\cG, \cA^{(1)})\,,
\end{equation}
modelled on $\cG$ and $\AtCA_\nabla(\cG, \cA^{(1)})$, respectively.
\end{enumerate}
\end{theorem}

\begin{remark}
Solutions of NSNS supergravity are in particular generalised Ricci solitons \cite{GFS:Gen_Ricci_flow}. In fact, most of our later statement regarding the moduli space of NSNS supergravity solutions apply \emph{mutatis mutandis} to the moduli space of generalised Ricci solitons.
\qen
\end{remark}

This resolves the apparent discrepancy between the two approaches to NSNS supergravity moduli:
while the two models are not equivalent at the level of configurations, and also do not have equivalent symmetry 2-groups, these differences cancel each other in the passage to the moduli $\infty$-prestacks.
It will be highly interesting to explore similar results in the setting of heterotic supergravity in the future, where we expect analogous links between connections on Chern-Simons 2-gerbes~\cite{Waldorf:String_Cons, Bunke:String_strs_and_Pfaffians} and structures on String algebroids~\cite{GFRT:Holomorphic_String_Algds, GF:Lectures_on_Strominger}.

%%%%%%%%%%%%%%%%%%%%%%%%%%%%%%%%%%%%%%%%%%%%%%%%%%%%%%%%%%%%%%%%%%%%%%%%%%%%

\subsection*{Notation}

%%%%%%%%%%%%%%%%%%%%%%%%%%%%%%%%%%%%%%%%%%%%%%%%%%%%%%%%%%%%%%%%%%%%%%%%%%%%

\begin{itemize}
\item $\Mfd$ denotes the category of smooth manifolds and smooth maps.

\item $\Cart \subset \Mfd$ is the full subcategory of cartesian spaces.

\item $\Cartfam$ denotes the category whose objects are families of cartesian spaces, parameterised by objects $c \in \Cart$ (Definition~\ref{def:Cartfam}).

\item $\Gpd$ denotes the 2-category of (small) groupoids, and $\Cat$ the 2-category of (small) categories.

\item $\scS$ is the $\infty$-category of spaces (or, equivalently, $\infty$-groupoids).

\item By an $\infty$-category, we mean a simplicial set satisfying the inner horn-lifting conditions in the sense of, for instance,~\cite{Lurie:HTT, Cisinski:HiC_and_HoA}.

\item We write $\Fun(-,-)$ to denote a category of functors between ordinary categories and $\scFun(-,-)$ to denote $\infty$-categories of functors between $\infty$-categories.

\item We write $\sSet = \Fun(\bbDelta^\opp, \Set)$ for the category of simplicial sets.

\item If $\scC$ is a simplicially enriched category, we write $\ul{\scC}(-,-) \colon \scC^\opp \times \scC \to \sSet$ for the simplicially enriched hom functor.

\item For an $\infty$-category $A$, we abbreviate the $\infty$-category of $\infty$-presheaves on $A$ as $\scP(A) \coloneqq \scFun(A^\opp, \scS)$.

\item For a category $\scC$ and object $c \in \scC$, we write $\scC_{/c}$ and $\scC_{c/}$ for the slice categories over and under $c$, respectively.
The same applies if $\scC$ is an $\infty$-category.

\item $\Diff(M)$ is the diffeomorphism group of a manifold $M$.
We write $\Diff_{[\cG]}(M)$ for the subgroup of $\Diff(M)$ which preserves the equivalence class of a given geometric structure $\cG$ on $M$ (Definition~\ref{def:Diff_[cG] and D_[cG]}).

\item Given a group $H$, we let $\rmB H \in \Gpd$ denote its delooping:
this is the groupoid with a single object, whose group of automorphisms is $H$.

\item We write $\rmB\scD$ for the Grothendieck construction of the functor $\Cart^\opp \to \Gpd$, $c \mapsto \rmB \Diff(M)$.
Analogously, we write $\rmB\scD[\cG]$ for the Grothendieck construction of $\Cart^\opp \to \Gpd$, $c \mapsto \rmB \Diff_{[\cG]}(M)$.

\item We introduce a category $\GCov^\scD$ whose objects are smooth families of good open coverings of $c \times M$, where $c \in \Cart$ is a cartesian space.
Its morphisms are a combination of pulling coverings back along fibrewise diffeomorphisms and refining the result by a good open covering (Definition~\ref{def:GCov^D(M)}).

\item We write $\Ch_{\geq 0}$ for the category of non-negatively graded chain complexes of abelian groups.

\item $\Ab_\Delta$ denotes the category of simplicial abelian groups.

\item We write $\varGamma \colon \Ch_{\geq 0} \to \Ab_\Delta$ for (one of the two functors in the) Dold-Kan correspondence.

\item We use bold-face font to denote $\infty$-presheaves and roman font to denote ordinary or simplicial presheaves.

\item We write $\textint F \to N\scC$ to denote the left fibration associated to a (1-categorical) functor $F \colon \scC \to \sSet$.

\end{itemize}

%%%%%%%%%%%%%%%%%%%%%%%%%%%%%%%%%%%%%%%%%%%%%%%%%%%%%%%%%%%%%%%%%%%%%%%%%%%%

\subsection*{Acknowledgements}

%%%%%%%%%%%%%%%%%%%%%%%%%%%%%%%%%%%%%%%%%%%%%%%%%%%%%%%%%%%%%%%%%%%%%%%%%%%%

The authors would like to thank André Henriques, Lukas Müller, Dmitri Pavlov, Sakura Schäfer-Nameki, and Richard Szabo for helpful discussions. SB's research was funded by the Deutsche Forschungsgemeinschaft (DFG, German Research Foundation) under the project number 468806966.
The work of C.S.S. is funded by the Germany Excellence Strategy \emph{Quantum Universe} - 390833306 and the 2022 Leonardo Grant for Researchers and Cultural Creators of the BBVA Foundation.

%%%%%%%%%%%%%%%%%%%%%%%%%%%%%%%%%%%%%%%%%%%%%%%%%%%%%%%%%%%%%%%%%%%%%%%%%%%%

\part{Moduli $\infty$-stacks of higher geometric structures on manifolds}
\label{part:Higher Geometry}

%%%%%%%%%%%%%%%%%%%%%%%%%%%%%%%%%%%%%%%%%%%%%%%%%%%%%%%%%%%%%%%%%%%%%%%%%%%%

This first part of the paper develops a general formalism for constructing smooth families of higher geometric structures on manifolds, taking into account the smooth action of the diffeomorphisms of the manifold.
Both of these aspects are key to obtaining a full picture of the global moduli $\infty$-stacks of higher geometric structures.

Higher geometric structures are, in full generality, described in terms of $\infty$-sheaves%
\footnote{We use the term $\infty$-sheaf and $\infty$-stack interchangeably.}
with respect to (good) open coverings.
As we work on 1-categorical sites, these $\infty$-sheaves can always be presented in terms of simplicial presheaves which satisfy homotopy descent (see, for instance,~\cite[Thm.~7.9.8, Cor.~7.9.9]{Cisinski:HiC_and_HoA}).
Since such presentations provide additional computational tools, we start our set-up from simplicial homotopy sheaves on an appropriate category with Grothendieck coverage (Sections~\ref{sec:sites for families}, \ref{sec:families of hgeo structures} and~\ref{sec:averaging good open coverings}).

We then construct the smooth higher automorphism and symmetry groups of higher geometric structures and prove their universal property (Sections~\ref{sec:restricting to scD[cG]} and~\ref{sec:higher symmetry groups of hgeo strs}).
In order to encode the smooth and homotopy coherent actions of these groups which underlie the definition of moduli $\infty$-stacks, we pass to the language of left-fibrations (another model for $\infty$-presheaves~\cite[Thm.~7.8.9]{Cisinski:HiC_and_HoA}).

Finally, using the formalism developed so far, we construct moduli $\infty$-prestacks of higher geometric data (Section~\ref{sec:moduli oo-prestacks of hgeo strs on M}) and prove a useful criterion for when two different moduli $\infty$-prestacks are equivalent.
The set-up we present is motivated by field-theoretic problems, but applies in more general situations.
In Section~\ref{sec:descent for moduli oo-prestacks} we show that under certain conditions our moduli $\infty$-prestacks satisfy descent with respect to good open coverings of parameter spaces, i.e.~form $\infty$-stacks.

%%%%%%%%%%%%%%%%%%%%%%%%%%%%%%%%%%%%%%%%%%%%%%%%%%%%%%%%%%%%%%%%%%%%%%%%%%%%

\section{Sites for smooth families of manifolds and diffeomorphisms}
\label{sec:sites for families}

%%%%%%%%%%%%%%%%%%%%%%%%%%%%%%%%%%%%%%%%%%%%%%%%%%%%%%%%%%%%%%%%%%%%%%%%%%%%

In this section we set the stage for our study of higher geometric structures on manifolds.
We introduce a particular category with a Grothendieck coverage, whose sheaves encode smooth families of geometric data on a fixed manifold $M$, parameterised by cartesian spaces, and the smooth action of the diffeomorphism group $\Diff(M)$.

%%%%%%%%%%%%%%%%%%%%%%%%%%%%%%%%%%%%%%%%%%%%%%%%%%%%%%%%%%%%%%%%%%%%%%%%%%%%

\subsection{Smooth families of cartesian spaces}
\label{sec:Families of cartesian spaces}

%%%%%%%%%%%%%%%%%%%%%%%%%%%%%%%%%%%%%%%%%%%%%%%%%%%%%%%%%%%%%%%%%%%%%%%%%%%%

We begin by constructing a (Grothendieck) site which we use to parameterise smooth families of higher geometric structures.

\begin{definition}
Let $\Mfd$ be the category of smooth manifolds and smooth maps.
Let $\Cart \subset \Mfd$ be the full subcategory on those manifolds $c$ which are diffeomorphic to $\RR^n$, where $n$ ranges over all non-negative integers.
\end{definition}

\begin{definition}
\label{def:Cartfam}
Let $\Cartfam$ denote the following category:
its objects are smooth, locally trivial fibre bundles $\hat{c} \to c$, where $c \in \Cart$, and whose fibre is a cartesian space%
\footnote{Note that the bundles $(\hat{c} \to c)$ are also globally trivial by the fact that $c$ is smoothly contractible and~\cite[Thm.~44.24]{KM:Convenient_global_analysis}.
It follows that $\hat{c}$ is itself a cartesian space.}.
The morphisms $(\hat{c} \to c) \longrightarrow (\hat{d} \to d)$ are commutative squares
\begin{equation}
\begin{tikzcd}
	\hat{c} \ar[r, "\hat{f}"] \ar[d]
	& \hat{d} \ar[d]
	\\
	c \ar[r, "f"']
	& d
\end{tikzcd}
\end{equation}
of morphisms in $\Cart$.
\end{definition}

In particular, for each object $(\hat{c} \to c)$ in $\Cartfam$ there is a cartesian space $d$ and some isomorphism in $\Cartfam$ of the form
\begin{equation}
\begin{tikzcd}
	\hat{c} \ar[r, "\cong"] \ar[d]
	&c \times d \ar[d, "\pr_c"]
	\\
	c \ar[r, equal]
	& c
\end{tikzcd}
\end{equation}

\begin{definition}
We say an open covering $\cU = \{U_a\}_{a \in \Lambda}$ of a manifold $M$ is \textit{good} if each finite intersection $U_{a_0 \cdots a_n} \coloneqq U_{a_0} \cap \cdots, \cap U_{a_n}$, for $a_0, \ldots, a_n \in \Lambda$, is either empty or a cartesian space.
\end{definition}

\begin{remark}
Every open covering of each manifold $M$ admits a refinement by a good open covering~\cite[Appendix~A]{FSS:Cech_for_diff_classes}.
\qen
\end{remark}

\begin{definition}
\label{def:coverings in Cart_fam}
A \textit{covering} of an object $(\hat{c} \to c) \in \Cartfam$ is a family $(\hat{\cU} \to \cU) = \{(\hat{U}_a \to U_a)\}_{a \in \Lambda}$ of subsets
\begin{equation}
\begin{tikzcd}
	\hat{U}_a \ar[r, hookrightarrow] \ar[d]
	& \hat{c} \ar[d]
	\\
	U_a \ar[r, hookrightarrow]
	& c
\end{tikzcd}
\end{equation}
such that
\begin{enumerate}
\item $\hat{\cU}$ and $\cU$ are good open coverings of $\hat{c}$ and $c$, respectively, and

\item Every finite intersection $(\hat{U}_{a_0 \cdots a_n} \to U_{a_0 \cdots a_n})$ is either empty or an object in $\Cartfam$.
\end{enumerate}
\end{definition}

\begin{example}
Let $c,d \in \Cart$ with good open coverings $\cU = \{U_a\}_{a \in \Lambda}$ of $c$ and $\cV = \{V_b\}_{b \in \Xi}$ of $d$.
Then, the product covering $\{U_a \times V_b\}_{a \in \Lambda, b \in \Xi}$ is a covering of $(\pr_c \colon c \times d \to c)$ in $\Cartfam$.
\qen
\end{example}

\begin{example}
The second condition in Definition~\ref{def:coverings in Cart_fam} is not implied by the first:
for instance, let x$c = d = \RR$ and consider the object $(\pr_c \colon c \times d \to c) \in \Cartfam$.
The covering of $c \times d$ and $c$ depicted here:
\begin{center}
	\includegraphics[scale=0.2]{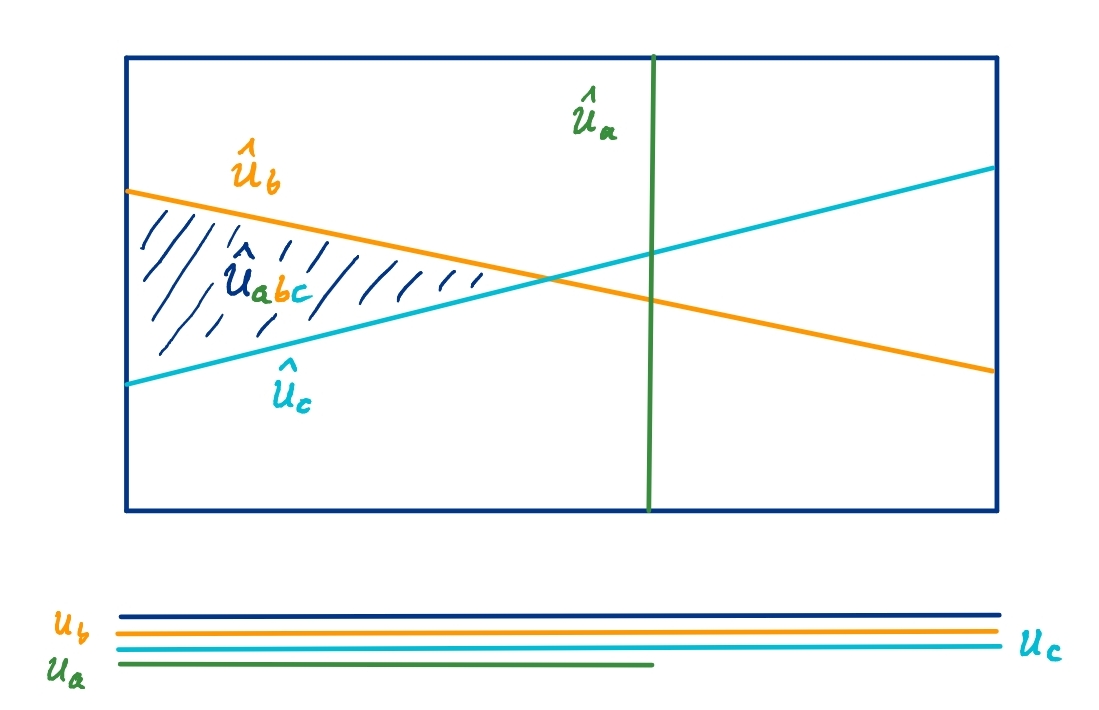}
\end{center}
satisfies condition (1), but violates condition (2) as the projection $\hat{U}_{abc} \to U_{abc}$ is not surjective.
\qen
\end{example}

\begin{remark}
Consider a covering $(\hat{\cU} \to \cU)$ of an object $(\hat{c} \to c) \cong (c \times c' \to c) \in \Cartfam$.
For each $x \in c$, we obtain an induced open covering $\hat{\cU}_{|x}$ of the fibre $d$.
This covering can vary from point to point, as illustrated in the following example for $c = c' = \RR$:
\begin{center}
	\includegraphics[scale=0.2]{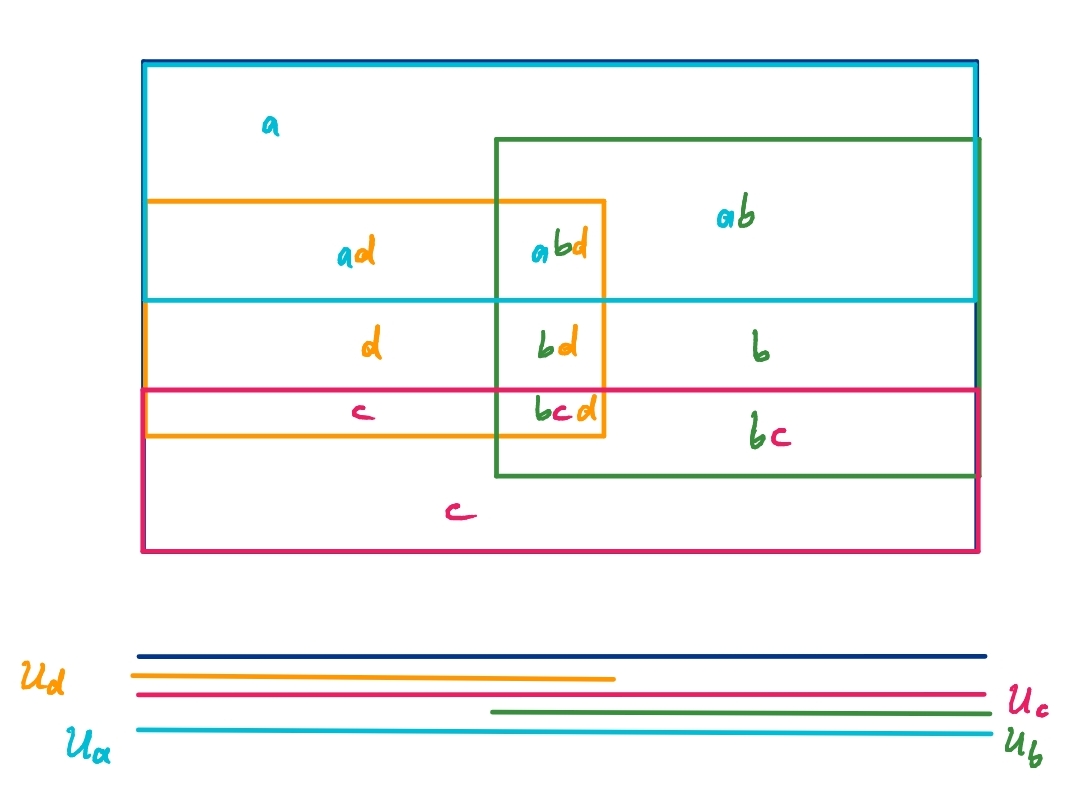}
\end{center}
Here each rectangle of single boundary colour in the top space $c \times c' \cong \RR^2$ is a patch in the good open covering $\hat{\cU} = \{\hat{U}_a, \hat{U}_b, \hat{U}_c, \hat{U}_d\}$ of $c \times c'$, and we have displayed in each rectangle which arises as an intersection of the patches all indices of the patches which enter in the intersection.
In particular, this will allow us to describe smooth families of geometric data on $M$ without having to fix a good open covering of $c \times M$ for each family first.
\qen
\end{remark}

\begin{definition}
Let $\scC$ be a category.
A \textit{Grothendieck coverage} on $\scC$ is a collection $\tau$ of \textit{coverings}, for each object $c \in \scC$; a covering is a family $\cU = \{f_a \colon c_a \to c\}_{a \in \Lambda}$ of morphisms with codomain $c$.
These coverings satisfy the following property:
if $\cU$ is a covering of $c \in \scC$ and $h \in \scC(d,c)$ is any morphism, then there exists a covering $\cV = \{g_b \colon d_b \to d\}_{b \in \Xi} \in \tau$ of $d$ and a map $\lambda \colon \Xi \to \Lambda$ such that, for each $b \in \Xi$, there is a commutative diagram in $\scC$:
\begin{equation}
\begin{tikzcd}
	d_b \ar[r, "h_b"] \ar[d, "g_b"']
	& c_{\lambda(b)} \ar[d, "f_{\lambda(b)}"]
	\\
	d \ar[r, "h"']
	& c
\end{tikzcd}
\end{equation}
\end{definition}

\begin{remark}
The good open coverings endow the category $\Cart$ with a Grothendieck coverage (see~\cite[Cor.~A.1]{FSS:Cech_for_diff_classes}).
\qen
\end{remark}

\begin{lemma}
\label{st:tau_rmfam is coverage}
The coverings on $\Cartfam$ define a Grothendieck coverage on $\Cartfam$, which we denote by $\tau_\rmfam$.
\end{lemma}

\begin{proof}
Consider a covering $(\hat{\cU} \to \cU)$ of $(\hat{d} \to d)$ and a morphism $(\hat{f},f) \colon (\hat{c} \to c) \longrightarrow (\hat{d} \to d)$.
Since $\hat{c} \to c$ is a trivialisable bundle, we find an isomorphism $(\varphi, 1_c) \colon (c \times \tilde{c} \to c) \longrightarrow (\hat{c} \to c)$, where the domain object is the trivial $\tilde{c}$-bundle over $c$.
The open covering $\hat{\cU}$ of $\hat{d}$ pulls back to an open covering $(\hat{f} \varphi)^{-1}(\hat{\cU})$ of $c \times \tilde{c}$.
By the properties of the product topology, this admits a refinement of the form $\hat{\cV}' = \{V'_b \times \tilde{V}'_b\}_{b \in \Xi}$, where $V'_b$ and $\tilde{V}'_b$ are open subsets of $c$ and $\tilde{c}$, respectively.
Projecting onto each factor, we obtain open coverings $\cV' = \{V'_b\}_{b \in \Xi}$ of $c$ and $\tilde{\cV}' = \{ \tilde{V}'_b\}_{b \in \Xi}$ of $\tilde{c}$.
We find good refinements $\cW = \{W_i\}_{i \in I}$ and $\tilde{\cW} = \{W_j\}_{j \in J}$ of these coverings (see the proof of~\cite[Prop.~A.1]{FSS:Cech_for_diff_classes}).
Then, the families $\hat{\cW} = \{\varphi(W_i \times \tilde{W}_j)\}_{i \in I, j \in J}$ and $\cW$ fit together to give a covering $(\hat{\cW} \to \cW)$ of $(\hat{c} \to c)$ with the desired properties.
\end{proof}

%%%%%%%%%%%%%%%%%%%%%%%%%%%%%%%%%%%%%%%%%%%%%%%%%%%%%%%%%%%%%%%%%%%%%%%%%%%%

\subsection{Smooth families of diffeomorphisms of $M$}

%%%%%%%%%%%%%%%%%%%%%%%%%%%%%%%%%%%%%%%%%%%%%%%%%%%%%%%%%%%%%%%%%%%%%%%%%%%%

Next, we include the action of the diffeomorphisms group $\Diff(M)$ of $M$.
We view $\Diff(M)$ as a group object in $\Fun(\Cart^\opp, \Set)$:
it assigns to each $c \in \Cart$ the set of all smooth maps $\varphi^\dashv \colon c \times M \to c \times M$ which commute with the projection to $c$, restrict to a diffeomorphism of $M$ at each $x \in c$, and which admit an inverse map which also has these properties.
Equivalently, we can describe the data $\varphi^\dashv$ as a smooth map $\varphi \colon c \to \Diff(M)$.
Let $\Gpd$ denote the 2-category of groupoids.
Given a group object $H \in \Fun(\Cart^\opp, \Set)$, we let $\rmB H \colon \Cart^\opp \to \Gpd$ denote the strict functor which assigns to $c \in \Cart$ the groupoid $\rmB(H(c))$ consisting of a single object, which has the group $H(c)$ as its automorphisms.

Fix a manifold $M$.
We introduce the shorthand notation
\begin{equation}
\begin{tikzcd}
	\rmB \scD \coloneqq \smallint \big( \rmB \Diff(M) \big) \ar[r, "\pi_M"]
	& \Cart
\end{tikzcd}
\end{equation}
for the the Grothendieck construction of $\rmB \Diff(M)$.
Concretely, the objects of $\rmB \scD$ are the same as those of $\Cart$, but a morphism $c \to d$ in $\rmB \scD$ is a pair $(f, \varphi)$ of a morphism $f \in \Cart(c,d)$ and a smooth map $\varphi \colon c \to \Diff(M)$; these data induce a commutative square in $\Mfd$,
\begin{equation}
\begin{tikzcd}
	c \times M \ar[r, "f \wr \varphi"] \ar[d]
	& d \times M \ar[d]
	\\
	c \ar[r, "f"']
	& d
\end{tikzcd}
\end{equation}
Here, for $x \in c$ and $z \in M$, we set
\begin{equation}
\label{eq:wr notation}
	(f \wr \varphi)(x,z) = \big( f(x), \varphi^\dashv(x,z) \big)\,.
\end{equation}
Applying the Grothendieck construction to the contravariant functor $\rmB \Diff(M) \colon \Cart^\opp \to \Gpd$ and taking nerves produces a \textit{right} fibration
\begin{equation}
	N\pi_M \colon N \rmB \scD \longrightarrow N \Cart\,.
\end{equation}

\begin{remark}
\label{rmk:BD(M)^opp and Diff(M)^rev}
In this article the opposite category $\rmB\scD^\opp$ plays an important role.
Its nerve is the domain of the left fibration $N\pi_M \colon N\rmB\scD^\opp \longrightarrow N\Cart^\opp$ (where we suppress the $(-)^\opp$ on the morphism for ease of notation; this will be clear from context).
We remark that, for each $c \in \Cart$, the fibre $(\rmB\scD^\opp)_{|c}$ is the delooping of the \textit{opposite group} $\Diff(M)^\rev$---i.e.~$\Diff(M)(c)$ with its multiplication reversed---rather than $\Diff(M)(c)$.
That is, there is a canonical isomorphism
\begin{equation}
	N\rmB\scD
	\cong r_{\Cart^\opp}^* \rmB \big( \Diff(M)^\rev \big)
\end{equation}
of simplicial sets over $N\Cart^\opp$, where $r_{\Cart^\opp}^* \colon \Fun(\Cart^\opp, \sSet) \longrightarrow \sSet_{/N \Cart^\opp}$ is the rectification functor for presheaves of simplicial sets from Definition~\ref{def:r_C^*}.
This reversal is indeed natural for us to consider since the $\Diff(M)$-actions we will encounter are via pullback of geometric structures; therefore, they naturally arise as \textit{right} actions, but we will often encounter them as left actions of $\Diff(M)^\rev$.
\qen
\end{remark}

Observe that there is a canonical functor
\begin{equation}
	e_M \colon \Cart \longrightarrow \rmB \scD\,,
\end{equation}
which is the identity on objects and sends a morphism $f \colon c \to d$ in $\Cart$ to the morphism $(f, 1_M)$ in $\rmB \scD$.
It further satisfies the identity
\begin{equation}
	\pi_M \circ e_M = 1_{\Cart}\,,
\end{equation}
i.e.~$e_M$ is a section of $\pi_M$.

Next, we associate to $M$ a category which encodes smooth families of good open coverings of $M$ and the action of $\Diff(M)$ on these coverings.

\begin{definition}
\label{def:GCov^D(M)}
We define a category $\GCov^\scD$, whose \textit{objects} are pairs $(c, \hat{\cU} \to \cU)$ of the following data: $c \in \Cart$ is a cartesian space, and $(\hat{\cU} \to \cU)$ is a family
\begin{equation}
	(\hat{\cU} \to \cU) = \{(\hat{U}_a \to U_a) \in \Cartfam\}_{a \in \Lambda}
\end{equation}
of objects $(\hat{U}_a \to U_a) \in \Cartfam$ satisfying the following properties:
\begin{enumerate}
\item $\hat{U}_a \subset c \times M$ is an open subset,

\item $U_a \subset c$ is an open subset,

\item for each $a \in \Lambda$, the canonical diagram
\begin{equation}
\begin{tikzcd}
	\hat{U}_a \ar[r, hookrightarrow] \ar[d]
	& c \times M \ar[d]
	\\
	U_a \ar[r, hookrightarrow]
	& c
\end{tikzcd}
\end{equation}
of manifolds and smooth maps commutes.

\item $\{\hat{U}_a\}_{a \in \Lambda}$ and $\{U_a\}_{a \in \Lambda}$ form good open coverings of $c \times M$ and $c$, respectively, and

\item each pair $(\hat{U}_{a_0 \cdots a_l} \to U_{a_0 \cdots a_l})$ of finite intersections is again an object in $\Cartfam$.
\end{enumerate}
A \textit{morphism} in $\GCov^\scD$ is defined as follows:
consider two objects $(c, \hat{\cU} \to \cU)$ and $(d, \hat{\cV} \to \cV)$ in $\GCov^\scD$, whose coverings read as $(\hat{\cU} \to \cU) = \{\hat{U}_a \to U_a\}_{a \in \Lambda}$ and $(\hat{\cV} \to \cV) = \{\hat{V}_b \to V_b\}_{b \in \Xi}$.
A morphism $(c, \hat{\cU} \to \cU) \longrightarrow (d, \hat{\cV} \to \cV)$ in $\GCov^\scD$ consists of a morphism $(f, \varphi) \colon c \to d$ in $\rmB \scD$ and a map $\lambda \colon \Lambda \to \Xi$ of indexing sets such that the restriction of $(f \wr \varphi \to f) \colon (c {\times} M \to c) \longrightarrow (d {\times} M \to d)$ to $(\hat{U}_a \to U_a)$ factors through $(\hat{V}_{\lambda(a)} \to V_{\lambda(a)})$, for each $a \in \Lambda$.
\end{definition}

\begin{example}
Let $\cW = \{W_i \hookrightarrow M\}_{i \in \Xi}$ be a good open covering of $M$, and let $\cU = \{U_a \hookrightarrow c\}_{a \in \Lambda}$ be a good open covering of $c \in \Cart$.
We obtain the product covering
\begin{equation}
	\hat{\cU} = \cW {\times} \cU = \{U_a {\times} W_i \hookrightarrow c \times M\}_{a \in \Lambda, i \in \Xi}\,.
\end{equation}
Then, the canonical projection $(\hat{\cU} \to \cU)$ is an object in $\GCov^\scD$.
Let $\varphi \colon c \to \Diff(M)$ be a smooth family of diffeomorphisms and $f \colon c \to c$ a diffeomorphism.
Set $\hat{U}'_{ai} \coloneqq (f \wr \varphi)^{-1}(U_a \times W_i) \subset c \times M$ and $U'_a \coloneqq f^{-1}(U_a) \subset c$.
We obtain a new object $(\hat{\cU}' \to \cU') \in \GCov^\scD$, where $\hat{\cU}' = \{\hat{U}'_{ai}\}_{a \in \Lambda, i \in \Xi}$ and $\cU' = \{U'_a\}_{a \in \Lambda}$.
This comes with a canonical isomorphism
\begin{equation}
	(\lambda = \id, f \wr \varphi \to f) \colon (\hat{\cU}' \to \cU') \longrightarrow (\hat{\cU} \to \cU)
\end{equation}
in $\GCov^\scD$.
Note that the new covering $(\hat{\cU}' \to \cU')$ is, in general, no longer of product form; rather, it has been twisted by the diffeomorphism $f \wr \varphi \colon c {\times} M \to c {\times} M$.
\qen
\end{example}

\begin{remark}
We may depict morphisms in $\GCov^\scD$ as commutative cubes
\begin{equation}
\begin{tikzcd}[column sep={1.5cm,between origins}, row sep={1.cm,between origins}]
	\hat{\cU} \ar[dr] \ar[rr] \ar[dd]
	& & \hat{\cV} \ar[dr] \ar[dd]
	&
	\\
	& c \times M \ar[rr, "f \wr \varphi" {pos=0.25}, crossing over]
	& & d \times M \ar[dd]
	\\
	\cU \ar[rr] \ar[dr]
	& & \cV \ar[dr]
	&
	\\
	& c \ar[rr, "f"'] \ar[from=uu, crossing over]
	& & d
\end{tikzcd}
\end{equation}
leaving the map $\lambda$ implicit.
\qen
\end{remark}

\begin{remark}
\label{rmk:varpi and its fibres}
There is a canonical projection functor
\begin{equation}
\label{eq:varpi^scD}
	\varpi^\scD \colon \GCov^\scD \longrightarrow \rmB \scD\,,
\end{equation}
whose fibre $(\varpi^\scD)^{-1}(c) = \GCov^\scD_{|c}$ over $c \in \rmB\scD$ is the cofiltered category of good open coverings of $(c {\times} M \to c)$ (the cofilteredness is seen analogously to the proof of Lemma~\ref{st:tau_rmfam is coverage}).
\qen
\end{remark}

\begin{lemma}
\label{st:fibres and slices of GCov(M) --> BD(M)}
Let $c \in \rmB \scD$.
The inclusion $\jmath^\scD_c \colon \GCov^\scD_{|c} \hookrightarrow \GCov^\scD_{c/}$ of the fibre into the under-category is homotopy cofinal.
\end{lemma}

\begin{proof}
An object in $X \in \GCov^\scD_{/c}$ consists of a morphism $(f, \varphi) \colon c \to d$ in $\rmB\scD$, for some object $d \in \rmB \scD$, together with an object $(\hat{\cV} \to \cV)$ in $\GCov^\scD_{|d}$.
We have to show that, for each $X \in \GCov_{/c}$, the nerve
\begin{equation}
	N(\jmath_c/X)
	\cong N \big( \GCov^\scD_{|c} \big) \underset{N (\GCov^\scD_{c/})}{\times} N \big( (\GCov^\scD_{|c})_{/X} \big)
\end{equation}
is a weakly contractible simplicial set.
However, the comma category $\jmath_c/X$ can be described as the category of all those good open coverings of $(c {\times} M \to c)$ which refine the open covering $((f \wr \varphi)^{-1}(\hat{\cV}) \to f^{-1}(\cV))$ of $(c {\times} M \to c)$ (the latter is, in general, not differentiably good).
This category is cofiltered, so that its nerve is indeed contractible.
\end{proof}

\begin{lemma}
\label{st:varphi^scD is smooth}
The $\infty$-functor $N\varpi^\scD \colon N \GCov^\scD \longrightarrow N \rmB \scD$ satisfies the following:
\begin{enumerate}
\item It is an isofibration.

\item It is smooth (in the sense of~\cite[Defs.~4.4.1, 4.4.15]{Cisinski:HiC_and_HoA}).
\end{enumerate}
\end{lemma}

\begin{proof}
The nerve of any functor between (1-)categories is an inner fibration~\cite[Rmk.~3.3.15]{Cisinski:HiC_and_HoA}; thus, $N\varpi^\scD$ is an inner fibration.
To see that it is even an isofibration, let $(c_1, \hat{\cU} \to \cU)$ be an object in $\GCov^\scD$, where $\cU = \{U_i\}_{i \in \Lambda}$ and $\hat{\cU} = \{\hat{U}_i\}_{i \in \Lambda}$ (compare Definition~\ref{def:GCov^D(M)}).
Let $(f, \varphi) \colon c_0 \to c_1$ be an isomorphism in $\rmB\scD$; that is, $f \colon c_0 \to c_1$ is a diffeomorphism of cartesian spaces.
Let $V_i \coloneqq f^{-1}(U_i)$ and $\hat{V}_i \coloneqq (f \wr \varphi)^{-1}(\hat{U}_i)$, for each $i \in \Lambda$, and set $\cV = \{V_i\}_{i \in \Lambda}$ and $\hat{\cV} = \{\hat{V}_i\}_{i \in \Lambda}$.
Then, the triple $(f, \varphi, \lambda = 1_\Lambda)$ provides an isomorphism $(c_0, \hat{\cV} \to \cV) \to (c_0, \hat{\cU} \to \cU)$ in $\GCov^\scD$ which maps to $(f, \varphi)$ under $\varpi^\scD$.
This shows that $N\varpi^\scD$ is and isofibration.
The second claim now follows from Lemma~\ref{st:fibres and slices of GCov(M) --> BD(M)} and~\cite[Thm.~4.4.36]{Cisinski:HiC_and_HoA}.
\end{proof}

\begin{definition}
Let $\GCov \subset \GCov^\scD$ be the wide%
\footnote{A subcategory is called \textit{wide} if it contains all objects of its ambient category.}
subcategory whose morphisms are only those morphisms in $\GCov^\scD$ which are of the form $(f, \varphi) = (f,1_M)$.
Thus, $\GCov$ comes with a canonical projection functor $\varpi \colon \GCov \to \Cart$.
\end{definition}

\begin{remark}
The $\infty$-functor $N\varpi \colon N\GCov \to N\Cart$ is a smooth isofibration by the same arguments as in Lemma~\ref{st:varphi^scD is smooth}.
\qen
\end{remark}

We obtain a commutative diagram of categories
\begin{equation}
\begin{tikzcd}[column sep=1.25cm, row sep=1cm]
	\GCov \ar[r, hookrightarrow, "\hat{e}_M"] \ar[d, "\varpi"']
	& \GCov^\scD \ar[d, "\varpi^\scD"]
	\\
	\Cart \ar[r, shift left=0.1cm, "e_M"]
	& \rmB \scD \ar[l, shift left=0.1cm, "\pi_M"]
\end{tikzcd}
\end{equation}

By definition of the subcategory $\GCov \subset \GCov^\scD$, we have the following lemma:

\begin{lemma}
\label{st:GCov cartesian diagram}
The commutative square of simplicial sets
\begin{equation}
\begin{tikzcd}[column sep=1.25cm, row sep=1cm]
	N\GCov \ar[r, hookrightarrow, "N\hat{e}_M"] \ar[d, "N\varpi"']
	& N\GCov^\scD \ar[d, "N\varpi^\scD"]
	\\
	N\Cart \ar[r, "Ne_M"]
	& N\rmB \scD
\end{tikzcd}
\end{equation}
is cartesian.
\end{lemma}

\begin{lemma}
\label{st:cofinality on GCov slices}
For each $c \in \Cart$, the following statements hold true:
\begin{enumerate}
\item The inclusion $\jmath^\scD_c \colon \GCov_{|c} \hookrightarrow \GCov_{c/}$ is homotopy cofinal.

\item The canonical inclusion $\imath_c \colon \GCov_{/c} \hookrightarrow \GCov^\scD_{c/}$ is homotopy cofinal
\end{enumerate}
\end{lemma}

\begin{proof}
Claim 1 follows analogously to Lemma~\ref{st:fibres and slices of GCov(M) --> BD(M)}, noting that there is a coincidence of fibres
\begin{equation}
	\GCov_{|c} = \GCov^\scD_{|c}\,,
\end{equation}
for each $c \in \Cart$.
Claim 2 is then a combination of Claim 1, Lemma~\ref{st:fibres and slices of GCov(M) --> BD(M)} and~\cite[Cor.~4.1.9]{Cisinski:HiC_and_HoA}.
\end{proof}

%%%%%%%%%%%%%%%%%%%%%%%%%%%%%%%%%%%%%%%%%%%%%%%%%%%%%%%%%%%%%%%%%%%%%%%%%%%%

\section{Smooth families of higher geometric structures}
\label{sec:families of hgeo structures}

%%%%%%%%%%%%%%%%%%%%%%%%%%%%%%%%%%%%%%%%%%%%%%%%%%%%%%%%%%%%%%%%%%%%%%%%%%%%

In this section we present a functorial construction which produces simplicial presheaves on $\rmB\scD$ from simplicial homotopy sheaves on $\Cartfam$.
This allows us to build smooth families of global geometric data on $M$ from local models.

\begin{definition}
\label{def:scH and scH_rmfam}
We make the following definitions:
\begin{enumerate}
\item Let $\scH \coloneqq \Fun(\Cart^\opp, \sSet)$ denote the category of simplicial presheaves on $\Cart$.
We view this as endowed with the projective model structure.

\item Let $\scH^{loc}$ denote the left Bousfield localisation of $\scH$ at the \v{C}ech nerves of good open coverings of cartesian spaces.

\item Let $\scH_\rmfam \coloneqq \Fun(\Cartfam^\opp, \sSet)$ denote the category of simplicial presheaves on $\Cartfam$, endowed with the projective model structure.

\item By $\scH_\rmfam^{loc}$ we denote the left Bousfield localisation of $\scH_\rmfam$ at the \v{C}ech nerves of $\tau_\rmfam$-coverings (see Definition~\ref{def:coverings in Cart_fam}).
\end{enumerate}
\end{definition}

\begin{remark}
Each of the above model structures is simplicial.
We denote the simplicially enriched hom functor in a simplicial category $\scC$ by $\ul{\scC}(-,-) \colon \scC^\opp \times \scC \to \sSet$.
\qen
\end{remark}

There is a functor $\v{C} \colon \GCov^\scD \longrightarrow \scH_\rmfam$ defined as follows:
consider a good open covering $(\hat{\cU} \to \cU)$ of $(c \times M \to c)$, consisting of patches $(\hat{\cU} \to \cU) = \{(\hat{U}_a \to U_a)\}_{a \in \Lambda}$.
The functor $\v{C}$ sends this to the simplicial presheaf on $\Cartfam$ whose $l$-th level is
\begin{equation}
	\v{C}(c, \hat{\cU} \to \cU) = \coprod_{a_0 \cdots a_l \in \Lambda} h_{(\hat{U}_{a_0 \cdots a_l} \to U_{a_0 \cdots a_l})}\,,
\end{equation}
where $h$ denotes the Yoneda embedding of $\Cartfam$.
The action of $\v{C}$ on morphisms is canonically induced by the Yoneda embedding of $\Cartfam$.
Since $\v{C}(c, \hat{\cU} \to \cU)$ is levelwise a coproduct of representables, it is cofibrant in $\scH_\rmfam$.

\begin{remark}
Consider the presheaf \smash{$\ul{(c \times M \to c)}$} on $\Cartfam$ which sends an object $(\hat{d} \to d)$ to the set of commutative square of smooth maps
\begin{equation}
\begin{tikzcd}
	\hat{d} \ar[d] \ar[r]
	& c \times M \ar[d, "\pr_c"]
	\\
	d \ar[r]
	& c
\end{tikzcd}
\end{equation}
The above \v{C}ech nerve is even a cofibrant approximation in $\scH_\rmfam^{loc}$ of this presheaf:
this follows, for instance, from the fact that the canonical morphism
\begin{equation}
	\v{C}(c, \hat{\cU} \to \cU)_0 = \coprod_{a \in \Lambda} h_{(\hat{U}_a \to U_a)}
	\longrightarrow \ul{(c \times M \to c)}
\end{equation}
is a local epimorphism with respect to the coverage $\tau_\rmfam$ on $\Cartfam$ (see Section~\ref{sec:sites for families}).
Its \v{C}ech nerve is then a weak equivalence in the local model structure by~\cite[Cor.~A.3]{DHI:Hypercovers_and_sPSHs}).
\qen
\end{remark}

\begin{construction}
\label{cstr:sPShs on GCov^D from sPShs on Cart_fam}
There exists a functor
\begin{align}
\label{eq:sPShs on GCov^D from sPShs on Cart_fam}
	&\check{C}^{\scD*} \colon \scH_\rmfam^{loc} \longrightarrow \Fun \big( (\GCov^\scD)^\opp, \sSet \big)\,,
	\qquad
	G \longmapsto \check{C}^{\scD*} G\,,
	\\
	&\check{C}^{\scD*} G (c, \hat{\cU} \to \cU)
	\coloneqq \ul{\scH_\rmfam} \big( \v{C}(c, \hat{\cU} \to \cU), G \big)
\end{align}
Endowing $\Fun ((\GCov^\scD)^\opp, \sSet)$ with the projective model structure, this functor preserves fibrations, trivial fibrations, limits, as well as weak equivalences between fibrant objects.
In particular, it maps fibrant objects in $\scH_\rmfam^{loc}$ to fibrant objects of $\Fun ((\GCov^\scD)^\opp, \sSet)$.

Analogously, we define the functor
\begin{align}
	&\check{C}^{M*} \colon \scH_\rmfam^{loc} \longrightarrow \Fun \big( \GCov^\opp, \sSet \big)\,,
	\qquad
	G \longmapsto \check{C}^{M*} G\,,
	\\
	&\check{C}^{M*} G (c, \hat{\cU} \to \cU)
	\coloneqq \ul{\scH_\rmfam} \big( \v{C}(c, \hat{\cU} \to \cU), G \big)
\end{align}
This functor has the same properties as $\check{C}^{\scD*}$ above.
\qen
\end{construction}

\begin{lemma}
\label{st:wtG ess const on fibres of varpi^D}
Let $\check{C}^{\scD*} G$ be a simplicial presheaf on $\GCov^\scD$ obtained via Construction~\ref{cstr:sPShs on GCov^D from sPShs on Cart_fam} from a fibrant object $G \in \scH_\rmfam^{loc}$.
For each $c \in \rmB\scD$, the functor \smash{$\check{C}^{\scD*} G$} restricts to an essentially constant functor \smash{$(\GCov^\scD)_{|c}^\opp \longrightarrow \sSet$} on the fibre over $c \in \Cart$, i.e.~\smash{$\check{C}^{\scD*} G$} sends each morphism in \smash{$\GCov^\scD_{|c}$} to a weak equivalence.
The same applies to the functor $\check{C}^{M*}$.
\end{lemma}

\begin{proof}
This is a direct consequence of the fact that $G$ satisfies descent with respect to $\tau_\rmfam$.
\end{proof}

Our goal now is to rid ourselves of the dependence on choices of good open coverings.
We achieve this by averaging over all such coverings by means of the following left Kan extensions:

\begin{proposition}
\label{st:Lan_varpi}
The (1-categorical) left Kan extensions
\begin{align}
	\Lan_{\varpi^\scD} \colon \Fun \big( (\GCov^\scD)^\opp, \sSet \big) &\longrightarrow \Fun \big( \rmB\scD^\opp, \sSet \big)
	\quad \text{and}
	\\
	\Lan_\varpi \colon \Fun \big( \GCov^\opp, \sSet \big) &\longrightarrow \Fun \big( \Cart^\opp, \sSet \big)
\end{align}
preserve finite limits, as well as projective weak equivalences and fibrations (with respect to the Kan-Quillen model structure on $\sSet$).
In particular, they preserves projectively fibrant objects.
\end{proposition}

\begin{proof}
We present the proof of the statement for $\Lan_{\varpi^\scD}$; the proof for $\Lan_\varpi$ is entirely analogous.
Let $\wtG^\scD \colon (\GCov^\scD)^\opp \to \sSet$ be a functor.
By Lemmas~\ref{st:fibres and slices of GCov(M) --> BD(M)} and~\ref{st:cofinality on GCov slices}, for each $c \in \Cart$, the canonical morphism of simplicial sets
\begin{equation}
\label{eq:Lan comparison mp}
	\colim \big( \wtG^\scD \colon (\GCov^\scD_{/c})^\opp \to \sSet \big)
	\longrightarrow \colim \big( \wtG^\scD \colon (\GCov^\scD_{|c})^\opp \to \sSet \big)
\end{equation}
is an isomorphism.
The isomorphisms~\eqref{eq:Lan comparison mp} are not natural in $c$ since the fibres \smash{$\GCov^\scD_{|c}$} do not depend on $c$ functorially (the assignment \smash{$c \mapsto \GCov^\scD_{|c}$} is not a functor $\rmB \scD^\opp \to \Cat$).
For fixed $c \in \Cart$, they are, however, natural in the diagram \smash{$\wtG^\scD$}.
We thus obtain, for each $c \in \Cart$, a canonical isomorphism
\begin{align}
	\big( \Lan_{\varpi^\scD} \wtG^\scD \big) (c) &\cong \colim \Big( \wtG^\scD \colon (\GCov^\scD)_{|c}^\opp \to \sSet \Big)\,,
\end{align}
natural in \smash{$\wtG^\scD$}.
Since all the properties we need to check are objectwise and stable under isomorphism, it suffices to show that, for each $c \in \Cart$, the functor
\begin{equation}
	\Fun \big( (\GCov^\scD)^\opp, \sSet \big) \longrightarrow \sSet\,,
	\qquad
	\wtG^\scD \longmapsto \colim \big( \wtG^\scD \colon (\GCov^\scD_{|c})^\opp \to \sSet \big)
\end{equation}
preserves finite limits, fibrations and weak equivalences.

The indexing category \smash{$\GCov^\scD_{|c}$} is cofiltered (see Remark~\ref{rmk:varpi and its fibres}), so that the above colimit commutes with finite limits.
In particular, it follows that the colimit of a constant diagram \smash{$(\GCov^\scD)_{|c}^\opp \to \sSet$} with value $K \in \sSet$ is canonically isomorphic to $K$.
Thus, $\Lan_{\varpi^\scD}$ preserves final objects.

Next, since Kan fibrations are stable under filtered colimits~\cite[Lemma~3.1.24]{Cisinski:HiC_and_HoA} and closed under isomorphism, it follows that \smash{$\Lan_{\varpi^\scD}$} preserves Kan fibrations.
As $\Lan_{\varpi^\scD}$ also preserves final objects, we obtain that it preserves projectively fibrant objects.

Finally, again because $(\GCov^\scD)_{|c}^\opp$ is filtered, it follows from~\cite[Prop.~7.3]{Dugger:Combinatorial_Mocats} that the above colimit is, in fact, a homotopy colimit.
Consequently, the left Kan extension $\Lan_{\varpi^\scD}$ is even a homotopy left Kan extension and thus preserves objectwise weak equivalences.
\end{proof}

The two left Kan extensions in Proposition~\ref{st:Lan_(varpi^D) wtG as a localisation} are related by a base change formula:

\begin{lemma}
\label{st:base change for 1Cat Lan_varpi}
There is a canonical natural isomorphism
\begin{equation}
	e_M^* \circ \Lan_{\varpi^\scD}
	\cong 
	\Lan_\varpi \circ \hat{e}_M^*
\end{equation}
of functors $\Fun((\GCov^\scD)^\opp, \sSet) \longrightarrow \Fun(\Cart^\opp, \sSet)$.
\end{lemma}

\begin{proof}
For each $c \in \Cart$, the functor $e_M$ induces a functor $\GCov_{c/} \to (\GCov^\scD)_{c/}$; this is precisely the functor $\jmath^\scD_c$ considered in Lemma~\ref{st:cofinality on GCov slices}.
By that lemma, the functor $\jmath^\scD_c$ is homotopy cofinal and hence, in particular, cofinal.
\end{proof}

We also have the following $\infty$-categorical version of Lemma~\ref{st:base change for 1Cat Lan_varpi}.
Let $\scS$ denote the $\infty$-category of spaces.
For $\scC$ an $\infty$-category, we write $\scP(\scC) \coloneqq \scFun(\scC^\opp, \scS)$ for the $\infty$-category of $\infty$-presheaves on $\scC$.

\begin{proposition}
\label{st:proper base change for wtG^D}
Let \smash{$\widetilde{\bbG}^\scD \in \scP(N\GCov^\scD)$} be \textit{any} $\infty$-presheaf on \smash{$N\GCov^\scD$}.
The canonical base change morphism
\begin{equation}
	N \varpi_! \circ N\hat{e}_M^* \widetilde{\bbG}^\scD \longrightarrow
	N e_M^* \circ N\varpi^\scD_! \widetilde{\bbG}^\scD
\end{equation}
is an equivalence in $\scP(N\GCov^\scD)$.
\end{proposition}

\begin{proof}
This follows from Lemmas~\ref{st:varphi^scD is smooth} and~\ref{st:GCov cartesian diagram}, together with the Proper Base Change Theorem~\cite[Prop.~6.4.3, Thm.~6.4.13]{Cisinski:HiC_and_HoA}.
\end{proof}

We finish this section by summarising our main constructions so far:

\begin{theorem}
\label{st:fctrs from H_fam to sPSh(rmBscD)}
We have the following functors:
\begin{enumerate}
\item The functor
\begin{equation}
	\check{C}^{\scD*} \colon \scH_\rmfam^{loc} \longrightarrow \Fun \big( (\GCov^\scD)^\opp, \sSet \big)\,,
	\qquad
	G \longmapsto \check{C}^{\scD*} G = \ul{\scH}_\rmfam \big( \v{C}(-), G \big)\,,
\end{equation}
which preserves limits, fibrations, trivial fibrations and weak equivalences between fibrant objects.
Further, it maps fibrant objects to functors which send refinements of coverings to weak equivalences.

\item The functor
\begin{equation}
	\Lan_{\varpi^\scD} \colon \Fun \big( (\GCov^\scD)^\opp, \sSet \big) \longrightarrow \Fun(\rmB\scD^\opp, \sSet)\,,
	\qquad
	\wtG^\scD \longmapsto \Lan_{\varpi^\scD} \wtG^\scD\,,
\end{equation}
preserves colimits, finite limits, fibrations and weak equivalences between fibrant objects.

\item The composition of the above functors,
\begin{equation}
	\Lan_{\varpi^\scD} \circ \check{C}^{\scD*} \colon \scH_\rmfam^{loc} \longrightarrow \Fun(\rmB\scD^\opp, \sSet)\,,
	\qquad
	G \longmapsto \Lan_{\varpi^\scD} \check{C}^{\scD*} G\,,
\end{equation}
preserves finite limits and fibrations (hence also fibrant objects), as well as weak equivalences between fibrant objects.
\end{enumerate}
Analogous statements hold true for the functors $\check{C}^{M*}$ and $\Lan_\varpi$.
\end{theorem}

%%%%%%%%%%%%%%%%%%%%%%%%%%%%%%%%%%%%%%%%%%%%%%%%%%%%%%%%%%%%%%%%%%%%%%%%%%%%

\section{Averaging over open coverings as an $\infty$-categorical localisation}
\label{sec:averaging good open coverings}

%%%%%%%%%%%%%%%%%%%%%%%%%%%%%%%%%%%%%%%%%%%%%%%%%%%%%%%%%%%%%%%%%%%%%%%%%%%%

In Proposition~\ref{st:Lan_varpi} we saw that averaging over open coverings in $\GCov^\scD$ has nice technical properties.
Here we provide a deeper interpretation of this left Kan extension:
in Theorem~\ref{st:Lan_(varpi^D) wtG as a localisation} we show that on associated left fibrations it amounts precisely to an $\infty$-categorical localisation at the morphisms arising from refinements of good open coverings.
This shows that the construction in Theorem~\ref{st:fctrs from H_fam to sPSh(rmBscD)} is indeed a presentation of the correct way to pass from local presentations of geometric data to families of global objects on $M$.
While that is conceptually important and gives credence to our formalism, the reader interested mainly in applications may skip this section.

\begin{proposition}
\label{st:good InnFibs are localisations}
\emph{\cite[Prop.~7.1.12]{Cisinski:HiC_and_HoA}}
Let $q \colon C \to D$ be a morphism of $\infty$-categories.
Suppose that $q$ has the following properties:
\begin{enumerate}
\item it is an inner fibration,

\item it is smooth or proper, and

\item for each vertex $d \in D$, the fibre $C_{|d} = q^{-1}(d) \subset C$ is a weakly contractible simplicial set.
\end{enumerate}
Then, for each cartesian square
\begin{equation}
\begin{tikzcd}
	A \ar[r, "u"] \ar[d, "q'"']
	& C \ar[d, "q"]
	\\
	B \ar[r, "v"']
	& D
\end{tikzcd}
\end{equation}
of $\infty$-categories, the $\infty$-functor $q'$ exhibits $\scB$ as the $\infty$-categorical localisation of $\scA$ at those morphisms in $\scA$ which become identities under $q'$.
\end{proposition}

\begin{corollary}
\label{st:good InnFibs are localisations; id case}
Let $q \colon C \to D$ be a morphism of $\infty$-categories.
Suppose $q$ is an inner fibration which is also smooth or proper  and all of whose fibres are weakly contractible.
Then, $q$ exhibits $D$ as the $\infty$-categorical localisation of $C$ at those morphisms in $C$ whose image under $q$ is an identity morphism in $D$.
\end{corollary}

\begin{proof}
This is the Proposition~\ref{st:good InnFibs are localisations} in the case where $u$ and $v$ are identities.
\end{proof}

In the situation of Corollary~\ref{st:good InnFibs are localisations; id case}, let $W \subset C_1$ be the class of edges which become identity morphisms under $q$.
Let $\overline{W} \subset C_1$ denote the saturation of $W$ in the sense of~\cite[Rmk.~7.1.5]{Cisinski:HiC_and_HoA}, i.e.~the class of edges $f$ in $C$ such that $q(f)$ is an equivalence in the $\infty$-category $D$.
Then, $q$ equivalently exhibits $D$ as the localisation of $C$ at $\overline{W}$.

\begin{corollary}
\label{st:varpi(^scD) as a localisation functor}
The following statements hold true:
\begin{enumerate}
\item The morphism of $\infty$-categories
\begin{equation}
	N \varpi^\scD \colon N \GCov^\scD \longrightarrow N\rmB \scD
\end{equation}
is an $\infty$-categorical localisation at the class $W^\scD$ of refinements of $\tau_\rmfam$-coverings, or, equivalently at its saturation $\overline{W}{}^\scD$; this is the class of morphisms $((\hat{f}, f), \varphi)$ in $\GCov^\scD$ whose image in $\rmB\scD$ is an isomorphism.
This is equivalent to $f$ being a diffeomorphism of cartesian spaces.

\item The morphism of $\infty$-categories
\begin{equation}
	N \varpi \colon N \GCov \longrightarrow N\Cart
\end{equation}
is an $\infty$-categorical localisation at the class $W$ of refinements of $\tau_\rmfam$-coverings, or, equivalently at its saturation $\overline{W}$; this is the class of morphisms in $\GCov$ whose image in $\Cart$ is a diffeomorphism of cartesian spaces.
\end{enumerate}
\end{corollary}

\begin{proof}
The functor $N\varpi^\scD$ is a smooth inner fibration by Lemma~\ref{st:varphi^scD is smooth}.
Remark~\ref{rmk:varpi and its fibres} implies that it also has weakly contractible fibres.
\end{proof}

\begin{lemma}
\label{st:locs and special cartesian squares}
Consider a cartesian diagram of $\infty$-categories
\begin{equation}
\begin{tikzcd}
	A \ar[r, "u"] \ar[d, "p"']
	& C \ar[d, "q"]
	\\
	B \ar[r, "v"']
	& D
\end{tikzcd}
\end{equation}
Suppose that $q$ has the properties in Proposition~\ref{st:good InnFibs are localisations} and that, additionally, $v$ is conservative (i.e.~reflects equivalences).
Let $\overline{W} \subset C_1$ denote the saturation of the class of morphisms at which $q$ localises $C$ (compare Corollary~\ref{st:good InnFibs are localisations; id case}).
Then, $p$ exhibits $B$ as the $\infty$-categorical localisation of $A$ at those edges in $A$ whose image under $u$ lies in $\overline{W}$.
That is,
\begin{equation}
	B \simeq L_{u^{-1}(\overline{W})} A\,.
\end{equation}
\end{lemma}

\begin{proof}
By Proposition~\ref{st:good InnFibs are localisations} we have that $p$ exhibits $B$ as the localisation of $A$ at the class $V$ of morphisms whose image under $p$ are identities.
The saturation $\overline{V} \subset A_1$ of this class consists of those morphisms in $A$ which are mapped to equivalences in $B$ under $p$.
Thus,
\begin{equation}
	B \simeq L_{\overline{V}} A\,.
\end{equation}
We claim that an edge $f \in A_1$ lies in $\overline{V}$ precisely if $u(f) \in \overline{W}$.
Since $v$ is conservative by assumption, $\overline{V}$ is the class of edges in $A$ whose image under $v \circ p$ is an equivalence in $D \simeq L_{\overline{W}}C$.
That is, $f \in \overline{V}$ if and only if $vp(f) \in D_1$ is an equivalence.
By the commutativity of the diagram, $vp(f) = qu(f)$, and $qu(f)$ is an equivalence if and only if $u(f) \in \overline{W}$.
\end{proof}

\begin{lemma}
\label{st:cov weqs between LFibs are exactly cat weqs}
Let $S \in \sSet$ be a simplicial set, and consider a commutative triangle
\begin{equation}
\begin{tikzcd}
	A \ar[rr, "f"] \ar[dr, "p"']
	& & B \ar[dl, "q"]
	\\
	& S &
\end{tikzcd}
\end{equation}
Suppose that $p$ and $q$ are left fibrations.
Then, $f$ is a covariant weak equivalence in $\sSet_{/S}$ if and only if it is a categorical weak equivalence, i.e.~a weak equivalence in the Joyal model structure on $\sSet$.
\end{lemma}

\begin{proof}
Let $T$ denote a class of morphisms in a model category $\scC$ such that the left Bousfield localisation $L_T \scC$ exists.
Then, a morphism between fibrant objects in $L_T \scC$ is a weak equivalence precisely if it is a weak equivalence in the original model category $\scC$.
We apply this to the following setting:
the Joyal model structure on $\sSet$ induces a model structure on the slice category $\sSet_{/S}$ (see, for instance,~\cite[Thm.~7.6.5]{Hirschhorn:MoCats_and_localisations}).
The covariant model structure on $\sSet_{/S}$ is the left Bousfield localisation of this induced model structure at all morphisms~\cite[p.~5]{HM:Left_fibs_and_hocolims_I}
\begin{equation}
\begin{tikzcd}
	\Lambda^n_0 \ar[rr, hookrightarrow] \ar[dr]
	& & \Delta^n \ar[dl]
	\\
	& S &
\end{tikzcd}
\end{equation}
where $n$ ranges over $\NN_0$ and, for each $n$, we consider all possible morphisms $\Delta^n \to S$.
This shows the claim.
\end{proof}

We now apply the above statements to the left fibrations associated to the simplicial presheaves from Section~\ref{sec:families of hgeo structures}.
These left fibrations are computed using the rectification functor (see Appendix~\ref{app:rectification}, as well as~\cite{HM:Left_fibs_and_hocolims_I, Bunk:Localisation_of_sSet} for more background):
given a category $\scC$, rectification is a functor
\begin{equation}
	r_\scC^* \colon \Fun(\scC, \sSet) \longrightarrow \sSet_{/N\scC}\,.
\end{equation}
Endowing $\Fun(\scC, \sSet)$ with the projective model structure and $\sSet_{/N\scC}$ with the covariant model structure, $r_\scC^*$ is even a right Quillen equivalence (see Theorem~\ref{st:r_C^* as Quillen equivalence}).
Given a simplicial presheaf $F^\scD \colon \rmB \scD^\opp \to \sSet$ or $F^M \colon \Cart^\opp \to \sSet$, we write
\begin{equation}
\label{eq:base change on left fibs tint F^scD}
	\textint F^\scD \coloneqq r_{\rmB\scD^\opp}^* F^\scD
	\qquad \text{and} \qquad
	\textint F^\scD \coloneqq r_{\Cart^\opp}^* F^M\,,
\end{equation}
respectively.
By Lemma~\ref{st:r_C^* and pullbacks} there is a canonical isomorphism
\begin{equation}
	Ne_M^* \textint F^\scD
	\cong \textint e_M^*F^\scD
\end{equation}
in $\sSet_{/N \Cart^\opp}$.

For a category $\scC$, we denote the $\infty$-categorical localisation of $\Fun(\scC, \sSet)$ at the objectwise weak homotopy equivalences by
\begin{equation}
	\gamma_\scC \colon N\Fun(\scC, \sSet) \longrightarrow \scFun(N\scC, \scS)\,.
\end{equation}
Recall further that $r_\scC^*$ produces the left fibrations classified by the $\infty$-presheaves obtained via the localisation functor $\gamma_\scC$ (Theorem~\ref{st:r_C^* and oo-categorical localisation}).
In other words, for $F \colon \scC \to \sSet$, the $\infty$-functor $\gamma_\scC F \colon N\scC \to \scS$ classifies the left fibration $r_\scC^*F \to N\scC$.
Given this insight, we see that the isomorphism~\eqref{eq:base change on left fibs tint F^scD} is a manifestation on the level of left fibrations of the base change equivalence in Proposition~\ref{st:proper base change for wtG^D}.

Recall the projection functor $\varpi^\scD \colon \GCov^\scD \to \rmB\scD$ from~\eqref{eq:varpi^scD}.

\begin{lemma}
\label{st:wtG^D sim varpi^*Lan_(w^D) wtG^D}
Let \smash{$\wtG^\scD \colon (\GCov^\scD)^\opp \to \sSet$} be a simplicial presheaf such that, for each $c \in \rmB\scD$ the restriction of \smash{$\wtG^\scD$} to the fibre \smash{$\GCov^\scD_{|c} = (\varpi^\scD){}^{-1}(c)$} to the fibre at $c$ is an essentially constant functor (i.e.~it sends each morphism in the domain to a weak equivalence in its codomain).
Then, there is a canonical covariant equivalence
\begin{equation}
	\textint \wtG^\scD
	\longrightarrow (N\varpi^\scD)^* \textint \Lan_{\varpi^\scD} \wtG^\scD
\end{equation}
of left fibrations over $(N\GCov^\scD)^\opp$.
\end{lemma}

\begin{proof}
Given an object $X = (c, \hat{\cU} \to \cU) \in \GCov^\scD$, consider the canonical cocone morphism
\begin{align}
	\eta_{|X} \colon \wtG^\scD(X)
	\longrightarrow &\colim \Big( \wtG^\scD \colon \big( \GCov^\scD_{/c} \big)^\opp \to \sSet \Big)
	\\
	&\cong \colim \Big( i_c^*\wtG^\scD \colon \big( \GCov^\scD_{|c} \big)^\opp \to \sSet \Big)\,.
\end{align}
The morphisms $\eta_{|X}$ are the components of a natural transformation:
this is the unit of the adjunction $(\varpi^\scD)^* \dashv \Lan_{\varpi^\scD}$.
The isomorphism in the second line stems from Proposition~\ref{st:Lan_varpi}; it is, in general \textit{not} natural in $c$.
However, by Remark~\ref{rmk:varpi and its fibres} the colimit in the second line is filtered; together with the assumption that \smash{$i_c^*\wtG^\scD$} is essentially constant, this shows that the composition of $\eta_{|X}$ with the above isomorphism is a weak equivalence, for each $X \in \GCov^\scD$.
It follows by the two-out-of-three property that $\eta$ is a weak equivalence between fibrant objects in the projective model structure on simplicial presheaves on $\GCov^\scD$.
By Theorem~\ref{st:r_C^* as Quillen equivalence}, we thus obtain a covariant weak equivalence
\begin{equation}
	\textint \wtG^\scD
	\longrightarrow r_{(\GCov^\scD)^\opp}^* \circ (\varpi^\scD)^* \circ \Lan_{\varpi^\scD} \wtG^\scD
\end{equation}
between fibrant objects in $\sSet_{/N(\GCov^\scD)^\opp}$.
Finally, on the right-hand side we have a further canonical isomorphism
\begin{equation}
	r_{(\GCov^\scD)^\opp}^* \circ (\varpi^\scD)^* \circ \Lan_{\varpi^\scD} \wtG^\scD
	\cong N(\varpi^\scD)^* \circ r_{\rmB\scD^\opp}^* \circ \Lan_{\varpi^\scD} \wtG^\scD
	= N(\varpi^\scD)^* \textint \Lan_{\varpi^\scD} \wtG^\scD\,,
\end{equation}
induced by Lemma~\ref{st:r_C^* and pullbacks}.
\end{proof}

\begin{theorem}
\label{st:Lan_(varpi^D) wtG as a localisation}
Let \smash{$\wtG^\scD \colon (\GCov^\scD)^\opp \to \sSet$} be a simplicial presheaf.
Suppose that \smash{$\wtG^\scD$} is projectively fibrant and essentially constant along the fibres of $\varpi^\scD \colon \GCov^\scD \to \rmB\scD$ (as in the assumptions of Lemma~\ref{st:wtG^D sim varpi^*Lan_(w^D) wtG^D}).
There is a commutative diagram
\begin{equation}
\begin{tikzcd}[row sep=0.75cm, column sep=1cm]
	\textint \wtG^\scD \ar[r, "v"] \ar[d, "q"']
	& N(\GCov^\scD)^\opp \ar[d, "N\varpi^\scD"]
	\\
	\textint \Lan_{\varpi^\scD} \wtG^\scD \ar[r, "u"']
	& N\rmB\scD^\opp
\end{tikzcd}
\end{equation}
in $\sSet$, and the following statements hold true:
\begin{enumerate}
\item This diagram is homotopy cartesian in the Joyal model structure on $\sSet$.

\item The horizontal morphisms are left fibrations.

\item The morphism $q$ exhibits \smash{$\textint \Lan_{\varpi^\scD} \wtG^\scD$} as the $\infty$-categorical localisation of \smash{$\textint \wtG^\scD$} at those morphisms whose image under $u$ is in the class $\overline{W}{}^\scD$ from Corollary~\ref{st:varpi(^scD) as a localisation functor}(1).
\end{enumerate}
An analogous statement holds true for the map \smash{$\textint \wtG^M \to \textint \Lan_\varpi \wtG^M$} for any projectively fibrant simplicial presheaf \smash{$\wtG^M \colon \GCov^\opp \to \sSet$} which is essentially constant along the fibres of the functor $\varpi \colon \GCov \to \Cart$ and the class $\overline{W}$ from Corollary~\ref{st:varpi(^scD) as a localisation functor}(2).
\end{theorem}

Note that any projectively fibrant simplicial presheaf \smash{$\wtG^\scD$} on $\GCov^\scD$ gives rise to a projectively fibrant simplicial presheaf \smash{$\wtG^M \coloneqq \iota^*\wtG^\scD$} on $\GCov$.
Moreover, if the former is essentially constant on the fibres of $\varpi^\scD$, then the latter is essentially constant along the fibres of $\varpi$.

Theorem~\ref{st:Lan_(varpi^D) wtG as a localisation} shows that the $\infty$-category \smash{$\textint \Lan_{\varpi^\scD} \wtG^\scD$} arises by starting from the $\infty$-category \smash{$\textint \wtG^\scD$} and inverting (in the $\infty$-categorical sense) those of its morphisms which stem from reparameterisations (i.e.~isomorphisms in $\Cart$), smooth families of diffeomorphisms of $M$ and refinements of $\tau_\rmfam$-coverings.
In the case of \smash{$\wtG^M$} we only invert morphisms coming from reparameterisations and refinements of $\tau_\rmfam$-coverings.

\begin{proof}[Proof of Theorem~\ref{st:Lan_(varpi^D) wtG as a localisation}]
The morphisms $q$ and $v$ are induced by the morphism in Lemma~\ref{st:wtG^D sim varpi^*Lan_(w^D) wtG^D} and the universal property of pullbacks; this produces the commutative diagram.
Lemmas~\ref{st:cov weqs between LFibs are exactly cat weqs} and~\ref{st:wtG^D sim varpi^*Lan_(w^D) wtG^D} imply that, in order to show Claim~(1), it suffices to show that the strict pullback
\begin{equation}
\label{eq:Nvarpi^* tint wtG^D-square}
\begin{tikzcd}[row sep=1cm, column sep=1.25cm]
	(N\varpi^\scD)^* \big( \textint \Lan_{\varpi^\scD} \wtG^\scD \big) \ar[r, "v'"] \ar[d, "q'"']
	& N(\GCov^\scD)^\opp \ar[d, "N\varpi^\scD"]
	\\
	\textint \Lan_{\varpi^\scD} \wtG^\scD \ar[r, "u"']
	& N\rmB\scD^\opp
\end{tikzcd}
\end{equation}
is homotopy cartesian in the Joyal model structure.
This follows because its top right, bottom left and bottom right vertices are $\infty$-categories and $N\omega^\scD$ is an isofibration between $\infty$-categories (see Lemma~\ref{st:varphi^scD is smooth}); therefore, it is a fibration in the Joyal model structure~\cite[Thm.~3.6.1]{Cisinski:HiC_and_HoA}.
Claim~(2) holds by Theorem~\ref{st:r_C^* as Quillen equivalence}.
Finally, we show Claim~(3):
first, note that by Claim~(2) and~\cite[Prop.~3.4.8]{Cisinski:HiC_and_HoA} it follows that $v'$ and $u$ are conservative isofibrations.
Thus, we can apply Lemma~\ref{st:locs and special cartesian squares} to obtain that the morphism $q'$ in~\eqref{eq:Nvarpi^* tint wtG^D-square} is an $\infty$-categorical localisation functor at those edges in \smash{$(N\varpi^\scD)^* \big( \textint \Lan_{\varpi^\scD} \wtG^\scD \big)$} whose image under $v'$ is in $\overline{W}{}^\scD$.
The claim now follows because Lemmas~\ref{st:cov weqs between LFibs are exactly cat weqs} and~\ref{st:wtG^D sim varpi^*Lan_(w^D) wtG^D} provide a weak categorical equivalence $\eta'$ which fits into the commutative diagram
\begin{equation}
\begin{tikzcd}
	\textint \wtG^\scD \ar[dr, "\eta'" description] \ar[drr, bend left=15, "v" description] \ar[ddr, bend right=30, "q" description]
	& &
	\\
	& (N \varpi^\scD)^* \textint \Lan_{\varpi^\scD} \wtG^\scD \ar[r, "v'"] \ar[d, "q'"']
	& N(\GCov^\scD)^\opp \ar[d, "N\varpi^\scD"]
	\\
	& \textint \Lan_{\varpi^\scD} \wtG^\scD \ar[r, "u"']
	& N\Cart^\opp
\end{tikzcd}
\end{equation}
The claim then follows since localisations are compatible with weak categorical equivalences~\cite[Prop.~7.1.8]{Cisinski:HiC_and_HoA}.
The proof for \smash{$\textint \wtG^M \to \textint \Lan_\varpi \wtG^M$} and the class $\overline{W}$ is analogous.
\end{proof}

%%%%%%%%%%%%%%%%%%%%%%%%%%%%%%%%%%%%%%%%%%%%%%%%%%%%%%%%%%%%%%%%%%%%%%%%%%%%

\section{Diffeomorphisms preserving the equivalence class of a geometric structure}
\label{sec:restricting to scD[cG]}

%%%%%%%%%%%%%%%%%%%%%%%%%%%%%%%%%%%%%%%%%%%%%%%%%%%%%%%%%%%%%%%%%%%%%%%%%%%%

We now specialise to diffeomorphisms which preserve the equivalence class of a given geometric structure $\cG$ on $M$.
They are precisely those diffeomorphisms of $M$ which admit a lift to $\cG$, and therefore may appear in symmetries of $\cG$.

Let \smash{$G^\scD$} be a projectively fibrant simplicial presheaf on $\rmB\scD$, and let $\cG \in G^M(\RR^0)$ be a section of \smash{$G^\scD = e_M^*G^\scD$} over $\RR^0$.
Equivalently, since $\RR^0 \in \Cart$ is a final object, a vertex in \smash{$G^M(\RR^0)$} is a morphism of simplicial presheaves
\begin{equation}
	\cG \colon \Delta^0 \longrightarrow G^M = e_M^* G^\scD\,,
\end{equation}
i.e.~a global section of $\G^M$ over $\Cart^\opp$.
Explicitly, this section reads as
\begin{equation}
	c \mapsto \coll_c^*\cG\,,
\end{equation}
where $\coll_c \colon c \to \RR^0$ is the collapse morphism in $\Cart$.
However, we will usually suppress the collapse morphisms and simply denote the section by $\cG$, as above.

\begin{definition}
\label{def:Diff_[cG] and D_[cG]}
Given a section $\cG \in G^M(\RR^0)$, let
\begin{equation}
	\Diff_{[\cG]}(M) \subset \Diff(M)
\end{equation}
denote the subpresheaf of groups on $\Cart$ defined by
\begin{equation}
	\Diff_{[\cG]}(M)(c)
	= \big\{ \varphi \in \Diff(M)(c)\, \big| \, (1_c, \varphi)^*\cG \simeq \cG \text{ in } G^\scD(c) \big\}\,.
\end{equation}
Equivalently, the morphism
\begin{equation}
	G^\scD(1_c, \varphi) \colon G^\scD(c) \longrightarrow G^\scD(c)
\end{equation}
of simplicial sets preserves the connected component of $\cG \in (e_M^*G^\scD)(c) = G^\scD(c)$.
We let
\begin{equation}
	\rmB \scD[\cG] \subset \rmB \scD
\end{equation}
denote the Grothendieck construction of the delooping of $\Diff_{[\cG]}(M)$, and we denote the associated inclusion map by
\begin{equation}
	\jmath_{[\cG]} \colon \rmB \scD[\cG] \hookrightarrow \rmB \scD\,.
\end{equation}
\end{definition}

As it will always be clear whether we are using $\rmB\scD$ or its subcategory $\rmB\scD[\cG]$, we will still denote the canonical functors between $\rmB\scD[\cG]$ and $\Cart$ by
\begin{equation}
\begin{tikzcd}
	e_M : \Cart \ar[r, shift left=0.05cm]
	& \rmB\scD[\cG] : \pi_M\,. \ar[l, shift left=0.05cm]
\end{tikzcd}
\end{equation}
Observe that each of these functors acts as the identity at the level of objects.

Analogously, we obtain an inclusion of a wide%
\footnote{A subcategory is \textit{wide} if it contains all objects of its ambient category.}
subcategory
\begin{equation}
	\iota_{[\cG]} \colon \GCov^{\scD[\cG]} \hookrightarrow \GCov^\scD\,,
\end{equation}
where \smash{$\GCov^{\scD[\cG]}$} is defined as the strict pullback of categories
\begin{equation}
\label{eq:GCov^D_[G] defining pb square}
\begin{tikzcd}
	\GCov^{\scD[\cG]} \ar[r, "\iota_{[\cG]}"] \ar[d, "\varpi^\scD_{[\cG]}"']
	& \GCov^\scD \ar[d, "\varpi^\scD"]
	\\
	\rmB \scD[\cG] \ar[r, "\jmath_{[\cG]}"']
	& \rmB \scD
\end{tikzcd}
\end{equation}

\begin{definition}
Given a simplicial presheaf $G^\scD \in \Fun(\rmB\scD^\opp, \sSet)$, we denote its restriction to $\rmB\scD[\cG]$ by $G^{\scD[\cG]} \in \Fun(\rmB\scD[\cG]^\opp, \sSet)$.
We denote its associated left fibration by
\begin{equation}
	\textint G^{\scD[\cG]} \coloneqq r_{\rmB\scD[\cG]^\opp}^* G^{\scD[\cG]}
	\longrightarrow N\rmB\scD[\cG]^\opp\,.
\end{equation}
\end{definition}

Note that by Lemma~\ref{st:r_C^* and pullbacks} there is a canonical isomorphism
\begin{equation}
	\textint G^{\scD[\cG]} \cong \jmath_{[\cG]}^* \textint G^\scD\,.
\end{equation}
Observe that $\cG$ defines a section of $G^M$, but not of $G^{\scD[\cG]}$.
Nevertheless, since $\Cart$, $\rmB\scD$ and $\rmB\scD[\cG]$ each have the same objects, for each $c \in \rmB\scD[\cG]$ there is a canonical identification $G^M(c) = G^{\scD[\cG]}(c)$.
Even though this is \textit{not} a natural transformation, it specifies a connected component of $G^\scD(c)$.
Restricting to $\rmB\scD[\cG]$, these connected components define a simplicial subpresheaf of $G^{\scD[\cG]}$:

\begin{definition}
\label{def:G_[cG]}
Let \smash{$G^\scD$} be a projectively fibrant simplicial presheaf on $\rmB\scD$, and let $\cG \in G^M(\RR^0)$.
We make the following definitions:
\begin{enumerate}
\item We let \smash{$G^{\scD[\cG]}_{[\cG]} \subset G^{\scD[\cG]}$} denote the simplicial subpresheaf on $\rmB \scD[\cG]$ such that, for each $c \in \rmB\scD[\cG]$,
\begin{equation}
	G^{\scD[\cG]}_{[\cG]}(c) \subset G^{\scD[\cG]}(c) = G^M(c)
\end{equation}
is the full connected component of the vertex $\cG(c) \in G^M(c) = G^{\scD[\cG]}(c)$.
In particular, the object \smash{$G^{\scD[\cG]}_{[\cG]} \in \Fun(\rmB\scD[\cG]^\opp, \sSet)$} is again projectively fibrant if $G^\scD$ was so in $\Fun(\rmB\scD^\opp, \sSet)$.

\item We also define the associated left fibration over $N\rmB\scD[\cG]^\opp$,
\begin{equation}
	\big( \textint G^{\scD[\cG]}_{[\cG]} \big)
	\coloneqq r_{\rmB\scD[\cG]^\opp}^* G^{\scD[\cG]}_{[\cG]}
	\subset \textint G^{\scD[\cG]}\,.
\end{equation}

\item Recall that $G^M \coloneqq e_M^* G^\scD$.
Observing that $G^\scD(\RR^0) = G^M(\RR^0)$, we also define a simplicial subpresheaf $G^M_{[\cG]} \subset G^M$ on $\Cart$ by requiring that
\begin{equation}
	G^M_{[\cG]}(c) \subset G^M(c)
\end{equation}
is the connected component of $\cG \in G^M(c)$, for each $c \in \Cart$.

\item We again also define the associated left fibration over $N\Cart^\opp$,
\begin{equation}
	\big( \textint G^M_{[\cG]} \big)
	\coloneqq r_{\Cart^\opp}^* G^M_{[\cG]}
	\subset \textint G^M\,.
\end{equation}
\end{enumerate}
\end{definition}

With these definitions, we have the following immediate observation:

\begin{lemma}
\label{st:Ne_M^* tint G^D_G = tint G^M_G}
In the setting of Definition~\ref{def:G_[cG]} there is a canonical isomorphism
\begin{equation}
	(N e_M)^* \textint G^{\scD[\cG]}_{[\cG]}
	\cong \textint G^M_{[\cG]}
\end{equation}
of left fibrations over $N\Cart^\opp$.
\end{lemma}

\begin{proof}
We have canonical isomorphisms
\begin{align}
	(N e_M)^* \textint G^{\scD[\cG]}_{[\cG]}
	= (N e_M)^* r_{\rmB\scD[\cG]^\opp}^* G^{\scD[\cG]}_{[\cG]}
	\cong r_{\Cart^\opp}^* e_M^* G^{\scD[\cG]}_{[\cG]}
	\cong r_{\Cart^\opp}^* G^M_{[\cG]}\,,
\end{align}
where we have used Lemma~\ref{st:r_C^* and pullbacks}.
\end{proof}

\begin{remark}
\label{rmk:section associated to cG}
We could have equivalently defined \smash{$\textint G^{\scD[\cG]}_{[\cG]}$} as follows:
the section $\cG \in \G^\scD(\RR^0)$ defines a section $\tilde{\sigma}\cG$ of the left fibration $\textint G^M \to N\Cart^\opp$.
From this, we obtain a section $\sigma \cG$ of the left fibration $(N \pi_M)_! \textint G^\scD \longrightarrow N\Cart^\opp$ (note that $N \pi_M$ is a left fibration), i.e.~a commutative diagram of $\infty$-categories
\begin{equation}
\begin{tikzcd}[column sep=1.25cm, row sep=0.75cm]
	N \Cart^\opp \ar[r, "\tilde{\sigma} \cG" description] \ar[rr, bend left=25, "\sigma \cG" description] \ar[dr, equal]
	& \textint G^M \ar[r, hookrightarrow] \ar[d]
	& \textint G^{\scD[\cG]} \ar[dl]
	\\
	& N\Cart^\opp &
\end{tikzcd}
\end{equation}
Then, the simplicial subset \smash{$\textint G^M_{[\cG]} \subset \textint G^M$} is the essential image of the $\infty$-functor $\tilde{\sigma}\cG$.
Analogously, we could have defined the simplicial subset
\begin{equation}
	\big( \textint G^{\scD[\cG]}_{[\cG]} \big)
	\subset \textint G^\scD
\end{equation}
as the essential image of the $\infty$-functor $\sigma$.
Note that here it is crucial that we restrict to the subcategory $\rmB\scD[\cG] \subset \rmB\scD$.
\qen
\end{remark}

\begin{remark}
We point out that while the canonical morphism \smash{$\textint G^{\scD[\cG]}_{[\cG]} \longrightarrow N\rmB\scD[\cG]^\opp$} is a left fibration, the composite morphism \smash{$\textint G^{\scD[\cG]}_{[\cG]} \longrightarrow N\rmB\scD[\cG]^\opp \hookrightarrow N\rmB\scD[\cG]^\opp$} does, in general, no longer have this property.
\qen
\end{remark}

\begin{lemma}
\label{st:cofinality for slices with [cG]}
Let $G^\scD$ be a projectively fibrant simplicial presheaf on $\rmB\scD$, and let $\cG \in G^\scD(\RR^0)$.
The morphism
\begin{equation}
	(N\iota_{[\cG]})_{c/} \colon \big( N\GCov^{\scD[\cG]} \big)_{c/}
	\hookrightarrow \big( N\GCov^\scD \big)_{c/}
\end{equation}
is a cofinal morphism of simplicial sets, for each $c \in \Cart$.
\end{lemma}

\begin{proof}
For each $c \in \Cart$, there is an identity of fibres
\begin{equation}
	\big( \GCov^{\scD[\cG]} \big)_{|c}
	= \big( \GCov^\scD \big)_{|c}\,.
\end{equation}
Moreover, for each $c \in \Cart$, there is a commutative diagram of simplicial sets
\begin{equation}
\begin{tikzcd}
	\big( N\GCov^\scD \big)_{|c} \ar[r] \ar[d]
	& \big( N\GCov^{\scD[\cG]} \big)_{c/} \ar[r, "{(N\iota_{[\cG]})_{c/}}"] \ar[d, "{(N\varpi^\scD_{[\cG]})_{c/}}"']
	& \big( N\GCov^\scD \big)_{c/} \ar[d, "{(N\varpi^\scD)_{c/}}"]
	\\
	\Delta^0 \ar[r, "\{c\}"']
	& \big( N\rmB \scD[\cG] \big)_{c/} \ar[r, "{(N\jmath_{[\cG]})_{c/}}"']
	& \big( N\rmB \scD \big)_{c/}
\end{tikzcd}
\end{equation}
Each square in this diagram is a pullback diagram.
Since smooth morphisms of simplicial sets are stable under pullback, each vertical morphism in this diagram is a smooth morphism in $\sSet$ by Lemma~\ref{st:varphi^scD is smooth}.
Then, it follows (for instance from~\cite[Def.4.4.1]{Cisinski:HiC_and_HoA}) that the top left horizontal morphism is cofinal.
Since we already know (see Lemma~\ref{st:cofinality on GCov slices}) that the composition of the top horizontal arrows is cofinal, we infer that the morphism $(N\iota_{[\cG]})_{c/}$ is cofinal as well~\cite[Cor.~4.1.9]{Cisinski:HiC_and_HoA}.
\end{proof}

\begin{lemma}
\label{st:Ne_M^*G cong G^(D_[G])}
Let \smash{$G^\scD$} be a projectively fibrant simplicial presheaf on \smash{$\rmB\scD$}, and let $\cG$ be an element of \smash{$G^\scD(\RR^0)$}.
The following statements hold true:
\begin{enumerate}
\item The morphism $N\varpi^\scD_{[\cG]} \colon N\GCov^{\scD[\cG]} \longrightarrow N\rmB\scD[\cG]$ is smooth.

\item For each simplicial presheaf \smash{$\widetilde{F}^\scD$} on \smash{$\GCov^\scD$}, there is a canonical isomorphism of simplicial presheaves on $\rmB \scD[\cG]$
\begin{equation}
	\jmath_{[\cG]}^* \Lan_{\varpi^\scD} \widetilde{F}^\scD
	\cong \Lan_{\varpi^\scD_{[\cG]}} \imath_{[\cG]}^* \widetilde{F}^\scD\,.
\end{equation}

\item Let $\wtF^\scD$ be a projectively fibrant simplicial presheaf on $\GCov^\scD$.
Set \smash{$F^\scD \coloneqq \Lan_{\varpi^\scD} \wtF^\scD$} and \smash{$F^M \coloneqq \Lan_\varpi \imath^* \wtF^\scD$}.
Suppose that $\cG \in F^\scD(\RR^0)$.
There are canonical isomorphisms of simplicial presheaves on $\Cart$
\begin{equation}
	N e_M^* \big( F^{\scD[\cG]}_{[\cG]} \big)
	\cong N e_M^* \Big( \big( \Lan_{\varpi^\scD_{[\cG]}} \imath_{[\cG]}^*\wtF^\scD \big)_{[\cG]} \Big)
	\cong \big( \Lan_\varpi \imath^* \wtF^\scD \big)_{[\cG]}\,.
\end{equation}
\end{enumerate}
\end{lemma}

\begin{proof}
Claim~(1) holds true because smooth maps of simplicial sets are stable under pullbacks and the definition of $\GCov^{\scD[\cG]}$ as a pullback (see Equation~\eqref{eq:GCov^D_[G] defining pb square}).
Claim~(2) follows from Lemma~\ref{st:cofinality for slices with [cG]} (essentially, this is a presentation of the smooth base change for $\infty$-presheaves associated to the square~\eqref{eq:GCov^D_[G] defining pb square}).

To see Claim~(3), first recall that $e_M \colon \Cart \to \rmB\scD[\cG]$ is the identity on objects.
Consequently, taking the connected components of $\cG$ over each $c \in \Cart$ commutes with restricting along $e_M$; that is, there is a canonical isomorphism
\begin{equation}
	N e_M^* \Big( \big( \Lan_{\varpi^\scD_{[\cG]}} \imath_{[\cG]}^*\wtF^\scD \big)_{[\cG]} \Big)
	\cong \Big( N e_M^* \Big( \big( \Lan_{\varpi^\scD_{[\cG]}} \imath_{[\cG]}^*\wtF^\scD \big) \Big)_{[\cG]}\,.
\end{equation}
Next, consider the commutative diagram
\begin{equation}
\begin{tikzcd}[column sep=1.25cm, row sep=0.75cm]
	\GCov \ar[r, hookrightarrow, "\hat{e}_M"] \ar[d, "\varpi"']
	& \GCov^{\scD[\cG]} \ar[d, "\varpi^\scD_{[\cG]}"]
	\\
	\Cart \ar[r, "e_M"']
	& \rmB\scD[\cG]
\end{tikzcd}
\end{equation}
This is a strict pullback square of categories.
The associated base change morphism
\begin{equation}
	\Lan_\varpi \circ \hat{e}_M^*
	\longrightarrow	e_M^* \circ \Lan_{\varpi^\scD_{[\cG]}}
\end{equation}
is an isomorphism; this follows by the same argument as Lemma~\ref{st:base change for 1Cat Lan_varpi}, using that Lemmas~\ref{st:cofinality on GCov slices} and~\ref{st:fibres and slices of GCov(M) --> BD(M)} still apply with $\GCov^{\scD[\cG]}$ in place of $\GCov^\scD$.
Thus, we obtain a canonical isomorphism
\begin{equation}
	N e_M^* \big( \Lan_{\varpi^\scD_{[\cG]}} (\imath_{[\cG]}^*\wtF^\scD) \big)
	\cong \Lan_\varpi \imath'{}^* (\imath_{[\cG]}^* \wtF^\scD)\,.
\end{equation}
That completes the proof.
\end{proof}

\begin{remark}
Let $G^\scD$ be a projectively fibrant simplicial presheaf on $\rmB\scD$.
We emphasise that all definitions in this section only depend on the \textit{connected component} of $\cG$ in $\G^\scD(\RR^0)$, and not on the actual choice of $\cG$ itself.
That is, if $\cG \simeq \cG'$ in $G^M(\RR^0)$, then we have
\begin{alignat}{5}
	\Diff_{[\cG]}(M)
	&= \Diff_{[\cG']}(M)\,,
	& \qquad \quad
	\rmB\scD[\cG]
	&= \rmB\scD[\cG']\,,
	& \qquad \quad
	\GCov^{\scD[\cG]}
	&= \GCov^{\scD[\cG']}\,,
	\\
	G^{\scD[\cG]}
	&= G^{\scD[\cG']}\,,
	& \qquad \quad
	G^{\scD[\cG]}_{[\cG]}
	&= G^{\scD[\cG']}_{[\cG']}\,,
	\\
	\textint G^{\scD[\cG]}
	&= \textint G^{\scD[\cG']}\,,
	& \qquad \quad
	\textint G^{\scD[\cG]}_{[\cG]}
	&= \textint G^{\scD[\cG']}_{[\cG']}\,,
\end{alignat}
and similarly for the pullbacks of the above simplicial presheaves and left fibrations to the categories $\GCov$ and $\Cart$.
\qen
\end{remark}

%%%%%%%%%%%%%%%%%%%%%%%%%%%%%%%%%%%%%%%%%%%%%%%%%%%%%%%%%%%%%%%%%%%%%%%%%%%%

\section{Higher symmetry groups of higher geometric structures}
\label{sec:higher symmetry groups of hgeo strs}

%%%%%%%%%%%%%%%%%%%%%%%%%%%%%%%%%%%%%%%%%%%%%%%%%%%%%%%%%%%%%%%%%%%%%%%%%%%%

In this section we use the left fibrations associated to the simplicial presheaves considered so far to construct the smooth higher automorphism and symmetry groups of a fixed higher geometric structure $\cG$ on $M$.
More precisely, for each smooth action $\Phi \colon \bbGamma \to \Diff_{[\cG]}(M)$ of a smooth higher group $\bbGamma$ on $M$ we construct a smooth $\infty$-group of symmetries which lift the action of elements of $\bbGamma$ to $\cG$.
We show that these higher symmetry groups are extensions of $\bbGamma$ by the smooth higher automorphism group of $\cG$ (Theorem~\ref{st:Aut-Sym-Diff[cG] extension} and Corollary~\ref{st:Aut-Sym_phi-Gamma extension}) and that the smooth higher symmetry groups completely classify smooth equivariant structures on $\cG$ (Theorem~\ref{st:scSym represents Equivar} and Corollary~\ref{st:char of equivar structures}).

%%%%%%%%%%%%%%%%%%%%%%%%%%%%%%%%%%%%%%%%%%%%%%%%%%%%%%%%%%%%%%%%%%%%%%%%%%%%

\subsection{The smooth higher automorphism and symmetry groups}

%%%%%%%%%%%%%%%%%%%%%%%%%%%%%%%%%%%%%%%%%%%%%%%%%%%%%%%%%%%%%%%%%%%%%%%%%%%%

First, given a smooth higher group action $\Phi \colon \bbGamma \to \Diff_{[\cG]}(M)$, we define the smooth higher automorphism and symmetry groups of a fixed higher geometric structure $\cG$ on $M$.
We show that there is a short exact sequence of smooth $\infty$-groups, exhibiting the smooth higher symmetry group of $\cG$ as an extension of $\bbGamma$ by the smooth higher automorphism group of $\cG$.

\begin{definition}
\label{def:BGAU(G), BSYM(G)}
Let $G^\scD$ be a projectively fibrant simplicial presheaf on $\rmB\scD$, and let $\cG \in G^\scD(\RR^0)$.
Recall the sections $\tilde{\sigma} \cG$ and $\sigma \cG$ from Remark~\ref{rmk:section associated to cG}.
We define the following simplicial sets:
\begin{enumerate}
\item We denote by
\begin{equation}
	\rmB \AUT^\rev(\cG) \subset \textint G^M
\end{equation}
the full simplicial subset (i.e.~the full $\infty$-subcategory) on the vertices in the image of the section \smash{$\tilde{\sigma} \cG \colon N\Cart^\opp \to \textint G^M$} from Remark~\ref{rmk:section associated to cG}.

\item We denote by
\begin{equation}
	\rmB \SYM^\rev(\cG) \subset \textint G^{\scD[\cG]}
\end{equation}
the full simplicial subset (i.e.~the full $\infty$-subcategory) on the vertices in the image of the composition
\begin{equation}
\label{eq:section from (cG,cA^(k))}
\begin{tikzcd}
	\sigma \cG \colon N\Cart^\opp \ar[r, "\tilde{\sigma} \cG"]
	& \textint G^M \ar[r]
	& \textint G^{\scD[\cG]}\,.
\end{tikzcd}
\end{equation}
\end{enumerate}
\end{definition}

\begin{example}
The simplicial set $\rmB\SYM^\rev(\cG)$ has a unique vertex over each object $c \in \Cart$; it consists of the constant family $\tilde{\sigma}\cG(c) \in G^{\scD[\cG]}(c) = (\textint G^{\scD[\cG]})_{|c}$ (recall that $\tilde{\sigma}\cG(c) = \coll_c^*\cG$, where $\cG \in G^M(\RR^0)$ is the fixed chosen vertex and $\coll_c \colon c \to \RR^0$ is the canonical collapse morphism).
A 1-simplex in $\rmB\SYM^\rev(\cG)$ is a pair $((f, \varphi), \psi)$, consisting of a morphism $(f,\varphi) \colon c_0 \to c_1$ in $\rmB\scD[\cG]$ and an equivalence
\begin{equation}
	\psi \colon (f, \varphi)^* \big( \tilde{\sigma} \cG(c_1) \big) \longrightarrow \tilde{\sigma}\cG(c_0)
\end{equation}
in $G^\scD(c_0) = G^M(c_0)$.
Note that, in general, $(f, \varphi)^* (\tilde{\sigma} \cG(c_1)) \neq \tilde{\sigma} \cG(c_0)$, but as long as $(f,\varphi)$ is a morphism in $\rmB\scD[\cG] \subset \rmB\scD$, there is an equivalence in $G^\scD(c_0)$ as above, by definition of $\rmB\scD[\cG]$.
In particular, $\rmB\SYM^\rev(\cG)$ is not the rectification of a simplicial subpresheaf of $G^{\scD[\cG]}$.

Finally, the vertices in $\rmB\AUT^\rev(\cG)$ are the same as those in $\rmB\SYM^\rev(\cG)$, and morphisms are those morphisms in $\rmB\SYM^\rev(\cG)$ where $\varphi = \id_M$.
\qen
\end{example}

\begin{remark}
The superscript $\rev$ indicates that $\rmB\SYM^\rev(\cG)$ is naturally related to the \textit{opposite} of the group of symmetries of $\cG$; this is consistent with the fact that the codomain of the left fibration $\rmB\SYM^\rev(\cG) \to N\rmB\scD[\cG]^\opp$ describes the opposite of the diffeomorphism group $\Diff_{[\cG]}(M)$.
\qen
\end{remark}

\begin{lemma}
\label{st:limits in slices and contractibility}
Let $\scC$ be an $\infty$-category, let $c$ be an object of $\scC$, and let $F \colon K \to \scC_{/c}$ be a diagram in $\scC_{/c}$, indexed by some simplicial set $K$.
If $K$ is weakly contractible (i.e.~weakly equivalent to $\Delta^0$ in the Kan-Quillen model structure on $\sSet$), then the projection functor $\pr \colon \scC_{/c} \to \scC$ preserves and reflects limits.
That is, $\overline{F} \colon K^\triangleleft \to \scC_{/c}$ is a limit diagram if and only if $\pr \circ \overline{F} \colon K^\triangleleft \to \scC$ is so.
\end{lemma}

\begin{proof}
The limit of $F$ in $\scC_{/c}$ agrees with the limit of
\begin{equation}
	F * c \colon K^\triangleright = K * \Delta^0 \longrightarrow \scC\,,
\end{equation}
together with its canonical morphism to $c$, where $(-)*(-)$ denotes the join construction of simplicial sets.
The canonical inclusion $K \hookrightarrow K * \Delta^0$ is a cofinal morphism of simplicial sets if and only if $K$ is weakly contractible.
\end{proof}

\begin{remark}
\label{rmk:limits in 1-Cat slices and connectedness}
In the case where $\scC$, $\scD$ and $K$ are 1-categories and we are interested in 1-categorical limits it even suffices to have the indexing category $K$ connected; this is enough to ensure that the inclusion $K \hookrightarrow K * [0]$ is a cofinal functor of 1-categories (though, in general, it will not be homotopy cofinal).
\qen
\end{remark}

\begin{proposition}
\label{st:BSYM and BGAU properties}
Consider the setting of Definition~\ref{def:BGAU(G), BSYM(G)}.
For the commutative diagram
\begin{equation}
\begin{tikzcd}
	\rmB \AUT^\rev(\cG) \ar[r, hookrightarrow] \ar[d]
	& \rmB \SYM^\rev(\cG) \ar[d]
	\\
	N\Cart^\opp \ar[r, hookrightarrow, "Ne_M"']
	& N \rmB \scD[\cG]^\opp
\end{tikzcd}
\end{equation}
of simplicial sets, the following statements hold true:
\begin{enumerate}
\item Both vertical morphisms are left fibrations.

\item The square is a pullback diagram, and hence (by Claim~(1)) even a homotopy pullback square in the Joyal model structure.

\item Augmenting by the map \smash{$N \rmB \scD[\cG] \longrightarrow N\Cart^\opp$} produces a cartesian square in $\sSet_{/N \Cart^\opp}$, which is even homotopy cartesian in the covariant model structure.
\end{enumerate}
\end{proposition}

\begin{proof}
For Claim~(1) it suffices to show that the morphism
\begin{equation}
	\rmB \SYM^\rev(\cG) \longrightarrow \rmB \scD[\cG]^\opp
\end{equation}
is a left fibration.
The statement for the left-hand vertical morphism will then follow from Claim~(2).
We have a sequence
\begin{equation}
\begin{tikzcd}
	\rmB \SYM^\rev(\cG) \ar[r, hookrightarrow]
	&  \textint G^{\scD[\cG]}_{[\cG]} \ar[r]
	& N \rmB \scD[\cG]^\opp\,.
\end{tikzcd}
\end{equation}
The first morphism is an inclusion of a full $\infty$-subcategory, and the second morphism is a left fibration.
We need to show that the composition has the right lifting property with respect to all left horn inclusions $\Lambda^r_s \hookrightarrow \Delta^r$, for $0 \leq s < r \in \NN_0$.
Since the first morphism is an inclusion of a full $\infty$-subcategory, it automatically has the right lifting property with respect to all left horn inclusions except possibly the case $\Lambda^1_0 \hookrightarrow \Delta^1$.
As the second morphism is a left fibration, the only non-trivial lifting problems for the composition arise for the horn inclusion $\Lambda^1_0 \hookrightarrow \Delta^1$.
We check the lifting property explicitly:
a commutative diagram
\begin{equation}
\begin{tikzcd}
	\Lambda^1_0 \ar[r] \ar[d, hookrightarrow]
	& \rmB \SYM^\rev(\cG) \ar[d]
	\\
	\Delta^1 \ar[r]
	& N \rmB \scD[\cG]^\opp
\end{tikzcd}
\end{equation}
is equivalent to specifying a 1-simplex in \smash{$N\rmB\scD[\cG]^\opp$} and a lift of its initial vertex to \smash{$\rmB \SYM^\rev(\cG)$}.
The 1-simplex is a pair $(f,\varphi)$ of a smooth map $f \colon c_1 \to c_0$ of cartesian spaces and a smooth map $\varphi \colon c_1 \to \Diff_{[\cG]}(M)$.
Let $\coll_c \colon c \to \RR^0$ denote the collapse map in $\Cart$.
A solution to the lifting problem amounts to specifying an equivalence
\begin{equation}
	(\coll_{c_1}, 1_M)^* (\cG)
	\simeq (f, \varphi)^* (\coll_{c_0}, 1_M)^* (\cG)
	= (1_{c_1}, \varphi)^* (\coll_{c_1}, 1_M)^* (\cG)
\end{equation}
in in \smash{$G^\scD(c_1)$} (see also the construction of the rectification functor $r_\scC^*$ in~\cite[Sec.~4]{HM:Left_fibs_and_hocolims_I} and~\cite[Sec.~2.2]{Bunk:Localisation_of_sSet}).
Such an equivalence exists since $f$ takes values in $\Diff_{[\cG]}(M)$ (see Definition~\ref{def:Diff_[cG] and D_[cG]}).

Claim~(2) follows by inspection of the diagram:
the inclusion $\rmB\AUT^\rev(\cG) \hookrightarrow \rmB\SYM^\rev(\cG)$ is a bijection on vertices, and a higher simplex in $\rmB\SYM^\rev(\cG)$ lies in $\rmB\AUT^\rev(\cG)$ precisely if its image in $N\rmB\scD[\cG]$ consists of a composable chain $(f_0, \varphi_0), \ldots, (f_n, \varphi_n)$ of morphisms in $\rmB\scD[\cG]$ with $\varphi_i = 1_M$, for each $i = 1, \ldots, n$.
Claim~(3) follows readily from Lemma~\ref{st:limits in slices and contractibility} (see also Remark~\ref{rmk:limits in 1-Cat slices and connectedness}) and Claim~(2), since the covariant fibrations between left fibrant objects are precisely the left fibrations (see the opposite of~\cite[Thm.~4.1.5]{Cisinski:HiC_and_HoA}).
\end{proof}

\begin{lemma}
\label{st:SYM into GRB and GAU into GRB are cov weqs}
In the setting of Definition~\ref{def:BGAU(G), BSYM(G)} the canonical inclusion morphisms
\begin{equation}
	\rmB \SYM^\rev(\cG) \hookrightarrow \textint G^{\scD[\cG]}_{[\cG]}\,,
	\qquad
	\rmB \AUT^\rev(\cG) \hookrightarrow \textint G^M_{[\cG]}
\end{equation}
are covariant weak equivalences in \smash{$\sSet_{/N \rmB \scD[\cG]^\opp}$} and \smash{$\sSet_{/N \Cart^\opp}$}, respectively.
\end{lemma}

\begin{proof}
We show the claim for $\rmB\SYM^\rev(\cG)$; the proof for $\rmB\AUT^\rev(\cG)$ is analogous.
First, note that both $\rmB\SYM^\rev(\cG)$ and \smash{$\textint G^{\scD[\cG]}_{[\cG]}$} are fibrant in the covariant model structure on $\sSet_{/N \rmB \scD[\cG]^\opp}$ (by Lemma~\ref{st:BSYM and BGAU properties} and Definition~\ref{def:BGAU(G), BSYM(G)}, respectively).
It follows that both simplicial sets are $\infty$-categories.
Recall that the simplicial subset
\begin{equation}
	\textint G^{\scD[\cG]}_{[\cG]} \subset \textint G^{\scD[\cG]}
\end{equation}
is the essential image of the functor~\eqref{eq:section from (cG,cA^(k))} (by Remark~\ref{rmk:section associated to cG}).
The further simplicial subset
\begin{equation}
	\rmB\SYM^\rev(\cG) \subset \textint G^{\scD[\cG]}_{[\cG]}
\end{equation}
is the full simplicial subset on all vertices in the image of the functor ~\eqref{eq:section from (cG,cA^(k))}; it follows readily that the inclusion
\begin{equation}
	\rmB\SYM^\rev(\cG) \subset \textint G^{\scD[\cG]}_{[\cG]}
\end{equation}
is both fully faithful and essentially surjective as a functor between $\infty$-categories.
That is, it is a weak equivalence in the Joyal model structure~\cite[Thm.~3.1.9]{Cisinski:HiC_and_HoA}.
Lemma~\ref{st:cov weqs between LFibs are exactly cat weqs} then implies that it is also a covariant weak equivalence in $\sSet_{/N \rmB \scD[\cG]^\opp}$.
\end{proof}

\begin{remark}
Note that, for each $c \in \Cart$, the fibres $\rmB\AUT^\rev(\cG)_{|c}$ and $\rmB\SYM^\rev(\cG)_{|c}$ are \textit{reduced} simplicial sets, i.e.~they have a unique vertex.
The looping-delooping Quillen equivalence between reduced simplicial sets and simplicial groups~\cite[Prop.~V.6.3]{GJ:Simplicial_HoThy} shows that the left fibrations $\rmB\AUT^\rev(\cG) \to N\Cart^\opp$ and $\rmB\SYM^\rev(\cG) \to N\Cart^\opp$ present presheaves of $\infty$-groups on $N\Cart$ by~\cite[Sec.~3.5, Cor.~3.34]{NSS:Pr_ooBdls_II} (see also the following remark and Definition~\ref{def:scSym scAut}).
\qen
\end{remark}

\begin{remark}
\label{rmk:oo-groups and ptd conn objs}
We recall the equivalence between pointed connected objects in an $\infty$-topos $\scX$ and group objects in $\scX$~\cite[Lemma~7.2.2.11]{Lurie:HTT}.
Concretely, there are inverse equivalences of $\infty$-categories
\begin{equation}
\label{eq:Omega -| B equivalence}
\begin{tikzcd}
	\Omega : \scX^{*/}_{\geq 1} \ar[r, shift left=0.1cm]
	& \Grp(\scX) : \rmB\,, \ar[l, shift left=0.1cm]
\end{tikzcd}
\end{equation}
where $\scX^{*/}_{\geq 1} \subset \scX$ denotes the $\infty$-subcategory of pointed connected objects in $\scX$, and $\Grp(\scX)$ denotes the $\infty$-category of group objects in $\scX$ (see, for instance,~\cite[below Def.~7.2.2.1]{Lurie:HTT} and~\cite{NSS:Pr_ooBdls_I, Bunk:Pr_ooBdls_and_String}).
The functor $\rmB$ takes the classifying object of a group object, while $\Omega$ takes the loop object (with its homotopy coherent multiplication) of a pointed connected object in $\scX$.
Here we will mostly be interested in $\infty$-topoi which arise as $\infty$-categories of presheaves of spaces.
\qen
\end{remark}

\begin{definition}
A \textit{smooth $\infty$-group} is a group object in the $\infty$-topos $\scP(N\Cart)$.
The \textit{$\infty$-category of smooth $\infty$-groups} is $\Grp(\scP(N\Cart))$.
\end{definition}

Let $\rho \colon \bbDelta \to \bbDelta$ be the functor defined by
\begin{equation}
	\rho[n] = [n]\,,
	\qquad
	\rho(\lambda)(i) = n - \lambda(m-i)
\end{equation}
for each $\lambda \in \bbDelta([m], [n])$ and $i \in \{0, \ldots, m\}$ (compare, for instance,~\cite[Par.~1.5.7]{Cisinski:HiC_and_HoA}).
This functor is a strict involution, i.e.~$\rho \circ \rho = 1_\bbDelta$.
In particular, it is an auto-equivalence of $\bbDelta$.

\begin{definition}
Let $\scC$ be an $\infty$-category.
We define the involution
\begin{equation}
	\rev \coloneqq (N\rho^\opp)^* \colon \scFun(N\bbDelta^\opp, \scC) \longrightarrow \scFun(N\bbDelta^\opp, \scC)
\end{equation}
on the $\infty$-category of simplicial objects in $\scC$.
\end{definition}

An $\infty$-category $\scC$ is said to have geometric realisations if all $\bbDelta^\opp$-shaped diagrams in $\scC$ have colimits.

\begin{lemma}
\label{st:|revX| = |X|}
If $\scC$ is an $\infty$-category which admits geometric realisations and $X \in \scFun(N\bbDelta^\opp, \scC)$ is a simplicial object in $\scC$, then there is a canonical equivalence
\begin{equation}
	|\rev X| \simeq |X|
\end{equation}
in $\scC$, natural in $X$.
\end{lemma}

\begin{proof}
Since $\rho$ is an equivalence of categories, the induced functor $N\rho^\opp$ on $N \bbDelta^\opp$ is an equivalence.
Thus, it is final~\cite[Prop.~5.3.1]{Cisinski:HiC_and_HoA}.
\end{proof}

\begin{definition}
\label{def:scSym scAut}
In the setting of Definition~\ref{def:BGAU(G), BSYM(G)}, we make the following definitions:
\begin{enumerate}
\item We denote the $\infty$-presheaf which classifies the left fibration $\pi_\SYM \colon \rmB \SYM^\rev(\cG) \longrightarrow N \Cart^\opp$ by
\begin{equation}
	\rmB \scSym^\rev(\cG) \colon N \Cart^\opp \longrightarrow \scS\,.
\end{equation}
Since the fibres of $\pi_\SYM$ are reduced simplicial sets, the values of $\rmB \scSym^\rev(\cG)$ are pointed, connected spaces.
By Remark~\ref{rmk:oo-groups and ptd conn objs}, there is thus a unique group object associated to $\rmB\scSym^\rev(\cG)$, which we denote by
\begin{equation}
	\scSym^\rev(\cG) \in \Grp \big( \scFun(N \Cart^\opp, \scS) \big)\,.
\end{equation}

\item The reverse simplicial object
\begin{equation}
	\scSym(\cG) \coloneqq \big( \scSym^\rev(\cG) \big)^\rev
	\quad \in \Grp \big( \scFun (N \Cart^\opp, \scS) \big)
\end{equation}
is the smooth higher group of \textit{symmetries of $\cG \in \G^\scD(\RR^0)$.}

\item Similarly, we let
\begin{equation}
	\rmB\scAut^\rev(\cG) \colon N\Cart^\opp \longrightarrow \scS
\end{equation}
denote the $\infty$-presheaf which classifies the left fibration $\pi_\AUT \colon \rmB\AUT^\rev(\cG) \longrightarrow N\Cart^\opp$.
We denote its associated group object by
\begin{equation}
	\scAut^\rev(\cG) \in \Grp \big( \scFun (N \Cart^\opp, \scS) \big)\,.
\end{equation}

\item The reverse simplicial object
\begin{equation}
	\scAut(\cG) \coloneqq \big( \scAut^\rev(\cG) \big)^\rev
	\quad \in \Grp \big( \scFun (N \Cart^\opp, \scS) \big)
\end{equation}
is the smooth higher group of \textit{automorphisms of $\cG \in G^\scD(\RR^0)$.}
\end{enumerate}
\end{definition}

We denote by
\begin{equation}
	\bbG^{\scD[\cG]}_{[\cG]} = \gamma_{\rmB\scD[\cG]^\opp}^* G^{\scD[\cG]}_{[\cG]}
\end{equation}
the $\infty$-presheaf associated to \smash{$G^{\scD[\cG]}_{[\cG]}$}.

\begin{lemma}
\label{st:BSym^rev --> pi_M! G^D[G]_G}
There is a canonical commutative triangle 
\begin{equation}
\begin{tikzcd}[column sep={2cm,between origins}, row sep=0.75cm]
	\rmB\scSym^\rev(\cG) \ar[rr] \ar[dr]
	& & (N\pi_M)_! \bbG^{\scD[\cG]}_{[\cG]} \ar[dl]
	\\
	& \rmB \Diff_{[\cG]}^\rev(M) &
\end{tikzcd}
\end{equation}
The top morphism is an objectwise equivalence.
\end{lemma}

\begin{proof}
For each $\infty$-category $A$, there is a canonical equivalence between the $\infty$-categorical localisation of $\sSet_{/A}$ at the covariant weak equivalences and the $\infty$-category $\scFun(A, \scS)$~\cite[Thm.~7.8.9]{Cisinski:HiC_and_HoA}.
By construction, the triangle in the statement is presented under this equivalence by the commutative triangle
\begin{equation}
\begin{tikzcd}[column sep={2cm,between origins}, row sep=0.75cm]
	\rmB\SYM^\rev(\cG) \ar[rr, hookrightarrow] \ar[dr]
	& & \textint G^{\scD[\cG]}_{[\cG]} \ar[dl]
	\\
	& N \rmB\scD[\cG] &
\end{tikzcd}
\end{equation}
of simplicial sets over $N\Cart^\opp$.
Lemma~\ref{st:SYM into GRB and GAU into GRB are cov weqs} shows that the top morphism is a covariant weak equivalence over $N\rmB\scD[\cG]^\opp$.
Since each object in the covariant model structure is cofibrant, under the left Quillen functor
\begin{equation}
	(N\pi_M)_! \colon \sSet_{/N \rmB \scD[\cG]^\opp} \longrightarrow \sSet_{/N \Cart^\opp}
\end{equation}
this becomes a covariant weak equivalence over $N\Cart^\opp$.
\end{proof}

\begin{definition}
\cite[Def.~4.26]{NSS:Pr_ooBdls_I}
Let $\scX$ be an $\infty$-topos.
An extension of group objects in $\scX$ is a sequence of morphisms $\bbA \to \bbH \to \bbGamma$ in the $\infty$-category $\Grp(\scX)$ of group objects in $\scX$ such that the induced sequence $\rmB\bbA \to \rmB\bbH \to \rmB\bbGamma$ is a fibre sequence in $\scX$.
\end{definition}

Proposition~\ref{st:BSYM and BGAU properties}(3) has the following direct consequence:

\begin{corollary}
\label{st:Aut-Sym-Diff[cG] extension}
We have an extension of group objects in $\scFun(N\Cart^\opp, \scS)$:
\begin{equation}
\begin{tikzcd}
	\scAut(\cG) \ar[r]
	& \scSym(\cG) \ar[r]
	& \Diff_{[\cG]}(M)\,.
\end{tikzcd}
\end{equation}
\end{corollary}

\begin{proof}
By Proposition~\ref{st:BSYM and BGAU properties}(3) and the presentation of the $\infty$-category $\scP(N\Cart)$ by the covariant model category $\sSet_{/N \Cart^\opp}$ (see~\cite[Thm.~7.8.9]{Cisinski:HiC_and_HoA}) the square
\begin{equation}
\begin{tikzcd}
	\rmB \scAut^\rev(\cG) \ar[r] \ar[d]
	& \rmB \scSym^\rev(\cG) \ar[d]
	\\
	* \ar[r]
	& \rmB\Diff_{[\cG]}^\rev(M)
\end{tikzcd}
\end{equation}
is a fibre sequence in $\scFun(N\Cart^\opp, \scS)$.
Lemma~\ref{st:|revX| = |X|} shows that this is equivalent to fact that the fibre square for the classifying objects of the non-reversed group objects is a fibre sequence.
\end{proof}

The smooth $\infty$-group $\scSym(\cG)$ can be interpreted as consisting of all possible lifts of elements $\varphi \in \Diff_{[\cG]}(M)$ to $\cG$ (see, in particular, Theorem~\ref{st:scSym represents Equivar} and Corollary~\ref{st:char of equivar structures} below).
We now associate similar universal groups to any smooth group action on $M$.

\begin{definition}
\label{def:BSYM_(Q,phi)(G)}
Consider a covariantly fibrant object $Q \to N\Cart^\opp$ in $\sSet_{/N \Cart^\opp}$ and suppose that the fibre $Q_{|c}$ is a reduced simplicial set, for each $c \in \Cart$.
Suppose $\phi \colon Q \to N\rmB\scD[\cG]^\opp$ is a morphism in $\sSet_{/N \Cart^\opp}$.
We define a left fibration over $Q$ as the pullback (this indeed produces a left fibration by Proposition~\ref{st:BSYM and BGAU properties}(1))
\begin{equation}
\begin{tikzcd}
	\rmB\SYM^\rev_\phi(\cG) \ar[r] \ar[d, "\psi_\phi"']
	& \rmB\SYM^\rev(\cG) \ar[d]
	\\
	Q \ar[r, "\phi"']
	& N\rmB\scD[\cG]^\opp
\end{tikzcd}
\end{equation}
\end{definition}

Note that there is a canonical factorisation
\begin{equation}
\begin{tikzcd}[column sep={1.5cm,between origins}, row sep=0.75cm]
	N\Cart^\opp \ar[rr, "\eta_M"] \ar[dr, "Ne_M"']
	& & Q \ar[dl, "\phi"]
	\\
	& N\rmB\scD[\cG]^\opp
\end{tikzcd}
\end{equation}
of the morphism $e_M \colon N\Cart^\opp \longrightarrow N\rmB\scD[\cG]^\opp$ through $\phi$.

\begin{example}
\label{eg:BSYM_(BH,phi)(G) for smooth group action}
Let $H$ be a presheaf of groups on $\Cart$, and let $\psi \colon H \to \Diff_{[\cG]}(M)$ be a morphism of group objects in $\scH$.
Recall the isomorphism
\begin{equation}
	N\rmB\scD[\cG]^\opp \cong r_{\Cart^\opp}^* \big( N \circ \rmB \Diff_{[\cG]}(M)^\rev \big)
\end{equation}
from Remark~\ref{rmk:BD(M)^opp and Diff(M)^rev}.
Consider the simplicial set
\begin{equation}
	Q = \rmB H^\rev
	\coloneqq r_{\Cart^\opp}^* \rmB H^\rev\,,
\end{equation}
where $\rmB$ is the delooping functor, together with the canonical morphism $\phi \coloneqq r_{\Cart^\opp}^* \rmB(\phi^\rev) \colon Q \to N\rmB\scD[\cG]$.
Then, the left fibration $\rmB\SYM^\rev_\phi(\cG)$ encodes those symmetries of $\cG$ which cover the action of elements of $H$.
\qen
\end{example}

We have the following relative version of Proposition~\ref{st:BSYM and BGAU properties}:

\begin{lemma}
\label{st:BAUT BSYM_phi fib seq}
In the situation of Definition~\ref{def:BSYM_(Q,phi)(G)}, we have a commutative square
\begin{equation}
\begin{tikzcd}
	\rmB \AUT^\rev(\cG) \ar[r, hookrightarrow] \ar[d, "\psi_\Aut"']
	& \rmB \SYM^\rev_\phi(\cG) \ar[d]
	\\
	N\Cart^\opp \ar[r, hookrightarrow, "\eta_M"']
	& Q
\end{tikzcd}
\end{equation}
of simplicial sets.
Moreover, the following statements hold true:
\begin{enumerate}
\item Both vertical morphisms are left fibrations.

\item The square is a pullback diagram, and hence (by Claim~(1)) even a homotopy pullback square in the Joyal model structure.

\item Augmenting by the map \smash{$Q \to N\Cart^\opp$}, this produces a cartesian square in $\sSet_{/N \Cart^\opp}$, which is even homotopy cartesian in the covariant model structure.
\end{enumerate}
\end{lemma}

\begin{proof}
Claim~(1) follows from Proposition~\ref{st:BSYM and BGAU properties}(1) (for the left-hand vertical morphism) and Definition~\ref{def:BSYM_(Q,phi)(G)} (for the right-hand vertical morphism).
Claim~(2) follows by the pasting law for pullbacks, and Claim~(3) follows from Claim~(2) and Lemma~\ref{st:limits in slices and contractibility} (see also Remark~\ref{rmk:limits in 1-Cat slices and connectedness}).
\end{proof}

The left fibration $Q \to N\Cart^\opp$ is classified by an $\infty$-presheaf $\rmB\bbGamma^\rev \colon N\Cart^\opp \to \scS$ which takes values in pointed, connected spaces.
It thus presents an essentially unique group object $\bbGamma = (\Omega \rmB \bbGamma^\rev)^\rev$ in $\scP(N\Cart)$ (recall the equivalence~\eqref{eq:Omega -| B equivalence}).
Then, the morphism $\phi \colon Q \to N\rmB\scD[\cG]^\opp$ presents a morphism $\Phi \colon \bbGamma \to \Diff_{[\cG]}(M)$ of group objects in $\scP(N\Cart)$.

\begin{definition}
\label{def:Sym_(Gamma, Phi)(cG)}
Let $\bbGamma$ be a smooth $\infty$-group, i.e.~a group object in $\scP(N\Cart)$, and let $\Phi \colon \bbGamma \to \Diff_{[\cG]}(M)$ be a morphism of smooth $\infty$-groups.
We define a smooth $\infty$-group as the pullback in $\Grp(\scP(N\Cart))$,
\begin{equation}
\begin{tikzcd}
	\scSym_\Phi(\cG) \ar[r] \ar[d]
	& \scSym(\cG) \ar[d]
	\\
	\bbGamma \ar[r, "\Phi"']
	& \Diff_{[\cG]}(M)
\end{tikzcd}
\end{equation}
\end{definition}

\begin{corollary}
\label{st:Aut-Sym_phi-Gamma extension}
By construction, this is an extension of $\bbGamma$ by $\scAut(\cG)$, i.e.~we have a fibre sequence in $\Grp(\scP(N\Cart))$:
\begin{equation}
\begin{tikzcd}
	\scAut(\cG) \ar[r]
	& \scSym_\Phi(\cG) \ar[r]
	& \bbGamma\,.
\end{tikzcd}
\end{equation}
\end{corollary}

\begin{proof}
This follows readily from Corollary~\ref{st:Aut-Sym-Diff[cG] extension} and the pasting law for pullbacks.
\end{proof}

\begin{remark}
In the preceding discussion the left fibration $\rmB\SYM_\phi^\rev(\cG) \to Q$ presents the morphism
\begin{equation}
	\scSym_\Phi(\cG)^\rev \to \bbGamma^\rev
\end{equation}
of smooth $\infty$-groups.
\qen
\end{remark}

%%%%%%%%%%%%%%%%%%%%%%%%%%%%%%%%%%%%%%%%%%%%%%%%%%%%%%%%%%%%%%%%%%%%%%%%%%%%

\subsection{The universal property of the group of symmetries}

%%%%%%%%%%%%%%%%%%%%%%%%%%%%%%%%%%%%%%%%%%%%%%%%%%%%%%%%%%%%%%%%%%%%%%%%%%%%

In this subsection we prove the relation between the smooth higher symmetry group of a higher geometric structure $\cG$ and smooth equivariant structures on $\cG$.

Let $G^\scD$ be a projectively fibrant simplicial presheaf on $\rmB\scD$, and let $\cG \in \G^\scD(\RR^0)$ be a section.
Let \smash{$G^{\scD[\cG]}_{[\cG]}$} be the associated projectively fibrant simplicial presheaf on $\rmB\scD[\cG]$ (see Definition~\ref{def:G_[cG]}).
We denote by
\begin{equation}
	\bbG^{\scD[\cG]}_{[\cG]} \coloneqq \gamma_{\rmB\scD[\cG]^\opp}^*G^{\scD[\cG]}_{[\cG]}
\end{equation}
the associated $\infty$-presheaf on $N\rmB\scD[\cG]$.
By construction, there is a canonical morphism of $\infty$-presheaves \smash{$\bbG^{\scD[\cG]}_{[\cG]} \to \rmB\Diff_{[\cG]}^\rev(M)$}.
In this section we exhibit the geometric and structural significance of the smooth $\infty$-group $\scSym(\cG)$ associated to any higher geometric structure $\cG \in G^\scD$ on $M$:
it arises from a universal property (it represents a certain $\infty$-functor) and completely controls smooth families of equivariant structures on $\cG$.
As before, let $\phi \colon Q \to N\rmB\scD[\cG]^\opp$ be a morphism of left fibrations over $N\Cart^\opp$ which presents the $\infty$-group action $\Phi \colon \bbGamma \to \Diff_{[\cG]}(M)$.
We denote the left fibration $Q \to N\Cart^\opp$ by $q$.

\begin{definition}
\label{def:Equivar_(H,Phi)(G,A^(k))}
Let $\Phi \colon \bbGamma \to \Diff(M)$ be a morphism of group objects in $\scP(N\Cart)$.
The space of $(\bbGamma, \Phi)$-equivariant structures on $\cG$ is the pullback of $\infty$-groupoids
\begin{equation}
\begin{tikzcd}[column sep=1.5cm]
	\Equivar_\Phi(\cG) \ar[r] \ar[d]
	& \scP(N\Cart)_{/\rmB\Diff_{[\cG]}^\rev(M)} \big( \rmB \bbGamma^\rev, (N\pi_M)_! \bbG^{\scD[\cG]}_{[\cG]} \big) \ar[d]
	\\
	\Delta^0 \ar[r, "\{\cG\}"']
	& \big( (N\pi_M)_! \bbG^{\scD[\cG]}_{[\cG]} \big)(\RR^0)
\end{tikzcd}
\end{equation}
\end{definition}

The right-hand vertical morphism in this square arises from the fact that $\rmB \bbGamma^\rev$ is a pointed object of $\scP(N\Cart)$, i.e.~it has a distinguished global section.
This induces a morphism
\begin{equation}
	\scP(N\Cart)_{/\rmB\Diff_{[\cG]}^\rev(M)} \big( \rmB \bbGamma^\rev, (N\pi_M)_! \bbG^{\scD[\cG]}_{[\cG]} \big)
	\longrightarrow \big( (N\pi_M)_! \bbG^{\scD[\cG]}_{[\cG]} \big) (\RR^0)\,.
\end{equation}
The morphism $e_M \colon \Cart^\opp \to \rmB\scD[\cG]$ induces a canonical morphism
\begin{equation}
	\bbG^M_{[\cG]}
	= Ne_M^* \bbG^{\scD[\cG]}_{[\cG]} \longrightarrow
	\big( (N\pi_M)_! \bbG^{\scD[\cG]}_{[\cG]} \big)\,.
\end{equation}
The bottom horizontal morphism is the composition of this morphism by \smash{$\cG \colon \Delta^0 \to Ne_M^* \bbG^{\scD[\cG]}$}.

Let us justify Definition~\ref{def:Equivar_(H,Phi)(G,A^(k))}.
The $\infty$-functor \smash{$\bbG^{\scD[\cG]}_{[\cG]} \colon N\rmB\scD[\cG]^\opp \to \scS$} encodes the smooth right action of $\Diff_{[\cG]}(M)$ on \smash{$\bbG^{M}_{[\cG]} \in \scP(N\Cart)$}.
Analogously, the $\infty$-functor \smash{$\phi^*\bbG^{\scD[\cG]}_{[\cG]} \colon Q \to \scS$} encodes the smooth right action of $\bbGamma$ on \smash{$\bbG^{M}_{[\cG]} \in \scP(N\Cart)$} via the morphism $\Phi \colon \bbGamma \to \Diff_{[\cG]}(M)$.
The $\infty$-groupoid of sections of \smash{$\bbG^{M}_{[\cG]}$} endowed with smooth homotopy fixed-point data is given by the mapping space \smash{$\scFun(Q, \scS) (\sfe_Q, \phi^* \bbG^{\scD[\cG]}_{[\cG]})$}, where $\sfe_Q \in \scFun(Q,\scS)$ is the constant functor with value the final object in $\scS$.
Therefore, the space of $(\bbGamma, \Phi)$-equivariant structures on $\cG$ is the full $\infty$-subgroupoid of \smash{$\scFun(Q, \scS) (\sfe_Q, \bbG^{M}_{[\cG]})$} on those vertices whose underlying section of \smash{$\bbG^{\scD[\cG]}_{[\cG]}$} is the constant section $\cG$ (we could instead have defined $\Equivar_\Phi(\cG)$ in this way).

\begin{lemma}
Consider a commutative triangle of $\infty$-categories
\begin{equation}
\begin{tikzcd}
	A \ar[rr, "\phi"] \ar[dr, "q"']
	& & B \ar[dl, "\pi"]
	\\
	& C &
\end{tikzcd}
\end{equation}
where $q$ and $\pi$ are left fibrations.
Let $\sfe_A \in \scFun(A, \scS)$ be the $\infty$-functor with constant value the final object in $\scS$, and analogously for $\sfe_B \in \scFun(B, \scS)$.
For each $F \in \scFun(B, \scS)$, there is a canonical equivalence
\begin{equation}
	\phi^*F
	\simeq \sfe_A \underset{q^* \pi_! \sfe_B}{\times} q^* \pi_! F\,,
\end{equation}
in $\scFun(A,\scS)$, natural in $F$.
\end{lemma}

\begin{proof}
Let $\textint F \to B$ denote the left fibration classified by $F$.
Since $\pi$ is a left fibration, the left Kan extension $\pi_!F \colon C \to \scS$ classifies the composed left fibration $\textint F \to B \to C$ (see~\cite[Prop.~6.1.14]{Cisinski:HiC_and_HoA}).
The left fibration classified by $e_B$ is the identity morphism $B \to B$.
We have a commutative diagram of simplicial sets
\begin{equation}
\begin{tikzcd}
	\phi^* \textint F \ar[r] \ar[d]
	& q^* \textint F \ar[r] \ar[d]
	& \textint F \ar[d]
	\\
	A \ar[r] \ar[dr, equal]
	& q^*B \ar[r] \ar[d]
	& B \ar[d, "\pi"]
	\\
	& A \ar[r, "q"']
	& C
\end{tikzcd}
\end{equation}
Its middle horizontal composition is $\phi$, and by the pasting law each square is a pullback square.
In particular, the left upper square is homotopy cartesian in the covariant model structure on $\sSet_{/A}$; it presents the equivalence in the claim.
\end{proof}

We obtain a pullback square of $\infty$-groupoids
\begin{equation}
\begin{tikzcd}
	\scFun(Q, \scS) \big( \sfe_Q, \phi^* \bbG^{\scD[\cG]}_{[\cG]} \big) \ar[r] \ar[d]
	& \scFun(Q, \scS) \big( \sfe_Q, q^* (N\pi_M)_! \bbG^{\scD[\cG]}_{[\cG]} \big) \ar[d]
	\\
	\Delta^0 \simeq \scFun(Q, \scS) \big( e_Q, e_Q \big) \ar[r]
	& \scFun(Q, \scS) \big( \sfe_Q, q^* (N\pi_M)_! \sfe_{N\rmB\scD[\cG]^\opp} \big)
\end{tikzcd}
\end{equation}
where the bottom morphism is induced by the action $\Phi$ (which corresponds to a morphism $\sfe_Q \to \phi^* \sfe_{N\rmB\scD[\cG]^\opp}$ in $\scFun(Q, \scS)$) and the unit of the adjunction $(N\pi_M)_! \dashv (N\pi_M)^*$.
Finally, recall that $q_! \sfe_Q$ is the $\infty$-functor presented by the left fibration $q$.
Using this, we see that the Kan extension adjunctions induce canonical equivalences
\begin{align}
	\scFun(Q, \scS) \big( \sfe_Q, q^* (N\pi_M)_! \sfe_{N\rmB\scD[\cG]^\opp} \big)
	&\simeq \scFun(N\Cart^\opp, \scS) \big( \rmB\bbGamma^\rev, \rmB\Diff_{[\cG]}^\rev(M) \big)\,,
	\\
	\scFun(Q, \scS) \big( \sfe_Q, q^* (N\pi_M)_! \bbG^{\scD[\cG]}_{[\cG]} \big)
	&\simeq \scFun(N\Cart^\opp, \scS) \big( \rmB\bbGamma^\rev, (N\pi_M)_! \bbG^{\scD[\cG]}_{[\cG]} \big)\,.
\end{align}
Combining this with the above pullback square, we obtain a canonical equivalence
\begin{equation}
	\scFun(Q, \scS) \big( \sfe_Q, \phi^* \bbG^{\scD[\cG]}_{[\cG]} \big)
	\simeq \scP(N\Cart)_{/\rmB\Diff_{[\cG]}^\rev(M)} \big( \rmB \bbGamma^\rev, (N\pi_M)_! \bbG^{\scD[\cG]}_{[\cG]} \big)\,,
\end{equation}
where on the left-hand side we have the $\infty$-groupoid of smooth homotopy fixed points of the $\bbGamma$-action on \smash{$\bbG^M_{[\cG]}$}.
Then, the further pullback in Definition~\ref{def:Equivar_(H,Phi)(G,A^(k))} restricts the $\infty$-subgroupoid of such homotopy fixed points whose underlying section is $\cG$, i.e.~the smooth equivariant structures on $\cG$.

Varying the pair $(\bbGamma, \Phi)$, the above construction canonically assembles into a functor
\begin{equation}
	\Equivar(\cG) \colon \big( \Grp \big( \scP(N\Cart) \big)_{/\Diff_{[\cG]}(M)} \big)^\opp \longrightarrow \scS
\end{equation}
which assigns to a $\bbGamma$-action $(\bbGamma, \Phi)$ on $M$ the space of $(\bbGamma, \Phi)$-equivariant structures on $\cG$ as in Definition~\ref{def:Equivar_(H,Phi)(G,A^(k))}.
We can further extend this to a functor on the $\infty$-category \smash{$\Grp \big( \scP(N\Cart) \big)_{/\Diff(M)}$} of all smooth $\infty$-group actions on $M$ by setting $\Equivar_\Phi(\cG) = \emptyset$ whenever $\Phi \colon \bbGamma \to \Diff(M)$ does not factor through $\Diff_{[\cG]}(M)$ (in that case, no $(\bbGamma, \Phi)$-equivariant structure can exist on $\cG$).

\begin{theorem}
\label{st:scSym represents Equivar}
The $\infty$-functor
\begin{equation}
	\Equivar(\cG) \colon \big( \Grp \big( \scP(N\Cart) \big)_{/\Diff(M)} \big)^\opp \longrightarrow \scS
\end{equation}
is representable by the morphism $\scSym(\cG) \to \Diff_{[\cG]}(M) \hookrightarrow \Diff(M)$.
That is, there is a canonical natural equivalence
\begin{equation}
\label{eq:Equivar equivalence}
	\Equivar_\Phi(\cG)
	\simeq \Grp \big( \scP(N\Cart) \big)_{/\Diff(M)} \big( \bbGamma, \scSym(\cG) \big)
\end{equation}
between $(\bbGamma,\Phi)$-equivariant structures on $\cG$ and lifts of the morphism $\Phi \colon \bbGamma \to \Diff(M)$ of smooth $\infty$-groups along the morphism $\scSym(\cG) \to \Diff_{[\cG]}(M)$.
\end{theorem}

\begin{proof}
We first observe that if $\Phi$ does not take values purely in the subgroup $\Diff_{[\cG]}(M) \subset \Diff(M)$, then both sides of~\eqref{eq:Equivar equivalence} are the empty space.
Thus, we may restrict our attention to the case where $\Phi \colon \Gamma \to \Diff_{[\cG]}(M)$, i.e.~$\Phi$ factors through the inclusion $\Diff_{[\cG]}(M) \hookrightarrow \Diff(M)$.
In that case, there is an equivalence
\begin{align}
	\Grp \big( \scP(N\Cart) \big)_{/\Diff(M)} \big( \bbGamma, \scSym(\cG) \big)
	\simeq \Grp \big( \scP(N\Cart) \big)_{/\Diff_{[\cG]}(M)} \big( \bbGamma, \scSym(\cG) \big)\,,
\end{align}
The equivalence $\rev \colon \Grp(\scP(N\Cart)) \to \Grp(\scP(N\Cart))$ provides a canonical equivalence
\begin{equation}
	\Grp \big( \scP(N\Cart) \big)_{/\Diff_{[\cG]}(M)} \big( \Gamma, \scSym(\cG) \big)
	\\
	\simeq \Grp \big( \scP(N\Cart) \big)_{/\Diff_{[\cG]}^\rev(M)} \big( \Gamma^\rev, \scSym^\rev(\cG) \big)\,.
\end{equation}
By the equivalence~\eqref{eq:Omega -| B equivalence}, there is a further equivalence
\begin{align}
	\Grp \big( \scP(N\Cart) \big)_{/\Diff_{[\cG]}^\rev(M)} \big( \bbGamma^\rev, \scSym^\rev(\cG) \big)
	&\simeq \big( \scP(N\Cart)^{*/}_{\geq 1} \big)_{/\rmB \Diff_{[\cG]}^\rev(M)} \big( \rmB \bbGamma^\rev, \rmB\scSym^\rev(\cG) \big)
	\\
	&\simeq \big( \scP(N\Cart)^{*/} \big)_{/\rmB \Diff_{[\cG]}^\rev(M)} \big( \rmB \bbGamma^\rev, \rmB\scSym^\rev(\cG) \big)\,.
\end{align}
Now we use that, by construction of $\scSym(\cG)$, the canonical morphism
\begin{equation}
	\rmB\scSym^\rev(\cG) \hookrightarrow (N\pi_M)_! \bbG^{\scD[\cG]}_{[\cG]}
\end{equation}
is an equivalence over $\rmB\Diff_{[\cG]}^\rev(M)$ (see Lemma~\ref{st:BSym^rev --> pi_M! G^D[G]_G}).
This canonical equivalence also endows the object \smash{$(N\pi_M)_! \bbG^{\scD[\cG]}_{[\cG]}$} in $\scP(N\Cart)$ with a choice of base point, transferred from $\rmB\scSym^\rev(\cG)$.
With this choice of base point, there is a canonical equivalence
\begin{equation}
	\big( \scP(N\Cart)^{*/} \big)_{/\rmB \Diff_{[\cG]}^\rev(M)} \big( \rmB \bbGamma^\rev, \rmB\scSym^\rev(\cG) \big)
	\simeq \big( \scP(N\Cart)^{*/} \big)_{/\rmB \Diff_{[\cG]}^\rev(M)} \big( \rmB \bbGamma^\rev, (N\pi_M)_! \bbG^{\scD[\cG]}_{[\cG]} \big)
\end{equation}
Explicitly, the base point of \smash{$(N\pi_M)_! \bbG^{\scD[\cG]}_{[\cG]}$} is the global section corresponding to the element
\begin{equation}
	\cG' \in \big( (N\pi_M)_! \bbG^{\scD[\cG]}_{[\cG]} \big)(\RR^0)
\end{equation}
(as can be seen, for instance, on the level of left fibrations as shown in the proof of Lemma~\ref{st:BSym^rev --> pi_M! G^D[G]_G}).
Consequently, the commutative square of $\infty$-groupoids
\begin{equation}
\begin{tikzcd}[column sep=1.5cm]
	\big( \scP(N\Cart)^{*/} \big)_{/\rmB \Diff_{[\cG]}^\rev(M)} \big( \rmB \bbGamma^\rev, (N\pi_M)_! \bbG^{\scD[\cG]}_{[\cG]} \big) \ar[r] \ar[d]
	& \scP(N\Cart)_{/\rmB\Diff_{[\cG]}^\rev(M)} \big( \rmB \bbGamma^\rev, (N\pi_M)_! \bbG^{\scD[\cG]}_{[\cG]} \big) \ar[d]
	\\
	* \ar[r, "\{\cG\}"']
	& \big( (N\pi_M)_! \bbG^{\scD[\cG]}_{[\cG]} \big) (\RR^0)
\end{tikzcd}
\end{equation}
is cartesian.
By comparison with Definition~\ref{def:Equivar_(H,Phi)(G,A^(k))} this completes the proof.
\end{proof}

Recall the smooth $\infty$-group $\scSym_\Phi(\cG)$ from Definition~\ref{def:Sym_(Gamma, Phi)(cG)}.

\begin{corollary}
\label{st:char of equivar structures}
The following $\infty$-presheaves
\begin{equation}
	\big( \Grp \big( \scP(N\Cart) \big)_{/\Diff(M)} \big)^\opp \longrightarrow \scS
\end{equation}
on the category of smooth $\infty$-group actions on $M$ are canonically equivalent:
\begin{enumerate}
\item the functor $\Equivar(\cG)$, assigning to an object $(\bbGamma, \Phi)$ the space of $(\bbGamma, \Phi)$-equivariant structures on $\cG$,

\item the functor assigning to an object $(\bbGamma, \Phi)$ the space of lifts of $\Phi \colon \bbGamma \to \Diff_{[\cG]}(M)$ through the morphism $\scSym(\cG) \longrightarrow \Diff_{[\cG]}(M)$,

\item the functor assigning to an object $(\bbGamma, \Phi)$ the space of splittings of the extension
\begin{equation}
\begin{tikzcd}
	\scAut(\cG) \ar[r]
	& \scSym_\Phi(\cG) \ar[r]
	& \bbGamma\,.
\end{tikzcd}
\end{equation}
\end{enumerate}
\end{corollary}

\begin{proof}
The equivalence between (1) and (2) is a rephrasing of the assertion of Theorem~\ref{st:scSym represents Equivar}.
The equivalence between (2) and (3) follows readily from the construction of $\scSym_\Phi(\cG)$ as a pullback of $\scSym(\cG)$ along the group morphism $\Phi$.
\end{proof}

%%%%%%%%%%%%%%%%%%%%%%%%%%%%%%%%%%%%%%%%%%%%%%%%%%%%%%%%%%%%%%%%%%%%%%%%%%%%

\section{Moduli $\infty$-prestacks of higher geometric structures}
\label{sec:moduli oo-prestacks of hgeo strs on M}

%%%%%%%%%%%%%%%%%%%%%%%%%%%%%%%%%%%%%%%%%%%%%%%%%%%%%%%%%%%%%%%%%%%%%%%%%%%%

Having constructed and analysed the smooth higher symmetries and automorphisms of a fixed higher geometric structure $\cG$ on $M$, we now construct moduli $\infty$-stacks of higher geometric data defined \textit{on} $\cG$, modulo the action of smooth higher symmetries or automorphisms of $\cG$ on these data.
We extract information about the topology of the resulting moduli $\infty$-prestacks and prove a criterion for when two of our moduli $\infty$-prestacks are equivalent; the key results here are Proposition~\ref{st:SOL^(n,D_cG(M))(cG,cA^(k)) depends only on [cG]} and Theorem~\ref{st:equiv result for Mdl oo-stacks}.
A priori, these moduli arise as $\infty$-\textit{pre}sheaves (or, equivalently, $\infty$-prestacks).
However, we show in Section~\ref{sec:descent for moduli oo-prestacks} that whenever the smooth higher group $\bbGamma$, which we allow as symmetries acting on the underlying manifold $M$, is described by a sheaf of 1-groups, our moduli $\infty$-prestacks are, in fact, $\infty$-stacks, i.e.~satisfy descent (Theorem~\ref{st:descent for spl pshs on tint BH}).

%%%%%%%%%%%%%%%%%%%%%%%%%%%%%%%%%%%%%%%%%%%%%%%%%%%%%%%%%%%%%%%%%%%%%%%%%%%%

\subsection{Families of configurations and solutions}
\label{sec:Fams of Confs and Sols}

%%%%%%%%%%%%%%%%%%%%%%%%%%%%%%%%%%%%%%%%%%%%%%%%%%%%%%%%%%%%%%%%%%%%%%%%%%%%

In this section we introduce the setting which we use to treat moduli of higher geometric structures.
The idea is that we wish to study some additional data on a given, fixed higher geometric structure $\cG$ on $M$.
These additional data may be constrained in a certain way, such as by requiring that they form solutions to certain field equations.
For instance, $\cG$ could be a higher bundle on $M$, and the additional data could consist of higher connections on $M$ and Riemannian metrics on $M$ which together satisfy a coupled set of field equations (see, in particular, Sections~\ref{sec:examples with higher U(1)-connections} and~\ref{sec:comparing GRic moduli} below).
Further, the additional data are acted on smoothly and homotopy coherently by the smooth higher symmetry group of $\cG$, and we wish to study the (homotopy) quotient by this action.

\begin{definition}
Given a simplicial presheaf $\Fix^\scD$ on $\rmB\scD$, a \textit{configuration presheaf over $\Fix^\scD$} is a simplicial presheaf $\Conf^\scD$ on $\rmB\scD$ together with a projective fibration
\begin{equation}
	p \colon \Conf^\scD \longrightarrow \Fix^\scD\,.
\end{equation}
\end{definition}

\begin{remark}
For $c \in \rmB\scD$, we think of the vertices of a configuration presheaf $\Conf^\scD(c)$ as $c$-parameterised families of configurations of geometric structures on $M$ (such as higher bundles).
Higher simplices of $\Conf^\scD$ encode $c$-families of (higher) isomorphisms of configurations, such as gauge transformations, and the functorial dependence on objects $c \in \rmB\scD$ encodes restriction and reparameterisation of families, as well as the smooth action of $\Diff(M)$ on families of configurations.

If $p \colon \Conf^\scD \to \Fix^\scD$ is a configuration presheaf over $\Fix^\scD$, the morphism $p$ encodes the assumption that the configurations have underlying higher geometric data given by a section of $\Fix^\scD$.
For instance, an important case is where $\Conf^\scD$ encodes smooth families of higher bundles with connection, $\Fix^\scD$ encodes smooth families of higher bundles without connection, and the morphism $\Conf^\scD \to \Fix^\scD$ forgets the connection data.
\qen
\end{remark}

Note that $\Conf^\scD$ is necessarily itself projectively fibrant.
Let $\cG$ be an element in $\Fix^\scD(\RR^0)$, or, equivalently, a global section of \smash{$e_M^*\Fix^\scD$}.
Suppose that $p \colon \Conf^\scD \to \Fix^\scD$ is a configuration presheaf over $\Fix^\scD$.
As before, we obtain restricted projectively fibrant simplicial presheaves
\begin{align}
	\Conf^{\scD[\cG]},\ \Fix^{\scD[\cG]} \colon \rmB \scD[\cG]^\opp &\longrightarrow \sSet
	\quad \text{and}
	\\
	\Conf^M \coloneqq e_M^*\Conf^{\scD[\cG]},\ \Fix^M
	\coloneqq e_M^*\Fix^{\scD[\cG]} \colon \Cart^\opp &\longrightarrow \sSet\,.
\end{align}

We denote the associated left fibrations of these simplicial presheaves (obtained via the rectification functor; see Appendix~\ref{app:rectification}) by
\begin{align}
	\textint \Conf^\scD \coloneqq r_{\rmB\scD^\opp}^* \Conf^\scD &\longrightarrow N \rmB \scD^\opp\,,
	\\
	\textint \Conf^{\scD[\cG]} \coloneqq r_{\rmB\scD[\cG]}^* \Conf^{\scD[\cG]} &\longrightarrow N \rmB \scD[\cG]^\opp\,,
	\\
	\textint \Conf^M \coloneqq r_{\Cart^\opp}^* \Conf^M &\longrightarrow N \Cart^\opp\,,
\end{align}
respectively, and analogously for $\Fix^\scD$.
Note that the morphisms
\begin{align}
	\Conf^{\scD[\cG]} \longrightarrow \Fix^{\scD[\cG]}
	\quad \text{and} \quad
	\Conf^M \longrightarrow \Fix^M\,,
\end{align}
are again objectwise Kan fibrations.
They induce canonical left fibrations
\begin{align}
\label{eq:LFibs CONF --> G}
	\textint \Conf^\scD
	\longrightarrow \textint \Fix^\scD\,,
	\qquad
	\textint \Conf^{\scD[\cG]}
	\longrightarrow \textint \Fix^{\scD[\cG]}\,,
	\qquad
	\textint \Conf^M
	\longrightarrow \textint \Fix^M\,,
\end{align}
over $N\rmB\scD^\opp$, $N\rmB\scD[\cG]^\opp$ and $N\Cart^\opp$, respectively.

\begin{definition}
\label{def:solution presheaf}
Let $p \colon \Conf^\scD \to \Fix^\scD$ be a projective fibration.
A \textit{solution presheaf} is a simplicial subpresheaf
\begin{equation}
	\Sol^\scD \subset \Conf^\scD
\end{equation}
on $\rmB\scD$ such that, for each $c \in \rmB\scD$, the inclusion $\Sol^\scD(c) \hookrightarrow \Conf^\scD(c)$ is an inclusion of a disjoint union of connected components.
\end{definition}

\begin{remark}
We think of the vertices $\Sol^\scD$ as consisting of those configurations which satisfy a certain condition, such as being solutions to a set of equations, imposed on the vertices of $\Conf^\scD$; see Part~\ref{part:Higher U(1) gauge fields} for examples.
The fact that we require each $\Sol^\scD(c) \hookrightarrow \Conf^\scD(c)$ to be an inclusion of connected components is equivalent to demanding that whether a family is a solution depends only on its isomorphism class:
if $\psi \in \Conf^\scD_1(c)$ is a 1-simplex, representing an equivalence between two families of configurations $d_1 \psi$ and $d_0 \psi$, then $d_1 \psi$ is a solution if and only if $d_0 \psi$ is so.
As for $\Conf^\scD$, the functoriality over $\rmB\scD$ implements that families of solutions are stable under restricting or reparameterising families, as well as pulling back along smooth families of diffeomorphisms on $M$.
\qen
\end{remark}

Consider a configuration presheaf $\Conf^\scD$ on $\rmB\scD$ and any solution subpresheaf $\Sol^\scD \subset \Conf^\scD$.
Being an inclusion of connected components, the inclusion morphism \smash{$\Sol^\scD \subset \Conf^\scD$} is an objectwise monomorphism as well as a projective fibration.
On the level of left fibrations this gives rise to a morphism
\begin{equation}
	\textint \Sol^\scD \hookrightarrow \textint \Conf^\scD
\end{equation}
(in \smash{$\sSet_{/N\rmB\scD^\opp}$}) on the associated left fibrations which is both a monomorphism (Lemma~\ref{st:r_C^* preserves injective cofibrations}) and a left fibration (Theorem~\ref{st:r_C^* as Quillen equivalence}).
Suppose that $p \colon \Conf^\scD \to \Fix^\scD$ is a configuration presheaf over $\Fix^\scD$.
Composing with the top left fibration from~\eqref{eq:LFibs CONF --> G}, we obtain a left fibration
\begin{equation}
	\textint \Sol^\scD \longrightarrow \textint \Fix^\scD\,.
\end{equation}

\begin{definition}
\label{def:tint Sol^(D[G])(cG)}
Let $\cG \in \Fix^\scD(\RR^0)$ be a global section of $e_M^*\Fix^M$.
We define the simplicial set \smash{$\textint \Conf^{\scD[\cG]}(\cG)$} and \smash{$\textint \Sol^{\scD[\cG]}(\cG)$} by demanding that the squares in the diagram
\begin{equation}
\label{eq:tint Sol^(D[G])(cG)}
\begin{tikzcd}[column sep=1.5cm]
	\textint \Sol^{\scD[\cG]}(\cG) \ar[r, hookrightarrow] \ar[d, hookrightarrow]
	& \textint \Conf^{\scD[\cG]}(\cG) \ar[r] \ar[d, hookrightarrow]
	& \textint \Fix^{\scD[\cG]}_{[\cG]} \ar[d, hookrightarrow]
	\\
	\textint \Sol^{\scD[\cG]} \ar[r, hookrightarrow]
	& \textint \Conf^{\scD[\cG]} \ar[r]
	& \textint \Fix^{\scD[\cG]}
\end{tikzcd}
\end{equation}
be cartesian.
\end{definition}

The vertical arrows in diagram~\eqref{eq:tint Sol^(D[G])(cG)} are inclusions on full simplicial subsets on certain vertices, and the horizontal arrows are left fibrations (by Theorem~\ref{st:r_C^* as Quillen equivalence}).
The resulting left fibrations
\begin{equation}
	\textint \Conf^{\scD[\cG]}(\cG) \longrightarrow N\rmB\scD[\cG]^\opp
	\qquad \text{and} \qquad
	\textint \Sol^{\scD[\cG]}(\cG) \longrightarrow N\rmB\scD[\cG]^\opp
\end{equation}
encode smooth families of configurations and solutions, respectively, whose image in $\textint \Fix^{\scD[\cG]}$ is equivalent to the constant family $\cG$.

\begin{remark}
By definition, the simplicial subpresheaf \smash{$\Fix^{\scD[\cG]}_{[\cG]} \subset \Fix^{\scD[\cG]}$} depends only on the equivalence class $[\cG]$ of $\cG$ in $\Fix^M(\RR^0)$.
Hence, so does the simplicial subset \smash{$\textint \Fix^{\scD[\cG]}(\cG) \subset \textint \Fix^{\scD[\cG]}$}.
It follows that also the simplicial subsets
\begin{equation}
	\textint \Sol^{\scD[\cG]}(\cG) \subset \textint \Sol^{\scD[\cG]}
	\qquad \text{and} \qquad
	\textint \Conf^{\scD[\cG]}(\cG) \subset \textint \Conf^{\scD[\cG]}
\end{equation}
depend only on the equivalence class of of $\cG$ in $\Fix^M(\RR^0)$.
\qen
\end{remark}

\begin{definition}
\label{def:objwise pi_0-surjective}
Let $\scC$ be a small category, and let $q \colon F_0 \to F_1$ be a projective fibration between projectively fibrant%
\footnote{One can also define 0-connectedness for generic morphisms $q \colon F_0 \to F_1$, by requiring that all homotopy fibres of \smash{$q_{|c}$} are connected.
For $q$ as in Definition~\ref{def:objwise pi_0-surjective}, the strict fibres already compute the homotopy fibres.}
simplicial presheaves on $\scC$.
We say that $q$ is \textit{0-connected} if, for each $c \in \scC$, the map
\begin{equation}
	\pi_0 q_{|c} \colon \pi_0 F_0(c) \longrightarrow \pi_0 F_1(c)
\end{equation}
induced by $q$ on connected components is a bijection.
\end{definition}

\begin{remark}
Equivalently, a left fibration between fibrant objects in $\sSet_{N\scC^\opp}$,
\begin{equation}
	q \colon X_0 \longrightarrow \textint X_1
\end{equation}
is called \textit{0-connected} if, for each $c \in \scC$, the restriction of $q$ to the fibres over $c$ induces a bijection between the connected components of the fibres.
\qen
\end{remark}

\begin{remark}
\label{rmk:(-)_[G] and isos on conn comps}
Let $q \colon G_0^\scD \to G_1^\scD$ be a morphism of projectively fibrant simplicial presheaves on $\rmB\scD$ which is 0-connected.
Let $\cG_0 \in G^M(\RR^0)$, and set $\cG_1 = q(\cG_0) \in G^M(\RR^0)$.
In this case, we have that
\begin{equation}
	\Diff_{[\cG_0]}(M) = \Diff_{[\cG_1]}(M)
	\qquad \text{and thus} \qquad
	\rmB\scD[\cG_0] = \rmB\scD[\cG_1]\,.
\end{equation}
Therefore, it makes sense to write
\begin{equation}
	G_{0 [\cG_1]}^{\scD[\cG_1]} \coloneqq G_{0 [\cG_0]}^{\scD[\cG_0]}\,,
\end{equation}
where \smash{$[\cG_0] \coloneqq (\pi_0 q)^{-1}([\cG_1]) \in \pi_0 G_0^{\scD_{[\cG_1]}}$}.
The morphism $p$ restricts to a morphism
\begin{equation}
	q \colon G_{0 [\cG_1]}^{\scD[\cG_1]} \longrightarrow G_{1 [\cG_1]}^{\scD[\cG_1]}\,,
\end{equation}
which we denote again by $q$ for ease of notation.
\qen
\end{remark}

\begin{remark}
Note that if $q \colon F_0 \to F_1$ is such that the induced morphism $\pi_0 q \colon \pi_0 F_0 \to \pi_0 F_1$ is only injective, we can replace $F_1$ by the full simplicial subpresheaf $F'_1 \subset F_1$ which assigns to $c \in \scC$  the subset $F'_1(c) \subset F_1(c)$ consisting of all those connected components of $F_1(c)$ which contain a vertex in the image of the map $q_{|c} \colon F_0(c) \to F_1(c)$.
\qen
\end{remark}

Consider a configuration presheaf $p \colon \Conf^\scD \longrightarrow \Fix_0^\scD$ over $\Fix_0^\scD$.
Let
\begin{equation}
	q \colon \Fix_0^\scD \longrightarrow \Fix_1^\scD
\end{equation}
be a projective fibration of projectively fibrant simplicial presheaves on $\rmB\scD$.
Then, we have a sequence of morphisms
\begin{equation}
\begin{tikzcd}
	\Sol^\scD \ar[r, hookrightarrow]
	& \Conf^\scD \ar[r] 
	& \Fix_0^\scD \ar[r, "q"]
	& \Fix_1^\scD\,,
\end{tikzcd}
\end{equation}
and the composite $\Conf^\scD \to \Fix_1^\scD$ is also a projective fibration between projectively fibrant simplicial presheaves on $\rmB\scD^\opp$.
In particular, we can consider families of solutions relative to fixed data either in $\Fix_0^\scD$, or in $\Fix_1^\scD$.
Further suppose that $q$ is 0-connected.
Let $\cG_0 \in \Fix_0^M(\RR^0)$ be a global section of $e_M^*\Fix_0^M$, and set $\cG_1 = q(\cG_0) \in \Fix_1^M(\RR^0)$.
By Remark~\ref{rmk:(-)_[G] and isos on conn comps} we may write \smash{$\Fix_{0 [\cG_1]}^{\scD[\cG_1]} = \Fix_{0 [\cG_0]}^{\scD[\cG_0]}$}.
The first of two key observations in this section is:

\begin{proposition}
\label{st:SOL^(n,D_cG(M))(cG,cA^(k)) depends only on [cG]}
In the above set-up there is an identity of left fibrations over $N\rmB\scD[\cG_0]^\opp = N\rmB\scD[\cG_1]^\opp$
\begin{equation}
	\textint \Sol^{\scD[\cG_0]}(\cG_0)
	= \textint \Sol^{\scD[\cG_1]}(\cG_1)\,.
\end{equation}
\end{proposition}

\begin{proof}
It suffices to show that
\begin{equation}
	\textint \Conf^{\scD[\cG_0]}(\cG_0)
	= \textint \Conf^{\scD[\cG_1]}(\cG_1)\,.
\end{equation}
The claim then follows from the construction of $\textint \Sol^{\scD[\cG_i]}(\cG_i)$ as the left-hand pullback square in~\eqref{eq:tint Sol^(D[G])(cG)}.
In fact, since by construction \smash{$\textint \Conf^{\scD[\cG_i]}(\cG_i) \subset \textint \Conf^{\scD[\cG_i]}$} is a \textit{full} simplicial subset on a certain subset of vertices, it suffices to check the above identity on the level of vertices.

To that end, first note that $\textint \Conf^{\scD[\cG_0]} = \textint \Conf^{\scD[\cG_1]}$.
For $i = 0,1$, a vertex $x \in \textint \Conf^{\scD[\cG_i]}$ lies in the simplicial subset \smash{$\textint \Conf^{\scD[\cG_i]}(\cG_i) \subset \textint \Conf^{\scD[\cG_i]}$} precisely if its image under $\textint p \colon \textint \Conf^\scD \to \textint \Fix_i^\scD$ lies in the essential image of the functor \smash{$\sigma \cG_i \colon N\Cart^\opp \to \textint \Fix_i^{\scD[\cG_i]}$}, which is associated to the section $\cG_i$ by Remark~\ref{rmk:section associated to cG}.
By construction, we have that
\begin{equation}
	\textint q \circ \sigma \cG_0 = \sigma \cG_1\,.
\end{equation}
Thus, whenever $\textint p(x)$ is in the essential image of $\sigma \cG_0$, then $\textint q \circ \textint p(x)$ lies in the essential image of $\sigma \cG_1$.
That is
\begin{equation}
	\textint \Conf^{\scD[\cG_0]} (\cG_0) \subset \textint \Conf^{\scD[\cG_1]} (\cG_1)\,.
\end{equation}

We claim that the converse holds true as well:
assume that $x \in \textint \Conf^{\scD[\cG_0]}$ lies in the essential image of $\sigma \cG_1$.
We have to show that $\textint p(x)$ lies in the essential image of $\sigma \cG_0$.
Let $\alpha$ be a 1-simplex establishing an equivalence from $\sigma \cG_1(c')$ to $\textint q \circ \textint p(x)$, for some $c' \in \Cart$.
Let $c$ be the image in $\Cart$ of $x$.

First we show the auxiliary claim that we can always choose $c' = c$:
consider the commutative square
\begin{equation}
\begin{tikzcd}
	\Lambda^0_1 \ar[r] \ar[d, hookrightarrow]
	& \textint \Conf^{\scD[\cG_0]} \ar[d]
	\\
	\Delta^2 \ar[r]
	& N\Cart^\opp
\end{tikzcd}
\end{equation}
where the bottom morphism is the image of $\alpha$ composed with the identity on $c$.
The top morphism is given by the morphism $\sigma \cG_1(c') \to \sigma \cG_1(c)$ obtained by applying the map $\sigma \cG_1$ to the image of $\alpha$ in $\Cart^\opp$ (this yields one leg of the horn), together with the 1-simplex $\alpha$ itself.
Since the right-hand vertical morphism is a left fibration, the above square admits a lift $\beta \colon \Delta^2 \to \textint \Conf^{\scD[\cG_0]}$.
Then, $\alpha' \coloneqq d_0 \beta \colon \Delta^1 \to \textint \Conf^{\scD[\cG_0]}$ provides a 1-simplex
\begin{equation}
	\alpha' \colon \sigma \cG_1(c) \to \textint q \circ \textint p(x)
\end{equation}
in the fibre
\begin{equation}
	\big( \textint q \circ \textint p(x) \big)_{|c}
	\cong \Conf^{\scD[\cG_0]}(c)\,.
\end{equation}
Here we have used the canonical isomorphism $(r_\scC^*F)_{|c} \cong F(c)$ for any small category $\scC$, object $c \in \scC$ and functor $F \colon \scC \to \sSet$ (this arises from Lemma~\ref{st:r_C^* and pullbacks}).
This proves the auxiliary claim.

Consequently, we have that $\textint q \circ \textint p(x)$ lies in the same connected component of $\Fix_1^\scD(c)$ as $\sigma \cG_1(c) = \textint q \circ \sigma \cG_0(c)$.
Since the morphism $\textint q$ induces a bijection
\begin{equation}
	\pi_0 \big( \textint q \big)_{|c} \colon \pi_0 \big( \textint \Fix_0^\scD \big)_{|c} \longrightarrow \pi_0 \big( \textint \Fix_1^\scD \big)_{|c}\,,
\end{equation}
it now follows that also $\textint p (x)$ lies in the same connected component of $\Fix_0^\scD(c)$ as $\sigma \cG_0(c)$.
In particular, we deduce that $x \in \textint \Conf^{\scD[\cG_0]}(\cG_0)$.
Therefore, we have that
\begin{equation}
	\textint \Conf^{\scD[\cG_1]} (\cG_1) \subset \textint \Conf^{\scD[\cG_0]} (\cG_0)\,,
\end{equation}
which completes the proof.
\end{proof}

We also define the restriction
\begin{equation}
	\Sol^M \coloneqq e_M^* \Sol^\scD \colon \Cart^\opp \longrightarrow \sSet
\end{equation}
and its associated left fibration
\begin{equation}
	\textint \Sol^M \coloneqq r_{\Cart^\opp}^* \Sol^M \longrightarrow N\Cart^\opp\,.
\end{equation}
Note that by Lemma~\ref{st:r_C^* and pullbacks} there is a canonical isomorphism
\begin{equation}
	\textint \Sol^M \cong Ne_M^* \textint \Sol^\scD = Ne_M^* \textint \Sol^{\scD[\cG]}
\end{equation}
of left fibrations over $N\Cart^\opp$.
In analogy to construction~\eqref{eq:tint Sol^(D[G])(cG)}, we define two simplicial subsets \smash{$\textint \Conf^M(\cG) \subset \textint \Conf^M$} and \smash{$\textint \Sol^M(\cG) \subset \textint \Sol^M$} by demanding that the squares in the diagram
\begin{equation}
\label{eq:tint Sol^M(cG)}
\begin{tikzcd}[column sep=1.5cm]
	\textint \Sol^M(\cG) \ar[r, hookrightarrow] \ar[d, hookrightarrow]
	& \textint \Conf^M(\cG) \ar[r] \ar[d, hookrightarrow]
	& \textint \Fix^M_{[\cG]} \ar[d, hookrightarrow]
	\\
	\textint \Sol^M \ar[r, hookrightarrow]
	& \textint \Conf^M \ar[r]
	& \textint \Fix^M
\end{tikzcd}
\end{equation}
be cartesian.
It follows that there is a canonical isomorphism
\begin{equation}
\label{eq:nE_M^*SOL(cG,cA^(k))}
	Ne_M^* \textint \Sol^{\scD[\cG]}(\cG)
	\cong \textint \Sol^M(\cG)\,.
\end{equation}
We have the following direct consequence:

\begin{proposition}
\label{st:SOL fibseq in sSet}
There is a commutative diagram
\begin{equation}
\label{eq:tint Sol^M(cG)}
\begin{tikzcd}[column sep=1.5cm, row sep=0.75cm]
	\textint \Sol^M(\cG) \ar[r] \ar[d, hookrightarrow]
	& \textint \Fix^M_{[\cG]} \ar[d, hookrightarrow] \ar[r]
	& N\Cart^\opp \ar[d, hookrightarrow, "Ne_M"]
	\\
	\textint \Sol^{\scD[\cG]}(\cG) \ar[r]
	& \textint \Fix^{\scD[\cG]}_{[\cG]} \ar[r]
	& N\rmB\scD[\cG]^\opp
\end{tikzcd}
\end{equation}
in which both squares are cartesian and all horizontal arrows are left fibrations.
In particular, both squares are homotopy cartesian in $\sSet_{/N \rmB \scD[\cG]^\opp}$.
\end{proposition}

So far, we have studied the maximal and minimal smooth group actions on $M$ which have objectwise lifts to $\cG$, i.e.~the actions of $\Diff_{[\cG]}(M)$ and the trivial group, respectively.
Now we extend this to provide smooth higher symmetry groups for generic smooth higher group actions on $M$.
Suppose we are given the following data:
\begin{itemize}
\item a projectively fibrant simplicial presheaf $\Fix^\scD$ on $\rmB\scD$,

\item a section $\cG \in e_M^*\Fix^\scD(\RR^0)$,

\item a left fibration $Q \to N\Cart^\opp$ with reduced fibres (i.e.~the Kan complex $Q_{|c}$ has a single vertex, for each $c \in \Cart$),

\item a morphism $\phi \colon Q \to N\rmB\scD[\cG]^\opp$ in $\sSet_{/N \Cart^\opp}$.
\end{itemize}
In this situation, recall the construction of $\rmB\SYM^\rev_\phi(\cG)$ from Definition~\ref{def:BSYM_(Q,phi)(G)} (see also Definition~\ref{def:BGAU(G), BSYM(G)}).

\begin{definition}
\label{def:tint Sol_phi(cG), Sol_res(cG)}
We define simplicial sets as the pullbacks
\begin{equation}
\label{eq:def tint Sol_phi(cG), Sol_res(cG)}
\begin{tikzcd}[column sep=1cm]
	\textint \Sol^{\scD[\cG]}_\phi(\cG) \ar[r, hookrightarrow] \ar[d]
	& \phi^* \textint \Sol^{\scD[\cG]}(\cG) \ar[d]
	& \textint \Sol^M_\res(\cG) \ar[r, hookrightarrow] \ar[d]
	& \textint \Sol^M(\cG) \ar[d]
	\\
	\rmB\SYM^\rev_\phi(\cG) \ar[r, hookrightarrow]
	& \phi^* \textint \Fix^{\scD[\cG]}_{[\cG]}
	& \rmB\AUT^\rev(\cG) \ar[r, hookrightarrow]
	& \textint \Fix^M_{[\cG]}
\end{tikzcd}
\end{equation}
\end{definition}

\begin{remark}
Consider the case $Q = N\rmB\scD[\cG]^\opp$ and $\phi$ the identity morphism.
In this case the left-hand square in~\eqref{eq:def tint Sol_phi(cG), Sol_res(cG)} becomes
\begin{equation}
\begin{tikzcd}[column sep=1cm]
	\textint \Sol^{\scD[\cG]}_\id(\cG) \ar[r, hookrightarrow] \ar[d]
	& \textint \Sol^{\scD[\cG]}(\cG) \ar[d]
	\\
	\rmB\SYM^\rev(\cG) \ar[r, hookrightarrow]
	& \textint \Fix^{\scD[\cG]}_{[\cG]}
\end{tikzcd}
\end{equation}
Note that, in general, the top horizontal arrow is a strict inclusion.
\qen
\end{remark}

\begin{remark}
\label{rmk:Sol_res in terms of Sol}
Consider the case $Q = N\Cart^\opp$ and $\phi = N e_M$; this corresponds to the case where the smooth higher group represented by the left fibration $Q \to N\Cart^\opp$ is the trivial group.
By Proposition~\ref{st:SOL fibseq in sSet} there is a canonical isomorphism
\begin{equation}
	\textint \Sol^{\scD[\cG]}_{Ne_M}(\cG)
	\cong \textint \Sol^M_\res(\cG)
\end{equation}
of left fibrations over $N\Cart^\opp$.
In particular, it follows from the pasting law for pullbacks that the square
\begin{equation}
\label{eq:SOL_res square}
\begin{tikzcd}[column sep=1cm, row sep=0.75cm]
	\textint \Sol^M_\res(\cG) \ar[d, hookrightarrow] \ar[r]
	& \rmB\AUT^\rev(\cG) \ar[d, hookrightarrow]
	\\
	\textint \Sol^{\scD[\cG]}_\phi(\cG) \ar[r]
	& \rmB\SYM^\rev_\phi(\cG)
\end{tikzcd}
\end{equation}
is a pullback square as well.
Hence, upon augmenting this square by the canonical maps to $N\Cart^\opp$, it becomes cartesian in $\sSet_{/N\Cart^\opp}$ (by Lemma~\ref{st:limits in slices and contractibility}).
\qen
\end{remark}

The following theorem is the second main result of this section.
Consider again the situation of Proposition~\ref{st:SOL^(n,D_cG(M))(cG,cA^(k)) depends only on [cG]}.
That is, consider projective fibrations $p \colon \Conf^\scD \to \Fix_0^\scD$ and $q \colon \Fix_0^\scD \to \Fix_1^\scD$ between projectively fibrant simplicial presheaves on $\rmB\scD$, and assume that $q$ is 0-connected (see Definition~\ref{def:objwise pi_0-surjective}).
Suppose that $\cG_0 \in \Fix_0^M(\RR^0)$ is a section and set $\cG_1 \coloneqq q_{|\RR^0} (\cG_0) \in \Fix_1^M(\RR^0)$.
Recall that $\rmB\scD[\cG_0] = \rmB\scD[\cG_1]$ in this case (see Remark~\ref{rmk:(-)_[G] and isos on conn comps}).
We obtain pullback squares
\begin{equation}
\begin{tikzcd}[column sep=1cm, row sep=0.75cm]
	\textint \Sol^M_\res(\cG_i) \ar[d, hookrightarrow] \ar[r]
	& \rmB\AUT^\rev(\cG_i) \ar[d, hookrightarrow] \ar[r]
	& N\Cart^\opp \ar[d, "Ne_M"]
	\\
	\textint \Sol^{\scD[\cG_i]}_\phi(\cG_i) \ar[r]
	& \rmB\SYM^\rev_\phi(\cG_i) \ar[r]
	& Q
\end{tikzcd}
\end{equation}
for each $i = 0,1$.

\begin{theorem}
\label{st:Solution zig-zag}
In the above situation there is a canonical zig-zag
\begin{equation}
\label{eq:Solution zig-zag}
\begin{tikzcd}
	\textint \Sol^{\scD[\cG_0]}_\phi(\cG_0) \ar[r, hookrightarrow]
	& \phi^* \textint \Sol^{\scD[\cG_0]}(\cG_0) \ar[r, equal]
	& \phi^* \textint \Sol^{\scD[\cG_1]}(\cG_1)
	& \textint \Sol^{\scD[\cG_1]}_\phi(\cG_1) \ar[l, hookrightarrow]
\end{tikzcd}
\end{equation}
of covariant weak equivalences between fibrant objects in $\sSet_{/Q}$.
In particular, each of these maps is a weak categorical equivalence.
\end{theorem}

Before proving Theorem~\ref{st:Solution zig-zag} we show the following proposition:

\begin{proposition}
\label{st:SOL_res --> SOL is equivalence}
The squares in~\eqref{eq:def tint Sol_phi(cG), Sol_res(cG)} are homotopy cartesian in the covariant model structure on $\sSet_{/Q}$ and $\sSet_{/N \Cart^\opp}$, respectively.
Its horizontal morphisms are covariant weak equivalences over $Q$ and $N\Cart^\opp$, respectively.
\end{proposition}

\begin{proof}
We show the claim for the left-hand diagram; the case of the right-hand diagram is analogous.
The cospan diagram underlying this pullback square has the following properties:
(1) each of its objects is a left fibration over $Q$, (2) its vertical arrows are left fibrations, and (3) its bottom horizontal morphism is a covariant weak equivalence over $Q$.
Each of these properties follow from the fact that the functor
\begin{equation}
	\phi^* \colon \sSet_{/N \rmB \scD[\cG]^\opp} \longrightarrow \sSet_{/Q}
\end{equation}
is right Quillen with respect to the covariant model structures.
The squares are cartesian in $\sSet$ an thus, by Lemma~\ref{st:limits in slices and contractibility}, also in $\sSet_{/Q}$.
By the dual statement of~\cite[Prop.~A.2.4.4]{Lurie:HTT}, properties (1) and (2) then imply that the pullback square under consideration is homotopy cartesian in $\sSet_{/Q}$.
Combining this with property (3), we obtain that its top morphism is a weak equivalence in the covariant model structure as well.
\end{proof}

\begin{proof}[Proof of Theorem~\ref{st:equiv result for Mdl oo-stacks}]
The identity in the middle of~\eqref{eq:Solution zig-zag} is a direct consequence of Proposition~\ref{st:SOL^(n,D_cG(M))(cG,cA^(k)) depends only on [cG]}.
The statement for the left and right inclusion morphisms follows from Proposition~\ref{st:SOL_res --> SOL is equivalence}.
Finally, Lemma~\ref{st:cov weqs between LFibs are exactly cat weqs} (together with the fact that $Q$ is an $\infty$-category) implies that the morphisms in~\eqref{eq:Solution zig-zag} are also categorical weak equivalences.
\end{proof}

For later convenience we also record the following results:

\begin{corollary}
\label{st:SOL_res square is hoCart}
The square~\eqref{eq:SOL_res square}, augmented by the canonical maps to $N\Cart^\opp$, is homotopy cartesian in the covariant model structure on $\sSet_{/N \Cart^\opp}$.
\end{corollary}

\begin{proof}
The square is a pullback in $\sSet$.
Hence, its augmentation is a pullback square in $\sSet_{/N \Cart^\opp}$ by Lemma~\ref{st:limits in slices and contractibility} (actually only using its 1-categorical version; see Remark~\ref{rmk:limits in 1-Cat slices and connectedness}).
Each of its vertices is fibrant in the covariant model structure on $\sSet_{/N \Cart^\opp}$, and the horizontal morphisms are left fibrations; thus, the square is indeed homotopy cartesian in the covariant model structure on $\sSet_{/N \Cart^\opp}$.
\end{proof}

\begin{corollary}
\label{st:SOL fibseq in sSet_(/NCart^op)}
Both squares in the diagram
\begin{equation}
\begin{tikzcd}[column sep=1cm, row sep=0.75cm]
	\textint \Sol^M_\res(\cG) \ar[d, hookrightarrow] \ar[r]
	& \rmB\AUT^\rev(\cG) \ar[d, hookrightarrow] \ar[r]
	& N\Cart^\opp \ar[d, "Ne_M"]
	\\
	\textint \Sol^{\scD[\cG]}_\phi(\cG) \ar[r]
	& \rmB\SYM^\rev_\phi(\cG) \ar[r]
	& Q
\end{tikzcd}
\end{equation}
are cartesian and homotopy cartesian in $\sSet_{/N \Cart^\opp}$.
It follows that the outer square is both cartesian and homotopy cartesian in $\sSet_{/N \Cart^\opp}$ as well.
\end{corollary}

\begin{proof}
The left-hand square has both properties by Corollary~\ref{st:SOL_res square is hoCart}.
The right-hand square is cartesian and homotopy cartesian by Lemma~\ref{st:BAUT BSYM_phi fib seq}.
For the outer square, the claim follows from the pasting law for (homotopy) pullbacks.
\end{proof}

%%%%%%%%%%%%%%%%%%%%%%%%%%%%%%%%%%%%%%%%%%%%%%%%%%%%%%%%%%%%%%%%%%%%%%%%%%%%

\subsection{Solutions modulo symmetries and automorphisms}

%%%%%%%%%%%%%%%%%%%%%%%%%%%%%%%%%%%%%%%%%%%%%%%%%%%%%%%%%%%%%%%%%%%%%%%%%%%%

In the previous section we have investigated left fibrations of smooth families of solutions.
We now consider their classifying $\infty$-presheaves and the moduli stacks of solutions obtained by dividing out the actions of automorphisms and symmetries of higher geometric structures in a homotopy coherent way.
In particular, we prove a criterion for when two moduli $\infty$-prestacks in our formalism are equivalent.

The left fibration
\begin{equation}
	\textint \Sol^{\scD[\cG]} \longrightarrow \textint \Fix^{\scD[\cG]}
\end{equation}
is classified by a unique $\infty$-functor
\begin{equation}
	\scSol^{\scD[\cG]} \colon \textint \Fix^{\scD[\cG]} \longrightarrow \scS\,.
\end{equation}
Given a section $\cG \in \Fix^M(\RR^0)$, we let
\begin{equation}
	\scSol^{\scD[\cG]}(\cG) \colon \textint \Fix^{\scD[\cG]}_{[\cG]} \longrightarrow \scS
\end{equation}
denote the restriction of $\scSol^{\scD[\cG]}$ to \smash{$\textint \Fix^{\scD[\cG]}_{[\cG]}$}; this $\infty$-functor classifies the left fibration
\begin{equation}
	\textint \Sol^{\scD[\cG]}(\cG) \longrightarrow \textint \Fix^{\scD[\cG]}_{[\cG]}\,.
\end{equation}

As before, let $\phi \colon Q \to \rmB\scD[\cG]^\opp$ be a morphism of left fibrant objects over $N\Cart^\opp$, where $Q \to N\Cart^\opp$ has reduced fibres.
Recall from the discussion before Definition~\ref{def:BSYM_(Q,phi)(G)} that these data present a smooth $\infty$-group action $\Phi \colon \bbGamma \to \Diff_{[\cG]}(M)$.
Let
\begin{equation}
	\iota_\phi \colon \rmB\SYM^\rev_\phi(\cG) \hookrightarrow \phi^* \textint \Fix^{\scD[\cG]}_{[\cG]}
\end{equation}
be the canonical inclusion, and write $\pi_\phi \colon \rmB\SYM_\phi^\rev(\cG) \to N\Cart^\opp$ for the canonical projection morphism, i.e.~$\pi_\phi \coloneqq N\pi_M \circ \phi \circ \psi_\phi$ in the notation of Definition~\ref{def:BSYM_(Q,phi)(G)}.
Finally, we let
\begin{equation}
\label{eq:scSol_phi(cG) as oo-functor}
	\scSol_\phi(\cG) \colon \rmB\SYM^\rev_\phi(\cG) \longrightarrow \scS
\end{equation}
be the $\infty$-functor classifying the left fibration \smash{$\textint \Sol^{\scD[\cG]}_\phi(\cG) \longrightarrow \rmB\SYM^\rev_\phi(\cG)$}.

\begin{remark}
\label{rmk:scSol_phi(cG) as oo-action}
The $\infty$-functor~\eqref{eq:scSol_phi(cG) as oo-functor} encodes a right action in the $\infty$-topos $\scP(N\Cart)$ of the group object $\scSym_\Phi(\cG)$ on the object
\begin{equation}
	(\sigma \cG)^* \scSol^{\scD[\cG]}(\cG) = (\tilde{\sigma}\cG)^* \scSol^M(\cG)
	\quad \in \scP(N\Cart)\,,
\end{equation}
where $\sigma \cG \colon N\Cart \to \textint G^{\scD[\cG]}$ and $\tilde{\sigma}\cG \colon N\Cart \to \textint G^M$ are the morphisms from Remark~\ref{rmk:section associated to cG}.
In particular, if $\phi = Ne_M$ presents the action of the trivial group, the $\infty$-functor $\scSol_{Ne_M}(\cG) \colon \rmB\AUT^\rev(\cG) \to \scS$ encodes a smooth action of the group object $\scAut(\cG)$ in $\scP(N\Cart)$ on the object $(\tilde{\sigma}\cG)^* \scSol^M(\cG) \in \scP(N\Cart)$.
\qen
\end{remark}

\begin{definition}
\label{def:scMdl_phi(cG)}
The \textit{moduli prestack $\scMdl_\Phi(\cG)$ of solutions on $\cG$ modulo the action of $\scSym_\Phi(\cG)$} is the $\infty$-categorical left Kan extension of \smash{$\scSol_\Phi(\cG)$} along the left fibration
\begin{equation}
	\pi_\phi \colon \rmB\SYM^\rev_\phi(\cG) \longrightarrow N\Cart^\opp\,.
\end{equation}
That is, we define the $\infty$-presheaf
\begin{equation}
	\scMdl_\Phi(\cG) \coloneqq (\pi_\phi)_!\, \scSol_\Phi(\cG)
	\colon N\Cart^\opp \longrightarrow \scS\,.
\end{equation}
\end{definition}

The left Kan extension along $\pi_\phi$ takes the quotient of $\scSol^M(\cG)(c)$ by the action of $\scSym_\Phi(\cG)(c)$, for each $c \in \Cart$, functorially in $c$.

\begin{proposition}
\label{st:left fibration classd by Mdl_phi}
The $\infty$-presheaf $\scMdl_\Phi(\cG)$ classifies the composed left fibration
\begin{equation}
\begin{tikzcd}
	\textint \Sol_\phi^{\scD[\cG]}(\cG) \ar[r]
	& \rmB\SYM^\rev_\phi(\cG) \ar[r, "\pi_\phi"]
	& N\Cart^\opp\,.
\end{tikzcd}
\end{equation}
\end{proposition}

\begin{proof}
This follows readily from~\cite[Prop.~6.1.14]{Cisinski:HiC_and_HoA}.
\end{proof}

\begin{remark}
Consider the following two limiting cases:
\begin{enumerate}
\item For $Q = \rmB\SYM^\rev(\cG)$ and $\phi = 1_{\rmB\SYM^\rev(\cG)}$, the resulting moduli prestack $\scMdl_\id(\cG)$ describes smooth families of solutions on $\cG$ modulo smooth families of \textit{symmetries} of $\cG$.

\item For $Q = N\Cart^\opp$ and $\phi = N e_M$, the resulting moduli prestack $\scMdl_{Ne_M}(\cG)$ describes smooth families of solutions on $\cG$ modulo smooth families of \textit{automorphisms} of $\cG$ (i.e.~those symmetries which cover the identity diffeomorphism of $M$).
The $\infty$-functor
\begin{equation}
	(\psi_{Ne_M})_! \scSol_{Ne_M}(\cG) \colon \rmB\AUT^\rev(\cG) \longrightarrow \scS
\end{equation}
classifies the left fibration \smash{$\textint \Sol^M_\res(\cG) \longrightarrow \rmB\AUT^\rev(\cG)$}.
\qen
\end{enumerate}
\end{remark}

With this preparation and the constructions in the previous sections, we can now state and prove the main theorem of this section:
consider again the situation of Proposition~\ref{st:SOL^(n,D_cG(M))(cG,cA^(k)) depends only on [cG]} and Theorem~\ref{st:Solution zig-zag}:
let \smash{$p \colon \Conf^\scD \longrightarrow \Fix_0^\scD$} and \smash{$q \colon \Fix_0^\scD \to \Fix_1^\scD$} be projective fibrations between projectively fibrant simplicial presheaves on $\rmB\scD$.
Suppose that $q$ is $0$-connected (see Definition~\ref{def:objwise pi_0-surjective}).
Further, suppose that $\cG_0 \in \Fix_0^M(\RR^0)$ is a section and set $\cG_1 \coloneqq q_{|\RR^0} (\cG_0) \in \Fix_1^M(\RR^0)$.

\begin{theorem}
\label{st:equiv result for Mdl oo-stacks}
In the above situation there is a canonical equivalence in $\scP(N\Cart)$
\begin{equation}
	\scMdl_\Phi(\cG_0)
	\simeq \scMdl_\Phi(\cG_1)\,.
\end{equation}
\end{theorem}

We can think of the two moduli $\infty$-prestacks $\scMdl_\Phi(\cG_0)$ and $\scMdl_\Phi(\cG_1)$ as describing smooth families of geometric data on $M$, where in each case we keep part of the data fixed.
However, which part we keep fixed differs between the two cases, and consequently also the data we vary differs.
In general, this lead to different moduli problems with different moduli stacks.
One interpretation of Theorem~\ref{st:equiv result for Mdl oo-stacks} is then that, in this situation, the change in the data we vary (including its morphisms and higher morphisms) is compensated for by a simultaneous change in the smooth higher symmetry groups of the fixed data, leading to equivalent homotopy quotients.
See Sections~\ref{sec:examples with higher U(1)-connections} and~\ref{sec:GRic solitons} for concrete examples and applications.

\begin{proof}[Proof of Theorem~\ref{st:equiv result for Mdl oo-stacks}.]
By Proposition~\ref{st:left fibration classd by Mdl_phi} the $\infty$-functor $\scMdl_\Phi(\cG_i)$ classifies the left fibration
\begin{equation}
	\textint \Sol_\phi^{\scD[\cG_i]}(\cG_i) \longrightarrow N\Cart^\opp\,,
\end{equation}
for $i = 0,1$.
Since $\scP(N\Cart)$ is the $\infty$-categorical localisation of $\sSet_{/N \Cart^\opp}$ at the covariant weak equivalences~\cite[Thm.~7.8.9]{Cisinski:HiC_and_HoA}, the canonical zig-zag of covariant weak equivalences in Theorem~\ref{st:Solution zig-zag} induces a canonical equivalence
\begin{equation}
	\scMdl_\Phi(\cG_0)
	\simeq \scMdl_\Phi(\cG_1)
\end{equation}
in $\scP(N\Cart)$, as claimed.
\end{proof}

%%%%%%%%%%%%%%%%%%%%%%%%%%%%%%%%%%%%%%%%%%%%%%%%%%%%%%%%%%%%%%%%%%%%%%%%%%%%

\subsection{The underlying space of an $\infty$-presheaf on $\Cart$}
\label{sec:underlying spaces}

%%%%%%%%%%%%%%%%%%%%%%%%%%%%%%%%%%%%%%%%%%%%%%%%%%%%%%%%%%%%%%%%%%%%%%%%%%%%

In this section we include some background on the \textit{underlying space} of an $\infty$-presheaf $\bbF \in \scP(N\Cart)$ (it also goes by the names \textit{smooth singular complex}~\cite{Bunk:R-loc_HoThy}, \textit{concordance space}~\cite{BEBdBP:Classifying_spaces_of_oo-sheaves}, or \textit{shape}~\cite{SS:Equivar_pr_infty-bundles} of $\bbF$).
Moreover, it interacts well with principal $\infty$-bundles and therefore $\infty$-categorical group extensions in the $\infty$-topos $\scP(N\Cart)$, as we recall here.
As a consequence, we can derive important information about the homotopy types of the smooth higher symmetry and automorphism groups of any higher geometric structure $\cG$ on $M$, as well as moduli $\infty$-prestacks of higher geometric data on $\cG$.

Let $\scX$ be an $\infty$-topos and $\bbH$ a group object in $\scX$ (see, for instance,~\cite{Lurie:HTT, NSS:Pr_ooBdls_I} for background).
A \textit{$\bbH$-principal $\infty$-bundle} in $\scX$ is an effective epimorphism $\bbP \to \bbX$ in $\scX$ together with a $\bbH$-action $\bbP \dslash \bbH$ on $\bbP$ such that the canonical morphism $\bbP \dslash \bbH \to \check{C}(\bbP \to \bbX)$ from the action groupoid to the \v{C}ech nerve of $\bbP \to \bbX$ is an equivalence of simplicial objects in $\scX$ (this is equivalent to the original definition in~\cite{NSS:Pr_ooBdls_I} by~\cite[Thm.~3.32]{Bunk:Pr_ooBdls_and_String}).
Here we will mostly be interested in the $\infty$-topoi $\scP(N\Cart)$ and $\scS$.

We also recall the underlying space $\infty$-functor, following mostly~\cite{Bunk:R-loc_HoThy, Bunk:Pr_ooBdls_and_String}:
let
\begin{equation}
	\Delta_e \colon \bbDelta \to \Cart\,,
	\quad
	[k] \mapsto \Delta_e^k = \big\{x \in \RR^{k+1}\, \big| \, \textstyle\sum_{i = 0}^k x^i = 1 \big\}
\end{equation}
be the \textit{extended simplex functor} (the maps between the $\Delta_e^k$ are defined in the same way as for the topological standard simplices $|\Delta^k|$).
Then, $S$ is the functor of $\infty$-categories
\begin{equation}
\label{eq:ul space fctr S}
	S \colon \scP(N\Cart) \to \scS\,,
	\qquad
	S(\bbF) = \underset{\bbDelta^\opp}{\colim}\ (N\Delta_e)^* \bbF\,.
\end{equation}
It is a left adjoint which preserves finite products, and therefore maps principal $\infty$-bundles in $\scP(N\Cart)$ to principal $\infty$-bundles in $\scS$~\cite[Thm.~3.48]{Bunk:Pr_ooBdls_and_String}.
As a further consequence, it maps extensions of group objects in $\scP(N\Cart)$ to extensions of group objects in $\scS$~\cite[Cor.~3.52]{Bunk:Pr_ooBdls_and_String}.

The functor $S$ is presented by the following left Quillen functor $S_Q$:
let $\delta \colon \bbDelta \to \bbDelta \times \bbDelta$ denote the diagonal functor, and recall the notation $\scH = \Fun(\Cart^\opp, \sSet)$ (Definition~\ref{def:scH and scH_rmfam}).
Then, the functor
\begin{equation}
\label{eq:S_Q, presenting S}
	S_Q \colon \scH \to \sSet\,,
	\qquad
	F \mapsto \delta^* (\Delta_e^* F)\,,
\end{equation}
presents the $\infty$-functor $S \colon \scP(N\Cart) \to \scS$~\cite{Bunk:Pr_ooBdls_and_String, Bunk:R-loc_HoThy}.

\begin{definition}
\label{def:underlying space of a presheaf}
For a simplicial presheaf $F \colon \Cart^\opp \to \sSet$, we call $S_Q(F) \in \sSet$ the \textit{underlying space} of $F$.
By a slight abuse of terminology, for an $\infty$-presheaf $\bbF \colon N\Cart^\opp \to \scS$, we call $S(\bbF) \in \scS$ the \textit{underlying space} of $\bbF$.
\end{definition}

We investigate some of the properties of $S$:

\begin{definition}
Two morphisms $f_0, f_1 \colon F \to G$ in $\scH$ are \textit{smoothly homotopic} whenever there exists a commutative diagram in $\scH$ of the form
\begin{equation}
\begin{tikzcd}
	F \times \Delta_e^{\{0\}} \ar[d, hookrightarrow] \ar[dr, bend left=20, "f_0"]
	&
	\\
	F \times \Delta_e^1 \ar[r, "h" description]
	& G
	\\
	F \times \Delta_e^{\{1\}} \ar[u, hookrightarrow] \ar[ur, bend right=20, "f_1"']
	&
\end{tikzcd}
\end{equation}
\end{definition}

The functor $S_Q$ maps smoothly homotopic morphisms in $\scH$ to simplicially homotopic morphisms in $\sSet$~\cite[Lemma~3.10]{Bunk:R-loc_HoThy}.
For a manifold $N$, let $\ul{N} \in \scH$ be the object defined by $\ul{N}(c) = \Mfd(c,N)$.
The presheaf $\ul{\RR} \in \scH$ is canonically even a presheaf of rings.

\begin{lemma}
\label{st:spl smooth R-mods have trivial HoType}
Let $V \in \scH$ be a simplicial presheaf which is also a module over $\ul{\RR} \in \scH$.
That is, for each $c \in \Cart$, the simplicial set $V(c)$ is a simplicial module over $\ul{\RR}(c) = \Mfd(c,\RR)$, and this module structure is compatible with the presheaf structures on $V$ and $\ul{\RR}$.
Then, $S_Q(V) \in \sSet$ is weakly contractible.
\end{lemma}

\begin{proof}
Let $\sfc \colon \sSet \to \scH$ denote the constant-presheaf functor.
Consider the morphisms of simplicial presheaves
\begin{equation}
\begin{tikzcd}
	p : V \ar[r, shift left=0.075cm]
	& \sfc \Delta^0 = \{0\} : 0. \ar[l, shift left=0.075cm]
\end{tikzcd}
\end{equation}
The composition $0 \circ p$ is the identity on $\sfc \Delta^0$.
It thus suffices to provide a smooth homotopy between $p \circ 0$ and $1_V$ in $\scH$.
Such a smooth homotopy is given by the map
\begin{equation}
	h \colon V \times \Delta_e^1 \longmapsto V\,,
	\qquad
	\big( v \in V(c),\, f \in \Delta_e^1(c) \big)
	\longmapsto (\pr_0 \circ f) \cdot v \in V(c)\,,
\end{equation}
where $\pr_0 \colon \RR^2 \to \RR$ is the projection onto the first coordinate.
\end{proof}

The following statement fully explains the homotopy-theoretic significance of the underlying-space functor $S$.
It was shown indirectly in~\cite{Bunk:Pr_ooBdls_and_String} (see the last paragraph of Section~2.2 in that reference, and see~\cite[Sec.~4.2]{Bunk:R-loc_HoThy} for a model-categorical account).
A direct proof first appeared in~\cite[Prop.~12.10]{Pavlov:Proj_MoStrs_and_smooth_Oka_principle}; see also the book~\cite{ADH:Differential_cohomology} for an alternative $\infty$-categorical treatment.
Here we include a short proof for completeness.

\begin{proposition}
\label{st:NDelta_e is cofinal}
The extended simplex functor $\Delta_e \colon \bbDelta \to \Cart$ is homotopy cofinal, i.e.~its nerve is a cofinal morphism of simplicial sets.
\end{proposition}

\begin{proof}
By~\cite[Cor.~4.4.31]{Cisinski:HiC_and_HoA} it suffices to show that the slice $\infty$-category
\begin{equation}
	(N \Delta_e)_{/c}
	= N \bbDelta \underset{N \Cart}{\times} N \Cart_{/c}
	\cong  N \big( \bbDelta \underset{\Cart}{\times} \Cart_{/c} \big)
\end{equation}
is a weakly contractible simplicial set, for each $c \in \Cart$.
Since each object $c \in \Cart$ is isomorphic to $\RR^n$, for some $n \in \NN_0$, we may suppose that $c = \RR^n$.
Observe that there is a canonical weak homotopy equivalence
\begin{equation}
	(N \Delta_e)_{/\RR^n}
	\simeq \hocolim \big( \bbDelta^\opp \to \Cart\,, \ [k] \longmapsto \Cart(\Delta_e^k, \RR^n) \big)\,,
\end{equation}
where we can model the homotopy colimit using the two-sided bar construction (see, for instance,~\cite[Cor.~5.1.3]{Riehl:Cat_HoThy}).
In particular, this is the homotopy colimit of a simplicial diagram in $\sSet$, i.e.~a bisimplicial set.
The Bousfield-Kan map provides a further weak homotopy equivalence~\cite[Defs.~18.7.1, 18.7.3, Cor.~18.7.7]{Hirschhorn:MoCats_and_localisations} from this homotopy colimit to the diagonal of this bisimplicial set.
That, is we have a weak homotopy equivalence of simplicial sets
\begin{equation}
	N (\Delta_{e/\RR^n}) \wequiv
	S_Q \ul{\RR^n}\,.
\end{equation}
To finish the proof, we observe that $h_{\RR^n}$ is an $\ul{\RR}$-module and invoke Lemma~\ref{st:spl smooth R-mods have trivial HoType}.
\end{proof}

\begin{corollary}
\emph{\cite[Thm.~4.14]{Bunk:R-loc_HoThy}}
There is a canonical natural weak equivalence
\begin{equation}
	S_Q \simeq \underset{\Cart^\opp}{\hocolim}
\end{equation}
of functors $\scH = \Fun(\Cart^\opp, \sSet)$.
Consequently, the left Quillen functor $S_Q$ presents the $\infty$-functor $\colim \colon \scFun(N\Cart^\opp, \scS)$.
\end{corollary}

Combining the fact that $S \colon \scFun(N\Cart^\opp, \scS) \to \scS$ preserves extensions of group objects~\cite[Thm.~3.48]{Bunk:Pr_ooBdls_and_String} with Theorem~\ref{st:Aut-Sym_phi-Gamma extension}, we obtain:

\begin{theorem}
For each smooth $\infty$-group action $\Phi \colon \bbGamma \to \Diff_{[\cG]}(M)$, there is an extension of group objects in $\scS$,
\begin{equation}
\begin{tikzcd}
	S\scAut(\cG) \ar[r]
	& S\scSym_\Phi(\cG) \ar[r]
	& S\bbGamma\,.
\end{tikzcd}
\end{equation}
Consequently, there is a long exact sequence of homotopy groups
\begin{equation}
\begin{tikzcd}[column sep=0.75cm]
	\cdots \ar[r]
	& \pi_r \big( S\scAut(\cG) \big) \ar[r]
	& \pi_r \big( S \scSym_\Phi(\cG) \big) \ar[r] \ar[d, phantom, ""{coordinate, name=MidPoint}]
	& \pi_r (S \bbGamma)
	\ar[dll, rounded corners, to path={-- ([xshift=2ex]\tikztostart.east) |- (MidPoint) \tikztonodes -| ([xshift=-2ex]\tikztotarget.west) -- (\tikztotarget)}]
	& 
	\\
	& \pi_{r-1} \big( S\scAut(\cG) \big) \ar[r]
	& \pi_{r-1} \big( S \scSym_\Phi(\cG) \big) \ar[r]
	& \pi_{r-1} (S \bbGamma) \ar[r]
	& \cdots
\end{tikzcd}
\end{equation}
\end{theorem}

%%%%%%%%%%%%%%%%%%%%%%%%%%%%%%%%%%%%%%%%%%%%%%%%%%%%%%%%%%%%%%%%%%%%%%%%%%%%

\subsection{Homotopy types of solution stacks}
\label{sec:HoTypes of solution stacks}

%%%%%%%%%%%%%%%%%%%%%%%%%%%%%%%%%%%%%%%%%%%%%%%%%%%%%%%%%%%%%%%%%%%%%%%%%%%%

To finish this section, we use the formalism developed so far to extract information about the underlying spaces of our moduli $\infty$-prestacks.
As an important special case, which arises whenever the configurations and field equations have a linear structure, we consider the case where the underlying space of the simplicial solution subpresheaf is contractible.
Then, we use the interaction of the underlying-space functor and principal $\infty$-bundles to obtain a long exact sequence for the homotopy groups of the moduli $\infty$-prestacks of solutions on a fixed higher geometric structure $\cG$ (Theorem~\ref{st:LES for S(Mdl_phi)}).

%%%%%%%%%%%%%%%%%%%%%%%%%%%%%%%%%%%%%%%%%%%%%%%%%%%%%%%%%%%%%%%%%%%%%%%%%%%%

\subsubsection{Contractible solution spaces}
\label{sec:contractible S(Sol)}

%%%%%%%%%%%%%%%%%%%%%%%%%%%%%%%%%%%%%%%%%%%%%%%%%%%%%%%%%%%%%%%%%%%%%%%%%%%%

Let $\Conf^\scD \colon \rmB\scD^\opp \longrightarrow \sSet$ be a configuration presheaf over $\Fix^\scD \colon \rmB\scD^\opp \longrightarrow \sSet$, and let $\cG \in e_M^*\Fix^\scD$ be a section.
Let $\phi \colon Q \to \rmB\scD[\cG]^\opp$ be a morphism of left fibrations over $N\Cart^\opp$, where $Q \to N\Cart^\opp$ has reduced fibres (see also Definition~\ref{def:BSYM_(Q,phi)(G)}).
Let $\Phi \colon \bbGamma \to \Diff_{[\cG]}(M)$ denote the group action in $\scP(N\Cart)$ presented by these data.
Recall that the left fibration $\rmB\SYM_\phi^\rev(\cG) \longrightarrow N\Cart^\opp$ has a canonical section \smash{$e_\phi \colon N\Cart^\opp \longrightarrow \rmB\SYM_\phi^\rev(\cG)$}.

For any choice of solution presheaf \smash{$\Sol^\scD \subset \Conf^\scD$}, the $\infty$-functor \smash{$\scSol_\Phi(\cG) \colon \rmB\SYM_\phi^\rev(\cG) \to \scS$} encodes the right action of $\scSym_\Phi(\cG)$ on $(\sigma \cG)^* \scSol^{\scD[\cG]}(\cG)$ (compare Remark~\ref{rmk:scSol_phi(cG) as oo-action}).
We can describe the $\infty$-presheaf \smash{$\scMdl_\phi(\cG)$} as the $\infty$-categorical quotient of this action.

Equivalently, the morphism
\begin{equation}
	\textint \Sol_\phi(\cG) \longrightarrow \rmB\SYM_\phi^\rev(\cG)
\end{equation}
of left fibrations over $N\Cart^\opp$ describes a canonical morphism
\begin{equation}
	\scMdl_\phi(\cG) = (\pi_\phi)_! \scSol_\phi(\cG) \longrightarrow \rmB\scSym_\Phi^\rev(\cG)
\end{equation}
in $\scP(N\Cart)$.
Moreover, the fibre of this morphism over the canonical section $\sigma \cG \colon * \to \rmB\scSym_\Phi^\rev(\cG)$ presented by the morphisms from Remark~\ref{rmk:section associated to cG} in $\scP(N\Cart)$ is $(\sigma \cG)^* \scSol^{\scD[\cG]}(\cG)$.
Since $\rmB\scSym_\Phi^\rev(\cG)$ is the classifying object for $\scSym(\cG)$ right actions in $\scP(N\Cart)$, this gives rise to a simplicial object $(\sigma \cG)^* \scSol^{\scD[\cG]}(\cG) \dslash \scSym_\Phi(\cG)$ in $\scP(N\Cart)$ which encodes the aforementioned right action of of $\scSym_\Phi(\cG)$ on $(\sigma \cG)^* \scSol^{\scD[\cG]}(\cG)$ (see~\cite[Thm.~3.17]{NSS:Pr_ooBdls_I}).
In particular, by the same theorem, this endows the canonical morphism $(\sigma \cG)^* \scSol^{\scD[\cG]}(\cG) \to \scMdl_\Phi(\cG)$ with the structure of a principal $\infty$-bundle in $\scP(N\Cart)$ for the structure group $\scSym_\Phi(\cG)$.
Therefore, we have a canonical equivalence
\begin{equation}
\label{eq:scMdl as oo-quotient}
	\scMdl_\phi(\cG)
	\simeq \big| (\sigma \cG)^* \scSol^{\scD[\cG]}(\cG) \dslash \scSym_\phi(\cG) \big|\,,
\end{equation}
in $\scP(N\Cart)$, where $|{-}|$ denotes the colimit of simplicial diagrams; see~\cite{Bunk:Pr_ooBdls_and_String, NSS:Pr_ooBdls_I} for more background.

\begin{proposition}
\label{st:SMdl for SSol = *}
Whenever the underlying space
\begin{equation}
	S \big( (\sigma \cG)^* \scSol^{\scD[\cG]}(\cG) \big)
	= S \big( (\tilde{\sigma} \cG)^* \scSol^M(\cG) \big)
	\quad \in \scS
\end{equation}
is equivalent to $\Delta^0$, we have canonical equivalences
\begin{align}
	S \big( \scMdl_\Phi(\cG) \big)
	&\simeq S \big( \rmB \scSym_\Phi(\cG) \big)
\end{align}
\end{proposition}

\begin{proof}
Since $S$ preserves principal $\infty$-bundles, colimits and finite products, and using our assumption on the $\infty$-presheaf of solutions, we obtain equivalences in $\scS$ of the form
\begin{align}
	S \big( \scMdl_\Phi(\cG) \big)
	&\simeq \big| \big( S \big( (\sigma \cG)^* \scSol^{\scD[\cG]}(\cG) \big) \big) \dslash S \big( \scSym_\Phi(\cG) \big) \big|
	\\
	&\simeq \big| \Delta^0 \dslash S \big( \scSym_\Phi(\cG) \big) \big|
	\\
	&\simeq S\, \rmB \scSym_\Phi(\cG)\,.
\end{align}
This shows the claim.
\end{proof}

Consider again the situation of Proposition~\ref{st:SOL^(n,D_cG(M))(cG,cA^(k)) depends only on [cG]} and Theorem~\ref{st:Solution zig-zag}:
let \smash{$p \colon \Conf^\scD \longrightarrow \Fix_0^\scD$} and \smash{$q \colon \Fix_0^\scD \to \Fix_1^\scD$} be projective fibrations between projectively fibrant simplicial presheaves on $\rmB\scD$ which is 0-connected.
Further, suppose that $\cG_0 \in \Fix_0^M(\RR^0)$ is a section and set $\cG_1 \coloneqq q_{|\RR^0} (\cG_0) \in \Fix_1^M(\RR^0)$.

\begin{corollary}
\label{st:equivs of SSym(cG_0) and SSym(cG_1)}
In the above situation, suppose additionally that $\scSol^M(\cG_i)$ has contractible underlying space, for $i = 0,1$.
Then, there is a canonical equivalence
\begin{equation}
	S \big( \scSym_\Phi(\cG_0) \big)
	\simeq S \big( \scSym_\Phi(\cG_1) \big)
\end{equation}
of group objects in $\scS$.
\end{corollary}

Note that we do not obtain an equivalence of group objects in $\scP(N\Cart)$, but only on the level of underlying group objects in $\scS$.

\begin{proof}
By combining Theorem~\ref{st:equiv result for Mdl oo-stacks} and Proposition~\ref{st:SMdl for SSol = *} we obtain a canonical chain of equivalences of spaces
\begin{equation}
	S \big( \rmB \scSym_\Phi(\cG_0) \big)
	\simeq S \big( \scMdl_\Phi(\cG_0) \big)
	\simeq S \big( \scMdl_\Phi(\cG_1) \big)
	\simeq S \big( \rmB \scSym_\Phi(\cG_1) \big)\,.
\end{equation}
Since $S$ preserves finite products and geometric realisations (in fact all colimits), there is a canonical equivalence
\begin{equation}
	S \rmB \bbH \simeq \rmB S \bbH\,,
\end{equation}
for each group object $\bbH \in \Grp(\scP(N\Cart))$.
The claim now follows from the equivalence in Remark~\ref{rmk:oo-groups and ptd conn objs}.
\end{proof}

%%%%%%%%%%%%%%%%%%%%%%%%%%%%%%%%%%%%%%%%%%%%%%%%%%%%%%%%%%%%%%%%%%%%%%%%%%%%

\subsubsection{Exact sequences for solution stacks}
\label{sec:LES for solution stacks}

%%%%%%%%%%%%%%%%%%%%%%%%%%%%%%%%%%%%%%%%%%%%%%%%%%%%%%%%%%%%%%%%%%%%%%%%%%%%

\begin{proposition}
\label{st:Mod-Mdl-BDiff pullback}
There is a canonical pullback square in $\scP(N\Cart)$:
\begin{equation}
\begin{tikzcd}[column sep=1.25cm]
	\scMdl_{N e_M}(\cG) \ar[r] \ar[d]
	& * \ar[d]
	\\
	\scMdl_\Phi(\cG) \ar[r]
	& \rmB \bbGamma
\end{tikzcd}
\end{equation}
\end{proposition}

\begin{proof}
This follows from Corollary~\ref{st:SOL fibseq in sSet_(/NCart^op)}:
at the level of left fibrations over $N\Cart^\opp$ the diagram of $\infty$-presheaves on $\Cart$ that we are interested in here is presented by the outer square of the diagram in Corollary~\ref{st:SOL fibseq in sSet_(/NCart^op)}.
There we showed that this outer square is homotopy cartesian in $\sSet_{/N \Cart^\opp}$.
\end{proof}

\begin{theorem}
\label{st:Mod-Mdl-Diff prBun}
The morphism of $\infty$-presheaves on $\Cart$
\begin{equation}
\label{eq:Mod-Mdl-Diff prBun}
	\scMdl_{Ne_M}(\cG)
	\longrightarrow \scMdl_\Phi(\cG)
\end{equation}
canonically carries the structure of a $\bbGamma$-principal $\infty$-bundle in $\scP(N\Cart)$.
\end{theorem}

\begin{proof}
This follows by combining Proposition~\ref{st:Mod-Mdl-BDiff pullback} with~\cite[Prop.~3.13]{NSS:Pr_ooBdls_I} (see also~\cite[Props.~3.33 and~3.41]{Bunk:Pr_ooBdls_and_String}).
\end{proof}

We obtain a group object in $\scS$ by applying the functor $S$ from~\eqref{eq:ul space fctr S} to the smooth $\infty$-group $\bbGamma$.
In particular, for $Q = \rmB \scD[\cG]$, the resulting group is canonically equivalent to the subgroup $\Diff_{[\cG]}(M) \subset \Diff(M)$ of the diffeomorphism group, endowed with the usual topology; this follows from~\cite[Props.~3.15, 3.16]{OT:Smooth_HoCoh_Actions} (there the result is proven for the full group $\Diff(M)$, but the subgroup $\Diff_{[\cG]}(M) \subset \Diff(M)$ consists of a disjoint union of connected components of $\Diff(M)$ and so the result carries over directly).
We can thus derive important information about the homotopy type of the space underlying the moduli $\infty$-(pre)stack $\scMdl_\Phi(\cG)$:

\begin{corollary}
\label{st:moduli prBun in spaces}
The morphism of underlying spaces
\begin{equation}
\label{eq:moduli prBun in spaces}
	S \big( \scMdl_{Ne_M}(\cG) \big)
	\longrightarrow S \big( \scMdl_\Phi(\cG) \big)
\end{equation}
canonically carries the structure of a principal $\infty$-bundle in $\scS$ with structure group $S \bbGamma$.
In particular, its homotopy fibre is $S \bbGamma$.
\end{corollary}

\begin{corollary}
\label{st:LES for S(Mdl_phi)}
There is a long exact sequence of homotopy groups
\begin{equation}
\label{eq:hogroup LES}
\begin{tikzcd}[column sep=0.75cm]
	\cdots \ar[r]
	& \pi_r S(\bbGamma) \ar[r]
	& \pi_r S \big( \scMdl_{Ne_M}(\cG) \big) \ar[r] \ar[d, phantom, ""{coordinate, name=MidPoint}]
	& \pi_r S \big( \scMdl_\Phi(\cG) \big)
	\ar[dll, rounded corners, to path={-- ([xshift=2ex]\tikztostart.east) |- (MidPoint) \tikztonodes -| ([xshift=-2ex]\tikztotarget.west) -- (\tikztotarget)}]
	& 
	\\
	& \pi_{r-1} S(\bbGamma) \ar[r]
	& \pi_{r-1} S \big( \scMdl_{Ne_M}(\cG) \big) \ar[r]
	& \pi_{r-1} S \big( \scMdl_\Phi(\cG) \big) \ar[r]
	& \cdots
\end{tikzcd}
\end{equation}
\end{corollary}

%%%%%%%%%%%%%%%%%%%%%%%%%%%%%%%%%%%%%%%%%%%%%%%%%%%%%%%%%%%%%%%%%%%%%%%%%%%%

\section{The descent property for moduli $\infty$-prestacks}
\label{sec:descent for moduli oo-prestacks}

%%%%%%%%%%%%%%%%%%%%%%%%%%%%%%%%%%%%%%%%%%%%%%%%%%%%%%%%%%%%%%%%%%%%%%%%%%%%

In this section we consider the important special case where $\bbGamma$ (the higher smooth group acting on $M$) is given by a \textit{sheaf of groups} on $\Cart$.
We show that, in this case, our moduli $\infty$-prestacks of solutions on a fixed higher geometric structure $\cG$ modulo symmetries which lift the action of elements of $\bbGamma$ satisfy descent with respect to good open coverings of cartesian spaces; that is, our moduli $\infty$-prestacks are even $\infty$-stacks.
The most important case for geometric applications is where $H$ is a Lie group which acts smoothly on the base manifold $M$.

Let $H \colon \Cart^\opp \to \Grp$ be a presheaf of groups, let $\rmB H \colon \Cart^\opp \to \Gpd$ be its delooping, and let $\textint \rmB H$ denote the Grothendieck construction of $\rmB H$.
We denote its associated left fibration by
\begin{equation}
	Q_H \coloneqq N \textint \rmB H \longrightarrow N\Cart^\opp\,.
\end{equation}
A smooth (left) action of $H$ on $M$ is a morphism
\begin{equation}
	\phi' \colon H \to \Diff(M)
\end{equation}
of presheaves of groups on $\Cart$.
We obtain a morphism
\begin{equation}
	\phi \coloneqq N \textint \rmB \phi' \colon Q_H \longrightarrow \rmB\scD
\end{equation}
of left fibrations over $N\Cart^\opp$.

\begin{example}
For $H = \{e\}$ the trivial group we obtain $\rmB \{e\} = \Cart$, and for $H = \Diff(M)$ we have $\textint \rmB H = \rmB\scD$.
\qen
\end{example}

\begin{example}
Let $G^\scD \colon \rmB\scD \longrightarrow \sSet$ be a simplicially fibrant simplicial presheaf on $\rmB\scD$, and let $\cG \in G^\scD(\RR^0)$ be a section.
Let $H = \Diff_{[\cG]}(M) \subset \Diff(M)$ be the smooth subgroup of diffeomorphisms which preserve the equivalence class of $\cG$.
In this case, the inclusion of $H$ into $\Diff(M)$ induces a canonical action of $H$ on $M$.
\qen
\end{example}

For any presheaf of groups $H$ on $\Cart$, we obtain a Grothendieck coverage $\tau_H$ on $\rmB H$:
explicitly, let $\{(f_i, \varphi_i) \colon c_i \to c\}_{i \in \Lambda}$ be a family of morphisms in $\textint \rmB H$ with codomain $c$.
That is, $f_i \in \Cart(c_i, c)$ is a morphism in $\Cart$, and $\varphi_i \in H(c_i)$, for each $i \in \Lambda$.
The family is a covering of $c \in \rmB H$ with respect to $\tau_H$ if and only if its image under the canonical functor $\textint \rmB H \to \Cart$ is a covering in $\Cart$; that is, if an only if the family $\{f_i \colon c_i \to c\}_{i \in \lambda}$ is a good open covering of $c$.
For $H = \Diff(M)$ we abbreviate $\tau \coloneqq \tau_{\Diff(M)}$.

\begin{remark}
\label{rmk:tau-coverings as M-bundle maps}
We can think of a $\tau$-covering $\{(\iota_i, \varphi_i) \colon c_i \to c\}_{i \in \Lambda}$ in $\rmB\scD$ as a particular type of open covering of the product manifold $c \times M$:
the family $\{ \iota_i \colon c_i \to c \}_{i \in \Lambda}$ provides a good open covering of $c$ in $\Cart$.
The family $\{\varphi\}_{i \in \Lambda}$ then enhances this to a covering of $c \times M$ by patches of the form $c_i \times M$ which are included into $c \times M$ along the smooth maps $\iota_i \wr \varphi_i$ (recall the notation from~\eqref{eq:wr notation}); they can be described as maps (covering $\iota_i$) of bundles with fibre $M$ which restrict to diffeomorphisms on fibres.
\qen
\end{remark}

For the sake of brevity, we set
\begin{equation}
	\scK_H \coloneqq \Fun(\textint \rmB H^\opp, \sSet)
	\quad \text{and} \quad
	\scK \coloneqq \scK_{\Diff(M)} = \Fun(\rmB \scD^\opp, \sSet)\,.
\end{equation}

\begin{lemma}
\label{st:pbs in tint BH}
Let $\phi' \colon H \to \Diff(M)$ be a morphism of presheaves of groups on $\Cart$.
Consider a cospan
\begin{equation}
\begin{tikzcd}[column sep=1cm]
	c_0 \ar[r, "{(\iota_0, \psi_0)}"]
	& c
	& c_1 \ar[l, "{(\iota_1, \psi_1)}"']
\end{tikzcd}
\end{equation}
in $\rmB H$ such that $\iota_i \in \Cart(c_i,c)$ is an open embedding, for $i = 0,1$, and suppose that the pullback $c_0 \times_c c_1$ in $\Cart$ exists.
Then,
\begin{enumerate}
\item the pullback $c_0 \times_c c_1$ in $\textint \rmB H$ exists, and

\item it is preserved by the functor $\textint \rmB \phi' \colon \textint \rmB H \to \rmB \scD$.
\end{enumerate}
\end{lemma}

\begin{proof}
Let $c_{01}$ be an object representing the pullback $c_0 \times_c c_1$ in $\Cart$.
We thus have a commutative square
\begin{equation}
\begin{tikzcd}
	c_{01} \ar[r, "\jmath_1"] \ar[d, "\jmath_0"']
	& c_1 \ar[d, "\iota_1"]
	\\
	c_0 \ar[r, "\iota_0"']
	& c
\end{tikzcd}
\end{equation}
in $\Cart$.
We lift this to a square in $\textint \rmB H$:
\begin{equation}
\label{eq:pb in tint BH}
\begin{tikzcd}[column sep=1cm, row sep=0.75cm]
	c_{01} \ar[r, "{(\jmath_1, e_H)}"] \ar[d, "{(\jmath_0, (\jmath_0^* \psi_0)^{-1} \cdot \jmath_1^*\psi_1)}"']
	& c_1 \ar[d, "{(\iota_1, \psi_1)}"]
	\\
	c_0 \ar[r, "{(\iota_0, \psi_0)}"']
	& c
\end{tikzcd}
\end{equation}
We claim that this is a pullback square.
To check this, we augment the above commutative square to an arbitrary commutative diagram as given by the solid arrows in
\begin{equation}
\begin{tikzcd}[column sep=1cm, row sep=0.75cm]
	d \ar[rrd, bend left=20, "{(f_1, \varphi_1)}"] \ar[ddr, bend right=20, "{(f_0, \varphi_0)}"'] \ar[dr, dashed, "{(f_{01}, \varphi_{01})}" {description, xshift=0.1cm}]
	& &
	\\
	& c_{01} \ar[r, "{(\jmath_1, e_H)}"] \ar[d]
	& c_1 \ar[d, "{(\iota_1, \psi_1)}"]
	\\
	& c_0 \ar[r, "{(\iota_0, \psi_0)}"']
	& c
\end{tikzcd}
\end{equation}
in $\textint \rmB H$.
The commutativity implies, in particular, the identity
\begin{equation}
	f_1^* \psi_1 \cdot \varphi_1
	= f_0^* \psi_0 \cdot \varphi_0
\end{equation}
in the group $H(d)$.
To fill in the dashed arrow, first note that the smooth map $f_{01} \colon d \to c_{01}$ exists and is uniquely determined since we constructed $c_{01}$ as a pullback in $\Cart$.
Second, the commutativity of the upper triangle forces upon us the \textit{unique} choice $\varphi_{01} \coloneqq \varphi_1$.
By the above identity of sections of $H$ over $d$, the left-hand triangle commutes as well, and we obtain that~\eqref{eq:pb in tint BH} is indeed a pullback square in $\textint \rmB H$.

Since this construction did not depend on the choice of presheaf of groups $H$, and the functor $\textint \rmB \phi'$ acts as the identity at the level of underlying morphisms of cartesian spaces, it follows that $\textint \rmB \phi'$ preserves pullbacks of the type~\eqref{eq:pb in tint BH}.
\end{proof}

\begin{remark}
\label{rmk:pbs in Cech nerves in tint BH}
Each of the pullbacks that appear in the \v{C}ech nerves of $\tau_H$-covering are of the form considered in Lemma~\ref{st:pbs in tint BH}
\qen
\end{remark}

\begin{lemma}
For each morphism $\phi' \colon H \to \Diff(M)$ of presheaves of groups on $\Cart$, there is a Quillen adjunction
\begin{equation}
\begin{tikzcd}
	(\rmB \phi')_! : L_{\tau_H} \scK_H \ar[r, shift left=0.125cm, "\perp"' yshift=0.05cm]
	& L_\tau \scK : (\rmB \phi')^*\,. \ar[l, shift left=0.125cm]
\end{tikzcd}
\end{equation}
between the left Bousfield localisations of the projective model structures at the \v{C}ech nerves of coverings with respect to $\tau_H$ and $\tau$, respectively.
\end{lemma}

\begin{proof}
We readily observe that the adjunction
\begin{equation}
\begin{tikzcd}
	(\rmB \phi')_! : \scK_H \ar[r, shift left=0.125cm, "\perp"' yshift=0.05cm]
	& \scK : (\rmB \phi')^* \ar[l, shift left=0.125cm]
\end{tikzcd}
\end{equation}
is a Quillen adjunction with respect to the projective model structures.
By~\cite[Prop.~3.3.18]{Hirschhorn:MoCats_and_localisations} it suffices to show that $(\rmB \phi')_!$ maps \v{C}ech nerves of $\tau_H$-coverings to \v{C}ech nerves of $\tau$-coverings.
This follows from Lemma~\ref{st:pbs in tint BH} and Remark~\ref{rmk:pbs in Cech nerves in tint BH}.
\end{proof}

\begin{proposition}
Let $G \colon \Cartfam^\opp \to \sSet$ be fibrant in $\scH_\rmfam^{loc}$, and consider the associated functor
\begin{equation}
	\check{C}^{\scD*} G \coloneqq \ul{\scH}_\rmfam \big( \v{C}(-), G \big)
	\colon (\GCov^\scD)^\opp \longrightarrow \sSet
\end{equation}
from Construction~\ref{cstr:sPShs on GCov^D from sPShs on Cart_fam}.
The (homotopy) left Kan extension
\begin{equation}
	G^\scD \coloneqq \Lan_{\varpi^\scD} \check{C}^{\scD*} G \colon \rmB\scD^\opp \longrightarrow \sSet
\end{equation}
is fibrant in $L_\tau \scK$, the left Bousfield localisation of the projective model structure on $\scK$ at the $\tau$-coverings.
\end{proposition}

\begin{proof}
The strategy of this proof is as follows:
let $c \in \rmB\scD$, and let \smash{$\widehat{\cU} = \{(f_i, \psi_i) \colon c_i \to c\}_{i \in \Lambda}$} be a $\tau$-covering in $\rmB\scD$.
We need to show that the canonical morphism
\begin{equation}
	(\Lan_{\varpi^\scD} \check{C}^{\scD*} G)(c)
	\longrightarrow \holim_{[l] \in \bbDelta} \ul{\scK}(\v{C}_l\widehat{\cU}, \Lan_{\varpi^\scD} \check{C}^{\scD*} G)
\end{equation}
is a weak homotopy equivalence in $\sSet$.
Recall from Remark~\ref{rmk:tau-coverings as M-bundle maps} that we can understand \smash{$\widehat{\cU}$} as a covering of $c \times M$ by maps of bundles with fibre $M$.
We then further resolve each patch $c_{i_0 \cdots i_l} \times M$ by a $\tau_\rmfam$-covering, using a good open covering of $M$.
That allows us to apply arguments as in Proposition~\ref{st:Lan_varpi} to get control over the left Kan extension.

Thus, let $\{d_a \hookrightarrow M\}_{a \in \Xi}$ be a good open covering of $M$.
Set
\begin{equation}
	\hat{c}_{ia} \coloneqq (f_i \wr \psi_i)^{-1} \big( f_i(c_i) \times d_a \big)
	\quad \subset \quad
	c_i \times M\,.
\end{equation}
The manifold $\hat{c}_{ia}$ comes with a canonical smooth map $\hat{c}_{ia} \to c_i$, making it into an object in $\Cartfam$ (it can be described as the bundle over $c_i$ with fibre $d_a$, twisted by the smooth family of diffeomorphism $\psi_i$).
Moreover, there are canonical diffeomorphisms of manifolds
\begin{equation}
	\hat{c}_{ia} \underset{c \times M}{\times} \hat{c}_{jb}
	\cong (c_i \underset{M}{\times} c_j) \times (d_a \cap d_b)
	\eqqcolon \hat{c}_{ij,ab}
\end{equation}
(where the pullbacks are formed in $\Mfd$), and similarly for iterated pullbacks.
Thus, we obtain an object
\begin{equation}
	\cV \coloneqq \{(\hat{c}_{ia} \to c_i)\}_{i \in \Lambda, a \in \Xi}
	\quad \in \GCov^\scD_{|c}\,,
\end{equation}
where \smash{$\GCov^\scD_{|c}$} is the fibre of the functor $\varpi^\scD \colon \GCov^\scD \to \rmB\scD$ at $c \in \rmB\scD$.
We also obtain an object
\begin{equation}
	\cV_{i_0 \cdots i_l} \coloneqq \{(\hat{c}_{i_0 \cdots i_l, a} \to c_{i_0 \cdots i_l})\}_{a \in \Xi}
	\quad \in \GCov^\scD_{|c_{i_0 \cdots i_l}}\,,
\end{equation}
for each $l \in \NN_0$ and each fixed $i_0, \ldots, i_l \in \Lambda$.

By the same argument as in Proposition~\ref{st:Lan_varpi}, the canonical morphism
\begin{equation}
	\wtG(\cV) = \ul{\scH}_\rmfam (\v{C}\cV, G)
	\longrightarrow (\Lan_{\varpi^\scD} \check{C}^{\scD*} G)(c)
\end{equation}
is a weak homotopy equivalence.
Similarly, for each $l \in \NN_0$ and $i_0, \ldots, i_l \in \Lambda$, the canonical morphism
\begin{equation}
	\ul{\scH}_\rmfam (\v{C}\cV_{i_0 \cdots i_l}, G)
	\longrightarrow (\Lan_{\varpi^\scD} \check{C}^{\scD*} G)(c_{i_0 \cdots i_l})
\end{equation}
is a weak homotopy equivalence.
Further, both sides are Kan complexes since the left Kan extension is computed by a filtered colimit, and Kan fibrations are stable under filtered colimits~\cite[Lemma~3.1.24]{Cisinski:HiC_and_HoA}.
Therefore, the a canonical map
\begin{align}
	\holim_{[l] \in \bbDelta} \bigg( \prod_{i_0 \cdots i_l \in \Lambda} \ul{\scH}_\rmfam (\v{C}\cV_{i_0 \cdots i_l}, G) \bigg)
	\longrightarrow
	&\holim_{[l] \in \bbDelta} \bigg( \prod_{i_0 \cdots i_l \in \Lambda} (\Lan_{\varpi^\scD} \check{C}^{\scD*} G)(c_{i_0 \cdots i_l}) \bigg)
	\\
	&= \holim_{[l] \in \bbDelta}\ \ul{\scK}(\v{C}_l\widehat{\cU}, \Lan_{\varpi^\scD} \check{C}^{\scD*} G)
\end{align}
is also a weak homotopy equivalence.
By an application of the Bousfield-Kan map, the canonical morphism
\begin{equation}
	\hocolim_{[k] \in \bbDelta} \v{C}_k\cV_{i_0 \cdots i_l}
	\longrightarrow \v{C}\cV_{i_0 \cdots i_l}
\end{equation}
is a projective weak equivalence in $\scH_\rmfam$ (where on the left-hand side we have a homotopy colimit of simplicially constant simplicial presheaves).
We thus arrive at a commutative diagram
\begin{equation}
\begin{tikzcd}[column sep=1.5cm, row sep=0.75cm]
	\underset{[m] \in \bbDelta}{\holim}\ \ul{\scH}_\rmfam (\v{C}_m \cV, G) \ar[r]
	&  \underset{[l] \in \bbDelta}{\holim}\ \underset{[k] \in \bbDelta}{\holim}\ \bigg( \displaystyle\prod_{i_0 \cdots i_l \in \Lambda} \ul{\scH}_\rmfam (\v{C}_k\cV_{i_0 \cdots i_l}, G) \bigg)
	\\
	\ul{\scH}_\rmfam (\v{C}\cV, G) \ar[d, "\sim"] \ar[r] \ar[u, "\sim"]
	& \underset{[l] \in \bbDelta}{\holim}\ \bigg( \displaystyle\prod_{i_0 \cdots i_l \in \Lambda} \ul{\scH}_\rmfam (\v{C}\cV_{i_0 \cdots i_l}, G) \bigg) \ar[d, "\sim"] \ar[u, "\sim"]
	\\
	(\Lan_{\varpi^\scD} \check{C}^{\scD*} G)(c) \ar[r]
	& \underset{[l] \in \bbDelta}{\holim}\ \bigg( \displaystyle\prod_{i_0 \cdots i_l \in \Lambda} (\Lan_{\varpi^\scD} \wtG^\scD)(c_{i_0 \cdots i_l}) \bigg)
\end{tikzcd}
\end{equation}
The top horizontal morphism corresponds precisely to the inclusion of the diagonal $\bbDelta \hookrightarrow \bbDelta \times \bbDelta$; but this functor is homotopy cofinal (see, for instance,~\cite[Example~8.5.13]{Riehl:Cat_HoThy}), and so the top map is a weak homotopy equivalence.
The commutativity of the above diagram together with the two-out-of-three property of weak equivalences then yields the claim.
\end{proof}

\begin{corollary}
For any fibrant object $G \in \scH_\rmfam^{loc}$, the simplicial presheaf
\begin{equation}
	\big( \textint \rmB \phi' \big)^* \Lan_{\varpi^\scD} \check{C}^{\scD*} G
	 \colon \big( \textint \rmB H \big)^\opp \longrightarrow \sSet
\end{equation}
is fibrant in $L_{\tau_H} \scK_H$.
\end{corollary}

\begin{lemma}
\label{st:comm squares and equivs on fibs}
Consider a commutative square of simplicial sets
\begin{equation}
\begin{tikzcd}
	A \ar[r] \ar[d]
	& X \ar[d]
	\\
	B \ar[r, "f"']
	& Y
\end{tikzcd}
\end{equation}
Suppose that both vertical morphisms are left fibrations and that $f$ is a categorical weak equivalence (i.e.~a weak equivalence in the Joyal model structure on $\sSet$).
Then, the following are equivalent
\begin{enumerate}
\item The canonical morphism $A \to f^*X = X \times_Y B$ is a covariant weak equivalence in $\sSet_{/B}$.

\item The canonical morphism $f_! A \to X$ is a covariant weak equivalence in $\sSet_{/Y}$.

\item For each vertex $b \in B$, the canonical morphism on fibres $A_{|b} \to X_{|f(b)}$ is a weak homotopy equivalence (i.e.~a weak equivalence in the Kan-Quillen model structure on $\sSet$).
\end{enumerate}
\end{lemma}

\begin{proof}
(1) and (3) are equivalent even without the assumption that $f$ is a categorical weak equivalence; see, for instance,~\cite[Thm.~4.1.16]{Cisinski:HiC_and_HoA}.
The equivalence of (1) and (2) follows from~\cite[Prop.~F]{HM:Left_fibs_and_hocolims_I}, using that the morphisms $A \to B$ and $X \to Y$ are each fibrant as well as cofibrant objects in the covariant model structures on $\sSet_{/B}$ and $\sSet_{/Y}$, respectively.
\end{proof}

\begin{lemma}
\label{st:comm squares and weak Cat equivs}
In the situation of Lemma~\ref{st:comm squares and equivs on fibs}, suppose that one---and thus all---of the equivalent criteria in that lemma are satisfied.
Further, suppose that $B$ and $Y$ are $\infty$-categories, 
Then the top morphism $A \to X$ is a weak categorical equivalence.
\end{lemma}

\begin{proof}
By our additional assumption, each vertex of the commutative square in Lemma~\ref{st:comm squares and equivs on fibs} is now an $\infty$-category.
Consider the pullback square of simplicial sets
\begin{equation}
\begin{tikzcd}
	B \times_Y X \ar[r, "\widehat{f}"] \ar[d]
	& X \ar[d]
	\\
	B \ar[r, "f"']
	& Y
\end{tikzcd}
\end{equation}
As the right-hand vertical map is a left fibration between $\infty$-categories, it is also an isofibration~\cite[Prop.~3.4.8]{Cisinski:HiC_and_HoA}, and thus a fibration in the Joyal model structure~\cite[Thm.~3.6.1]{Cisinski:HiC_and_HoA}.
In particular, the above square is homotopy cartesian in the Joyal model structure, and it follows that $\widehat{f}$ is a categorical weak equivalence.
The claim then follows by combining property (1) of Lemma~\ref{st:comm squares and equivs on fibs} with Lemma~\ref{st:cov weqs between LFibs are exactly cat weqs}.
\end{proof}

\begin{remark}
\label{rmk:H sheaf of groups and NBH}
Let $H \colon \Cart^\opp \longrightarrow \Grp$ be a presheaf of groups.
Then, $H \colon \Cart^\opp \longrightarrow \Grp$ is a \textit{sheaf} of groups if and only if the simplicial presheaf $N \rmB H \colon \Cart^\opp \longrightarrow \Set$ satisfies homotopy descent with respect to good open coverings in $\Cart$.
\qen
\end{remark}

\begin{theorem}
\label{st:descent for spl pshs on tint BH}
Let $H \colon \Cart^\opp \longrightarrow \Grp$ be a \emph{sheaf} of groups.
Let $F \colon (\textint \rmB H)^\opp \longrightarrow \sSet$ be fibrant in $L_{\tau_H} \scK_H$.
Then, the homotopy left Kan extension
\begin{equation}
	\hoLan_{\pi_M^H} F \colon \Cart^\opp \longrightarrow \sSet
\end{equation}
is fibrant in $\scH^{loc}$ (recall the notation from Definition~\ref{def:scH and scH_rmfam}).
\end{theorem}

\begin{proof}
The proof proceeds in several steps.

\textit{(1) Computation and properties of $\hoLan_{\pi_M^H}$.}
First, we may compute the homotopy left Kan extension as
\begin{equation}
	(\hoLan_{\pi_M^H} F)(c)
	\simeq \hocolim \big( F_{|c} \colon \rmB H^\rev(c) \longrightarrow \sSet \big)\,,
\end{equation}
i.e.~as a homotopy colimit over the fibres of $\pi_M^H$ rather than its slices (since $(\textint N \rmB H)^\opp \to N\Cart^\opp$ is a left fibration and hence proper), where $F_{|c}$ is the restriction of $F$ to the fibre $(\pi_M^H)^{-1}(c) = \rmB H^\rev(c)$.
Note, further, that here the fibres even depend on $c$ functorially because we use the Grothendieck construction of a \textit{strict} groupoid-valued functor.
Moreover, since $\rmB H^\rev(c)$ is a groupoid and $F(c)$ is a Kan complex, for each $c \in \Cart$, we can choose a model for the homotopy colimit in such a way that the canonical morphism
\begin{equation}
	\hoLan_{\pi_M^H} F \longrightarrow \rmB H^\rev
\end{equation}
becomes a fibration in $\scH$, i.e.~an objectwise Kan fibration.
Concretely, we can use the two-sided bar construction~\cite[Secs.~4.2, 5.1]{Riehl:Cat_HoThy} to model the homotopy colimit, or the rectification functor (see~\cite[Thm.~C]{HM:Left_fibs_and_hocolims_I}); in the present setting this has the properties described above.

The functor $F_{|c}$ can equivalently be described as the simplicial set $F(c)$, acted upon from the right by $H(c)$ (the action being induced by the functoriality of $F$).
The homotopy colimit $(\hoLan_{\pi_M^H} F)(c)$ is a model for the homotopy quotient,
\begin{equation}
	(\hoLan_{\pi_M^H} F)(c) \simeq F(c) \dslash H(c)\,.
\end{equation}
More concretely, let $F(c) \circlearrowleft H(c)$ denote the bisimplicial set which in horizontal degree $n$ is the action groupoid of the action of $H(c)$ on $F_n(c)$; that is,
\begin{equation}
	\big( F(c) \circlearrowleft H(c) \big)_{n,k}
	= F_n(c) \times H(c)^k\,.
\end{equation}
The two-sided bar construction which computes the homotopy colimit of $F_{|c} \colon \rmB H^\rev(c) \longrightarrow \sSet$ is canonically weakly equivalent to the diagonal of this bisimplicial set:
\begin{equation}
	(\hoLan_{\pi_M^H} F)(c)
	\simeq B \big( *, \rmB H^\rev(c), F_{|c} \big)
	= \delta^* \big( F(c) \circlearrowleft H(c) \big)
	= F(c) \dslash H(c)\,.
\end{equation}
Finally, we emphasise again that each of these constructions depends functorially on $c \in \Cart$.

\textit{(2) Using the descent property.}
Let $c \in \Cart$, and let $\cU = \{f_i \colon c_i \hookrightarrow c\}_{i \in \Lambda}$ be a good open covering of $c$.
Our goal is to show that the canonical map
\begin{equation}
	(\hoLan_{\pi_M^H} F)(c)
	\longrightarrow \holim_{[l] \in \bbDelta} \ul{\scH}(\v{C}_l \cU, \hoLan_{\pi_M^H}F)
\end{equation}
is a weak equivalence in $\sSet$.
By an application of the Bousfield-Kan map~\cite[Cor.~18.7.7]{Hirschhorn:MoCats_and_localisations}, we have weak equivalences
\begin{equation}
	\ul{\scH} \big( \v{C}\cU, \hoLan_{\pi_M^H}F \big)
	\wequiv \ul{\scH} \big( \hocolim_{[l] \in \bbDelta}\v{C}_l\cU, \hoLan_{\pi_M^H}F \big)
	\cong \holim_{[l] \in \bbDelta} \ul{\scH}(\v{C}_l\cU, \hoLan_{\pi_M^H}F)\,.
\end{equation}
Combining this with step~(1) of this proof, we equivalently need to show that the canonical map
\begin{equation}
	F(c) \dslash H(c)
	\simeq (\hoLan_{\pi_M^H} F)(c)
	\longrightarrow \ul{\scH} \big( \v{C}\cU, \hoLan_{\pi_M^H}F \big)
\end{equation}
is a weak equivalence.
We lift $\cU$ to a $\tau_H$-covering $\widehat{\cU}$, given as
\begin{equation}
	\widehat{\cU} = \{(f_i, e_H) \colon c_i \to c\}_{i \in \Lambda}\,.
\end{equation}
For each $i \in \Lambda$, the element $e_H \in H(c_i)$ is the neutral element.
In other words, $\widehat{\cU}$ is the image of $\cU$ under the section $e_H \colon \Cart \to \rmB H$.
Since by assumption $F$ satisfies descent with respect to $\tau_H$-coverings, the canonical morphism
\begin{equation}
	F(c) \dslash H(c)
	\longrightarrow \big( \ul{\scK}_H(\v{C}\widehat{\cU}, F) \big) \dslash H(c)
\end{equation}
is a weak equivalence in $\sSet$.
The action of $H(c)$ on the right-hand side is obtained as follows:
a section $\varphi \in H(c)$ restricts to a compatible family of sections $\varphi_i = f_i^*\varphi \in H(c_i)$, and thus induces an automorphism of $\v{C}\widehat{\cU}$ in $\scK_H$.
The action is by precomposition by this automorphism.
Since the right-hand side is still a homotopy colimit of a diagram valued in Kan-complexes and indexed by the groupoid $\rmB H^\rev(c)$, the canonical map
\begin{equation}
	\big( \ul{\scK}(\v{C}\widehat{\cU}, F) \big) \dslash H(c)
	\longrightarrow N \rmB H^\rev(c)
\end{equation}
is again a Kan fibration.

Next, since the \v{C}ech nerve $\v{C}\cU$ is cofibrant in $\scH$, the projective fibration $\hoLan_{\pi_M^H} F \longrightarrow \rmB H^\rev$ induces a Kan fibration
\begin{equation}
	\ul{\scH}(\v{C}\cU, \hoLan_{\pi_M^H} F) \longrightarrow \ul{\scH} \big( \v{C}\cU, N \rmB H^\rev \big)\,.
\end{equation}
Thus, we arrive at a commutative square
\begin{equation}
\label{eq:comm square for descent}
\begin{tikzcd}[column sep=2cm]
	\big( \ul{\scK}(\v{C}\widehat{\cU}, F) \big) \dslash H(c) \ar[r] \ar[d]
	& \ul{\scH}(\v{C}\cU, \hoLan_{\pi_M^H} F) \ar[d]
	\\
	N \rmB H^\rev(c) \ar[r]
	& \ul{\scH} \big( \v{C}\cU, N \rmB H^\rev \big)
\end{tikzcd}
\end{equation}
in $\sSet$, whose vertices are Kan complexes, whose vertical arrows are Kan fibrations, and whose bottom arrow is a weak homotopy equivalence (see Remark~\ref{rmk:H sheaf of groups and NBH}).
The theorem will follow if we can show that the top arrow is a weak homotopy equivalence.

\textit{(3) Using formal properties from the preceding lemmas.}
To achieve this, we deploy Lemma~\ref{st:comm squares and weak Cat equivs}:
since a morphism between Kan complexes is a weak categorical equivalence if and only if it is a weak homotopy equivalence, it suffices to show that the horizontal maps in~\eqref{eq:comm square for descent} induce weak homotopy equivalences on all fibres; that is, we show that condition~(3) of Lemma~\ref{st:comm squares and equivs on fibs} is satisfied.
However, the groupoid $N \rmB H^\rev(c)$ has exactly one object, so there is only one fibre to consider.

The fibre of the left-hand vertical arrow in~\eqref{eq:comm square for descent} is the Kan complex $\ul{\scK}(\v{C}\widehat{\cU}, F)$.
A vertex in \smash{$\ul{\scH} (\v{C}\cU, N \rmB H^\rev)$} is the same as a \v{C}ech 1-cocycle with values in $H^\rev$ on $c$, subordinate to the covering $\cU$.
The bottom arrow in~\eqref{eq:comm square for descent} sends the unique object of $N \rmB H^\rev(c)$ to the trivial cocycle.
It follows (for instance using the bar construction to compute \smash{$\hoLan_{\pi_M^H} F$}) that the fibre of the right-hand vertical morphism over this vertex is also given by the Kan complex \smash{$\ul{\scK}(\v{C}\widehat{\cU}, F)$}, and the top horizontal arrow induces the identity map on these fibres.
\end{proof}

The direct consequence for the moduli problems of central interest in this paper is as follows:
let $\Conf^\scD, \Fix^\scD \colon \rmB\scD^\opp \longrightarrow \sSet$ be projectively fibrant simplicial presheaves and $p \colon \Conf^\scD \longrightarrow \Fix^\scD$ a projective fibration.
Let $\cG \in \Fix^M(\RR^0) = e_M^* \Fix(\RR^0)$ be a global section, and let $\Sol^\scD \subset \Conf^\scD$ be a solution presheaf.

\begin{corollary}
\label{st:Descent result for moduli oo-prestacks}
Let $H \colon \Cart^\opp \to \Grp$ be a \emph{sheaf} of groups, acting on $M$ via a morphism $\phi' \colon H \longrightarrow \Diff(M)$ of sheaves of groups on $\Cart$.
Let $\phi \colon Q_H = N \textint \rmB H \longrightarrow \rmB\scD$ denote the associated morphism in $\sSet_{/N \Cart^\opp}$, and suppose this factors through $\rmB\scD[\cG]$.
Suppose further that the restriction
\begin{equation}
	(\rmB \phi')^* \Sol^{\scD[\cG]}_{[\cG]} \colon (\textint \rmB H)^\opp \longrightarrow \sSet
\end{equation}
is fibrant in $L_{\tau_H} \scK_H$.
The associated moduli $\infty$-prestack $\scMdl_\Phi(\cG) \colon N\Cart^\opp \longrightarrow \scS$ is an $\infty$-stack, i.e.~it satisfies descent with respect to good open coverings in $\Cart$.
\end{corollary}

\begin{proof}
Recall the simplicial presheaf $\Sol^{\scD[\cG]}_{[\cG]} \colon \rmB\scD[\cG]^\opp \longrightarrow \sSet$ from Definition~\ref{def:G_[cG]}.
By construction and~\cite[Prop.~5.4]{Bunk:Localisation_of_sSet}, the $\infty$-presheaf $\scSol_\Phi(\cG) \colon N\Cart^\opp \longrightarrow \scS$ is presented by the simplicial presheaf
\begin{equation}
	\hoLan_{\pi_M^H} \big( (\rmB \phi')^* \Sol^{\scD[\cG]}_{[\cG]} \big)
	\colon \Cart^\opp \longrightarrow \sSet\,,
\end{equation}
possibly up to objectwise fibrant replacement.
The claim then follows from Theorem~\ref{st:descent for spl pshs on tint BH}.
\end{proof}

%%%%%%%%%%%%%%%%%%%%%%%%%%%%%%%%%%%%%%%%%%%%%%%%%%%%%%%%%%%%%%%%%%%%%%%%%%%%

\part{Applications to higher $\U(1)$-connections and NSNS supergravity}
\label{part:Higher U(1) gauge fields}

%%%%%%%%%%%%%%%%%%%%%%%%%%%%%%%%%%%%%%%%%%%%%%%%%%%%%%%%%%%%%%%%%%%%%%%%%%%%

In the second part of this paper we focus on a particularly important case of higher geometric structures consisting of higher principal bundles on $M$ whose structure groups are iterated deloopings of $\U(1)$ and the classifying stacks $\rmB_\nabla^k \U(1)$ for higher $\U(1)$-bundles with connection.
Such higher bundles are better known as $n$-gerbes with $k$-connections.
We study various aspects of the gauge theory of connections on these higher bundles, including higher Yang-Mills solutions, compute the higher gauge actions of these bundles from the formalism in Part~\ref{part:Higher Geometry} and apply Theorem~\ref{st:equiv result for Mdl oo-stacks} to reconcile two perspectives on moduli of NSNS supergravity solutions.

%%%%%%%%%%%%%%%%%%%%%%%%%%%%%%%%%%%%%%%%%%%%%%%%%%%%%%%%%%%%%%%%%%%%%%%%%%%%

\section{Smooth families of higher $\U(1)$-connections on a manifold}
\label{sec:smooth fams of higher U(1)-conns}

%%%%%%%%%%%%%%%%%%%%%%%%%%%%%%%%%%%%%%%%%%%%%%%%%%%%%%%%%%%%%%%%%%%%%%%%%%%%

%%%%%%%%%%%%%%%%%%%%%%%%%%%%%%%%%%%%%%%%%%%%%%%%%%%%%%%%%%%%%%%%%%%%%%%%%%%%

\subsection{Vertical differential forms and vertical Deligne complexes}
\label{sec:vertical forms and Deligne complexes}

%%%%%%%%%%%%%%%%%%%%%%%%%%%%%%%%%%%%%%%%%%%%%%%%%%%%%%%%%%%%%%%%%%%%%%%%%%%%

In terms of local data $n$-gerbes with $k$-connections are constructed from the \v{C}ech-Deligne complexes associated to good open coverings of manifolds~\cite{Gajer:Geo_of_Deligne_coho, FSS:Cech_for_diff_classes, Schreiber:DCCT_v2}.
Here we are interested in explicit models for \textit{smooth families} of \v{C}ech-Deligne cocycles---also called \textit{higher $\U(1)$-connections}, following~\cite{Brylinski:LSp_and_Geo_Quan, Gajer:Geo_of_Deligne_coho}---on a fixed but arbitrary manifold $M$.
In order to describe these, we first introduce smooth families of differential forms on cartesian spaces and then construct \v{C}ech data for smooth families of higher $\U(1)$-connections from these building blocks.

Recall the category $\Cartfam$ from Definition~\ref{def:Cartfam}.
We denote the tangent and cotangent bundles of a manifold $N$ by $TN$ and $T^\vee N$, respectively.
To an object $(p \colon \hat{c} \to c) \in \Cartfam$ we associate its \textit{vertical tangent bundle}
\begin{equation}
	T^v(\hat{c} \to c) = \{X \in T\hat{c}\, | \, p_*(X) = 0 \}
	= \ker(p_*)
	\subset T\hat{c}\,.
\end{equation}
Given a morphism $(\hat{f} \to f) \colon (p \colon \hat{c} \to c) \longrightarrow (q \colon \hat{d} \to d)$, we obtain a bundle morphism
\begin{equation}
	T^v(\hat{f} \to f) \colon T^v(\hat{c} \to c) \longrightarrow T^v(\hat{d} \to d)\,,
	\qquad
	X \longmapsto \hat{f}_*(X)\,,
\end{equation}
which covers the smooth map $f$.
Indeed, since $q \circ \hat{f} = f \circ p$, we have that
\begin{equation}
	p_*(X) = 0
	\quad \Longrightarrow \quad
	q_* \big( \hat{f}_*(X) \big)
	= (q \circ \hat{f})_*(X)
	= (f \circ p)_*(X)
	= 0\,,
\end{equation}
and thus $\hat{f}_*$ restricts to a bundle map
\begin{equation}
	T^v(\hat{f} \to f) = \hat{f}_{|\ker(p_*)} \colon T^v(\hat{c} \to c) \longrightarrow T^v(\hat{d} \to d)\,.
\end{equation}
For $k \in \NN_0$, we define the $C^\infty(\hat{c})$-module
\begin{equation}
	\Omega^{k,v}(\hat{c} \to c) \coloneqq \Gamma \big( \hat{c}; \Lambda^k (T^v (\hat{c} \to c))^\vee \big)
	\subset \Omega^k(\hat{c})\,,
\end{equation}
where, for a vector bundle $E \to \hat{c}$, the symbol $E^\vee$ denotes its dual vector bundle.
A morphism $(\hat{f} \to f) \colon (\hat{c} \to c) \longrightarrow (\hat{d} \to d)$ in $\Cartfam$ induces an $\RR$-linear map
\begin{equation}
	\Omega^{k,v}(\hat{f} \to f) = \hat{f}^* \colon \Omega^{k,v}(\hat{d} \to d) \longrightarrow \Omega^{k,v}(\hat{c} \to c)\,,
	\qquad
	(\hat{f}^*\omega)_{|x}(X_1, \ldots, X_k)
	= \omega_{|\hat{f}(x)} \big( \hat{f}_*(X_1), \ldots, \hat{f}_*(X_k) \big)\,,
\end{equation}
for $X_1, \ldots, X_k \in T^v(\hat{c} \to c)_{|x}$.
Thus, in particular, we obtain, for each $k \in \NN_0$, a presheaf
\begin{equation}
	\Omega^{k,v} \colon \Cartfam^\opp \longrightarrow \Ab
\end{equation}
of abelian groups ($\Ab$ denotes the category of abelian groups).

The de Rham differential $\dd \colon \Omega^k(\hat{c}) \to \Omega^{k+1}(\hat{c})$ restricts to a morphism $\dd^v \colon \Omega^{k,v} \to \Omega^{k+1,v}$:
given vertical vector fields $X_0, \ldots, X_k$ on $(\hat{c} \to c)$ and $\omega \in \Omega^{k,v}(\hat{c} \to c)$, one sets
\begin{align}
	(\dd^v \omega)(X_0, \ldots, X_k)
	\coloneqq &\sum_{i=0}^k (-1)^i \cL_{X_i} \big( \omega(X_0, \ldots, \widehat{X_i}, \ldots, X_k) \big)
	\\
	&+ \sum_{0 \leq i < j \leq k} (-1)^{i+j} \omega \big( [X_i, X_j], X_0, \ldots, \widehat{X_i}, \ldots, \widehat{X_j}, \ldots, X_k \big)\,,
\end{align}
where $\cL$ is the standard Lie derivative, $[X_i, X_j]$ is the standard commutator of vector fields, and the hat denotes omission of an entry.
Note that the commutator of two vertical vector fields is again vertical; see, for instance,~\cite[Prop.~8.30]{Lee:Smooth_Mfds_v2}.
The vertical de Rham differential is natural with respect to morphisms in $\Cartfam$; this follows since, for any $x \in c$, we have that
\begin{equation}
	(\dd^v \omega)_{|(\hat{c}_{|x})} = \dd (\omega_{|(\hat{c}_{|x})})\,,
\end{equation}
i.e.~first taking the vertical de Rham differential and then restricting to any fibre of $p \colon \hat{c} \to c$ give the same form on the fibre as first restricting to the fibre and then applying the ordinary de Rham differential.
Furthermore, morphisms in $\Cartfam$ restrict to fibres, and the ordinary de Rham differential $\dd$ is natural with respect to smooth maps.
Let $\Ch^{\geq 0}$ and $\Ch_{\geq 0}$ denote the categories of non-negatively graded cochain and chain complexes, respectively, of abelian groups.
We obtain a presheaf of non-negatively graded cochain complexes of abelian groups
\begin{equation}
	(\Omega^v, \dd^v) \colon \Cartfam^\opp \longrightarrow \Ch^{\geq 0}\,.
\end{equation}

We also consider the sheaf $\Omega^{k,v}[\hat{c} \to c] = \Gamma(-; \Lambda^k(T^v(\hat{c} \to c))^\vee)$ of vertical (or `relative') $k$-forms on the manifold $\hat{c}$.
This is, in particular, a sheaf of $C^\infty(\hat{c})$-modules on $\hat{c}$.
As such, it is \textit{fine}, i.e.~it admits partitions of unity on $\hat{c}$.
Thus, it is an acyclic sheaf on $\hat{c}$, and since $\hat{c}$ is paracompact Hausdorff, it also follows that the \v{C}ech cohomology of $\Omega^{k,v}[\hat{c} \to c]$ with respect to any open covering of $\hat{c}$ is trivial in positive degrees.
For later reference, we record this as the following lemma.
Recall the Grothendieck coverage $\tau_\rmfam$ on the category $\Cartfam$ from Definition~\ref{def:coverings in Cart_fam}.

\begin{lemma}
\label{st:acyclicity for vertical forms}
For each $k \in \NN_0$, the sheaf $\Omega^{k,v}[\hat{c} \to c]$ on $\hat{c}$ is fine and hence acyclic.
Consequently, for each $\tau_\rmfam$-covering $(\hat{\cU} \to \cU)$ of $(\hat{c} \to c)$ in $\Cartfam$, the total complex associated to the cosimplicial chain complex $\Omega^{k,v}(\v{C}\hat{\cU})$ is acyclic.
\end{lemma}

\begin{definition}
\label{def:b^n_(nabla) U(1)}
We define the presheaf of non-negatively graded chain complexes
\begin{align}
&\big( b^{n+1-k}\rmb^k_\nabla \U(1) \big)^v \colon \Cartfam^\opp \longrightarrow \Ch_{\geq 0}\,,
\\
&\big( b^{n+1-k}\rmb^k_\nabla \U(1) \big)^v =
\begin{tikzcd}[ampersand replacement=\&]
	\big( \Omega^{k,v}
	\& \Omega^{k-1,v} \ar[l, "\dd^v"']
	\& \cdots \Omega^{1,v} \ar[l, "\dd^v"']
	\& \U(1) \big)\,, \ar[l, "\dd^v \log"']
\end{tikzcd}
\end{align}
where $\U(1)$ is situated in degree $n{+}1$ (and $\Omega^{k,v}$ in degree $n{+}1{-}k$).
\end{definition}

\begin{lemma}
\label{st:bb_nabla U(1)^v is tau_fam-local}
For each $0 \leq k \leq n \in \NN_0$, the functor
\begin{equation}
	 \big( \rmb^{n+1-k} \rmb^k_\nabla \U(1) \big)^v \colon \Cartfam^\opp \longrightarrow \Ch_{\geq 0}
\end{equation}
satisfies homotopy descent with respect to $\tau_\rmfam$-coverings.
\end{lemma}

\begin{proof}
Let $(\hat{\cU} \to \cU) \longrightarrow (\hat{c} \to c)$ be a $\tau_\rmfam$-covering of an object $(\hat{c} \to c)$ in $\Cartfam$.
We need to show that the canonical morphism 
\begin{equation}
	\big( \rmb^{n+1-k} \rmb^k_\nabla \U(1) \big)^v (\hat{c} \to c)
	\longrightarrow \underset{l \in \bbDelta}{\holim} \big( \rmb^{n+1-k} \rmb^k_\nabla \U(1) \big)^v \big( \v{C}_l (\hat{\cU} \to \cU) \big)
\end{equation}
is a quasi-isomorphism of chain complexes.
The homotopy limit is modelled by the (truncation of the) total complex of the double complex given by the \v{C}ech resolution of each degree.
Note that the \v{C}ech complexes of the presheaves $\U(1)$ and $\Omega^{p,v}$, for $p = 1, \ldots, k$, viewed as presheaves on $\hat{c}$, are computed using only the good open covering $\hat{\cU}$ of the cartesian space $\hat{c}$.
Since $\Omega^{p,v}$ is a fine sheaf on $\hat{c}$, for each $p \in \NN_0$, its \v{C}ech resolution with respect to $\hat{\cU}$ is acyclic (see Lemma~\ref{st:acyclicity for vertical forms}).
Further, since $\hat{c}$ is cartesian and $\hat{\cU}$ is differentiably good, the \v{C}ech complex of $\U(1)$ with respect to $\hat{\cU}$ is acyclic as well (it computes sheaf cohomology of $\hat{c}$ with coefficients in $\U(1)$; this agrees with integer cohomology one degree higher, which is trivial since $\hat{c}$ is contractible).
Therefore, the $p$-th row in the above double complex is a resolution of the level-$k$ sheaf in $(\rmb^{n+1-k} \rmb^k_\nabla \U(1))^v$.
The claim then follows from~\cite[Lemma~8.5]{Voisin_I}.
\end{proof}

Let $\Ab$ denote the category of abelian groups and let $\Ab_\Delta = \Fun(\bbDelta^\opp, \Ab)$ denote the category of simplicial abelian groups.
The latter carries a model structure which is right induced from the forgetful functor $\Ab_\Delta \to \sSet$ and the Kan-Quillen model structure on $\sSet$.

\begin{definition}
\label{def:B^(n+1-k)B_nabla^k U(1)}
For each $0 \leq k \leq n \in \NN_0$, we define a simplicial presheaf
\begin{equation}
	(\rmB^{n+1-k} \rmB^k_\nabla \U(1))^v
	\coloneqq \varGamma \circ \big( \rmb^{n+1-k} \rmb^k_\nabla \U(1) \big)^v
	\colon \Cartfam^\opp \longrightarrow \sSet\,,
\end{equation}
where $\varGamma \colon \Ch_{\geq 0} \to \Ab_\Delta$ is part of the Dold-Kan correspondence (see also Appendix~\ref{app:Higher U(1)-bundles via sPShs} for more background and details) and where we have implicitly used the forgetful functor $\Ab_\Delta \to \sSet$.
\end{definition}

As a direct consequence of Lemma~\ref{st:bb_nabla U(1)^v is tau_fam-local} and the fact that both $\varGamma \colon \Ch_{\geq 0} \to \Ab_\Delta$ and the forgetful functor $\Ab_\Delta \to \sSet$ preserve homotopy limits, we obtain:

\begin{lemma}
\label{st:families of gerbes satisfy tau_fam-descent}
For each $0 \leq k \leq n \in \NN_0$, the functor
\begin{equation}
	\big( \rmB^{n+1-k} \rmB^k_\nabla \U(1) \big)^v \colon \Cartfam^\opp \longrightarrow \sSet
\end{equation}
is a fibrant object in $\scH_\rmfam^{loc}$.
\end{lemma}

\begin{remark}
\label{rmk:B^(n+k-1)B_nabla^kU(1)^v presents Picard oo-sheaf}
The $\infty$-categorical localisation of $\scH_\rmfam = \Fun(\Cartfam^\opp, \sSet)$ at the objectwise weak homotopy equivalences is equivalent to the $\infty$-category $\scP(\Cartfam) = \scFun(N\Cartfam^\opp, \scS)$ of $\infty$-functors from $N\Cartfam^\opp$ to the $\infty$-category $\scS$ of spaces.
The fact that the functor $(\rmB^{n+1-k} \rmB^k_\nabla \U(1))^v$ factors through the category $\Ab_\Delta$ of simplicial abelian groups implies that it presents an $\infty$-sheaf of Picard $\infty$-groupoids on $\Cartfam$ (by~\cite[Thm.~3.30, Cor.~3.34]{NSS:Pr_ooBdls_II}).
\qen
\end{remark}

The canonical projection morphisms
\begin{equation}
\begin{tikzcd}[column sep=0.75cm]
	\big( b^{n+1-k}\rmb^k_\nabla \U(1) \big)^v \ar[r, equal] \ar[d]
	& \big( \Omega^{k,v} \ar[d]
	& \Omega^{k-1,v} \ar[l, "\dd^v"'] \ar[d, "1"]
	& \Omega^{k-2,v} \ar[l, "\dd^v"'] \ar[d, "1"]
	& \cdots \Omega^{1,v} \ar[l, "\dd^v"'] \ar[d, "1"]
	& \U(1) \big)\,, \ar[l, "\dd^v \log"'] \ar[d, "1"]
	\\
	\big( b^{n+2-k}\rmb^{k-1}_\nabla \U(1) \big)^v \ar[r, equal]
	& \big( 0
	& \Omega^{k-1,v} \ar[l]
	& \Omega^{k-2,v} \ar[l, "\dd^v"]
	& \cdots \Omega^{1,v} \ar[l, "\dd^v"]
	& \U(1) \big)\,, \ar[l, "\dd^v \log"]
\end{tikzcd}
\end{equation}
are fibrations in $\Fun(\Cartfam^\opp, \Ch_{\geq 0})$, where both $\Ch_{\geq 0}$ and the functor category carry the projective model structure.
Composing with the Dold-Kan correspondence thus provides a projective (Kan) fibration
\begin{equation}
	p^k_{k-1} \colon \big( \rmB^{n+1-k} \rmB^k_\nabla \U(1) \big)^v
	\longrightarrow \big( \rmB^{n+2-k} \rmB^{k-1}_\nabla \U(1) \big)^v 
\end{equation}
of simplicial presheaves on $\Cartfam$.
For integers $0 \leq l \leq k \leq n+1$ we also define the composition
\begin{equation}
\label{eq:p^k_l}
	p^k_l \coloneqq p^{l+1}_l \circ \cdots \circ p^{k-1}_{k-2} \circ p^k_{k-1}
	\colon \big( \rmB^{n+1-k} \rmB^k_\nabla \U(1) \big)^v
	\longrightarrow \big( \rmB^{n+1-l} \rmB^l_\nabla \U(1) \big)^v\,.
\end{equation}

%%%%%%%%%%%%%%%%%%%%%%%%%%%%%%%%%%%%%%%%%%%%%%%%%%%%%%%%%%%%%%%%%%%%%%%%%%%%

\subsection{Presenting families of $n$-gerbes with connection on a manifold}
\label{sec:fams of n-gerbes with connection}

%%%%%%%%%%%%%%%%%%%%%%%%%%%%%%%%%%%%%%%%%%%%%%%%%%%%%%%%%%%%%%%%%%%%%%%%%%%%

Next, we describe families of $n$-gerbes with $k$-connection on any fixed manifold $M$.
In the formalism of Section~\ref{sec:families of hgeo structures} we consider the simplicial presheaf
\begin{equation}
	G = \big( \rmB^{n+1-k} \rmB_\nabla^k \U(1) \big)^v
	\colon \Cartfam^\opp \longrightarrow \sSet\,,
\end{equation}
for any $k,n \in \NN_0$ with $k \leq n+1$.
Following the notational conventions of Section~\ref{sec:families of hgeo structures} (see, in particular, Construction~\ref{cstr:sPShs on GCov^D from sPShs on Cart_fam}), we make the following definition:

\begin{definition}
\label{def:wtGrb}
For each $k,n \in \NN_0$ with $k \leq n+1$, we define the simplicial presheaf
\begin{align}
	\widetilde{\Grb}^{n,\scD}_{\nabla|k} \coloneqq \check{C}^{\scD*} \Grb^n_{\nabla|k} \colon (\GCov^\scD){}^\opp &\longrightarrow \sSet\,,
	\\
	(c, \hat{\cU} \to \cU) &\longmapsto \ul{\scH_\rmfam} \big( \v{C}(\hat{\cU} \to \cU), \big( \rmB^{n+1-k} \rmB_\nabla^k \U(1) \big)^v \big)\,.
\end{align}
We denote its restriction to $\GCov \subset \GCov^\scD$ by
\begin{equation}
	\widetilde{\Grb}^{n,M}_{\nabla|k} \colon \GCov^\opp \longrightarrow \sSet\,.
\end{equation}
\end{definition}

\begin{remark}
\label{rmk:Cech formulas for evaluating wtGrb}
One can describe \smash{$\widetilde{\Grb}^{n,\scD}_{\nabla|k}$} explicitly by means of the techniques in Appendix~\ref{app:Higher U(1)-bundles via sPShs} (see, in particular Example~\ref{eg:transition data for n-grbs w K-conn}).
\qen
\end{remark}

\begin{notation}
In the setting of Definition~\ref{def:wtGrb} we also write
\begin{equation}
	\widetilde{\Grb}^{n,\scD}_\nabla
	\coloneqq \widetilde{\Grb}^{n,\scD}_{\nabla|n+1}
	\qquad \text{and} \qquad
	\widetilde{\Grb}^{n,\scD}
	\coloneqq \widetilde{\Grb}^{n,\scD}_{\nabla|0}\,,
\end{equation}
and analogously for \smash{$\widetilde{\Grb}^{n,M}_\nabla$} and \smash{$\widetilde{\Grb}^{n,M}$}.
This notational convention will also apply to the following definitions in this subsection without us spelling it out explicitly in each adaptation.
\qen
\end{notation}

\begin{remark}
The morphisms $p^k_l$ from~\eqref{eq:p^k_l} induce projective fibrations
\begin{equation}
	\widetilde{p}^k_l \colon \widetilde{\Grb}{}^{n,\scD}_{\nabla|k} \longrightarrow \widetilde{\Grb}{}^{n,\scD}_{\nabla|l}\,,
\end{equation}
for each $l,k \in \NN_0$ with $0 \leq l \leq k \leq n+1$.
\qen
\end{remark}

As a consequence of Lemma~\ref{st:families of gerbes satisfy tau_fam-descent}, \smash{$\Grb^{n,\scD}_{\nabla|k}$} satisfies the conditions of Lemma~\ref{st:wtG ess const on fibres of varpi^D}.
Averaging over good open coverings (see, in particular, Proposition~\ref{st:Lan_varpi} and Section~\ref{sec:averaging good open coverings}), we obtain the following simplicial presheaves:

\begin{definition}
\label{def:Grb^M and Grb^D(M)}
For each $n \in \NN_0$ and $k = 0, \ldots, n+1$, we set
\begin{align}
\label{eq:Grb^M and Grb^D(M)}
	\Grb^{n,\scD}_{\nabla|k} &\coloneqq \Lan_{\varpi^\scD} \big( \widetilde{\Grb}^{n,\scD}_{\nabla|k} \big)
	\colon \rmB \scD^\opp \longrightarrow \sSet\,,
	\qquad \text{and}
	\\
	\Grb^{n,M}_{\nabla|k} &\coloneqq \Lan_\varpi \big( \iota^* \widetilde{\Grb}^{n,M}_{\nabla|k} \big)
	\colon \Cart^\opp \longrightarrow \sSet\,.
\end{align}
We refer to these as the simplicial presheaves of smooth families of \textit{$n$-gerbes with $k$-connection on $M$} (with or without the action of $\Diff(M)$, respectively).
\end{definition}

\begin{remark}
\label{rmk:p^k_l are proj fibs}
By Proposition~\ref{st:Lan_varpi} it follows that the morphisms $p^k_l$ from~\eqref{eq:p^k_l} induce projective fibrations
\begin{equation}
	\Grb^{n,\scD}_{\nabla|k} \longrightarrow \Grb^{n,\scD}_{\nabla|l}\,,
\end{equation}
for each $l,k \in \NN_0$ with $0 \leq l \leq k \leq n+1$.
By an abuse of notation, we again denote these morphisms by $p^k_l$.
\qen
\end{remark}

\begin{remark}
The notion of \textit{$k$-connections} on $n$-gerbes was introduced in~\cite{Gajer:Geo_of_Deligne_coho}.
Modulo minor shifts---the paper~\cite{Gajer:Geo_of_Deligne_coho} works with $\CC^*$ in place of $\U(1)$ and calls $n$-gerbes $(n{+}1)$-bundles---Gajer established $n$-gerbes with $k$-connection as a geometric model for the Deligne cohomology groups in any bidegree on $M$.
This paradigm was picked up and further enhanced by others, including in~\cite{Schreiber:DCCT_v2, FSS:Cech_for_diff_classes}, to define Kan complexes of $n$-gerbes with $k$-connection on $M$, thus encoding all levels of morphisms and higher morphisms between these objects.
Fiorenza, Rogers and Schreiber also introduced an abstract way of obtaining smooth families of $n$-gerbes with $k$-connection (and other higher geometric structures) on manifolds under the name of concretification~\cite{FRS:Higher_gerbe_connections} (this contained a mistake which was later corrected in~\cite{BSS:Stack_of_YM_fields, Schreiber:DCCT_v2}).
Here, instead, we have constructed directly simplicial presheaves on $\Cart$ which classify smooth families of $n$-gerbes with $k$-connections on any given manifold $M$.
Additionally, our construction directly encodes not only the higher algebraic structure of their automorphisms, but simultaneously encodes the smooth \textit{and} higher structure of their \textit{symmetries}, i.e.~automorphisms which cover possibly non-trivial diffeomorphisms of $M$.
This enables us to construct and investigate moduli $\infty$-stacks for Gajer's geometric model of Deligne cohomology in the following sections.
\qen
\end{remark}

\begin{remark}
Lemma~\ref{st:base change for 1Cat Lan_varpi} provides a canonical isomorphism
\begin{equation}
	e_M^* \Grb^{n,\scD}_{\nabla|k} \cong \Grb^{n,M}_{\nabla|k}
\end{equation}
of simplicial presheaves on $\Cart$.
\qen
\end{remark}

We also obtain the left fibrations associated to simplicial presheaves from~\eqref{eq:Grb^M and Grb^D(M)}:

\begin{definition}
\label{def:GRB LFib}
We define objects
\begin{alignat}{3}
	\textint \Grb^{n,\scD}_{\nabla|k} &\coloneqq r_{\rmB\scD^\opp}^* \Grb^{n,\scD}_{\nabla|k} &&
	\qquad &&\in \sSet_{/N\rmB\scD^\opp}
	\\
	\textint \Grb^{n,M}_{\nabla|k} &\coloneqq r_{\Cart^\opp}^* \Grb^{n,M}_{\nabla|k}
	&& \qquad &&\in \sSet_{/N \Cart^\opp}\,.
\end{alignat}
\end{definition}

By construction these are covariantly fibrant in $\sSet_{/N\rmB\scD^\opp}$ and $\sSet_{/N \Cart^\opp}$, respectively.

\begin{proposition}
\label{st:properties of GRB and p^k_l}
For each $n \in \NN_0$ and $0 \leq l \leq k \leq n+1$, the following hold true:
\begin{enumerate}
\item There is a canonical isomorphism over $N\Cart^\opp$,
\begin{equation}
	N e_M^* \textint \Grb^{n,\scD}_{\nabla|k} \cong \textint \Grb^{n,M}_{\nabla|k}\,.
\end{equation}

\item The forgetful maps
\begin{align}
	\textint p^k_l \colon \textint \Grb^{n,\scD}_{\nabla|k} \longrightarrow \textint \Grb^{n,\scD}_{\nabla|l}\,,
	\qquad
	\textint p^k_l \colon \textint \Grb^{n,M}_{\nabla|k} \longrightarrow \textint \Grb^{n,M}_{\nabla|l}
\end{align}
are fibrations between fibrant objects in the covariant model structures on $\sSet_{/N\rmB\scD^\opp}$ and $\sSet_{/N\Cart^\opp}$, respectively.
Thus, they are left fibrations by~\cite[Thm.~4.4.14]{Cisinski:HiC_and_HoA}.
\end{enumerate}
\end{proposition}

\begin{proof}
Claim~1 follows from Lemma~\ref{st:r_C^* and pullbacks} together with Lemma~\ref{st:base change for 1Cat Lan_varpi}.
Claim~2 is a consequence of Remark~\ref{rmk:p^k_l are proj fibs} and Theorem~\ref{st:r_C^* as Quillen equivalence}.
\end{proof}

Finally, the simplicial presheaves \smash{$\Grb^{n,\scD}_{\nabla|k}$} and \smash{$\Grb^{n,M}_{\nabla|k}$}, and equivalently their associated left fibrations from Definition~\ref{def:GRB LFib}, present $\infty$-presheaves on $N\rmB\scD$ and $N\Cart$, respectively.
Recall the $\infty$-categorical localisation functor $\gamma_\scC \colon N\Fun(\scC, \sSet) \longrightarrow \scFun(N\scC, \scS)$ from Appendix~\ref{app:rectification}, for any small category $\scC$.

\begin{definition}
For each $k,l \in \NN_0$ with $0 \leq l \leq k \leq n+1$, we define the following $\infty$-functors:
\begin{alignat}{3}
	\bbGrb^{n,\scD}_{\nabla|k} &\coloneqq \gamma_{\rmB\scD^\opp} \big( \Grb^{n,\scD}_{\nabla|k} \big)
	&& \qquad && \in \scP(N\rmB\scD)\,,
	\\
	\bbGrb^{n,M}_{\nabla|k} &\coloneqq \gamma_{\Cart^\opp} \big( \Grb^{n,M}_{\nabla|k} \big)
	&& \qquad && \in \scP(N\Cart)\,.
\end{alignat}
\end{definition}

Note that we again obtain induced projection morphisms
\begin{equation}
	p^k_l \colon \bbGrb^{n,\scD}_{\nabla|k} \longrightarrow \bbGrb^{n,\scD}_{\nabla|l}\,,
\end{equation}
for each $k,l \in \NN_0$ with $0 \leq l \leq k \leq n+1$.

Finally, we define the curvature of a smooth family of $n$-gerbes with connection:
consider an object $X = (c, \hat{\cU} \to \cU)$ in $\GCov^\scD$, and let $(\cG, \cA) = (g, A_1, \ldots, A_{n+1})$ be a vertex in \smash{$\widetilde{\Grb}^{n,\scD}_\nabla(X)$}.
That is, $g \colon \v{C}_{n+1} \hat{\cU} \to \U(1)$ is a smooth function, and $A_i \in \Omega^{i,v}(\v{C}_{n+1-i} \hat{\cU})$ is a vertical $i$-form, for $i = 1, \ldots, n+1$.
These data satisfy analogues of the equations in Example~\ref{eg:transition data for n-grbs w K-conn}, with the de Rham differential replaced by the vertical de Rham differential.
Then, by Lemma~\ref{st:acyclicity for vertical forms} there is a unique $H \in \Omega^{n+2,v,\scD}_\cl(c)$ such that $\delta H = \dd^v A_{n+1}$ on $c {\times} M$.
Analogously to the case for $n$-gerbes with connection (i.e.~without smooth families), one checks that this form is invariant under equivalences in \smash{$\widetilde{\Grb}^{n,\scD}_\nabla(X)$} and refinements of coverings, and, furthermore, compatible with pullbacks along morphisms in $\rmB\scD$.
In particular, we from this construction a morphism
\begin{equation}
\label{eq:curv for Grb^(n,scD)_nabla}
	\curv \colon \pi_0 \Grb^{n,\scD}_\nabla \longrightarrow \Omega^{n+2,v,\scD}
\end{equation}
in $\Fun(\rmB\scD^\opp, \sSet)$, where $\pi_0$ takes connected components objectwise over each $c \in \rmB\scD$.

\begin{definition}
\label{def:curv for Grb^(n,scD)_nabla}
Let $n \in \NN_0$.
The morphism~\eqref{eq:curv for Grb^(n,scD)_nabla} is called the \textit{curvature} of $n$-gerbes with connection.
An connection $\cA$ on an $n$-gerbe $\cG$ is called \textit{flat} if $\curv(\cG, \cA) = 0$.
We let
\begin{equation}
	\Grb^{n,\scD}_\flat \subset \Grb^{n,\scD}_\nabla
\end{equation}
denote the full simplicial subpresheaf on the $n$-gerbes with flat connection.
Finally, it is convenient to set \smash{$\Grb^{n,\scD}_{\nabla|k} \coloneqq \Grb^{n,\scD}_\flat$} for $k > n+1$.
\end{definition}

%%%%%%%%%%%%%%%%%%%%%%%%%%%%%%%%%%%%%%%%%%%%%%%%%%%%%%%%%%%%%%%%%%%%%%%%%%%%

\subsection{The structure of $k$-connections on a fixed $n$-gerbe}
\label{sec:Con_k on fixed n-gerbe}

%%%%%%%%%%%%%%%%%%%%%%%%%%%%%%%%%%%%%%%%%%%%%%%%%%%%%%%%%%%%%%%%%%%%%%%%%%%%

We aim to study higher moduli stacks of geometric data which involve connections on higher gerbes.
Here we first present the stacks of smooth families of connections on a given $n$-gerbes, or, equivalently, $\infty$-stacks of Deligne differential cocycles on any manifold $M$.
Let $n, k \in \NN_0$ with $0 \leq k \leq n+1$.
Recall from Remark~\ref{rmk:section associated to cG} that an $n$-gerbe with $k$-connection $(\cG, \cA^{(k)}) \in \Grb^{n,M}_{\nabla|k}(\RR^0)$ on $M$ induces a morphism
\begin{equation}
	\tilde{\sigma} (\cG, \cA^{(k)}) \colon N\Cart^\opp \longrightarrow \textint \Grb^{n,\scD[\cG]}_{\nabla|k}
\end{equation}
of left fibrations over $N\Cart^\opp$.

\begin{definition}
\label{def:conns on n-gerbes}
Let $0 \leq k \leq n+1 \in \NN_0$, and suppose $\cG \in \Grb^{n,M}(\RR^0)$ is an $n$-gerbe on $M$.
\begin{enumerate}
\item We define a left fibration over $N\Cart^\opp$ as the pullback
\begin{equation}
\label{eq:Con_k(cG)}
\begin{tikzcd}[column sep=1.25cm, row sep=0.75cm]
	\textint \Con_k(\cG) \ar[r] \ar[d]
	& \textint \Grb^{n,M}_{\nabla|k} \ar[d]
	\\
	N\Cart^\opp \ar[r, "\tilde{\sigma} \cG"']
	& \textint \Grb^{n,M}
\end{tikzcd}
\end{equation}
This left fibration presents the \textit{$\infty$-stack $\scCon_k(\cG)$ of (smooth families of) $k$-connections on the $n$-gerbe $\cG$}.

\item Suppose $\cA^{(k)}$ is a $k$-connection on $\cG$, i.e.~the pair $(\cG, \cA^{(k)})$ is an element in \smash{$\Grb^{n,M}_{\nabla|k}(\RR^0)$}.
We define a left fibration over $N\Cart^\opp$ as the pullback
\begin{equation}
\begin{tikzcd}[column sep=1.25cm, row sep=0.75cm]
	\textint \Con(\cG, \cA^{(k)}) \ar[r] \ar[d]
	& \textint \Grb^{n,M}_\nabla \ar[d]
	\\
	N\Cart^\opp \ar[r, "{\tilde{\sigma} (\cG, \cA^{(k)})}"']
	& \textint \Grb^{n,M}_{\nabla|k}
\end{tikzcd}
\end{equation}
This left fibration presents the \textit{$\infty$-stack $\scCon(\cG, \cA^{(k)})$ of (smooth families of) connections on the $n$-gerbe with $k$-connection $(\cG, \cA^{(k)})$}, i.e.~of ways to extend $\cA^{(k)}$ into a full $(n{+}1)$-connection on $\cG$ (for $k \neq 0)$.
\end{enumerate}
\end{definition}

These stacks have, to the best of our knowledge, not appeared in the literature before.
For $n > 0$ we see that, even if we disregard the smooth family aspect and evaluate $\scCon_k(\cG)$ on $\RR^0 \in \Cart$, the $k$-connections on a fixed $n$-gerbe have the structure of a (truncated) $\infty$-groupoid.
This is in stark contrast to the 1-categorical case of connections on principal bundles whose structure group is a Lie group:
the connections on any such bundle form an affine \textit{set}.
The 2-connections on a 1-gerbe, for instance, form an honest groupoid.
For a fixed good open covering $\cU$ of $M$, the morphisms of this groupoid are shifts
\begin{equation}
	(g, A_1, A_2) \longmapsto (g, A_1 + \delta C, A_2 + \dd C)\,,
\end{equation}
where $C \in \Omega^1(\check{C}_0 \cU)$ is any 1-form.
These morphism are neither part of the action of the gauge 2-group of the 1-gerbe, nor of an affine action analogous to that on connections on an principal bundle (see also the following paragraph and the proof of Theorem~\ref{st:forgetting conns is a principal map} below for more on the affine action in this case).
Instead, they are intrinsic to the notion of a 2-connection on a 1-gerbe, and a genuine feature of \textit{higher} connections on a fixed \textit{higher} principal bundle.
They have, so far, not been considered in the literature on higher bundles.

The above groupoid $\scCon(\cG)(\RR^0)$ of connections on a 1-gerbe $\cG$ on $M$ carries a further affine action by the (Baez-Crans) 2-vector space whose objects are pairs $(\omega_1, \omega_2)$ with $\omega_i \in \Omega^i(\check{C}_{2-i}\cU)$, satisfying $\delta \omega_1 = 0$ and $\dd \omega_1 = - \delta \omega_2$, and whose morphisms $(\omega_1, \omega_2) \to (\omega'_1, \omega'_2)$ are 1-forms $\eta_1 \in \Omega^1(\check{C}_0 \cU)$ satisfying $\delta \eta = \omega'_1 - \omega_1$ and $\dd \eta = \omega'_2 - \omega_2$ (as a groupoid, this coincides with the groupoid of 2-connections on the trivial 1-gerbe).

\begin{remark}
The pullback square~\eqref{eq:Con_k(cG)} is a homotopy fibre sequence in $\sSet_{/N \Cart^\opp}$.
It presents a fibre sequence in $\scP(N\Cart)$ of the form
\begin{equation}
\begin{tikzcd}
\label{eq:Con_k fibre sequence}
	\scCon_k(\cG) \ar[r] \ar[d]
	& \bbGrb^{n,M}_{\nabla|k} \ar[d]
	\\
	* \ar[r, "\tilde{\sigma} \cG"']
	& \bbGrb^{n,M}
\end{tikzcd}
\end{equation}
This fibre sequence is simultaneously a categorification and a smooth enhancement of the short exact sequence in~\cite[Thm.~A]{Gajer:Geo_of_Deligne_coho}.
For $k = n+1$ it can also be viewed as a smooth-family enhancement of the upwards-pointing diagonal short exact sequence in the differential cohomology hexagon~\cite{SS:Axiomatic_char_of_diff_coho}.
Since that original reference, categorifications for this diagram have been found~\cite{HS:Quadratic_functions, Schreiber:DCCT_v2, BNV:Diff_Coho_as_sheaves_of_spectra, ADH:Differential_cohomology}, but no smooth families and, in particular, no moduli of differential cocycles have been investigated so far;
in particular, the fibre sequence~\eqref{eq:Con_k fibre sequence} fits into a smooth-family refinement of the model for differential cohomology in~\cite{Schreiber:DCCT_v2}.
\qen
\end{remark}

\begin{example}
Let $\cI$ be the trivial 1-gerbe on $M$, and consider only the groupoid $\scCon_2(\cI)(\RR^0)$ of connections on $\cI$ (i.e.~disregarding smooth families).
This is the groupoid of pairs $(\omega_1,\omega_2)$ described in the preceding paragraph.
By Lemma~\ref{st:acyclicity for vertical forms} each object in $\scCon(\cI)(\RR^0)$ is isomorphic to an object of the form $(0, \omega_2)$.
We thus see that the isomorphism classes of connections on the trivial 1-gerbe satisfy
\begin{equation}
	\pi_0 \big( \scCon_2(\cI)(\RR^0) \big) \cong \frac{\Omega^2(M)}{\dd \Omega^1(M)}\,.
\end{equation}
The automorphisms of any object in this groupoid satisfy
\begin{equation}
	\pi_1 \big( \scCon_2(\cI)(\RR^0) \big) \cong \Omega^1_\cl(M)\,.
\end{equation}
In particular, $\scCon_2(\cI)$ is an honest groupoid, rather than a set.
\qen
\end{example}

%%%%%%%%%%%%%%%%%%%%%%%%%%%%%%%%%%%%%%%%%%%%%%%%%%%%%%%%%%%%%%%%%%%%%%%%%%%%

\section{The gauge action and higher form symmetries}
\label{st:higher gauge action for Grb^n_nabla|k}

%%%%%%%%%%%%%%%%%%%%%%%%%%%%%%%%%%%%%%%%%%%%%%%%%%%%%%%%%%%%%%%%%%%%%%%%%%%%

Higher moduli stacks of geometric data involving higher $\U(1)$-connections are quotients of higher stacks closely related to the higher stacks $\scCon(\cG, \cA^{(k)})$ of connections on a given $n$-gerbe with $k$-connection from Section~\ref{sec:Con_k on fixed n-gerbe}.
Recall that both $\scCon(\cG, \cA^{(k)})$ as well as the higher gauge group $\scAut(\cG, \cA^{(k)})$ have higher structure.
In order to better understand the gauge-theoretic structure of connections on $(\cG, \cA^{(k)})$, we compute the action of $\scAut(\cG, \cA^{(k)})$ on $\scCon(\cG, \cA^({k)})$ in this section.

The general formalism in Part~\ref{part:Higher Geometry} provides the intrinsic action of higher gauge transformations, encoded in terms of a left fibration of simplicial sets or a map of $\infty$-prestacks which classifies this action (see, for instance, Equation~\eqref{eq:def tint Sol_phi(cG), Sol_res(cG)} and the beginning of Section~\ref{sec:contractible S(Sol)}).
This had the great advantage that it allowed us to describe moduli $\infty$-prestacks in terms of an $\infty$-categorical left Kan extension.
While this does encode an $\infty$-categorical action of a smooth higher group, the fact that we can describe families $n$-gerbes with $k$-connection on $M$ very concretely in terms of \v{C}ech-Deligne complexes and the Dold-Kan construction poses the question whether we can describe also the higher gauge action in such explicit terms.
In particular, this should recover the familiar gauge action for $\U(1)$-bundles on $M$.
The first main goal of this section is thus to extract an explicit model for the higher gauge action for higher $\U(1)$-bundles from our formalism.
This amounts to strictifying the higher group action presented by the aforementioned left fibration.
In particular, it is an important consistency check that the strictified action which emerges from our formalism takes the expected form.

The second goal of this section is to explore the relation of higher $\U(1)$ gauge symmetries with higher form symmetries in the physics literature~\cite{GKSW:Gen_global_syms}.
We believe that the higher symmetry groups of higher geometric structures which we developed in Section~\ref{sec:higher symmetry groups of hgeo strs} are relevant much more broadly in the context of invertible higher categorical symmetries in field theories, but we restrict our attention to the case of higher $\U(1)$-connections here.

Let $k,l,n \in \NN_0$, with $0 \leq k < l \leq n+1$.
In the formalism of Section~\ref{sec:moduli oo-prestacks of hgeo strs on M}, we set \smash{$\Conf^\scD = \Grb^{n,\scD}_{\nabla|k}$} and \smash{$\Fix^\scD = \Grb^{n,\scD}_{\nabla|l}$}, with the canonical projection morphism \smash{$p^k_l \colon \Grb^{n,\scD}_{\nabla|k} \longrightarrow \Grb^{n,\scD}_{\nabla|l}$}.
To unravel the action of higher $\U(1)$-gauge transformations, let $\Sol^\scD = \Conf^\scD$.
Suppose that \smash{$(\cG, \cA^{(l)}) \in \Grb^{n,M}_{\nabla|k}(\RR^0)$} is an $n$-gerbe with $l$-connection on $M$.
Recall the pullback square~\eqref{eq:def tint Sol_phi(cG), Sol_res(cG)}, which here reads as
\begin{equation}
\begin{tikzcd}[column sep=1cm, row sep=0.75cm]
	\textint \Sol^M_\res(\cG, \cA^{(l)}) \ar[r, hookrightarrow] \ar[d]
	& \textint \Sol^M(\cG, \cA^{(l)}) \ar[d]
	\\
	\rmB\AUT^\rev(\cG, \cA^{(l)}) \ar[r, hookrightarrow]
	& \textint \Fix^M_{[\cG]}
\end{tikzcd}
\end{equation}
The vertical morphisms are left fibrations (by Proposition~\ref{st:SOL fibseq in sSet}).

\begin{remark}
We also recall that by Definition~\ref{def:scMdl_phi(cG)} the composed left fibration $\textint \Sol^M_\res(\cG, \cA^{(l)}) \to N\Cart^\opp$ presents the moduli $\infty$-stack of $k$-connections on $(\cG, \cA^{(l)})$ modulo gauge transformations, i.e.~automorphisms of $(\cG, \cA^{(l)})$.
We can thus equally write
\begin{equation}
	\textint \Mdl_{\res,k}(\cG, \cA^{(l)})
	\coloneqq \textint \Sol^M_\res(\cG, \cA^{(l)})
\end{equation}
for the domain of the left fibration over $N\Cart^\opp$ which describes $k$-connections on $(\cG, \cA^{(l)})$ modulo the action of automorphisms of $(\cG, \cA^{(l)})$.
\qen
\end{remark}

By construction we have that
\begin{equation}
	\rmB\AUT^\rev(\cG, \cA^{(l)})
	= r_{\Cart^\opp}^* \rmB \Aut^\rev(\cG, \cA^{(l)})\,,
\end{equation}
where $\rmB \Aut^\rev(\cG, \cA^{(l)})$ is the full simplicial subpresheaf of $\Grb^{n,M}_{\nabla|l}$ on the section $(\cG, \cA^{(l)})$.
That is,
\begin{equation}
	\rmB \Aut^\rev(\cG, \cA^{(l)})
	= \rmB \Omega_{(\cG, \cA^{(l)})} \Grb^{n,M}_{\nabla|l}\,,
\end{equation}
where $\Omega_{(\cG, \cA^{(l)})}$ denotes loops based at $(\cG, \cA^{(l)})$.
By the construction of $\Grb^{n,M}_{\nabla|l}$ via the Dold-Kan correspondence (compare Section~\ref{sec:smooth fams of higher U(1)-conns}) it follows that there is a canonical isomorphism
\begin{equation}
\label{eq:BAut = Grb^n-1}
	\Aut^\rev(\cG, \cA^{(l)})
	\cong \Grb^{n-1,M}_{\nabla|l}\,.
\end{equation}
This is a presheaf of simplicial abelian groups on $M$.
It is a presentation of the group object $\scAut(\cG, \cA^{(l)})$ in $\scP(N\Cart)$ (in the sense of~\cite[Prop.~3.35]{NSS:Pr_ooBdls_II}).

\begin{notation}
Recall our convention that
\begin{equation}
	\Grb^{n,M}_{\nabla|l}
	= \Grb^{n,M}_\flat
\end{equation}
whenever $l \geq n+2$.
In particular, for $l = n+1$ (i.e.~$(\cG, \cA^{(l)})$ an $n$-gerbe with full connection), the equivalence~\eqref{eq:BAut = Grb^n-1} reads as 
\begin{equation}
	\Aut(\cG, \cA^{(l)})
	= \Grb^{n-1,M}_\flat
\end{equation}
(where we were able to drop the superscript $\rev$ because both groups are abelian).
\qen
\end{notation}

We restrict our attention to families paramterised by any fixed cartesian space $c \in \Cart$.
By Lemma~\ref{st:wtG ess const on fibres of varpi^D}, Proposition~\ref{st:Lan_varpi} and~\cite[Cor.~5.3.6]{Cisinski:HiC_and_HoA} we may further restrict ourselves to working over any fixed object $(\hat{\cU} \to \cU)$ of $\GCov_{|c}$ (see Definition~\ref{def:GCov^D(M)}).
For brevity we let
\begin{equation}
	A \coloneqq \Tot^\times \big( \tau_{\geq 0} (\rmb^{n+1-k} \rmb_\nabla^k \U(1) (\v{C}\hat{\cU})) \big)\,
	\qqandqq
	C \coloneqq \Tot^\times \big( \tau_{\geq 0} (\rmb^{n+1-l} \rmb_\nabla^l \U(1) (\v{C}\hat{\cU})) \big)
\end{equation}
denote the total (product) complexes of \v{C}ech-Deligne double complexes (truncated to non-negative degrees) describing smooth families of $n$-gerbes on $M$ with $k$- and $l$-connections, respectively.
In the notation of Definition~\ref{def:wtGrb}, we thus have
\begin{equation}
	\widetilde{\Grb}^{n,M}_{\nabla|k}(\hat{\cU} \to \cU)
	= \varGamma A
	\qqandqq
	\widetilde{\Grb}^{n,M}_{\nabla|l}(\hat{\cU} \to \cU)
	= \varGamma C\,.
\end{equation}
By the above discussion, there is a canonical isomorphism
\begin{equation}
	\rmB\Aut^\rev(\cG, \cA^{(l)})(\hat{\cU} \to \cU)
	\cong \Gamma (\tau_{\geq 1} C)
	= \varGamma \big(
	\begin{tikzcd}
		0
		& Z_1(C) \ar[l]
		& C_2 \ar[l, "D_C"']
		& \cdots \big)\,. \ar[l, "D_C"']
	\end{tikzcd}
\end{equation}
Here $D_C$ is the differential in the complex $C$, and $Z_1(C)$ denotes the group of 1-cycles in $C$; it is situated in degree one.
We further define a simplicial subset of $\varGamma A$ as the pullback
\begin{equation}
\begin{tikzcd}[column sep=1.5cm, row sep=1cm]
	\Mdl_{\res,k}(\cG, \cA^{(l)})(\hat{\cU} \to \cU) \ar[r] \ar[d]
	& \varGamma A \ar[d, "\varGamma p^k_l"]
	\\
	\varGamma (\tau_{\geq 1} C) \ar[r, "{\{(\cG, \cA^{(l)})\}}"']
	& \varGamma C
\end{tikzcd}
\end{equation}
where the bottom morphism is the inclusion of the full simplicial subset on the single vertex $(\cG, \cA^{(l)})$.
Explicitly, in terms of the Dold-Kan construction (using the notation from Appendix~\ref{app:Higher U(1)-bundles via sPShs}), we have that
\begin{equation}
\label{eq:Mdl_k(cG,cA^(l)) via Cech data}
	\Mdl_{\res,k}(\cG, \cA^{(l)})(\hat{\cU} \to \cU)_r
	\coloneqq \big\{ X = (X_\varphi)_{\varphi \in [r] \twoheadrightarrow [t]} \in \varGamma A \, \big| \, X_\varphi \in A_t,\, (p^k_l)_0(X_{[r] \twoheadrightarrow [0]}) = (\cG, \cA^{(l)}) \big\}\,,
\end{equation}
for each $r \in \NN_0$.
Then, by construction there is a canonical equivalence
\begin{equation}
	\Mdl_{\res,k}(\cG, \cA^{(l)})(\hat{\cU} \to \cU)
	\simeq \Sol^M(\cG, \cA^{(l)})(c)\,.
\end{equation}
The Kan fibration $\varGamma p^k_l$ restricts to a Kan fibration (which we denote by the same name)
\begin{equation}
	\varGamma p^k_l \colon \Mdl_{\res,k}(\cG, \cA^{(l)})(\hat{\cU} \to \cU) \longrightarrow \varGamma (\tau_{\geq 1} C)\,.
\end{equation}
Note that $\varGamma (\tau_{\geq 1} C)$ is a reduced simplicial set, i.e.~it has a unique vertex.
There is a further pullback diagram
\begin{equation}
\begin{tikzcd}[column sep=1.5cm, row sep=1cm]
	\Con_k(\cG, \cA^{(l)})(\hat{\cU} \to \cU) \ar[r] \ar[d]
	& \Mdl_{\res,k}(\cG, \cA^{(l)})(\hat{\cU} \to \cU) \ar[d, "\varGamma p^k_l"]
	\\
	\Delta^0 \ar[r]
	& \varGamma (\tau_{\geq 1} C)
\end{tikzcd}
\end{equation}
in $\sSet$, where, explicitly,
\begin{align}
\label{eq:Con_k(cG,cA^(l)) via Cech data}
	&\Con_k(\cG, \cA^{(l)})(\hat{\cU} \to \cU)_r
	\\*
	&= \big\{ X = (X_\varphi)_{\varphi \in [r] \twoheadrightarrow [t]} \, \big| \, X_\varphi \in A_t,\, (p^k_l)_0(X_{\id_{[0]}}) = (\cG, \cA^{(l)}),\, (p^k_l)_r(X_\varphi) = 0\ \forall\, \varphi \colon [r] \twoheadrightarrow [t], t > 0 \big\}\,.
\end{align}
Finally, by construction there is a canonical equivalence
\begin{equation}
	\Con_k(\cG, \cA^{(l)})(c)
	\wequiv \Con_k(\cG, \cA^{(l)})(\hat{\cU} \to \cU)\,.
\end{equation}
Here we have used that $\Con_k(\cG, \cA^{(l)}) \in \scH$ is a homotopy sheaf (this follows, for instance, by Corollary~\ref{st:Descent result for moduli oo-prestacks} or the fact that the pullbacks in Definition~\ref{def:conns on n-gerbes} are homotopy pullbacks of left fibrations satisfying homotopy descent).

The higher gauge action, restricted to families parameterised by $c \in \Cart$, is thus modelled by the Kan fibration
\begin{equation}
\label{eq:Kan fib for gauge action}
	\varGamma p^k_l \colon \Mdl_{\res,k}(\cG, \cA^{(l)})(\hat{\cU} \to \cU)
	\longrightarrow \varGamma (\tau_{\geq 1} C)\,.
\end{equation}
More concretely, this Kan fibration presents an $\infty$-functor
\begin{equation}
\label{eq:vGamma tau_(geq 1) C --> scS}
	\rmB \big( \Aut(\cG, \cA^{(l)}) (\hat{\cU} \to \cU) \big)
	= \varGamma (\tau_{\geq 1} C) \longrightarrow \scS
\end{equation}
which encodes the action of the $\infty$-group $\scAut(\cG, \cA^{(l)})(c)$ presented by the reduced simplicial set $\varGamma (\tau_{\geq 1} C)$ on the space $\Con_k(\cG, \cA^{(l)})(c)$ presented by $\varGamma \ker(p^k_l)$.

We now present a strictification of this action via the Dold-Kan correspondence.
By~\cite[Rmk.~3.5, Rmk.~3.6, Lemma~3.7]{Bunk:Localisation_of_sSet}, an $\infty$-functor as in~\eqref{eq:vGamma tau_(geq 1) C --> scS} corresponds to the following collection of data:
for each morphism $Y \colon \Delta^r \to \varGamma (\tau_{\geq 1} C)$, we have to specify a left fibration
\begin{equation}
	\big( \Mdl_{\res,k}(\cG, \cA^{(l)})(\hat{\cU} \to \cU) \big)(Y) \longrightarrow \Delta^r
\end{equation}
together with a choice, for each $\psi \colon \Delta^s \to \Delta^r$, of a pullback diagram
\begin{equation}
\label{eq:coherence mps for Mdl_k(cG,cA^(l))(Y)}
\begin{tikzcd}[column sep=2cm, row sep=1cm]
	\big( \Mdl_{\res,k}(\cG, \cA^{(l)})(\hat{\cU} \to \cU) \big) (Y; \psi) \ar[r, "\widehat{\psi}"] \ar[d]
	& \big( \Mdl_{\res,k}(\cG, \cA^{(l)})(\hat{\cU} \to \cU) \big)(Y) \ar[d]
	\\
	\Delta^s \ar[r, "\psi"']
	& \Delta^r
\end{tikzcd}
\end{equation}
These data have to satisfy the coherence condition that
\begin{equation}
	\big( \Mdl_{\res,k}(\cG, \cA^{(l)})(\hat{\cU} \to \cU) \big) (Y; \psi)
	= \big( \Mdl_{\res,k}(\cG, \cA^{(l)})(\hat{\cU} \to \cU) \big) (Y \circ \psi)
\end{equation}
as objects in $\sSet_{/\Delta^s}$, for each $Y \colon \Delta^r \to \varGamma (\tau_{\geq 1} C)$ and each $\psi \colon \Delta^s \to \Delta^r$.
We obtain a canonical choice of such data by defining $(\Mdl_{\res,k}(\cG, \cA^{(l)})(\hat{\cU} \to \cU))(Y)$ to be the pullback
\begin{equation}
\begin{tikzcd}
	\big( \Mdl_{\res,k}(\cG, \cA^{(l)})(\hat{\cU} \to \cU) \big) (Y) \ar[r] \ar[d]
	& \Mdl_{\res,k}(\cG, \cA^{(l)})(\hat{\cU} \to \cU) \ar[d]
	\\
	\Delta^r \ar[r, "Y"']
	& \varGamma (\tau_{\geq 1} C)
\end{tikzcd}
\end{equation}
for each $Y \colon \Delta^r \to \varGamma (\tau_{\geq 1} C)$, where we use the canonical construction of pullbacks in $\Set$ to construct $(\Mdl_{\res,k}(\cG, \cA^{(l)})(\hat{\cU} \to \cU))(Y)$ (compare also~\cite[Lemma~3.7]{Bunk:Localisation_of_sSet}), and setting
\begin{equation}
	\big( \Mdl_{\res,k}(\cG, \cA^{(l)})(\hat{\cU} \to \cU) \big) (Y; \psi)
	\coloneqq \big( \Mdl_{\res,k}(\cG, \cA^{(l)})(\hat{\cU} \to \cU) \big) (Y \circ \psi)\,,
\end{equation}
for each $\psi \colon \Delta^s \to \Delta^r$.
This canonically induces choices for the morphisms \smash{$\widehat{\psi}$} as in~\eqref{eq:coherence mps for Mdl_k(cG,cA^(l))(Y)} (see also~\cite[Lemma~3.7]{Bunk:Localisation_of_sSet}).
Note that, explicitly,
\begin{align}
\label{eq:Mdl_k(cG,cA^(l))(Y) via Cech data}
	&\big( \Mdl_{\res,k}(\cG, \cA^{(l)})(\hat{\cU} \to \cU) \big)(Y)_M
	\\*
	&= \big\{ X = (X_\varphi)_{\varphi \in [m] \twoheadrightarrow [t]} \, \big| \, X_\varphi \in A_t,\, (p^k_l)_0(X_{[m] \twoheadrightarrow [0]}) = (\cG, \cA^{(l)}),\, (p^k_l)_m(X_\varphi) = Y_\varphi\ \forall\, \varphi \colon [m] \twoheadrightarrow [t], t > 0 \big\}\,.
\end{align}

\begin{remark}
\label{rmk:Kan fibs for gauge action}
Since $\Mdl_{\res,k}(\cG, \cA^{(l)})(\hat{\cU} \to \cU) \to \varGamma (\tau_{\geq 1})$ is even a Kan fibration, so are all the fibrations $(\varGamma A)(Y) \to \Delta^r$ and $(\varGamma A)(Y; \psi) \to \Delta^s$, since Kan fibrations are stable under pullback.
\qen
\end{remark}

Suppose we are given any such coherent collection of left fibrations and pullbacks.
Let $Y \colon \Delta^r \to \varGamma (\tau_{\geq 0} C)$ be any $r$-simplex.
Since the inclusion $\Delta^{\{0\}} \hookrightarrow \Delta^r$ of the initial vertex in $\Delta^r$ is left anodyne, for each $r \in \NN$, the left-hand vertical morphism in each solid square
\begin{equation}
\label{eq:lifts for strictified gauge action}
\begin{tikzcd}[column sep=1.cm, row sep=1cm]
	\Delta^{\{0\}} \times \Con_k(\cG, \cA^{(l)})(\hat{\cU} \to \cU) \ar[r, hookrightarrow] \ar[d, hookrightarrow]
	& \big( \Mdl_{\res,k}(\cG, \cA^{(l)})(\hat{\cU} \to \cU) \big)(Y) \ar[d]
	\\
	\Delta^r \times \Con_k(\cG, \cA^{(l)})(\hat{\cU} \to \cU) \ar[r, "\pr"'] \ar[ur, dashed, "\rho(Y)" description]
	& \Delta^r
\end{tikzcd}
\end{equation}
is a left anodyne monomorphisms of simplicial sets (by means of~\cite[Prop.~3.4.3]{Cisinski:HiC_and_HoA}).
As the right-hand vertical morphism is a left fibration (even a Kan fibration by Remark~\ref{rmk:Kan fibs for gauge action}), the above square admits a lift $\rho(Y)$, which is, moreover, unique up to contractible choice.

Further, combining Remark~\ref{rmk:Kan fibs for gauge action} with (the duals of)~\cite[Def.~4.4.1]{Cisinski:HiC_and_HoA} and~\cite[Prop.~4.4.11]{Cisinski:HiC_and_HoA}, we obtain that the top horizontal morphism in the above diagram is a cofinal monomorphism, and hence a left anodyne extension~\cite[Cor.~4.1.9]{Cisinski:HiC_and_HoA} as well.
If follows that any lift in the square is itself cofinal in this case, and in particular is a weak homotopy equivalence.

That means that we can trivialise the right-hand vertical Kan fibration.
Provided we are further able to find lifts in~\eqref{eq:lifts for strictified gauge action} for each $Y$ in a way that is coherent with the simplicial structure of $\varGamma (\tau_{\geq 1} C)$, this will produce an explicit, strict model for the $\infty$-functor~\eqref{eq:vGamma tau_(geq 1) C --> scS}.

\begin{theorem}
\label{st:Cech pres of gauge action}
Let $k,l,n \in \NN_0$ with $0 \leq l < k$, and let $(\cG, \cA^{(l)})$ be an $n$-gerbe with $l$-connection on $M$.
Let $(\hat{\cU} \to \cU) \in \GCov_{|c}$.
There exists a family of isomorphisms
\begin{equation}
	\rho(Y) \colon \Delta^r \times \Con_k(\cG, \cA^{(l)})(\hat{\cU} \to \cU)
	\longrightarrow \big( \Mdl_{\res,k}(\cG, \cA^{(l)})(\hat{\cU} \to \cU) \big)(Y)\,,
\end{equation}
indexed by all morphisms $Y \colon \Delta^r \to \varGamma (\tau_{\geq 1} C)$ in $\sSet$, which is, moreover, compatible with the simplicial structure of $\varGamma (\tau_{\geq 1} C)$.
That is, $\rho$ assembles into a natural isomorphism of functors $\bbDelta_{/\varGamma (\tau_{\geq 1} C)} \to \sSet$ which commutes with the projection to the functor $\Delta^\cdot \circ \pr \colon \bbDelta_{/\varGamma (\tau_{\geq 1} C)} \to \sSet$.
\end{theorem}

\begin{proof}
For each $(\hat{c} \to c) \in \Cartfam$ (see Section~\ref{sec:Families of cartesian spaces} for the notation) and each $t \in \NN_0$, there exists a morphism of abelian groups (i.e.~forgetting the differentials)
\begin{equation}
	\big( \rmb^{n+1-l} \rmb_\nabla^l \U(1) \big)_t (\hat{c} \to c)
	\longrightarrow \big( \rmb^{n+1-k} \rmb_\nabla^k \U(1) \big)_t (\hat{c} \to c)
\end{equation}
which is the zero map in degrees $t < n+1-k$ and the identity map in degrees $t \geq n+1-k$ (recall the notation from Section~\ref{sec:smooth fams of higher U(1)-conns}).
These morphism are compatible with the action of morphisms in $\Cartfam$, so that we obtain a family of morphism of presheaves of abelian groups
\begin{equation}
	s_t \colon \big( \rmb^{n+1-l} \rmb_\nabla^l \U(1) \big)_t
	\longrightarrow \big( \rmb^{n+1-k} \rmb_\nabla^k \U(1) \big)_t\,,
\end{equation}
which are sections of the morphisms
\begin{equation}
	(p^k_l)_t \colon \big( \rmb^{n+1-k} \rmb_\nabla^k \U(1) \big)_t
	\longrightarrow \big( \rmb^{n+1-l} \rmb_\nabla^l \U(1) \big)_t\,,
\end{equation}
for each $t \in \NN_0$.

By the compatibility of $s_t$ with the action of morphisms, $s_t$ is compatible with \v{C}ech differentials, and so we obtain, for each $t \in \NN_0$, an induced morphism of abelian groups
\begin{equation}
	\v{s}_t \colon (\tau_{\geq 1} C)_t \longrightarrow A_t\,.
\end{equation}
However, this is \textit{not} a morphism of chain complexes since the differential $D_C$ in $\tau_{\geq 1} C$ acts on the highest form degrees only with the \v{C}ech differential, whereas $D_A$ also acts with the (vertical) de Rham differential.
Since $s_t$ is a section of $(p^k_l)_t$, we have that
\begin{equation}
\label{eq:p circ v(s) = id}
	(p^k_l)_t \circ \v{s}_t = \id_{(\tau_{\geq 1} C)_t}
	\qquad
	\forall\, t \in \NN_0\,.
\end{equation}

In order to specify a collection of maps $\rho(Y)$ as in the statement, it suffices to specify, for each $Y$, the map
\begin{equation}
	\rho_{(r)}(Y) \colon \{\sigma_r\} \times \big( \Con_k(\cG, \cA^{(l)})(\hat{\cU} \to \cU) \big)_r
	\longrightarrow \big( \Mdl_{\res,k}(\cG, \cA^{(l)})(\hat{\cU} \to \cU) \big)(Y)_r\,,
\end{equation}
where $\sigma_r \in \Delta^r_r$ is the unique non-degenerate $r$-simplex.
The map $\rho(Y)$ will then be fixed by demanding that, for each $\psi \colon \Delta^s \to \Delta^r$, there is a a commutative square
\begin{equation}
\begin{tikzcd}[column sep=2cm, row sep=1cm]
	\Delta^s \times \Con_k(\cG, \cA^{(l)})(\hat{\cU} \to \cU) \ar[r, "\rho(Y \circ \psi)"] \ar[d, "\psi \times \id"]
	& \big( \Mdl_{\res,k}(\cG, \cA^{(l)})(\hat{\cU} \to \cU) \big)(Y \circ \psi) \ar[d]
	\\
	\Delta^r \times \Con_k(\cG, \cA^{(l)})(\hat{\cU} \to \cU) \ar[r, "\rho(Y)"']
	& \big( \Mdl_{\res,k}(\cG, \cA^{(l)})(\hat{\cU} \to \cU) \big)(Y)
\end{tikzcd}
\end{equation}
where the right-hand vertical morphism is fixed by the assignment $(\Mdl_{\res,k}(\cG, \cA^{(l)})(\hat{\cU} \to \cU))(-)$.

Recall the expressions~\eqref{eq:Con_k(cG,cA^(l)) via Cech data} and~\eqref{eq:Mdl_k(cG,cA^(l))(Y) via Cech data}.
The $r$-simplex $Y \colon \Delta^r \to \varGamma (\tau_{\geq 1} C)$ consists of a family
\begin{equation}
	Y = (Y_{\varphi \colon [r] \twoheadrightarrow [t]})
	\quad
	\in \bigoplus_{\varphi \colon [r] \twoheadrightarrow [t]} (\tau_{\geq 1} C)_t\,,
	\qquad
	\text{with}
	\qquad
	Y_{[r] \twoheadrightarrow [0]} = 0\,.
\end{equation}
Let $X = (X_{\varphi \colon [r] \twoheadrightarrow [t]})$ be an $r$-simplex of $\Con_k(\cG, \cA^{(l)})(\hat{\cU} \to \cU)$.
We set
\begin{equation}
	\big( \rho_{(r)}(Y)(X) \big)_{\varphi \colon [r] \twoheadrightarrow [t]}
	\coloneqq X_\varphi + \check{s}_t(Y_\varphi)
	\quad
	\in A_t\,.
\end{equation}
In particular, by~\eqref{eq:Con_k(cG,cA^(l)) via Cech data} and~\eqref{eq:p circ v(s) = id} we have that
\begin{align}
	\big( \rho_{(r)}(Y)(X) \big)_{[r] \twoheadrightarrow [0]}
	&= X_{[r] \twoheadrightarrow [0]}\,,
	\qquad \text{and}
	\\
	(p^k_l)_t \big( \rho_{(r)}(Y)(X) \big)_{\varphi \colon [r] \twoheadrightarrow [t]}
	&= (p^k_l)_t \big( X_\varphi + \v{s}_t(Y_\varphi) \big)
	= Y_\varphi\,,
\end{align}
for each surjection $\varphi \colon [r] \twoheadrightarrow [t]$ with $t > 0$.
Therefore, we see that, indeed,
\begin{equation}
	\rho_{(r)}(Y)(X) \in \Mdl_{\res,k}(\cG, \cA^{(l)})(\hat{\cU} \to \cU)
\end{equation}
(compare~\eqref{eq:Con_k(cG,cA^(l)) via Cech data} and~\eqref{eq:Mdl_k(cG,cA^(l))(Y) via Cech data}).
Moreover, the map $\rho_{(r)}(Y)$ has an inverse given by
\begin{align}
	\rho_{(r)}(Y)^{-1} \colon \big( \Mdl_{\res,k}(\cG, \cA^{(l)})(\hat{\cU} \to \cU) \big)_r
	&\longrightarrow \big( \Con_k(\cG, \cA^{(l)})(\hat{\cU} \to \cU) \big)_r\,,
	\\
	\big( X' = (X'_{\varphi \colon [r] \twoheadrightarrow [t]}) \big)
	&\longmapsto \big( X = (X'_{\varphi} - \check{s}_t(Y_\varphi)) \big)\,.
\end{align}
By construction, the induced map $\rho(Y)$ is compatible with the projection to $\Delta^r$ and the simplicial structure.
\end{proof}

\begin{remark}
We can thus unravel the higher gauge action of $\scAut(\cG, \cA^{(l)})$ on $\scCon_k(\cG, \cA^{(l)})$ as follows:
the assignment
\begin{equation}
	Y \longmapsto \Delta^r \times \Con_k(\cG, \cA^{(l)})(\hat{\cU} \to \cU)\,,
\end{equation}
is now manifestly constant on objects.
Instead, it contains all information of the higher gauge action in the structure morphisms
\begin{equation}
	\Delta^s \times \Con_k(\cG, \cA^{(l)})(\hat{\cU} \to \cU)
	\longrightarrow \Delta^r \times \Con_k(\cG, \cA^{(l)})(\hat{\cU} \to \cU)
\end{equation}
associated to simplicial maps $\psi \colon \Delta^s \to \Delta^r$.

However, by Theorem~\ref{st:Cech pres of gauge action} it is equivalent to the standard presentation~\eqref{eq:Mdl_k(cG,cA^(l))(Y) via Cech data} of the $\infty$-functor extracted from the left fibration (and even Kan) fibration
\begin{equation}
	\varGamma p^k_l \colon \Mdl_{\res,k}(\cG, \cA^{(l)})(\hat{\cU} \to \cU)
	\longrightarrow \varGamma (\tau_{\geq 1} C)\,.
\end{equation}
Thus, we have presented an equivalent $\infty$-functor $\varGamma (\tau_{\geq 0} C) \to \scS$ which also classifies the above left fibration, and, therefore, encodes the desired higher gauge action.
\qen
\end{remark}

We can read off the action of higher gauge transformations:
the $n$-gerbe with $l$-connection $(\cG, \cA^{(l)})$ is presented by a \v{C}ech-Deligne cocycle $(g, A_1, \ldots, A_l)$ with respect to the covering $(\hat{\cU} \to \cU)$.
Let the tuple $(g, A_1, \ldots, A_l, A_{l+1}, \ldots, A_k)$ present an extension of $(g, A_1, \ldots, A_l)$ to an $n$-gerbe with $k$-connection.
An automorphism of $(g, A_1, \ldots, A_l)$ is a \v{C}ech-Deligne \textit{chain} $(h, a_1, \ldots, a_{l-1}) \in (\tau_{\geq 1} C)_0$.
The identity one-simplex on $(g, A_1, \ldots, A_k)$ is given by the pair (compare Appendix~\ref{app:Higher U(1)-bundles via sPShs})
\begin{equation}
	X \coloneqq \big( (g, A_1, \ldots, A_k), (1, \underbrace{0, \ldots, 0}_{\text{$(k{-}1)$ many}}) \big)\,.
\end{equation}
Using the 1-simplex $Y \colon \Delta^1 \to \varGamma (\tau_{\geq 1} C)$ given by $(h, a_1, \ldots, a_{l-1})$, and employing the notation of the proof of Theorem~\ref{st:Cech pres of gauge action}, we obtain
\begin{equation}
	\rho_{(1)}(Y)(X)
	= \big( (g, A_1, \ldots, A_k), (h, a_1, \ldots, a_{l-1}, \underbrace{0, \ldots, 0}_{\text{$(k{-}l)$ many}}) \big)\,.
\end{equation}
The higher gauge action of $(h, a_1, \ldots, a_{l-1})$ on $(g, A_1, \ldots, A_k)$ is then the target of the above 1-simplex%
\footnote{Strictly speaking, it is the image under the isomorphism \smash{$\rho_{(0)}^{-1}(Y_0)$} of this target, where $Y_0 \colon \Delta^0 \to \varGamma(\tau_{\geq 1} C)$ is the unique morphism, but \smash{$\rho_{(0)}^{-1}(Y_0)$} is simply the identity morphism on $\Con_k(\cG, \cA^{(l)})(\hat{\cU} \to \cU)$.}; that is,
\begin{align}
	(g, A_1, \ldots, A_k) \triangleleft (h, a_1, \ldots, a_{l-1})
	&= d_0 \big( \rho_{(1)}(Y)(X) \big)
	\\
	&= (g, A_1, \ldots, A_{l-1}, A_l + (-1)^{n+1-l} \dd^v a_l, A_{l+1}, \ldots, A_k)\,.
\end{align}

\begin{example}
\label{eg:SES for Sharpe gen syms}
Consider the case
\begin{equation}
	\Conf^\scD = \Grb^{1,\scD}_{\nabla|k}
	\qqandqq
	\Fix^\scD = \Grb^{1,\scD}_{\nabla|l}\,,
\end{equation}
and let $\Sol^\scD \subset \Conf^\scD$ be some solution subpresheaf.
Here we take $l \in \{0,1\}$, $0 \leq l \leq k$, and we use the convention that \smash{$\Grb^{n,\scD}_{\nabla|k} = \Grb^{n,\scD}_\flat$} for all $k \geq n+2$.
Let $(\cG, \cA^{(l)}) \in \Grb^{1,M}_{\nabla|l}(\RR^0)$ be a 1-gerbe with $l$-connection on $M$.
Let $H$ be a Lie group acting on $M$ via a morphism $H \to \Diff_{[\cG]}(M)$; equivalently, the action is encoded in a morphism $\phi \colon \rmB H \to \rmB\scD[\cG]$ over $\Cart$.
This gives rise to a short exact sequence of group objects in $\scP(N\Cart)$
\begin{equation}
\label{eq:SES for Sharpe gen syms}
\begin{tikzcd}
	\bbGrb^{0,M}_{\nabla|l} \ar[r]
	& \scSym_\Phi(\cG, \cA^{(l)}) \ar[r]
	& H\,.
\end{tikzcd}
\end{equation}
Here we have used that, by construction of $\Grb^{n,\scD}_{\nabla|l}$, there is a canonical equivalence of abelian group objects
\begin{equation}
	\scAut(\cG, \cA^{(l)})
	\simeq \bbGrb^{n-1,M}_{\nabla|l}\,,
\end{equation}
for each $(\cG, \cA^{(l)}) \in \Grb^{n,M}_{\nabla|l}(\RR^0)$.
Recall that $\Grb^{0,M}_{\nabla|l}$ is (the nerve of) the presheaf of groupoids describing smooth families of, respectively, $\U(1)$-bundles on $M$ (for $l = 0$), $\U(1)$-bundles with connection (for $l = 1$), and $\U(1)$-bundles with flat connection (for $l > 1$).
The group structure is induced by the canonical tensor product of $\U(1)$-bundles on $M$.
In particular, the smooth 2-group $\Grb^{0,M}_\flat$ is a smooth enhancement of the 2-groups appearing in~\cite[Sec.~3.1]{Sharpe:Gen_global_syms} and~\cite[Sec.~4.1]{GKSW:Gen_global_syms}.

We consider some important special cases in order to make further contact with the physics literature:
\begin{enumerate}
\item Suppose that $\rmH^2(M;\ZZ) \cong 0$, i.e.~each $\U(1)$-bundle on $M$ is trivialisable.
Then, there is a canonical equivalence
\begin{equation}
	\Grb^{0,M} = \Grb^{0,M}_{\nabla|0}
	\simeq \rmB \big( \U(1)^M \big)\,,
\end{equation}
where $\U(1)^M$ is the sheaf of groups which assigns to $c \in \Cart$ the group of smooth maps $c {\times} M \to \U(1)$.
That is, under this assumption and for $l = 0$, the sequence~\eqref{eq:SES for Sharpe gen syms} is an extension of $H$ by $\rmB(\U(1)^M)$.

Under the same assumption, there is a canonical equivalence between $\Grb^{0,M}_\nabla$ and the simplicial presheaf of groups obtained via the Dold-Kan construction from the complex
\begin{equation}
\begin{tikzcd}
	\Omega^{1,v,M}
	& \U(1)^M\,, \ar[l, "\dd^v \log"']
\end{tikzcd}
\end{equation}
where $\Omega^{p,v,M}$ denotes the sheaf on $\Cart$ which assigns to $c \in \Cart$ the group of vertical $p$-forms on $c {\times} M$.
Analogously, there is an equivalence between $\Grb^{0,M}_\flat$ and the simplicial presheaf associated to the complex
\begin{equation}
\begin{tikzcd}
	\Omega^{1,v,M}_\cl
	& \U(1)^M\,, \ar[l, "\dd^v \log"']
\end{tikzcd}
\end{equation}
where $\Omega^{p,v,M}_\cl$ denotes the sheaf on $\Cart$ which assigns to $c \in \Cart$ the group of those vertical $p$-forms on $c {\times} M$ which are closed under the vertical de Rham differential $\dd^v$.

\item Suppose, additionally, that $M$ is 1-connected, i.e.~$\pi_i(M) = 0$ for $i = 0,1$.
In that case there is only one isomorphism class of $\U(1)$-bundles with flat connection on $M$, and we have a canonical equivalence
\begin{equation}
	\Grb^{0,M}_\flat
	\simeq \rmB \U(1)\,.
\end{equation}
Under these assumptions and for $l = 2$, the sequence~\eqref{eq:SES for Sharpe gen syms} is an extension of $H$ by $\rmB \U(1)$.
These extensions appear also in~\cite{Sharpe:Gen_global_syms}.

\item If the Lie group $H$ is a discrete group (such as any finite group) the short exact sequences~\eqref{eq:SES for Sharpe gen syms} are classified by group cohomology; this still takes coefficients in higher groups, but we can disregard the smooth structure of the higher groups involved in this case and the problem becomes purely algebraic.
Then, the sequence~\eqref{eq:SES for Sharpe gen syms} is classified by a cohomology class in
\begin{equation}
	\rmH^2 \big( H; \Grb^{0,M}_{\nabla|l}(\RR^0) \big)\,.
\end{equation}
In particular, under the simplifying assumptions from~(1) and~(2) in this example and for $l > 0$, we have a canonical equivalence
\begin{equation}
	 \rmH^2 \big( H; \Grb^{0,M}_{\nabla|l}(\RR^0) \big)
	\cong \rmH^2 \big( H; \rmB\U(1) \big)
	\cong \rmH^3 \big( H; \U(1) \big)\,,
\end{equation}
where $\U(1)$ is considered purely as an algebraic group.
That is, elements of this cohomology group describe obstructions to endowing a 1-gerbe $(\cG, \cA)$ with 2-connection on $M$ with a $H$-equivariant structure (see Corollary~\ref{st:char of equivar structures}).
\qen
\end{enumerate}
\end{example}

\begin{remark}
\label{rmk:higher form syms}
Let $(\cG, \cA^{(n)})$ be an $n$-gerbe with $n$-connection on $M$, and recall the canonical equivalence of smooth higher groups
\begin{equation}
	\scAut(\cG, \cA^{(n)})
	\simeq \bbGrb^{n-1,M}_\nabla\,.
\end{equation}
Each pair $(\Sigma^{(n)}, \sigma)$ of a closed oriented $n$-manifold $\Sigma^{(n)}$ and smooth map $\sigma \colon \Sigma^{(n)} \to M$ gives rise to a smooth group morphism
\begin{align}
	U_{(-)}(\Sigma^{(n)}, \sigma) \coloneqq \hol \big( \Sigma, \sigma^*(-) \big)
	\colon \pi_0\, \scAut(\cG, \cA^{(n)}) &\longrightarrow \U(1)\,,
	\\
	(\cG', \cA') &\longmapsto \hol \big( \Sigma^{(n)}, \sigma^*(\cG', \cA') \big)\,,
\end{align}
which sends an $(n{-}1)$-gerbe with connection $(\cG', \cA')$ to its holonomy around $(\Sigma^{(n)}, \sigma)$.
Note that the homotopy set $\pi_0$ is formed objectwise over $\Cart$, i.e.~in the $\infty$-topos $\scP(N\Cart)$ (the same applies to $\pi_p$ below).
Similarly, any smooth map $\sigma \colon \Sigma^{(n-p)} \to M$ from a closed oriented $(n{-}p)$-manifold $\Sigma^{(n-p)}$ induces a morphism of smooth groups
\begin{equation}
\label{eq:testing hoGrps of scAut(G,A)}
	U_{(-)}(\Sigma^{(n-p)}, \sigma)
	\colon \pi_p\, \scAut(\cG, \cA) \longrightarrow \U(1)\,,
	\qquad
	(\cG', \cA') \longmapsto \hol \big( \Sigma^{(n-p)}, \sigma^*(\cG', \cA') \big)\,,
\end{equation}
for each $p \in \{1, \ldots, n\}$, and this morphism depends only on the homotopy class of the map $\sigma \colon \Sigma^{(n-p)} \to M$ for $p > 0$.

In the case where $(\cG, \cA)$ is an $n$-gerbe with $(n{+}1)$-connection on $M$, we have
\begin{equation}
	\scAut(\cG, \cA)
	\simeq \bbGrb^{n-1,M}_\flat\,.
\end{equation}
In this situation also the smooth group morphism
\begin{equation}
	U_{(-)}(\Sigma^{(n)}, \sigma) \coloneqq \hol \big( \Sigma^{(n)}, \sigma^*(-) \big)
	\colon \pi_0\, \scAut(\cG, \cA) \longrightarrow \U(1)
\end{equation}
is topological, i.e.~it depends only on the smooth homotopy class of the map $\sigma \colon \Sigma^{(n)} \to M$.

The assignment $(\cG', \cA') \mapsto U_{(\cG', \cA')}(\Sigma, \sigma)$ bears precisely the structure of a $(d{-}n{-}1)$-form symmetry on $M$ in the sense of~\cite{GKSW:Gen_global_syms} (see Eqn.~(1.1)), where $\dim(M) = d$ (at least if we restrict to cases where $\sigma$ is an embedding).
Note that, more generally, this is true in any situation where the smooth higher automorphism group is given by $\bbGrb^{n-1,M}_\flat$.

For $n=1$, the above morphisms admit a geometric explanation and enhancement in terms of the smooth functorial field theories constructed from gerbes with connection in~\cite{BW:OCFFTs, BW:Transgression_of_D-branes} (the case of higher $n$ will be treated in forthcoming work).
We expect that this provides a link between the higher symmetry groups we obtained here and the formalism for higher symmetries in quantum field theories described in~\cite{FMT:Top_sym_in_QFT}.
Finally, note that for $p \geq 1$, the maps~\eqref{eq:testing hoGrps of scAut(G,A)} are not visible in the truncated approach to higher abelian gauge theory in~\cite{FMS:Uncertainty_of_fluxes, FMS:Heisenberg_and_NC_fluxes, Szabo:Quant_of_Higher_Ab_GT}, where the focus lies purely on $\pi_0 \scAut(\cG)$, for an $n$-gerbe $\cG$ on $M$.
\qen
\end{remark}

%%%%%%%%%%%%%%%%%%%%%%%%%%%%%%%%%%%%%%%%%%%%%%%%%%%%%%%%%%%%%%%%%%%%%%%%%%%%

\section{Forgetting higher connection data}

%%%%%%%%%%%%%%%%%%%%%%%%%%%%%%%%%%%%%%%%%%%%%%%%%%%%%%%%%%%%%%%%%%%%%%%%%%%%

The goal of this section is to show that projection morphisms
\begin{equation}
	\textint p^k_l \colon \textint \Grb^{n,\scD}_{\nabla|k} \longrightarrow \textint \Grb^{n,\scD}_{\nabla|l}
\end{equation}
are 0-connected (see Definition~\ref{def:objwise pi_0-surjective}) for each $k,l \in \NN_0$ with $0 \leq l \leq k \leq n$:

\begin{proposition}
\label{st:GRB conn comps}
For each $k,l,n \in \NN_0$ with $0 \leq l \leq k \leq n$ and each $c \in \rmB\scD$ (resp.~$c \in \Cart$), the morphisms induced on connected components,
\begin{align}
	\pi_0 (\textint p^k_{l|c}) \colon \pi_0 \Big( \big( \textint \Grb^{n,\scD}_{\nabla|k} \big)_{|c} \Big)
	&\longrightarrow \pi_0 \Big( \big( \textint \Grb^{n,\scD}_{\nabla|l} \big)_{|c} \Big)
	\\
	\pi_0 (\textint p^k_{l|c}) \colon \pi_0 \Big( \big( \textint \Grb^{n,M}_{\nabla|k} \big)_{|c} \Big)
	&\longrightarrow \pi_0 \Big( \big( \textint \Grb^{n,M}_{\nabla|l} \big)_{|c} \Big)
\end{align}
are bijections.
\end{proposition}

This will allow us to apply Theorem~\ref{st:equiv result for Mdl oo-stacks} to various moduli stacks involving connections on $n$-gerbes.
We prove the claim for\smash{ $\textint \Grb^{n,\scD}_{\nabla|k}$}; the proof for \smash{$\textint \Grb^{n,M}_{\nabla|k}$} is analogous.
Proposition~\ref{st:GRB conn comps} will follow directly from Corollary~\ref{st:Grb conn comps} below.

\begin{lemma}
\label{st:DiffCoho to Coho}
Let $M$ be a manifold and $c \in \Cart$ a cartesian space.
Consider $(\rmb^l_\nabla \U(1))^v$ as a $\Ch_{\geq 0}$-valued presheaf on $c {\times} M$ (recall the notation from Section~\ref{sec:smooth fams of higher U(1)-conns}).
The group homomorphism
\begin{equation}
	\kappa \colon \rmH_k \big( c {\times} M; (\rmb^l_\nabla \U(1))^v \big)
	\longrightarrow \rmH_k \big( c {\times} M; (\rmb^l \U(1))^v \big) = \rmH^{l+k+1}(M;\ZZ)
\end{equation}
has the following properties:
\begin{enumerate}
\item for $k \geq 1$, it is an isomorphism,

\item for $k = 0$, it is surjective.
\end{enumerate}
\end{lemma}

Here, for each $k,l \in \NN_0$, we can describe
\begin{equation}
	\rmH_k \big( c {\times} M; (\rmb^l_\nabla \U(1))^v \big)
	\cong \rmH^0 \big( c {\times} M; (\rmb^k \rmb^l_\nabla \U(1))^v \big)
\end{equation}
as hypercohomology in a chain complex of sheaves on $c {\times} M$.

\begin{proof}
We adapt an argument in~\cite{Brylinski:LSp_and_Geo_Quan} (see also\cite[Prop.~1.3.7]{Waldorf:Thesis}).
Here it is more convenient to work with the presheaf of chain complexes
\begin{equation}
	\rmb_\nabla^{l+1}\ZZ =
	\begin{tikzcd}
		\big( \Omega^{l,v}
		& \Omega^{l-1,v} \ar[l, "\dd^v"']
		& \cdots  \ar[l, "\dd^v"']
		& \Omega^{1,v}  \ar[l, "\dd^v"']
		& \Omega^0 \ar[l, "\dd^v"']
		& \ZZ \big)\,. \ar[l, hookrightarrow, "\iota"']
	\end{tikzcd}
\end{equation}
The morphism $\Omega^0 \to \U(1)$ induces a canonical objectwise weak equivalence
\begin{equation}
	 (\rmb^{l+1}_\nabla \ZZ)^v \wequiv (\rmb^l_\nabla \U(1))^v\,.
\end{equation}
For $0 \leq r \leq s \in \NN_0$, consider the chain complex of sheaves of $C^\infty(c)$-modules on $c {\times} M$,
\begin{equation}
	\Omega^v(r,s)[c {\times} M]  =
	\hspace{-0.2cm}
	\begin{tikzcd}
		\big( \Omega^{s,v}[c {\times} M]
		& \cdots \ar[l, "\dd^v"']
		& \Omega^{r,v}[c {\times} M] \big) \ar[l, "\dd^v"']\,,
	\end{tikzcd}
\end{equation}
with $\Omega^{s,v}[c {\times} M]$ situated in degree zero.
Here $\Omega^{s,v}[c {\times} M]$ denotes the sheaf which assigns to an open subset $U \subset c {\times} M$ the vertical differential $s$-forms on $U$ (i.e.~the $s$-forms with differentials only along the $M$-direction, but smoothly varying along the $c$-direction).
We use the notation $\Omega^{s,v}(c {\times} M)$ to denote the global sections of this sheaf on $c {\times} M$.
In each degree, the complex $\Omega(r,s)^v[c {\times} M]$ consists of a sheaf on $c {\times} M$ which is fine as a sheaf of $C^\infty(c {\times} M)$-modules (compare Lemma~\ref{st:acyclicity for vertical forms}); hence, each degree of $\Omega^v(r,s)[c {\times} M]$ is an acyclic sheaf of abelian groups on $c {\times} M$.
Consequently, we have a canonical isomorphism
\begin{align}
	\rmH_k \big( c {\times} M; \Omega(r,s)^v[c {\times} M] \big)
	&= \rmH_k
	\hspace{-0.2cm}
	\begin{tikzcd}[ampersand replacement=\&]
		\big(c {\times} M; \Omega^{s,v,M}[c {\times} M]
		\& \cdots \ar[l, "\dd^v"']
		\& \Omega^{r,v,M}[c {\times} M] \big) \ar[l, "\dd^v"']
	\end{tikzcd}
	\\
	&=
	\begin{cases}
		\Omega^{s,v}(c {\times} M) / \dd^v \big( \Omega^{s-1,v}(c {\times} M) \big)\,, & k = 0\,,
		\\
		\Omega^{s-k,v}_{cl}(c {\times} M) / \dd^v \big( \Omega^{s-k-1,v}(c {\times} M) \big)\,, & 1 \leq k < s-r\,,
		\\
		\Omega^{r,v}_{cl}(c {\times} M)\,, & k = s-r\,,
		\\
		0\,, & k < 0,\ k > s-r\,.
	\end{cases}
\end{align}
for each $k \in \NN_0$, we have a short exact sequence of chain complexes of sheaves
\begin{equation}
\label{eq:SES for b_nabla^(l+1)ZZ}
\begin{tikzcd}
	0 \ar[r]
	& \rmb^k \Omega(0,l)^v \ar[r]
	& (\rmb^k \rmb^{l+1}_\nabla \ZZ)^v \ar[r]
	& (\rmb^k \rmb^{l+1} \ZZ)^v \ar[r]
	& 0\,.
\end{tikzcd}
\end{equation}
The homology long exact sequence associated to this short exact sequence of chain complexes of sheaves now yields both claims.
\end{proof}

By construction of \smash{$\textint \Grb^{n,\scD}_{\nabla|k}$} in Definition~\ref{def:GRB LFib} and Lemma~\ref{st:r_C^* and pullbacks} there is a canonical isomorphism
\begin{equation}
	\big( \textint \Grb^{n,\scD}_{\nabla|k} \big)_{|c}
	 \cong \Grb^{n,\scD}_{\nabla|k}(c)
	 = \colim \Big( \widetilde{\Grb}^{n,\scD}_{\nabla|k} \colon (\GCov^\scD)_{|c}^\opp \to \sSet \Big)\,.
\end{equation}
Let $X = (\hat{\cU} \to \cU, c)$ be an object in $\GCov^\scD$.
Lemma~\ref{st:wtG ess const on fibres of varpi^D} implies that the canonical cocone morphism
\begin{equation}
	\eta_{|X} \colon \widetilde{\Grb}^{n,\scD}_{\nabla|k}(X)
	\longrightarrow \colim \Big( \widetilde{\Grb}^{n,\scD}_{\nabla|k} \colon( \GCov^\scD)_{|c}^\opp \to \sSet \Big)
\end{equation}
is a weak homotopy equivalence, for each choice of object $X = (\hat{\cU} \to \cU, c) \in \GCov^\scD$.
It thus suffices to show that, for any one \smash{$X \in \GCov^\scD_{|c}$}, the induced morphisms
\begin{equation}
	\widetilde{p}^k_l \colon \widetilde{\Grb}^{n,\scD}_{\nabla|k}(X) \longrightarrow \widetilde{\Grb}^{n,\scD}_{\nabla|l}(X)
\end{equation}
induce bijections of connected components.

Let $\cV = \{V_i\}_{i \in \Lambda}$ be a good open covering of $M$.
We obtain from this an object $X = (\hat{\cU} \to \cU, c)$ of $\GCov^\scD$ by setting $\cU$ to be the trivial covering of $c$ (it has the whole of $c$ as its only patch) and $\hat{\cU} = \{\hat{U}_i\}_{i \in \Lambda}$ with $\hat{U}_i = c \times V_i$ for each $i \in \Lambda$.
Recall the presheaves of chain complexes $\rmb^m \rmb_\nabla^k \U(1)$ from Section~\ref{sec:vertical forms and Deligne complexes}.

\begin{proposition}
\label{st:wtGrb conn comps}
Let $\cV = \{V_a\}_{a \in \Lambda}$ be a good open covering of $M$, and let $c$ be a cartesian space.
Let $X = (\hat{\cU} \to \cU, c)$ be the induced object of $\GCov^\scD$ constructed above.
For each $l \leq k \in \{0, \ldots, n\}$ and $c \in \Cart$, the morphism of Kan complexes
\begin{equation}
	\widetilde{p}^k_{l|c} \colon \widetilde{\Grb}^{n,\scD}_{\nabla|k}(X)
	\longrightarrow \widetilde{\Grb}^{n,\scD}_{\nabla|l}(X)
\end{equation}
induces a bijection on connected components.
This is, in general, not true for $k = n+1$.
Moreover, the isomorphisms on connected components are induced by the canonical isomorphisms of abelian groups
\begin{equation}
	\pi_0 \big( \widetilde{\Grb}^{n,\scD}_{\nabla|k}(X) \big)
	\cong \rmH^{n+2}(M;\ZZ)\,.
\end{equation}
\end{proposition}

\begin{proof}
By Definitions~\ref{def:B^(n+1-k)B_nabla^k U(1)} and~\ref{def:wtGrb} the problem is equivalent to showing that the morphism
\begin{equation}
\begin{tikzcd}
	\widetilde{p}^k_{0|\check{C}\cU} \colon \tau_{\geq 0} \Tot^\times \big( (\rmb^{n+1-k} \rmb^k_\nabla \U(1))^v(\v{C}(\hat{\cU} \to \cU)) \big) \ar[r]
	& \tau_{\geq 0} \Tot^\times \big( (\rmb^{n+1} \U(1))^v(\v{C}(\hat{\cU} \to \cU)) \big)\,,
\end{tikzcd}
\end{equation}
induces an isomorphism on $\rmH_0$, where we used the notation of Remark~\ref{rmk:Cech formulas for evaluating wtGrb}.
There are canonical isomorphisms
\begin{align}
	\rmH_i \big( \tau_{\geq 0} \Tot^\times \big( (\rmb^{n+1-k} \rmb_\nabla^{k+1} \ZZ(1))^v (\v{C}(\hat{\cU} \to \cU)) \big) \big)
	&\cong \rmH_i \big( \tau_{\geq 0} \Tot^\times \big( (\rmb^{n+1-k} \rmb_\nabla^k \U(1))^v (\v{C}(\hat{\cU} \to \cU)) \big) \big)\,,
	\\
	\rmH_i \big( \tau_{\geq 0} \Tot^\times \big( (\rmb^{n+1} \U(1))^v (\v{C}(\hat{\cU} \to \cU)) \big) \big)
	&\cong \rmH^{n+2-i}(M;\ZZ)\,.
\end{align}

We start by showing that we obtain a surjection on $\rmH_0$.
We have a commutative diagram
\begin{equation}
\begin{tikzcd}[column sep=1.5cm]
	\rmH_0 \big( \tau_{\geq 0} \Tot^\times \big( (\rmb^{n+2}_\nabla\ZZ)^v (\v{C}(\hat{\cU} \to \cU)) \big) \big)
	\ar[r, "\rmH_0 (\widetilde{p}^{n+2}_0)"] \ar[d, "\cong"']
	& \rmH_0 \big( \tau_{\geq 0} \Tot^\times \big( (\rmb^{n+2} \ZZ)^v (\v{C}(\hat{\cU} \to \cU)) \big) \big)
	\ar[d, "\cong"]
	\\
	\rmH^{n+2} \big( c {\times} M; (\rmb^{n+2}_\nabla \ZZ)^v \big) \ar[r]
	& \rmH^{n+2}(c {\times} M;\ZZ) \cong \rmH^{n+2}(M;\ZZ)
\end{tikzcd}
\end{equation}
Since each degree of \smash{$(\rmb^{n+2}_\nabla\ZZ)^v$} is an acyclic sheaf on cartesian spaces, and we can compute the sheaf hypercohomology in terms of \v{C}ech cohomology, the vertical morphisms in this diagram are isomorphisms.
The bottom morphism is surjective by Lemma~\ref{st:DiffCoho to Coho}(2), so that the top morphism must be surjective as well.
Finally, the map \smash{$\widetilde{p}^{n+2}_{0|\check{C}\cU}$} factors through \smash{$\widetilde{p}^k_{0|\check{C}\cU}$}, for $k \in \{0, \ldots, n+1\}$.
Thus, the latter must be surjective as well.

Next, we show that \smash{$\widetilde{p}^k_{0|}$} is also injective.
To see this, first observe that the chain complex \smash{$\tau_{\geq 0} \Tot^\times ((\rmb^{n+1-k} \rmb^{k+1}_\nabla\ZZ)^v (\v{C}(\hat{\cU} \to \cU))$} can be described as follows (compare also Remark~\ref{rmk:Cech formulas for evaluating wtGrb} and Appendix~\ref{app:Higher U(1)-bundles via sPShs}):
a degree-zero element is a tuple $(z, A_0, \ldots, A_k)$ with $z \in \ZZ(\check{C}_{n+2}\cV)$ and $i$-forms $A_i \in \Omega^{i,v}(c {\times} \check{C}_{n+1-i} \cV)$ for $i = \{0, \ldots, k\}$.
These data satisfy
\begin{align}
\label{eq:Cech-Deligne closedness}
	0 &= \delta z\,,
	\\
	0 &= \delta A_0 + (-1)^{n+2} \iota z
	\\
	0 &= \delta A_{i+1} + (-1)^{n+1-i} \dd^v A_i\,,
	\quad \text{for $i = 1, \ldots, k-1$}\,.
\end{align}
The map $\widetilde{p}^k_0$ projects the tuple $(z, A_0, \ldots, A_k)$ to $z$.
We now repeatedly use Lemma~\ref{st:acyclicity for vertical forms}.
Suppose that $(z, A_0, \ldots, A_k)$ is such that there exists a $y \in \ZZ(c {\times} \check{C}_{n+1}\cV)$ which satisfies $\delta y = z$.
It follows that $\delta(A_0 + (-1)^{n+1} \iota y) = 0$ and by Lemma~\ref{st:acyclicity for vertical forms} there exists some element $a_0 \in \Omega^{v,0}(c, \check{C}_n\cV)$ satisfying $A_0 = \delta a_0 + (-1)^{n+1} \iota y$.
Iterating this process, we find elements $a_i \in \Omega^{v,i}(c, \check{C}_{n-i}\cU)$ satisfying the identity $A_i = \delta a_i + (-1)^{n-i} \dd^v a_{i-1}$ for each $i = 0, \ldots, k$.
That is, $(z, A_0, \ldots, A_{k-1})$ is exact in the total complex.
However, in order to find $a_k$, it is essential that $k \in \{0, \ldots, n\}$, and we exclude the case $k = n+1$.
Thus, for $k \in \{0, \ldots, n+1\}$, the map \smash{$\rmH_0(\widetilde{p}^k_0)$} is injective, and hence an isomorphism.

For $k = n+1$ the map is not an isomorphism in general; Lemma~\ref{st:DiffCoho to Coho}(2) implies that it is surjective with kernel
\begin{equation}
	\frac{(\Omega^{n+1,v})(c {\times} M)}{\dd^v (\Omega^{n,v}(c {\times} M))}
	\bigg/ \rmH^{n+1}(M;\ZZ)\,,
\end{equation}
which is, in general, non-trivial.
\end{proof}

\begin{corollary}
\label{st:Grb conn comps}
For each $l < k \in \{0, \ldots, n\}$ and $c \in \Cart$, the projective fibration
\begin{equation}
	p^k_l \colon \Grb^{n,\scD}_{\nabla|k}
	\longrightarrow \Grb^{n,\scD}_{\nabla|l}
\end{equation}
of projectively fibrant simplicial presheaves on $\rmB\scD$ is 0-connected.
\end{corollary}

\begin{proof}
This follows from Proposition~\ref{st:wtGrb conn comps} and the fact that, for each object $X = (\hat{\cU} \to \cU, c)$ of $\GCov^\scD$, the canonical morphism
\begin{equation}
	\widetilde{\Grb}^{n,\scD}_{\nabla|k}(X)
	\longrightarrow \Grb^{n,\scD}_{\nabla|k}(c)
\end{equation}
is a weak equivalence:
the inclusion \smash{$\GCov^\scD_{|c} \hookrightarrow \GCov^\scD_{/c}$} is homotopy cofinal, so that the left Kan extension in the definition of \smash{$\Grb^{n,\scD}_{\nabla|k}(c)$} can be computed using $\GCov^\scD_{|c}$.
This category is cofiltered (hence has contractible nerve), and the diagram whose colimit we take sends each morphism in \smash{$\GCov^\scD_{|c}$} to a weak equivalence.
The claim then follows, for instance, from the fact that the inclusion of the object \smash{$X \colon * \to \GCov^\scD_{|c}$} is a homotopy cofinal functor and that filtered colimits of simplicial sets are homotopy colimits~\cite[Prop.~7.3]{Dugger:Combinatorial_Mocats}.
\end{proof}

Proposition~\ref{st:GRB conn comps} now follows by applying the rectification functor
\begin{equation}
	r_{\rmB\scD^\opp}^* \colon \Fun(\rmB\scD^\opp, \sSet) \longrightarrow \sSet_{/N \rmB \scD[\cG]^\opp}\,.
\end{equation}

\begin{remark}
Given any $n$-gerbe with $k$-connection $(\cG, \cA^{(k)})$ on $M$---i.e.~an element of \smash{$\Grb^{n,M}_{\nabla|k}(\RR^0)$}---let $(\cG, \cA^{(l)}) = p^k_l(\cG, \cA^{(k)})$ denote the $n$-gerbe with $l$-connection on $M$ obtained by forgetting part of the connection data (for $0 \leq k \leq n$).
In particular, we write $\cG = p^k_0(\cG, \cA^{(k)})$ for the underlying $n$-gerbe on $M$ without connection.
Remark~\ref{rmk:(-)_[G] and isos on conn comps} implies the identities
\begin{align}
	\Diff_{[\cG, \cA^{(k)}]} &= \Diff_{[\cG, \cA^{(l)}]}
	= \Diff_{[\cG]}\,,
	\\
	\rmB\scD[\cG, \cA^{(k)}] &= \rmB\scD[\cG, \cA^{(l)}]
	= \rmB\scD[\cG]\,,
	\\
	\big( \Grb^{n,\scD[\cG]}_{\nabla|k} \big)_{[\cG, \cA^{(k)}]}
	&= \big( \Grb^{n,\scD[\cG]}_{\nabla|k} \big)_{[\cG, \cA^{(l)}]}
	= \big( \Grb^{n,\scD[\cG]}_{\nabla|k} \big)_{[\cG]}
\end{align}
(using the notation introduced in Definitions~\ref{def:Diff_[cG] and D_[cG]} and~\ref{def:G_[cG]}).
\qen
\end{remark}

%%%%%%%%%%%%%%%%%%%%%%%%%%%%%%%%%%%%%%%%%%%%%%%%%%%%%%%%%%%%%%%%%%%%%%%%%%%%

\section{The homotopy type of the stack of $n$-gerbes with $k$-connection on $M$}
\label{sec:HoType of Grb^(n,M)}

%%%%%%%%%%%%%%%%%%%%%%%%%%%%%%%%%%%%%%%%%%%%%%%%%%%%%%%%%%%%%%%%%%%%%%%%%%%%

With our construction of moduli $\infty$-(pre)stacks of higher gerbes with connection on any manifold $M$ at hand, an interesting question about the global properties of these moduli stacks is to compute their homotopy type.
We achieve this by means of the functor $S \colon \scP(N\Cart) \to \scS$ from Section~\ref{sec:underlying spaces} and its presentation by the left Quillen functor $S_Q \colon \scH \to \sSet$~\eqref{eq:S_Q, presenting S}.

\begin{theorem}
\label{st:forgetting conns is a principal map}
Let $n \in \NN_0$ and $k \in \{0, \ldots, n\}$.
The morphism \smash{$p^k_0 \colon \bbGrb^{n,M}_{\nabla|k} \longrightarrow \bbGrb^{n, M}$ in $\scP(N\Cart)$}, which forgets all connection data, is a principal $\infty$-bundle in the $\infty$-topos $\scP(N\Cart)$.
Its structure $\infty$-group is presented by a simplicial $\ul{\RR}$-module in $\scH$ (compare Lemma~\ref{st:spl smooth R-mods have trivial HoType}).
\end{theorem}

\begin{proof}
The map $p^k_0$ is presented by the Kan fibration \smash{$p^k_0 \colon \Grb^{n,M}_{\nabla|k} \to \Grb^{k, M}$} of projectively fibrant simplicial presheaves.
It is an effective epimorphism in $\scP(N\Cart)$ since, by Corollary~\ref{st:Grb conn comps}, it induces a surjection (even an isomorphism) on connected components over each cartesian space $c \in \Cart$.
The domain of the presenting morphism \smash{$p^k_0 \colon \Grb^{n,M}_{\nabla|k} \to \Grb^{n, M}$} is a torsor for the presheaf of simplicial groups which arises via the Dold-Kan construction from
\begin{equation}
\begin{tikzcd}
	\big( \Omega^{k,v}
	& \Omega^{k-1,v} \ar[l, "\dd"']
	& \cdots \ar[l, "\dd"']
	& \Omega^{1,v} \big) \ar[l, "\dd"']
\end{tikzcd}
\end{equation}
(this is the simplicial presheaf of smooth families of $k$-connections on the trivial $n$-gerbe $\cI$ on $M$; see also Section~\ref{sec:Con_k on fixed n-gerbe}).
The torsor property is also evident from the fact that this chain complex appears in the short exact sequence~\eqref{eq:SES for b_nabla^(l+1)ZZ}.
It follows that these data present a principal $\infty$-bundle in the $\infty$-topos $\scP(N\Cart)$:
it remains to check that the homotopy quotient of \smash{$\Grb^{n,M}_{\nabla|k}$} by the action of the structure group is weakly equivalent to $\Grb^{n,M}$, but this follows by~\cite[Thm.~3.91]{NSS:Pr_ooBdls_II}.
\end{proof}

\begin{remark}
Theorem~\ref{st:forgetting conns is a principal map} makes manifest the fact that connections on a given $n$-gerbe form an affine space over a Baez-Crans-type~\cite{BC:Lie_2-algebras} higher vector space of differential forms on $M$.
This categorifies the fact that in classical gauge theory connections on a given principal bundle form an affine space over the vector space of 1-forms valued in the adjoint bundle.
\qen
\end{remark}

\begin{corollary}
\label{st:Sp^n_0 is equivalence}
Let $n \in \NN_0$ and $k \in \{0, \ldots, n\}$.
The morphism
\begin{equation}
	S(p^k_0) \colon S\big( \bbGrb^{n,M}_{\nabla|k} \big) \longrightarrow S \big( \bbGrb^{n,M} \big)
\end{equation}
in $\scS$ is an equivalence.
\end{corollary}

\begin{proof}
The functor $S \colon \scP(N\Cart) \to \scS$ maps principal $\infty$-bundles for any group object $\bbH$ in $\scP(N\Cart)$ to principal $\infty$-bundles over $S\bbH$ in $\scS$.
Thus, the fibre of the morphism \smash{$S(\bbGrb^{n,M}_{\nabla|k}) \longrightarrow S(\bbGrb^{n,M})$} in $\scS$ is given by $S$ applied to the structure group; however, the latter is equivalent to $\Delta^0 \in \scS$ as a consequence of Lemma~\ref{st:spl smooth R-mods have trivial HoType}.
\end{proof}

\begin{remark}
The space $S\big( \bbGrb^{n,M}_{\nabla|k} \big)$ can be described as space of $n$-gerbes with $k$-connection on $M$ and their \textit{concordances}, where connections are only defined along the fibres of a concordance (i.e.~there is no parallel transport in the directions of the concordance).
\qen
\end{remark}

We compute the homotopy groups of the underlying space of $n$-gerbes with $k$-connection on a manifold $M$:

\begin{theorem}
\label{st:spaces of n-gerbes}
For each $n \in \NN_0$, $k \in \{0, \ldots, n+1\}$ and $i \in \NN_0$, there are canonical isomorphisms
\begin{equation}
	\pi_i \big( S \bbGrb^{n,M}_{\nabla|k} \big)
	\cong \rmH^{n+2-i}(M; \ZZ)\,,
\end{equation}
and $\pi_i = 0$ otherwise.
\end{theorem}

\begin{proof}
By Corollary~\ref{st:Sp^n_0 is equivalence} we can restrict our attention to the case where $k = 0$.
There is an equivalence
\begin{equation}
	\Grb^{n,M} \simeq \big( \rmB^{n+1} \U(1) \big)^{\ul{M}}
\end{equation}
in $\scP(N\Cart)$ (this can be seen using $\scH$ to present the $\infty$-category $\scP(N\Cart)$ and choosing a \v{C}ech resolution of $M$ by a good open covering as a cofibrant replacement of $\ul{M}$).
Moreover, by the Smooth Oka Principle~\cite[Thm.~3.3.53]{SS:Equivar_pr_infty-bundles} (see also~\cite[Thm.~1.1]{BEBdBP:Classifying_spaces_of_oo-sheaves} and~\cite[Thm.~B, Cor.~6.4.8]{Clough:Thesis}) there is a weak equivalence of spaces
\begin{equation}
	S \Big( \rmB^{n+1} \U(1) \big)^{\ul{M}} \Big)
	\simeq \big( S \rmB^{n+1} \U(1) \big)^{S\ul{M}}\,.
\end{equation}
Using that $S$ commutes with colimits (it is itself an $\infty$-categorical colimit functor by Proposition~\ref{st:NDelta_e is cofinal}) and that $S\ul{M}$ is canonically equivalent to the underlying space of the manifold $M$ (see, for instance,~\cite[Thm.~4.15]{Bunk:R-loc_HoThy}) we obtain a composite weak equivalence
\begin{equation}
	S \big( \Grb^{n,M} \big)
	\simeq \scS \big( M, B^{n+1} \U(1) \big)\,.
\end{equation}
Note that this space has a canonical basepoint, which corresponds to the trivial $n$-gerbe $\cI$ on $M$.
A representative of an element in the $i$-th homotopy group of this space is a map of pointed spaces
\begin{equation}
	(\bbS^i, *) \longrightarrow \big( \scS \big( M, B^{n+1} \U(1) \big), \cI \big)\,,
\end{equation}
Using that $\rmB^{n+1}\U(1)$ is an Eilenberg-MacLane space $K(\ZZ, n+2)$, we can equivalently describe this as a map
\begin{equation}
	f \colon \bbS^i {\times} M \longrightarrow K(\ZZ, n+2)
\end{equation}
whose restriction along the inclusion $\iota \colon \{*\} {\times}M \hookrightarrow \bbS^i \times M$ represents the zero class in $\rmH^{n+2}(M; \ZZ)$.
That is, there is a canonical group isomorphism
\begin{equation}
	\pi_i \big( \scS \big( M, B^{n+1} \U(1) \big) \big)
	\simeq \big\{ x \in \rmH^{n+2}(\bbS^i \times M; \ZZ)\, \big| \, \iota^*x = 0 \in \rmH^{n+2}(M; \ZZ) \big\}\,.
\end{equation}
Since $\rmH^k(\bbS^i; \ZZ)$ is finitely generated and free in each degree, the Künneth Theorem in cohomology (see, for instance,~\cite[Thm.~3.16]{Hatcher:AT}) provides a canonical isomorphism
\begin{equation}
	\rmH^{n+2}(\bbS^i {\times} M; \ZZ)
	\cong \bigoplus_{r+s = n+2} \rmH^r(\bbS^i; \ZZ) \otimes \rmH^s(M; \ZZ)
	\cong \rmH^{n+2-i}(M; \ZZ) \oplus \rmH^{n+2}(M; \ZZ)\,.
\end{equation}
Under this isomorphism, the restriction along $\iota \colon \{*\} \times M \hookrightarrow \bbS^i \times M$ corresponds to the projection onto $\rmH^{n+2}(M;\ZZ)$.
Thus, we arrive at a canonical group isomorphism
\begin{equation}
	\pi_i \big( S \big( \bbGrb^{n,M} \big) \big)
	\cong \rmH^{n+2-i}(M; \ZZ)\,,
\end{equation}
for each $i \in \NN_0$, as claimed.
\end{proof}

%%%%%%%%%%%%%%%%%%%%%%%%%%%%%%%%%%%%%%%%%%%%%%%%%%%%%%%%%%%%%%%%%%%%%%%%%%%%

\section{Examples of moduli $\infty$-(pre)stacks involving higher $\U(1)$-connections}
\label{sec:examples with higher U(1)-connections}

%%%%%%%%%%%%%%%%%%%%%%%%%%%%%%%%%%%%%%%%%%%%%%%%%%%%%%%%%%%%%%%%%%%%%%%%%%%%

In this section we present and analyse various moduli $\infty$-stacks arising in field theories which involve higher $\U(1)$-connections.
We start with the bare moduli $\infty$-stack of connections on a given $n$-gerbe with $k$-connection on $M$, before moving on to, for instance, higher Maxwell and Einstein-Maxwell moduli $\infty$-stacks.

%%%%%%%%%%%%%%%%%%%%%%%%%%%%%%%%%%%%%%%%%%%%%%%%%%%%%%%%%%%%%%%%%%%%%%%%%%%%

\subsection{Moduli of $k$-connections on $n$-gerbes}

%%%%%%%%%%%%%%%%%%%%%%%%%%%%%%%%%%%%%%%%%%%%%%%%%%%%%%%%%%%%%%%%%%%%%%%%%%%%

In the last subsection we considered stacks of smooth families of higher $\U(1)$-connections on a fixed $n$-gerbe.
Here, we include the action of symmetries and consider \textit{moduli} stacks of higher $\U(1)$-connections on a fixed $n$-gerbe.
To that end, recall the notation for moduli $\infty$-prestacks of higher geometric structures and their presentation in terms of left fibrations from Section~\ref{sec:moduli oo-prestacks of hgeo strs on M}.
We use the following general set-up:
consider the choice
\begin{equation}
	\Sol^\scD = \Conf^\scD = \Grb^{n,\scD}_{\nabla|k}\,,
	\qquad
	\Fix^\scD = \Grb^{n,\scD}_{\nabla|l}\,,
\end{equation}
with $l \leq k \in \{0, \ldots, n+1\}$ and the canonical projection \smash{$p^k_l \colon \Grb^{n,\scD}_{\nabla|k} \longrightarrow \Grb^{n,\scD}_{\nabla|l}$} as the morphism $\Conf^\scD \to \Fix^\scD$.
Let $\cG \in \Grb^{n,M}(\RR^0)$ be an $n$-gerbe on $M$.
Further, consider a left fibration $Q \to N\Cart^\opp$ with reduced fibres and a morphism $\phi \in \sSet_{/N \Cart^\opp}(Q, N\rmB\scD[\cG]^\opp)$ (recall from Definition~\ref{def:BSYM_(Q,phi)(G)} and the discussion following that definition that these data present an action $\Phi \colon \bbGamma \to \Diff_{[\cG]}(M)$ of a smooth $\infty$-group on $M$ by diffeomorphisms which preserve the isomorphism class $[\cG]$ of $\cG$ in $\Grb^{n,M}(\RR^0)$).
Most of the examples we will consider will be modifications of the following moduli stacks:

\begin{definition}
\label{def:Mdl of k-conns on n-gerbes with l-conn}
In the above set-up, consider the left fibration over $N\Cart^\opp$ obtained as the pullback
\begin{equation}
\begin{tikzcd}[column sep=1.25cm, row sep=0.75cm]
	\textint \Mdl_{k,\phi}(\cG, \cA^{(l)}) \ar[r] \ar[d]
	& \textint \big( \Grb^{n,\scD[\cG]}_{\nabla|k} \big)^{[\cG]} \ar[d]
	\\
	\rmB\SYM_\phi^\rev(\cG, \cA^{(k)}) \ar[r]
	& \textint \big( \Grb^{n,\scD[\cG]}_{\nabla|l} \big)^{[\cG]}
\end{tikzcd}
\end{equation}
The \textit{$\infty$-prestack $\scMdl_{k,\Phi}(\cG, \cA^{(l)})$ of $k$-connections on the $n$-gerbe with $l$-connection $(\cG, \cA^{(l)})$} (for $0 \leq l \leq k \leq n+1$) modulo the action of symmetries of $(\cG, \cA^{(l)})$ which lift the action $\phi$ is the $\infty$-prestack on $N\Cart$ presented by the left fibration $\textint \Mdl_{k,\phi}(\cG, \cA^{(l)}) \longrightarrow N\Cart^\opp$.
\end{definition}

\begin{remark}
By Corollary~\ref{st:Descent result for moduli oo-prestacks} the moduli $\infty$-prestack $\scMdl_{k, \Phi}(\cG, \cA^{(l)})$ is an $\infty$-stack with respect to good open coverings of cartesian spaces whenever $\phi$ encodes the action of a sheaf of groups on $M$.
\qen
\end{remark}

\begin{remark}
For $l = 0$ and $k = n+1$, we obtain the moduli $\infty$-prestack  $\scMdl_{n+1,\Phi}(\cG)$ of connections on a fixed $n$-gerbe $\cG$ on $M$ modulo the higher smooth group of symmetries of $\cG$ which lift $\phi$.
At the same time, $\scMdl_{n+1,\Phi}(\cG)$ can also be viewed as the moduli $\infty$-prestack of differential refinements of the integer cocycle presented by $\cG$.
\qen
\end{remark}

\begin{remark}
For the choice $Q = \rmB\scD[\cG]^\opp$ and $\phi$ the identity on $\rmB\scD[\cG]^\opp$, the $\infty$-stack $\scMdl_{k,\id}(\cG)$ is the moduli $\infty$-stack of $k$-connections on $\cG$ modulo the action of the higher group of symmetries of $\cG$.
If we choose $Q = N\Cart^\opp$ and $\phi = Ne_M$ (corresponding to the action of the trivial group on $M$), we obtain the moduli $\infty$-stack $\scMdl_{k, Ne_M}(\cG)$ of $k$-connections on $M$ modulo the action of the higher group of automorphisms of $\cG$, i.e.~modulo the action of the higher gauge group of $\cG$.
\qen
\end{remark}

By a slight modification of this example, we obtain moduli stacks $\scMdl_{\flat, \Phi}(\cG)$ of flat connections on a fixed $n$-gerbe $\cG$ on $M$:
they arise from the choices
\begin{equation}
	\Conf^\scD = \Grb^{n,\scD}_\nabla\,,
	\qquad
	\Fix^\scD = \Grb^{n,\scD}_{\nabla|l}\,,
	\qquad
	\Sol^\scD = \Grb^{n,\scD}_\flat\,,
\end{equation}
where $\Grb^{n,\scD}_\flat \subset \Grb^{n,\scD}_\nabla$ is the full simplicial subpresheaf on those $n$-gerbes with connection whose (fibrewise) curvature $(n{+}2)$-forms vanish identically (note that this indeed defines a solution presheaf in the sense of Definition~\ref{def:solution presheaf}).
Let $(\cG, \cA^{(l)})$ be an $n$-gerbe with $l$-connection on $M$.

\begin{definition}
In the above set-up, consider the left fibration over $N\Cart^\opp$ obtained as the pullback
\begin{equation}
\begin{tikzcd}[column sep=1.25cm, row sep=0.75cm]
	\textint \Mdl_{\flat,\phi}(\cG, \cA^{(l)}) \ar[r] \ar[d]
	& \textint \big( \Grb^{n,\scD[\cG]}_\flat \big)^{[\cG]} \ar[d]
	\\
	\rmB\SYM_\phi^\rev(\cG, \cA^{(l)}) \ar[r]
	& \textint \big( \Grb^{n,\scD[\cG]}_{\nabla|l} \big)^{[\cG]}
\end{tikzcd}
\end{equation}
The \textit{$\infty$-prestack $\scMdl_{\flat,\Phi}(\cG, \cA^{(l)})$ of flat connections on $(\cG, \cA^{(l)})$} modulo the action of symmetries of $(\cG, \cA^{(l)})$ which lift the action $\phi$ is the $\infty$-prestack on $N\Cart$ presented by the left fibration $\textint \Mdl_{\flat,\phi}(\cG, \cA^{(l)}) \longrightarrow N\Cart^\opp$.
\end{definition}

Theorem~\ref{st:equiv result for Mdl oo-stacks} shows that we can analyse moduli $\infty$-prestacks of $(n{+}1)$-connections on any given $n$-gerbe in $n$ equivalent ways:
we may keep any part of the connection data fixed (as long as we do not fix the entire connection).
This changes which part of the connection data we are still allowed to vary (including its morphisms and higher morphisms), but these changes are compensated for by a simultaneous change of the smooth higher symmetries of the fixed data.
The precise statement reads as:

\begin{theorem}
\label{st:equiv thm for conns on n-gerbes}
Let $(\cG, \cA^{(k)})$ be an $n$-gerbe with $k$-connection on $M$, and let $(\cG, \cA^{(l)})$ denote its underlying $n$-gerbe with $l$-connection, for $0 \leq l \leq k \leq n \in \NN_0$.
There are canonical equivalences
\begin{equation}
	\scMdl_{n+1, \Phi}(\cG, \cA^{(k)})
	\simeq \scMdl_{n+1, \Phi}(\cG, \cA^{(l)})
	\qquad \text{and} \qquad
	\scMdl_{\flat, \Phi}(\cG, \cA^{(k)})
	\simeq \scMdl_{\flat, \Phi}(\cG, \cA^{(l)})\,.
\end{equation}
\end{theorem}

\begin{proof}
We only need to check that Theorem~\ref{st:equiv result for Mdl oo-stacks} applies, i.e.~that the projection morphisms
\begin{equation}
	\textint p^k_l \colon \textint \Grb^{n,\scD}_{\nabla|k} \longrightarrow \textint \Grb^{n,\scD}_{\nabla|l}
\end{equation}
are objectwise $\pi_0$-isomorphisms.
We already proved this in Proposition~\ref{st:GRB conn comps}.
\end{proof}

\begin{example}
Let $(\cI, 0)$ be the trivial $n$-gerbe, endowed with the zero $n$-connection, i.e.~$A_i = 0$ for all $i = 1, \ldots, n$.
There is an equivalence
\begin{equation}
	\scMdl_{\flat, Ne_M}(\cI, 0)
	\simeq \Omega^{n+1,v}_\cl \dslash \bbGrb^{n-1,M}_\nabla
	\colon N\Cart^\opp \longrightarrow \scS\,,
\end{equation}
where the action is via the curvature, $(B, (\cG', \cA')) \longmapsto B + \curv(\cG', \cA')$.
For any other $n$-gerbe with $n$-connection $(\cG, \cA)$, there is still an equivalence
\begin{equation}
	\scMdl_{\flat, Ne_M}(\cG, \cA)
	\simeq \Omega^{n+1,v}_\cl \dslash \bbGrb^{n-1,M}_\nabla
	\colon N\Cart^\opp \longrightarrow \scS\,,
\end{equation}
but this is no longer canonical (rather, $\scMdl_{\flat, Ne_M}(\cG, \cA)$ is a torsor over $\scMdl_{\flat, Ne_M}(\cI, 0)$ with the group structure given by the sum of $(n{+}1)$-forms).
In particular, taking $\pi_0$ in $\scFun(N\Cart^\opp, \scS)$, we obtain the presheaf of connected components
\begin{equation}
\label{eq:flat conns mod gauge}
	\pi_0 \big( \scMdl_{\flat, Ne_M}(\cG, \cA) \big)
	\simeq \frac{\Omega^{n+1,v}_\cl}{\Omega^{n+1,v}_{\cl,\ZZ}}
	\simeq \frac{\rmH^{n+1}(M;\RR)}{\rmH^{n+1}(M;\ZZ)}\,,
\end{equation}
where these groups become presheaves on $\Cart$ by interpreting the respective coefficient groups as group-valued presheaves on $\Cart$.
For $n = 0$, the smooth torus~\eqref{eq:flat conns mod gauge} is usually viewed as the moduli space of flat connection on a given line bundle.
In~\cite{MR:YM_for_BGrbs} this torus also appeared for $n = 1$ in the context of connections on gerbes.

However, forming $\pi_0$ in $\scFun(N\Cart^\opp, \scS)$ quotients away all higher structure and produces a $\Set$-valued presheaf on $\Cart$.
For instance, this process interferes with the ability to glue families of solutions (one can sheafify~\eqref{eq:flat conns mod gauge} to force gluing, but this does no longer describe families of flat connections; in particular, one cannot recover any gauge actions).
Moreover, the constructions in Remark~\ref{rmk:higher form syms} and Remark~\ref{rmk:hol and Mdl(G,A^1)} below fail if we work only with the truncation~\eqref{eq:flat conns mod gauge}.

The topology of the moduli space is, instead, encoded in its entirety in the underlying space $S(\scMdl_{\flat, Ne_M}(\cG, \cA^{(n)})) \in \scS$ (where $S$ is the functor $S \colon \scFun(N\Cart^\opp, \scS) \to \scS$ from Section~\ref{sec:higher symmetry groups of hgeo strs}).
Recall from Proposition~\ref{st:NDelta_e is cofinal} that there is a canonical equivalence
\begin{equation}
	S \simeq \underset{N\Cart^\opp}{\colim}
\end{equation}
between $S$ and the $\infty$-categorical colimit functor for $N\Cart^\opp$-shaped diagrams.
Note that there is no equivalence between $\pi_0 \circ S$ and $S \circ \pi_0$ (one produces a set and one a space).
Appealing to Corollary~\ref{st:Sp^n_0 is equivalence}, Proposition~\ref{st:SMdl for SSol = *} and Theorem~\ref{st:spaces of n-gerbes}, we have that
\begin{equation}
	S \big( \scMdl_{\flat, Ne_M}(\cG, \cA^{(n)}) \big)
	\simeq \rmB S(\bbGrb^{n-1,M})\,,
\end{equation}
and in particular
\begin{equation}
	\pi_i S \big( \scMdl_{\flat, Ne_M}(\cG, \cA^{(n)}) \big)
	\cong
	\begin{cases}
		\rmH^{n+2-i}(M;\ZZ)\,, & i = 1, \ldots, n+2\,,
		\\
		0 & \text{otherwise}.
	\end{cases}
\end{equation}
Moreover, the homotopy type of $\scMdl_{\flat, Ne_M}(\cG, \cA^{(n)})$ is not specific to using $n$-connections:
if $k < n$ and $\cA^{(k)}$ denotes the $k$-connection underlying $\cA^{(n)}$, then Theorem~\ref{st:equiv result for Mdl oo-stacks} implies that
\begin{equation}
	\scMdl_{\flat, Ne_M}(\cG, \cA^{(n)})
	\simeq \scMdl_{\flat, Ne_M}(\cG, \cA^{(k)})
	\simeq \scMdl_{\flat, Ne_M}(\cG)\,.
\end{equation}
This provides $n$ different ways of studying the moduli stack of flat connections on a given $n$-gerbe with $k$-connection.
\qen
\end{example}

\begin{remark}
\label{rmk:hol and Mdl(G,A^1)}
We can use smooth maps with target $M$ to test the higher structure of moduli stacks:
for a fixed $n$-gerbe with $n$-connection $(\cG, \cA^{(n)})$ on $M$, the holonomy morphisms of group objects in Remark~\ref{rmk:higher form syms} induce morphisms
\begin{equation}
	U_{(-)}(\Sigma^{(n+1-p)}, \sigma) \colon \pi_p \scMdl_{n+1, Ne_M}(\cG, \cA^{(n)})
	\longrightarrow U(1)
\end{equation}
(where the homotopy groups are formed in $\scP(N\Cart)$).

Using higher-dimensional holonomies instead extracts additional information:
for another example, let $LM = \Mfd(\bbS^1, M)$ denote the free loop space of $M$.
For $P \to M$ an ordinary $\U(1)$-bundle, the holonomy induces a map $\Con(P) \to \U(1)^{LM}$, where the internal hom is taken in $\scH$.
By the gauge invariance of holonomies, this further descends to a smooth map
\begin{equation}
	\hol \colon \Con(P)/\Aut(P) \longrightarrow \U(1)^{LM}\,.
\end{equation}
This map has analogues in higher $\U(1)$-gauge theory, but these have much richer structure:
let $(\cG, \cA^{(1)})$ be an $1$-gerbe with $1$-connection on $M$.
The holonomy for a connection on a 1-gerbe is equivalent to its transgression line bundle~\cite[Sec.~4.4]{BMS:2-Grp_Ext}.
Combining our results with~\cite{Waldorf:Transgression_II} we obtain a morphism of smooth stacks
\begin{equation}
\label{eq:hol mp from scMdl}
	\hol \colon \tau_{\leq 1} \scMdl_{2, Ne_M}(\cG, \cA^{(1)}) \longrightarrow \HLB_{\nabla, fus}(LM)
\end{equation}
from the 1-truncation (in $\scP(N\Cart)$) of the moduli stack $\scMdl_{2, Ne_M}(\cG, \cA^{(1)})$ to the smooth groupoid of hermitean fusion line bundles with superficial connection on $LM$ (see~\cite{Waldorf:Transgression_II} for details).
In particular, the isomorphism class of the underlying line bundle on $LM$ is fixed by $(\cG, \cA^{(1)})$, but the connection data on the right-hand side detects not just $\pi_0 \scMdl_{2, Ne_M}(\cG, \cA^{(1)})$, but it sends isomorphism classes of gauge transformations in $\scAut(\cG, \cA^{(1)})$ to gauge transformations of $\U(1)$-connections on the right-hand side (because of the functoriality of the transgression construction).
This holonomy map restricts to moduli of Maxwell solutions, for instance, and hence can also be used to test these moduli stacks.
More generally, it should be possible to gain even more information by replacing the target of~\eqref{eq:hol mp from scMdl} by a smooth groupoid of functorial field theories on $M$---the results in~\cite{BW:OCFFTs, BW:Transgression_of_D-branes} provide clear evidence for this---and find analogous maps for fixed $n$-gerbes with $n$-connections.
Finally, by Theorem~\ref{st:equiv thm for conns on n-gerbes}, the holonomy map~\eqref{eq:hol mp from scMdl} equivalently extracts equivalent information about $\scMdl_{2,Ne_M}(\cG)$, i.e.~the higher moduli stack of 2-connections on a fixed 1-gerbe $\scG$.
\qen
\end{remark}

%%%%%%%%%%%%%%%%%%%%%%%%%%%%%%%%%%%%%%%%%%%%%%%%%%%%%%%%%%%%%%%%%%%%%%%%%%%%

\subsection{Higher gauge-theoretic examples}

%%%%%%%%%%%%%%%%%%%%%%%%%%%%%%%%%%%%%%%%%%%%%%%%%%%%%%%%%%%%%%%%%%%%%%%%%%%%

We now present various moduli $\infty$-stacks which arise in field theories that involve higher $\U(1)$-connections.
Suppose $\dim(M) = d$.
For $r,s \in \NN_0$ with $r+s = d$, let $\Met^\scD_{r,s} \colon \rmB\scD^\opp \to \sSet$ be the (simplicially constant) simplicial presheaf which assigns to $c \in \rmB\scD$ the smooth families of pseudo-Riemannian metrics of signature $(r,s)$ on $M$; equivalently, these are fibrewise metrics on $c \times M$.
The action of $\Met_{r,s}$ on morphisms in $\rmB\scD$ is via the pullback of metrics along diffeomorphisms.

\begin{notation}
\label{nt:Riemannian geometry}
We introduce the following notation (for more details we refer the reader to~\cite{BS:EYM_Slices}):
let $(M,g)$ be a pseudo-Riemannian manifold.
\begin{itemize}
\item $\nabla^g$ is the Levi-Civita connection associated to $g$,

\item $\Ric^g$ is its Ricci curvature,

\item $*_g \colon \Omega^p(M) \to \Omega^{d-p}(M)$ is the Hodge star operator associated to $g$, where $d = \dim(M)$ and $p \in \{0, \ldots, d\}$,

\item $\<-,-\>_g \colon \Omega^p(M) \times \Omega^p(M) \to C^\infty(M)$ is the pairing defined by
\begin{equation}
	\alpha \wedge *_g \beta = \<\alpha, \beta\>_g \vol_g\,,
\end{equation}
where $\vol_g$ is the Riemannian volume form of $g$,

\item $\dd^*_g \colon \Omega^{p+1}(M) \to \Omega^p(M)$ is the de Rham codifferential associated to $g$ (i.e.~$\dd^*_g = (-1)^{pd} *_g \dd\, *_g$),

\item $\Delta_g \coloneqq (\dd + \dd^*_g)^2 \colon \Omega^*(M) \to \Omega^*(M)$ is the Laplacian on $(M,g)$,

\item we define a bilinear map
\begin{align}
	(-) \bullet_g (-) &\colon \Omega^p(M) \times \Omega^p(M) \longrightarrow \Gamma(T^\vee M \odot T^\vee M)\,,
	\\
	(\alpha \bullet_g \beta)(X_1, X_2) &\coloneqq \frac{1}{2} \big( \< \iota_{X_1} \alpha, \iota_{X_2} \beta \>_g + \< \iota_{X_2} \alpha, \iota_{X_1} \beta \>_g \big)
\end{align}

\item $\sharp_g \colon \Omega^*(M) \to \Gamma(\bigwedge^*TM)$ and $\flat_g \colon \Gamma(\bigwedge^*TM) \to \Omega^*(M)$ are the musical isomorphisms associated to $g$,

\item for $p \geq q$ we set
\begin{equation}
	\iota^g_{(-)} (-) \colon \Omega^q(M) \times \Omega^p(M) \longrightarrow \Omega^{p-q}(M)\,,
	\qquad
	\iota^g_\eta \omega \coloneqq \iota_{\sharp_g \eta} \omega\,,
\end{equation}

\item we define the map
\begin{equation}
	|{-}|_g^2 \colon \Omega^p(M) \to C^\infty(M)\,,
	\qquad
	|\alpha|_g^2 \coloneqq \<\alpha, \alpha\>_g\,.
\end{equation}
\end{itemize}
Finally, note that each of the above objects have straightforward adaptations to fibrewise objects on a product $c {\times} M$ whose fibres over $c$ are endowed a smooth family of (pseudo-)Riemannian metrics.
\qen
\end{notation}

\begin{example}
\label{eg:Mdl stacks from Grb x Met}
Consider configurations of $n$-gerbes with connection and metrics of signature $(r,s)$ on $M$; these are encoded in the choice
\begin{equation}
	\Conf^\scD = \Grb^{n,\scD}_\nabla \times \Met_{r,s}^\scD\,.
\end{equation}
For a cartesian space $c$, a simplex in $\Grb^{n,\scD}_\nabla(c) \times \Met_{r,s}^\scD(c)$ is a triple $(\cG, \cA, g)$, where $(\cG, \cA)$ is a smooth $c$-paramterised family of $n$-gerbes with connection on $M$ and $g$ is a smooth family of non-degenerate metrics of signature $(r,s)$ on the tangent bundle $TM$.
This configuration presheaf comes with several natural projective fibrations, and thus several natural choices of $\Fix^\scD$; these include the canonical projection morphisms
\begin{align}
\label{eq:Conf-->Sol for EM thy}
	\Conf^\scD = \Grb^{n,\scD}_\nabla \times \Met_{r,s}^\scD
	&\longrightarrow \Grb^{n,\scD}_{\nabla|k} \times \Met_{r,s}^\scD \eqqcolon \Fix^\scD
	\qquad \text{and}
	\\
	\Conf^\scD = \Grb^{n,\scD}_\nabla \times \Met_{r,s}^\scD
	&\longrightarrow \Grb^{n,\scD}_{\nabla|k} \eqqcolon \Fix^{\prime \scD}\,,
\end{align}
for any $0 \leq k \leq n$.

Three important choices of solution presheaves are reads as follows:
\begin{enumerate}
\item The simplicial presheaf of \textit{degree-$n$ vacuum Maxwell solutions} on $M$ is the full simplicial sub-presheaf
\begin{equation}
	\Sol_{Mw}^\scD \subset \Grb^{n,\scD}_\nabla \times \Met_{r,s}
\end{equation}
on those $n$-gerbes with connections and metrics on $M$ which satisfy the \textit{higher (vacuum) Maxwell equations}
\begin{equation}
\label{eq:Maxwell eqns}
	\dd^*_g \curv(\cG, \cA) = 0\,.
\end{equation}
This indeed defines a solution sub-presheaf in the sense of Definition~\ref{def:solution presheaf} because equivalent $n$-gerbes with connection on $M$ have the same curvature $(n{+}2)$-forms and all terms in the higher Maxwell equations~\eqref{eq:Maxwell eqns} are compatible with diffeomorphisms of $M$.
With the above choice $\Fix^\scD$, our formalism produces the moduli stack of connections on a fixed $n$-gerbe with $k$-connection $(\cG, \cA^{(k)})$ which satisfy the higher Maxwell equations with respect to a fixed metric $g$ on $M$.
These solutions are divided by the action of those symmetries of the fixed $n$-gerbe with $k$-connection $(\cG, \cA^{(k)})$ which lift an isometry of $(M,g)$.
Working instead with $\Fix^{\prime \scD}$, we obtain the moduli stack of pairs $(\cA, g)$ of a connection $\cA$ on a fixed $n$-gerbe $\cG$ on $M$ and a signature $(r,s)$-metric $g$ which together satisfy the equation~\eqref{eq:Maxwell eqns}.
Here, solutions are divided by the action of symmetries of the $n$-gerbe, which here may alter the metric $g$.

\item The simplicial presheaf of \textit{degree-$n$ Einstein-Maxwell solutions} on $M$ is the full simplicial sub-presheaf
\begin{equation}
	\Sol^{n,\scD}_{\mathrm{EM}} \subset \Grb^{n,\scD}_\nabla \times \Met_{r,s}^\scD
\end{equation}
on those simplices $(\cG, \cA, g)$ which solve the (fibrewise) Einstein-Maxwell equations
\begin{align}
\label{eq:Einstein-Maxwell eqns}
	\Ric^g - \frac{1}{2} s^g &= \frac{\kappa}{2} \big| \curv(\cG, \cA) \big|^2_g - \kappa \big( \curv(\cG, \cA) \bullet_g \curv(\cG, \cA) \big)\,,
	\\
	\dd^*_g \curv(\cG, \cA) &= 0\,,
\end{align}
where $\kappa \in \{-1, 1\}$ is a constant.
This defines a solution sub-presheaf in the sense of Definition~\ref{def:solution presheaf} for the same reasons as in the point~(1) above.
For the choice $\Fix^\scD$ in~\eqref{eq:Conf-->Sol for EM thy}, our formalism produces the moduli stacks of connections on a fixed $n$-gerbe with $k$-connection $(\cG, \cA^{(k)})$ which satisfy the Einstein-Maxwell equations for a fixed metric $g$ on $M$.
The symmetries in this case are generated by equivalences of $n$-gerbes with $k$-connection on $M$ and the isometries of the pseudo-Riemannian manifold $(M,g)$.
The choice $\Fix^{\prime \scD}$, in contrast, produces moduli stacks of pairs $(\cA, g)$ of  a metric $g$ on $M$ and a connection on a fixed $n$-gerbe with $k$-connection $(\cG, \cA^{(k)})$ on $M$, modulo symmetries of $(\cG, \cA^{(k)})$.

\item Suppose that $d = 2n + 4$.
The simplicial presheaf of \textit{metrics and degree-$n$ (anti-)self-dual $\U(1)$-instantons} on $M$ is the full simplicial sub-presheaf
\begin{equation}
	\Sol^{n,\scD}_{\mathrm{inst}} \subset \Grb^{n,\scD}_\nabla \times \Met_{r,s}^\scD
\end{equation}
on those simplices $(\cG, \cA, g)$ which solve the equations
\begin{equation}
	*_g \curv(\cG, \cA) = \pm \curv(\cG, \cA)
\end{equation}
i.e.~the (anti-)self-duality equations with respect to that metric.
Depending on our choice of $\Fix^\scD$ or $\Fix^{\prime \scD}$, we can consider moduli of such configurations on a fixed $n$-gerbe with $k$-connection and for a fixed metric (dividing out by symmetries lifting only isometries), or on a fixed $n$-gerbe with $k$-connection, allowing the metric to vary and dividing out by symmetries of the $n$-gerbe with $k$-connection.
\qen
\end{enumerate}
\end{example}

\begin{example}
\label{eg:Mdl stacks for SD, BF, CS}
We consider the following further examples:
\begin{enumerate}
\item Consider the configuration presheaf
\begin{equation}
	\Conf^\scD = \Grb^{n,\scD} \times \Grb^{d-n-4,\scD} \times \Met_{r,s}^\scD
\end{equation}
and the solution presheaf
\begin{equation}
	\Sol_{\mathrm{SD}}^\scD \subset \Grb^{n,\scD} \times \Grb^{d-n-4,\scD} \times \Met_{r,s}^\scD
\end{equation}
given as the full simplicial subpresheaf on those simplices $(\cG, \cA, \cG', \cA', g)$ such that the field strengths of $(\cG, \cA)$ and $(\cG', \cA')$ are in Poincaré, or electric/magnetic, duality, i.e.
\begin{equation}
	*_{g,v}\, \curv(\cG, \cA) = \curv(\cG', \cA')
\end{equation}
These solutions are the higher geometric objects underlying the considerations in differential cohomology in~\cite{BBSS:CS_diff_chars_and_PD}.

\item Another example is higher $\U(1)$-BF theory (see, for instance,~\cite{HMT:Gen_Ab_TV_and_U(1)-BF}):
let $M$ be a closed, oriented $d$-manifold.
Here we set
\begin{equation}
	\Conf^\scD_{BF;p,q} = \Grb^{p,\scD}_\nabla \times \Grb^{q,\scD}_\nabla
	\qqandqq
	\Fix^\scD_{BF;p,q} = \Grb^{p,\scD}_{\nabla|k} \times \Grb^{q,\scD}_{\nabla|l}\,,
\end{equation}
where $p + q = \dim(M)-1$, $k \in \{0, \ldots, p\}$ and $l \in \{0, \ldots, q\}$.
The solution subpresheaf $\Sol^\scD_{BF;p,q} \subset \Conf^\scD_{BF;p,q}$ is specified by extremising the action functional
\begin{align}
	S_{BF;p,q} \colon \Grb^{p,\scD}_\nabla \times \Grb^{q,\scD}_\nabla
	&\longrightarrow \pi_M^* \U(1)\,,
	\\
	\big( (\cG, \cA), (\cG', \cA') \big)
	&\longmapsto \hol \big( M; (\cG, \cA) \cup (\cG', \cA') \big)\,,
\end{align}
where
\begin{equation}
	\hol(M;-) \colon \pi_0 \Grb^{d+1,\scD}_\nabla \longrightarrow \U(1)
\end{equation}
denotes the higher-dimensional holonomy of a $(d{+}1)$-gerbe with connection around the closed oriented $d$-manifold $M$, and
\begin{equation}
	\cup \colon \Grb^{p,\scD}_\nabla \times \Grb^{q,\scD}_\nabla
	\longrightarrow \Grb^{p+q+2,\scD}_\nabla
\end{equation}
is induced by the Beilinson-Deligne cup product on \v{C}ech-Deligne cocycles (see, for instance,~\cite[Prop.~1.5.8]{Brylinski:LSp_and_Geo_Quan}).

\item Let $2p = \dim(M)-1$ and $k \in \{0, \ldots, p\}$.
Modifying the previous example to have
\begin{equation}
	\Conf^\scD_{BF;p,q} = \Grb^{p,\scD}_\nabla
	\qqandqq
	\Fix^\scD_{BF;p,q} = \Grb^{p,\scD}_{\nabla|k}\,,
\end{equation}
with action functional
\begin{align}
	S_{CS} \colon \Grb^{p,\scD}_\nabla
	\longrightarrow \pi_M^* \U(1)\,,
	\qquad
	(\cG, \cA)
	\longmapsto \hol \big( M; (\cG, \cA) \cup (\cG, \cA) \big)\,,
\end{align}
we obtain higher $\U(1)$-Chern-Simons theory on a closed oriented $d$-manifold $M$.
\qen
\end{enumerate}
\end{example}

We investigate the situation in Example~\ref{eg:Mdl stacks from Grb x Met}(1) more closely:
in particular, consider the choices
\begin{equation}
	\Conf^\scD = \Grb^{n,\scD}_\nabla \times \Met_{0,d}^\scD
	\longrightarrow \Grb^{n,\scD}_{\nabla|k} \times \Met_{0,d}^\scD \eqqcolon \Fix^\scD
\end{equation}
and the solution subpresheaf \smash{$\Sol_{Mw}^\scD \subset \Conf^\scD$}.
That is, we study moduli stacks of Maxwell connections on a fixed $n$-gerbe with $k$-connection $(\cG, \cA^{(k)})$ and with respect to a fixed Riemannian metric $g$ on $M$, modulo the action of symmetries of $(\cG, \cA^{(k)})$ which lift isometries of the Riemannian manifold $(M,g)$.

\begin{remark}
Solutions to the Maxwell equations on a $1$-gerbe have been investigated in~\cite{MR:YM_for_BGrbs}.
To the best of our knowledge, this is the only reference to date which considers moduli of higher-degree connections.
However, it does not take into account the higher structure that permeates their theory.
\qen
\end{remark}

First, we have the following generalisation of~\cite[Sec.~3.1]{MR:YM_for_BGrbs}:

\begin{theorem}
\label{st:MW existence (vacuum)}
Let $M$ be a compact manifold.
Suppose $M$ is endowed with a Riemannian metric $g$ and an $n$-gerbe with $k$-connection $(\cG, \cA^{(k)})$.
There exists a solution to the higher vacuum Maxwell equation~\eqref{eq:Maxwell eqns} on $(\cG, \cA^{(k)})$; that is, there exists an extension $\cA$ of $\cA^{(k)}$ to a full connection on $\cG$ such that
\begin{equation}
	\dd^*_g\, \curv(\cG, \cA) = 0\,.
\end{equation}
\end{theorem}

\begin{proof}
By the Hodge Decomposition Theorem (see, for instance~\cite{Ebert:Lectures_on_AS_Thm}), an $(n{+}2)$-form $H$ on $M$ satisfies $\dd H = 0$ and $\dd^*_g H = 0$ if and only if it is harmonic, i.e.~$\Delta_g H = 0$.
As a further consequence of the Hodge Decomposition Theorem, there exists a harmonic $(n{+}2)$-form on $M$ whose de Rham class coincides with the image of $[\cG] \in \rmH^{n+2}(M;\ZZ)$ under the canonical morphism $\rmH^{n+2}(M;\ZZ) \to \rmH^{n+2}(M;\RR)$.

Let $\cU = \{U_a\}_{a \in \Lambda}$ be a good open covering of $M$, and suppose that $(\cG, \cA^{(k)})$ is represented by cocycle data $(g, A_1, \ldots, A_k)$ with respect to this covering (where $k \leq n$).
We claim that there exists a connection $\cA$ on $(\cG, \cA^{(k)})$ with $\curv(\cG, \cA) = H$.
By the above arguments, providing such a connection will complete the proof.
To that end, let $\cA'$ be an arbitrary connection on $(\cG, \cA^{(k)})$ (note that $\cA'$ is of the form $(A_1, \ldots, A_k, A'_{k+1}, \ldots, A'_{n+1})$, with $\cA^{\prime (k)} = (A_1, \ldots, A_k) = \cA^{(k)})$ fixed by the input data $(\cG, \cA^{(k)})$).
Then, both $\curv(\cG, \cA')$ and $H$ represent the class of $\cG$ in de Rham cohomology; therefore, there exists an $\eta \in \Omega^{n+1}(M)$ with $\dd \eta = H - \curv(\cG, \cA')$.
We define $A_i \coloneqq A'_i$, for $i = k+1, \ldots, n$, and $A_{n+1} = A'_{n+1} + \delta \eta$, where $\delta$ is the \v{C}ech differential.
Then, $\delta A_{n+1} = \delta A'_{n+1}$, so that $\cA$ is indeed a connection on $(\cG, \cA^{(k)})$.
Furthermore, we have that
\begin{equation}
	\delta \curv(\cG, \cA)
	= \dd A_{n+1}
	= \dd A'_{n+1} + \dd \delta \eta
	= \delta (\curv (\cG, \cA') + \dd \eta)
	= \delta H\,.
\end{equation}
Since $\delta \colon \Omega^*(M) \to \Omega^*(\v{C}_0 \cU)$ is injective, we derive that $\curv(\cG, \cA) = H$.
\end{proof}

Let $\scMdl_{Mw}((\cG, \cA^{(k)}), g)$ denote the moduli $\infty$-prestack of higher Maxwell solutions on a fixed $n$-gerbe with $k$-connection on $M$, and with respect to a fixed pseudo-Riemannian metric $g$ of any signature $(r,s)$.

\begin{theorem}
\label{st:HoType of Sol_Mw}
The underlying space of the moduli $\infty$-stack \smash{$\scMdl_{Mw, Ne_M}((\cG, \cA^{(k)}), g)$} is equivalent to that of $\rmB \bbGrb^{n-1,M}$.
In particular,
\begin{equation}
	\pi_i S \big( \scMdl_{Mw, Ne_M}((\cG, \cA^{(k)}), g) \big)
	\simeq \rmH^{n+2-i}(M; \ZZ)\,,
	\qquad
	\forall\, i = 1, \ldots, n+1\,,
\end{equation}
and all other homotopy groups are trivial.
\end{theorem}

\begin{proof}
The solution presheaf $\Sol_{Mw}^M$ is affine over the simplicial $\ul{\RR}$-module presheaf which is obtained via the Dold-Kan construction from
\begin{equation}
\begin{tikzcd}
	\big( \ker( \dd^*_g \dd \colon \Omega^{n+1,v} \to \Omega^{n+1,v})
	& \Omega^{n,v} \ar[l, "\dd"']
	& \cdots \ar[l, "\dd"']
	& \Omega^{1,v} \big) \ar[l, "\dd"']
\end{tikzcd}
\end{equation}
Consequently, its underlying space is contractible by Lemma~\ref{st:spl smooth R-mods have trivial HoType}.
By Proposition~\ref{st:SMdl for SSol = *}, it follows that we have an equivalence
\begin{equation}
	\scMdl_{Mw, Ne_M} \big( (\cG, \cA^{(k)}), g \big)
	\simeq \rmB \scAut \big( (\cG, \cA^{(k)}), g \big))
	\simeq \rmB \scAut (\cG, \cA^{(k)})\,.
\end{equation}
Finally, the claim follows from the canonical equivalence $\scAut (\cG, \cA^{(k)}) \simeq \bbGrb^{n-1,M}_{\nabla|k}$ (as smooth higher abelian groups) and Corollary~\ref{st:Sp^n_0 is equivalence}.
\end{proof}

If additionally $\bbGamma$ is a smooth $\infty$-group acting on $(M,g)$ by isometries and preserving the equivalence class of the $n$-gerbe $\cG$ (see the formalism in Section~\ref{sec:LES for solution stacks}), a combination of Corollary~\ref{st:LES for S(Mdl_phi)} and Theorem~\ref{st:HoType of Sol_Mw} yields

\begin{theorem}
\label{st:hoLES for hYM Mdl}
Let $\bbGamma$ be a smooth $\infty$-group acting on $M$, with the action presented by a morphism $\phi \colon Q \to \rmB\scD[\cG, g]^\opp$ of left fibrations over $N\Cart^\opp$.
Then, there is a long exact sequence of homotopy groups (resp.~pointed sets in degree zero)
\begin{equation}
\begin{tikzcd}[column sep=0.75cm]
	\cdots \ar[r]
	& \pi_r S(\bbH) \ar[r]
	& \rmH^{n+2-r}(M;\ZZ) \ar[r] \ar[d, phantom, ""{coordinate, name=MidPoint}]
	& \pi_r S \big( \scMdl_{Mw,\Phi} ((\cG, \cA^{(k)}), g) \big)
	\ar[dll, rounded corners, to path={-- ([xshift=2ex]\tikztostart.east) |- (MidPoint) \tikztonodes -| ([xshift=-2ex]\tikztotarget.west) -- (\tikztotarget)}]
	& 
	\\
	& \pi_{r-1} S(\bbH) \ar[r]
	& \rmH^{n+1-r}(M;\ZZ) \ar[r]
	& \pi_{r-1} S \big( \scMdl_{Mw,\Phi} ((\cG, \cA^{(k)}), g) \big) \ar[r]
	& \cdots
\end{tikzcd}
\end{equation}
\end{theorem}

Moreover, an application of Theorem~\ref{st:equiv result for Mdl oo-stacks} yields the following result:

\begin{theorem}
\label{st:equiv of hYM Mds stacks}
Let $\bbGamma$ be a smooth $\infty$-group acting on $M$, with the action presented by a morphism $\phi \colon Q \to \rmB\scD[\cG, g]^\opp$ of left fibrations over $N\Cart^\opp$.
Further, let $0 \leq l < k \leq n$, and let $\cA^{(l)}$ denote the $l$-connection on $\cG$ obtained by forgetting the highest form degrees of $\cA^{(k)}$.
Then, there are canonical equivalences of $\infty$-presheaves
\begin{align}
	\scMdl_{Mw,\Phi} \big( (\cG, \cA^{(k)}), g \big)
	&\simeq \scMdl_{Mw,\Phi} \big( (\cG, \cA^{(l)}), g \big)\,,
	\\
	\scMdl_{Mw,\Phi} (\cG, \cA^{(k)})
	&\simeq \scMdl_{Mw,\Phi} (\cG, \cA^{(l)})\,.
\end{align}
If $\bbGamma = H$ is a sheaf of ordinary (rather than higher) groups, then these moduli $\infty$-prestacks are even $\infty$-stacks.
\end{theorem}

\begin{proof}
Both equivalences follow from Theorem~\ref{st:equiv result for Mdl oo-stacks}, noting that, by Corollary~\ref{st:Grb conn comps}, each of the morphisms
\begin{align}
	\Grb^{n,\scD}_{\nabla|k} \times \Met_{r,s}^\scD
	&\longrightarrow \Grb^{n,\scD}_{\nabla|l} \times \Met_{r,s}^\scD
	\qquad \text{and}
	\\
	\Grb^{n,\scD}_{\nabla|k}
	&\longrightarrow \Grb^{n,\scD}_{\nabla|l}
\end{align}
is 0-connected.
The last claim follows readily from Corollary~\ref{st:Descent result for moduli oo-prestacks} (one readily observes that the solution presheaves are homotopy sheaves as required in that corollary since the field equations only involve the curvature of $(\cG, \cA)$).
\end{proof}

\begin{remark}
Similar versions of Theorems~\ref{st:HoType of Sol_Mw},~\ref{st:hoLES for hYM Mdl} and~\ref{st:equiv of hYM Mds stacks} hold true for the other moduli $\infty$-(pre)stacks in Examples~\ref{eg:Mdl stacks from Grb x Met} and~\ref{eg:Mdl stacks for SD, BF, CS}.
\qen
\end{remark}

We finish this section with an example which encodes an action of a higher group on the simplicial presheaf of higher Maxwell solutions; a truncated version also appears in~\cite[Sec.~4.1]{GKSW:Gen_global_syms}:

\begin{example}
Define the full simplicial subpresheaf
\begin{equation}
	\tr\Grb^{n,\scD}_\nabla \subset \Grb^{n,\scD}_\nabla
\end{equation}
on all vertices $(\cG, \cA)$ whose underlying $n$-gerbe is given by a $\U(1)$-valued \v{C}ech cocycle which is identically equal to one, and whose connection is of the form $\cA = (A_1, \ldots, A_{n+1}) = (0, \ldots, 0, \delta \rho)$ for some smooth family of globally defined $(n{+}1)$-forms $\rho$ on $M$.
We denote these objects by $(\cI, \rho)$.
A morphism $(\cI, \rho_0) \to (\cI, \rho_1)$ in $\tr\Grb^{n,\scD}_\nabla$ is the same as an $(n{-}1)$-gerbe with connection $(\cG', \cA')$ such that
\begin{equation}
	\rho_1 - \rho_0
	= \curv(\cG', \cA')\,.
\end{equation}
We let 
\begin{equation}
	\tr\Grb^{n,\scD}_{\nabla, 1\flat} \subset \tr\Grb^{n,\scD}_\nabla
\end{equation}
denote the further simplicial subset defined by demanding that the 1-simplices $(\cG', \cA')$ are smooth families of \textit{flat} $(n{-}1)$-gerbes with connection.
We set
\begin{equation}
	\Conf^\scD
	\coloneqq \Met_{r,s}^\scD \times \big( (\tr\Grb^{n,\scD}_\nabla)_{(\cI, 0)/} \underset{\tr\Grb^{n,\scD}_\nabla}{\times} \tr\Grb^{n,\scD}_{\nabla, 1\flat} \big)\,.
\end{equation}
We can view this as a model for the higher action groupoid (or homotopy quotient)
\begin{align}
	\Conf^\scD
	\simeq \Met_{r,s}^\scD \times \big( \Grb^{n-1,\scD}_\nabla \dslash \Grb^{n-1,\scD}_\flat \big)
\end{align}
of the action of flat $(n{-}1)$-gerbes with connection on all $(n{-}1)$-gerbes with connection on $M$ via the product in $\Grb^{n-1,\scD}_\nabla$.
The advantage of writing $\Conf^\scD$ in the more complicated way above is that it is easier to see that the morphism
\begin{align}
\label{eq:inhom MW Conf --> Fix}
	&\Conf^\scD = \Met_{r,s}^\scD \times \big( (\tr\Grb^{n,\scD}_\nabla)_{(\cI, 0)/} \underset{\tr\Grb^{n,\scD}_\nabla}{\times} \tr\Grb^{n,\scD}_{\nabla, 1\flat} \big)
	\\
	&\longrightarrow \Met_{r,s}^\scD \times \big( \Omega^{d-n-1,v,\scD} \times \big( \{\cI\} \dslash \Grb^{n-1,\scD}_\flat \big) \big) \eqqcolon \Fix^\scD
\end{align}
is a projective fibration.
We describe this morphism in more detail:
first, the simplicial presheaf $\{\cI\} \dslash \Grb^{n-1,\scD}_\flat$ has the same $k$-simplices as $\tr\Grb^{n,\scD}_{\nabla, 1\flat}$ for each $k > 0$, whereas over each $c \in \Cart$ it has only a single vertex.
We can understand this as a Kan complex of trivial $n$-gerbes without connection, but where we still endow the higher simplices with connections as before.
Then, in simplicial level $k > 0$, the second component of the morphism~\eqref{eq:inhom MW Conf --> Fix} agrees with the canonical projection
\begin{equation}
	\Met_{r,s}^\scD \times \big( (\tr\Grb^{n,\scD}_\nabla)_{(\cI, 0)/} \underset{\tr\Grb^{n,\scD}_\nabla}{\times} \tr\Grb^{n,\scD}_{\nabla, 1\flat} \big)
	\longrightarrow \Met_{r,s}^\scD \times \tr\Grb^{n,\scD}_{\nabla, 1\flat}\,.
\end{equation}
On vertices, the morphism acts as
\begin{equation}
	\big( g,\, (\cG', \cA') \colon (\cI,0) \to (\cI, \rho) \big)
	\longmapsto \big( g,\, \dd^*_g \rho,\, \{\cI\} \big)\,.
\end{equation}
One readily observes that the morphism~\eqref{eq:inhom MW Conf --> Fix} has the horn-lifting properties which make it into a Kan fibration.
We further define
\begin{equation}
	\Sol^\scD = \Conf^\scD\,.
\end{equation}
Then, the fibres of~\eqref{eq:inhom MW Conf --> Fix} describe smooth families of higher Maxwell solutions for a given inhomogeneity $j = \dd^*_g \rho$.
This setting encodes the action of flat higher $\U(1)$-connections on Maxwell solutions as investigated, for instance, in~\cite[Sec.~4.1]{GKSW:Gen_global_syms}, in terms of a smooth higher group action.
\qen
\end{example}

%%%%%%%%%%%%%%%%%%%%%%%%%%%%%%%%%%%%%%%%%%%%%%%%%%%%%%%%%%%%%%%%%%%%%%%%%%%%

\section{Higher $\U(1)$-gauge theory and a new String group model}
\label{sec:String group}

%%%%%%%%%%%%%%%%%%%%%%%%%%%%%%%%%%%%%%%%%%%%%%%%%%%%%%%%%%%%%%%%%%%%%%%%%%%%

We next present a brief application of our results to String groups (as defined in~\cite{Stolz:Conj_on_pos_Ric}).
Here we will be using the formalism for smooth $\infty$-groups, i.e.~group objects in the $\infty$-category $\scP(N\Cart)$, and the notion of String group extensions in this setting as introduced in~\cite{Bunk:Pr_ooBdls_and_String} (see, in particular,~\cite[Defs.~4.1, 4.2]{Bunk:Pr_ooBdls_and_String}).
Let $H$ be a compact, simple and simply connected Lie group.
Then, a smooth String group extension of $H$ is an extension
\begin{equation}
	A \longrightarrow \widehat{H} \longrightarrow H
\end{equation}
of group objects in $\scP(N\Cart)$ (see~\cite[Def.~4.26]{NSS:Pr_ooBdls_I} and~\cite[Thm.~3.48]{Bunk:Pr_ooBdls_and_String}) which the functor $S \colon \scP(N\Cart) \to \scS$ sends to an ordinary String group extension in $\scS$.
We recall the following standard terminology for $k$-connections on 1-gerbes:

\begin{definition}
\label{def:connective structure and curving}
Let $\cG \in \Grb^1(M)$ be a 1-gerbe on a manifold $M$.
A \textit{connective structure} on $\cG$ is a 1-connection $\cA^{(1)}$ on $\cG$.
A \textit{curving} for a pair $(\cG, \cA^{(1)})$ of a 1-gerbe and a connective structure thereon is a family of locally defined 2-forms which completes $\cA^{(1)}$ into a full connection on $\cG$.
In other words, a connection on a 1-gerbe consists of a connective structure and a compatible curving.
\end{definition}

Let $L \colon H \to \Diff(H)$ denote the action of $H$ on itself via left multiplication.
Let $\cG_\bas \in \Grb^1(M)$ denote its basic gerbe~\cite{Meinrenken:The_basic_gerbe}; this is a gerbe whose class in $\rmH^3(H;\ZZ) \cong \ZZ$ is a generator.
Further, $\cG_\bas$ carries a canonical connection~\cite{Meinrenken:The_basic_gerbe}; we denote this connection by $\cA_\bas$ and its underlying connective structure by \smash{$\cA_\bas^{(1)}$}.

In~\cite{FRS:Higher_gerbe_connections}, Fiorenza, Rogers and Schreiber constructed a smooth String group model, which, in our formalism, agrees with the extension
\begin{equation}
	\scAut(\cG_\bas, \cA_\bas) \longrightarrow \bbString^{(2)}(H) \coloneqq \scSym_L(\cG_\bas, \cA_\bas) \longrightarrow H
\end{equation}
of group objects in $\scP(N\Cart)$ arising from Theorem~\ref{st:Aut-Sym_phi-Gamma extension}.
This is an extension of $H$ by $\rmB\U(1)$.
It was conjectured in~\cite{BMS:2-Grp_Ext} and proven in~\cite{Bunk:Pr_ooBdls_and_String} that also the extension
\begin{equation}
	\scAut(\cG_\bas) \longrightarrow \bbString^{(0)}(H) \coloneqq \scSym_L(\cG_\bas) \longrightarrow H
\end{equation}
is a smooth String group model.
By Corollary~\ref{st:Aut-Sym_phi-Gamma extension} it is an extension of $H$ as a group object in $\scP(N\Cart)$ by the smooth 2-group $\rmB(\bbGrb^{0,M}) \simeq \rmB(H^{\U(1)})$, where this equivalence uses that $\rmH^2(H;\ZZ) = 0$.
The underlying space of $\rmB (\U(1)^H)$ further has the homotopy type of $\rmB\U(1)$ since $\pi_1(H) \simeq 0$; see~\cite{BMS:2-Grp_Ext, Bunk:Pr_ooBdls_and_String} for details).

Here we are interested in the intermediate case, i.e.~the extension
\begin{equation}
	\scAut(\cG_\bas, \cA_\bas^{(1)}) \longrightarrow \bbString^{(1)}(H) \coloneqq \scSym_L(\cG_\bas, \cA_\bas^{(1)}) \longrightarrow H\,,
\end{equation}
where we fix the basic gerbe with only its \textit{connective structure}, rather than its full connection.
There is a canonical equivalence
\begin{equation}
	\scAut(\cG_\bas, \cA_\bas^{(1)})
	\simeq \bbGrb^{0,H}_\nabla\,,
\end{equation}
which is the smooth $\infty$-group of $\U(1)$-bundles with connection on $H$ (though note that each $\U(1)$-bundle on $H$ is trivialisable).
By the fact that $H$ is 2-connected (so that each $\U(1)$-bundle on $H$ is trivialisable) and Corollary~\ref{st:Sp^n_0 is equivalence} we thus obtain that
\begin{equation}
	S \big( \scAut(\cG_\bas, \cA_\bas^{(1)}) \big)
	\simeq S \bbGrb^{0,H}_\nabla
	\simeq S \bbGrb^{0,H}
	\simeq \rmB \U(1)\,.
\end{equation}
That is, it has the correct homotopy type for the fibre in a String group extension.

We check that $\bbString^{(1)}(H)$ indeed has the correct homotopy type for a String group extension of $H$:
first, the canonical forgetful map \smash{$p^1_0 \colon \bbGrb^{1,\scD}_{\nabla|1} \longrightarrow \bbGrb^{1,\scD}$} induces a canonical morphism of smooth $\infty$-groups
\begin{equation}
	q' \colon \bbString^{(1)}(H) = \scSym_L(\cG_\bas, \cA^{(1)}_\bas)
	\longrightarrow \scSym_L(\cG_\bas) = \bbString^{(0)}(H)\,.
\end{equation}

\begin{corollary}
The canonical forgetful map $q'$ induces an equivalence on the underlying spaces of $\bbString^{(1)}(H)$ and $\bbString^{(0)}(H)$.
\end{corollary}

\begin{proof}
This follows by a direct application of Corollary~\ref{st:equivs of SSym(cG_0) and SSym(cG_1)} to the setting where
\begin{equation}
	\Sol^\scD = \Conf^\scD = \Grb^{1,\scD}_\nabla\,,
	\qquad
	\Fix_0^\scD = \Grb^{1,\scD}_{\nabla|1}
	\qqandqq
	\Fix_1^\scD = \Grb^{1,\scD}_{\nabla|0}\,,
\end{equation}
with the canonical morphism $p^k_l \colon \Fix_0^\scD \to \Fix_1^\scD$ which forgets the curving of a gerbe connection (see also Theorem~\ref{st:forgetting conns is a principal map} and Corollary~\ref{st:Sp^n_0 is equivalence}).
\end{proof}

To see that $\bbString^{(1)}(H) \to H$ is indeed a string group extension, it thus remains to check that the morphism
\begin{equation}
	S \big( \bbString^{(1)}(H) \big)
	\longrightarrow H
\end{equation}
in $\scS$ represents a generator of $\rmH^3(H;\ZZ) \cong \ZZ$.
However, we have a commutative triangle
\begin{equation}
\begin{tikzcd}[column sep={2cm,between origins}]
	S \big( \bbString^{(1)}(H) \big) \ar[rr, "S q'"] \ar[dr]
	& & S \big( \bbString^{(0)}(H) \big) \ar[dl]
	\\
	& H &
\end{tikzcd}
\end{equation}
whose horizontal morphism is an equivalence.
Since the right-hand diagonal morphism represents a generator of $\rmH^3(H;\ZZ)$ (because we know that $\bbString^{(0)}(H)$ is a String group model), we obtain that also the left-hand morphism represents a generator of $\rmH^3(H;\ZZ)$.
In fact, the exact same arguments hold true if $\cA^{(1)}$ is \textit{any} 1-connection on $\cG_\bas$.
This establishes a whole family of higher smooth String group models:

\begin{theorem}
\label{st:String^1(H)}
For any 1-connection $\cA^{(1)}$, the extension of group objects
\begin{equation}
\begin{tikzcd}
	\scAut(\cG_\bas, \cA^{(1)}) \ar[r]
	& \bbString^{(1)}(H) \coloneqq \scSym_L(\cG_\bas, \cA^{(1)}) \ar[r]
	&  H
\end{tikzcd}
\end{equation}
in $\scP(N\Cart)$ is a model for the String group of $H$.
\end{theorem}

\begin{proof}
The arguments preceding Theorem~\ref{st:String^1(H)} prove the result for $\cA^{(1)} = \cA^{(1)}_{bas}$.
The general case then follow from Theorems~\ref{st:equiv thm for conns on n-gerbes} and~\ref{st:SMdl for SSol = *}.
\end{proof}

%%%%%%%%%%%%%%%%%%%%%%%%%%%%%%%%%%%%%%%%%%%%%%%%%%%%%%%%%%%%%%%%%%%%%%%%%%%%

\section{Moduli of NSNS supergravity solutions and gerbe connections}
\label{sec:GRic solitons}

%%%%%%%%%%%%%%%%%%%%%%%%%%%%%%%%%%%%%%%%%%%%%%%%%%%%%%%%%%%%%%%%%%%%%%%%%%%%

In this section we construct the higher moduli stacks of solutions to NSNS supergravity on $M$. Note that the same formalism directly applies to generalised Ricci solitons on the gerbe, of which supergravity NSNS solutions define a particular class.  
We do not require that the B-field is topologically trivial.
The main point of interest in this section is the mathematical structure of the B-field:
it is usually described either as a connection on a gerbe on $M$, or as an isotropic splitting of an exact Courant algebroid on $M$.
We compare these two perspectives by presenting a simplicial-presheaf description of exact Courant algebroids in the framework of this paper and adapt Hitchin's generalised tangent bundle construction to a map from the simplicial presheaf of gerbes with connective structure (resp.~and curving) to that of exact Courant algebroids (resp.~with isotropic splitting).
The groupoid of exact Courant algebroids on $M$ can be seen as a categorification of the \textit{set} $\rmH^3(M;\RR)$.
Our simplicial model also categorifies the abelian group structure on $\rmH^3(M;\RR)$.

We point out the problem that the field configurations of NSNS supergravity in both pictures are not equivalent.
In the first, they form an honest groupoid, whereas in the second they form a set.
However, we recall that the generalised tangent bundle associated to a gerbe with connective structure is its Atiyah algebroid as a principal 2-bundle.
Using this fact and invoking charge quantisation we are able to apply Theorem~\ref{st:equiv result for Mdl oo-stacks} to this problem and obtain that the higher moduli stacks of solutions are equivalent.

%%%%%%%%%%%%%%%%%%%%%%%%%%%%%%%%%%%%%%%%%%%%%%%%%%%%%%%%%%%%%%%%%%%%%%%%%%%%

\subsection{Simplicial presheaves of exact Courant algebroids}
\label{sec:ECA}

%%%%%%%%%%%%%%%%%%%%%%%%%%%%%%%%%%%%%%%%%%%%%%%%%%%%%%%%%%%%%%%%%%%%%%%%%%%%

Consider the simplicial presheaf
\begin{equation}
	\ECA \coloneqq \varGamma \circ \Omega^{2,v}_\cl[-1]
	\colon \Cartfam^\opp \longrightarrow \sSet
\end{equation}
which sends an object $(\hat{c} \to c) \in \Cartfam$ to the simplicial set obtained as the Dold-Kan construction of the chain complex
\begin{equation}
	\rmb \Omega^{2,v}_\cl = \Omega^{2,v}_\cl [-1]
	= \big( 0 \longleftarrow \Omega^{2,v}_\cl \longleftarrow 0 \big)\,.
\end{equation}
Here \smash{$\Omega^{2,v}_\cl$} is situated in degree one and denotes the sheaf (on $\Cartfam$) of vertical differential forms on \smash{$\hat{c}$} of degree two which are closed under the \textit{vertical} de Rham differential $\dd^v$.
The action of the functor $\ECA$ on morphisms is simply by pullback of differential forms.

There exists a canonical morphism of simplicial presheaves on $\Cartfam$
\begin{equation}
\label{eq:mp bb_nabla U(1) to ECA complex}
\begin{tikzcd}
	\rmb \rmb_\nabla \U(1) \ar[r, equal] \ar[d]
	& \big( 0 \ar[d, shift left=0.075cm]
	& \Omega^{1,v} \ar[l] \ar[d, "\dd^v"]
	& \U(1) \ar[l, "\dd^v \log"'] \ar[d]
	& 0 \ar[d, shift left=-0.075cm] \ar[l] \big)
	\\
	\rmb \Omega^{2,v}_\cl \ar[r, equal]
	& \big( 0
	& \Omega^{2,v}_\cl \ar[l]
	& 0 \ar[l]
	& 0 \big) \ar[l]
\end{tikzcd}
\end{equation}

The formalism in Section~\ref{sec:Families of cartesian spaces} provides a simplicial presheaf
\begin{equation}
	\widetilde{\ECA}{}^\scD
	\coloneqq \check{C}^{\scD*} \ECA \colon (\GCov^\scD)^\opp \longrightarrow \sSet\,,
	\qquad
	\widetilde{\ECA}{}^\scD(c, \hat{\cU} \to \cU)
	= \ul{\scH_\rmfam} \big( \v{C}(\hat{\cU} \to \cU), \ECA \big)\,.
\end{equation}
We can describe the simplicial set \smash{$\widetilde{\ECA}{}^\scD(c, \hat{\cU} \to \cU)$} explicitly:
it is the nerve of the groupoid with
\begin{itemize}
\item objects given by 2-forms $F \in \Omega^{2,v}(c, \v{C}_1\hat{\cU})$ which are closed under the vertical \v{C}ech \textit{and} de Rham differentials.

\item morphisms $F_0 \to F_1$ given by \textit{vertically} closed 2-forms $b \in \Omega^{2,v}(\v{C}_0\hat{\cU})$ whose vertical \v{C}ech differential satisfies that $\delta^v b = F_1 - F_0$.
\end{itemize}
In the following we will---by a slight abuse of notation---often discuss \smash{$\widetilde{\ECA}{}^\scD(c, \hat{\cU} \to \cU)$} in terms of the groupoid of which it is the nerve.

\begin{lemma}
\label{st:Omega^(2,v)_cl satisfies tau-fam-descent}
The functor
\begin{equation}
	\Omega^{2,v}_\cl[-1] \colon \Cartfam^\opp \longrightarrow \Ch_{\geq 0}
\end{equation}
satisfies descent along $\tau_\rmfam$-coverings.
Consequently, the object $\ECA \in \scH_\rmfam$ is fibrant in $\scH_\rmfam^{loc}$.
\end{lemma}

\begin{proof}
Let $(\hat{\cU} \to \cU)$ be a covering of $(\hat{c} \to c)$ in $\Cartfam$.
We need to show that the canonical morphism 
\begin{equation}
	\Omega^{2,v}_\cl[-1] (\hat{c} \to c)
	\longrightarrow \holim_\bbDelta \Big( \Omega^{2,v}_\cl[-1] \big( \v{C} (\hat{\cU} \to \cU) \big) \Big)
\end{equation}
is a quasi-isomorphism.
The homotopy limit is modelled by the truncated total chain complex of the double complex given by the \v{C}ech resolution in each degree.
Concretely, this is the two-term chain complex
\begin{equation}
\label{eq:Cech cplx of Omega^2v_cl}
\begin{tikzcd}
	\Omega^{2,v}_\cl \big( \check{C}_0(\hat{\cU} \to \cU) \big) \ar[r, "\delta"]
	& \ker \Big( \delta \colon \Omega^{2,v}_\cl \big( \check{C}_1(\hat{\cU} \to \cU) \big) \longrightarrow \Omega^{2,v}_\cl \big( \check{C}_2(\hat{\cU} \to \cU) \big) \Big)\,,
\end{tikzcd}
\end{equation}
where the second term lies in degree zero.
Note that this chain complex depends only on the covering $\hat{\cU}$ of the cartesian space $\hat{c}$.
We readily see that its first homology is $\Omega^{2,v}_\cl(\hat{c} \to c)$, and so it remains to show that its zeroth homology vanishes.
We check this explicitly:
let $\omega$ be a zero-cycle.
Since it is \v{C}ech closed, there exists some $\eta \in \Omega^{2,v}(\v{C}_0(\hat{\cU} \to \cU))$ with $\delta \eta = \omega$.
However, this $\eta$ is, in general, not $\dd^v$-closed.
To achieve $\dd^v$-closedness, first observe that
\begin{equation}
	\delta \dd^v \eta = \dd^v \delta \eta = 0\,,
\end{equation}
by our assumption on $\omega$.
Thus, there exists some (unique) $\alpha \in \Omega^{3,v}(\hat{c} \to c)$ with $\delta \alpha = \dd^v \eta$ (where we set \smash{$(\delta \alpha)_a \coloneqq \alpha_{|\hat{U}_a}$}).
Observe that the map
\begin{equation}
	\delta \colon \Omega^{p,v}(\hat{c} \to c)
	\longrightarrow \Omega^{p,v} \big( \v{C}_0 (\hat{\cU} \to \cU) \big)
\end{equation}
is injective, for each $p \in \NN_0$.
Combining this with the identities
\begin{equation}
	\delta \dd^v \alpha
	= \dd^v \delta \alpha
	= \dd^v \dd^v \eta
	= 0
\end{equation}
we deduce that $\dd^v \alpha = 0$ everywhere on $(\hat{c} \to c)$.
Up to isomorphism in $\Cartfam$, the object $(\hat{c} \to c)$ is of the form of a canonical projection map $(c \times d \to c)$, where $c,d \in \Cart$.
Therefore, by a family-version of the Poincaré Lemma (parameterised by $c$) we find some $\beta \in \Omega^{2,v}(\hat{c} \to c)$ that satisfies $\dd^v \beta = \alpha$.
Then, $\eta' \coloneqq \eta - \delta \beta$ is an element in $\Omega^{2,v}_\cl(\v{C}_0(\hat{U} \to \cU))$ as desired.
\end{proof}

\begin{lemma}
The functor \smash{$\widetilde{\ECA}{}^\scD \colon (\GCov^\scD)^\opp \longrightarrow \sSet$} restricts to essentially constant functors $(\GCov^\scD)^\opp_{|c} \longrightarrow \sSet$ on the fibre over each $c \in \Cart$.
\end{lemma}

\begin{proof}
This follows by an application of Lemma~\ref{st:wtG ess const on fibres of varpi^D}.
\end{proof}

We define a functor
\begin{equation}
	\ECA^\scD \colon \rmB\scD^\opp \longrightarrow \sSet\,,
	\qquad
	\ECA^\scD \coloneqq \Lan_{\varpi^\scD} \widetilde{\ECA}{}^\scD
\end{equation}
(compare also Definition~\ref{def:Grb^M and Grb^D(M)}); this is projectively fibrant by Proposition~\ref{st:Lan_varpi} and comes with a morphism
\begin{equation}
	\AtCA^\scD \colon \Grb^{1,\scD}_{\nabla|1} \longrightarrow \ECA^\scD
\end{equation}
in $\Fun(\rmB\scD^\opp, \sSet)$, which is induced by~\eqref{eq:mp bb_nabla U(1) to ECA complex}.
We call this morphism \textit{forming the Atiyah Courant algebroid} associated to a gerbe with connective structure on $M$.
We will elaborate more on this interpretation below.

Next we consider Hitchin's construction of the \textit{generalised tangent bundle} ~\cite[Sec.~2.3]{Hitchin:Brackets_forms_and_functionals} (see also~\cite[Sec.~7]{Hitchin:Generalised_CY_manifolds}).
Originally, this construction produces an exact Courant algebroid from the local transition data for a gerbe with connective structure.

\begin{remark}
One can slightly generalise Hitchin's generalised tangent bundle construction, so that it takes as input sections of \smash{$\widetilde{\ECA}{}^\scD$}; it does not actually depend on the choice of a gerbe with connective structure, but rather only on the collections of locally defined closed 2-forms which arise as the de Rahm differential of the connective structure (Hitchin already observed that the crucial property of the 2-forms entering in his construction is \textit{closedness} rather than that they arise from a connective structure on a gerbe~\cite[p.~544]{Hitchin:Brackets_forms_and_functionals}).
\qen
\end{remark}

We start by recalling the general definition of an exact Courant algebroid on a manifold $M$:

\begin{definition}
\label{def:exact Courant algebroids}
A \textit{Courant algebroid on $M$} is a quadruple $(E, \<-,-\>, \diamond, \rho)$, consisting of a vector bundle $E$ with a morphism $\rho \colon E \to TM$, called the \textit{anchor map}, a smooth, non-degenerate, symmetric bilinear pairing $\<-,-\>$ on $E$, and the \textit{Dorfman bracket}, an antisymmetric bilinear map $\diamond \colon \Gamma (M; E) \times \Gamma(M; E) \to \Gamma(M;E)$ on the sections of $E$.
These data satisfy that, for each $a,b,c \in \Gamma(M;E)$ and $f \in C^\infty(M)$, we have the following identities
\begin{enumerate}
\item (\textit{Jacobi identity}) $a \diamond (b \diamond c) = (a \diamond b) \diamond c + b \diamond (a \diamond c)$,

\item (\textit{$\rho$ preserves brackets}) $\rho(a \diamond b) = [\rho(a), \rho(b)]$, where on the right-hand side we use the Lie bracket of vector fields on $M$,

\item (\textit{Derivation property}) $a \diamond (fb) = f(a \diamond b) + \dd f(\rho(a)) \cdot b$,

\item (\textit{Compatibility of bracket and pairing}) $\iota_{\rho(a)} \dd \<b,c\> = \< a \diamond b, c\> + \<b, a \diamond c\>$,

\item (\textit{Failure of antisymmetry}) $a \diamond b + b \diamond a = D \<a, b\>$, where $D \colon C^\infty(M) \to \Gamma(M;E)$ is the differential operator given as
\begin{equation}
	D = \sharp_E \circ \rho^\vee \circ \dd\,.
\end{equation}
Here, $\sharp_E \colon E^\vee \to E$ is the musical isomorphism induced from the non-degenerate pairing $\<-,-\>$ on $E$, and $\rho^\vee \colon T^\vee M \to E^\vee$ is the dual of the vector bundle morphism $\rho \colon E \to TM$.
\end{enumerate}
A Courant algebroid $(E, \<-,-\>, \diamond, \rho)$ on $M$ is \textit{exact} if the sequence
\begin{equation}
\begin{tikzcd}
	0 \ar[r]
	& T^\vee M \ar[r, "\sharp_E \circ \rho^\vee"]
	& E \ar[r, "\rho"]
	& TM \ar[r]
	& 0
\end{tikzcd}
\end{equation}
of vector bundles on $M$ is exact.
\end{definition}

\begin{definition}
Let $(E_0, \<-,-\>_0, \diamond_0, \rho_0)$ and $(E_1, \<-,-\>_1, \diamond_1, \rho_1)$ be two exact Courant algebroids on $M$.
An \textit{isomorphism of exact Courant algebroids} $(E_0, \<-,-\>_0, \diamond_0, \rho_0) \to (E_1, \<-,-\>_1, \diamond_1, \rho_1)$ is a vector bundle isomorphism $\psi \colon E_0 \to E_1$ which preserves the pairings and Dorfman brackets.
\end{definition}

It then follows that $\psi$ also intertwines the anchor maps~\cite[p.~14]{GFS:Gen_Ricci_flow}.
Note that, for each $c \in \Cart$, this generalises straightforwardly to a family of exact Courant algebroids on $M$, parameterised by $c$:
this is a vector bundle $E \to c \times M$ with the same data as above, but where we replace the de Rham differential by the vertical de Rham differential $\dd^v$ everywhere in Definition~\ref{def:exact Courant algebroids}, and similarly for $T^\vee M$ and $TM$.

Further, observe that if $f \in \Diff(M)$ and $(E, \<-,-\>, \diamond, \rho)$ is an exact Courant algebroid on $M$, then we obtain a canonical Courant algebroid structure on the pullback bundle $f^*E$; we denote this resulting exact Courant algebroid by
\begin{equation}
	f^*(E, \<-,-\>, \diamond, \rho)
	= (f^*E, f^*\<-,-\>, f^*\diamond, f^{-1}_* \circ \rho)\,.
\end{equation}

\begin{remark}
In summary, we obtain a pseudo-functor
\begin{equation}
	\CalECA \colon \rmB\scD^\opp \longrightarrow \Gpd\,,
\end{equation}
which sends each object $c \in \Cart$ to the groupoid of $c$-parameterised families of exact Courant algebroids on $M$.
\qen
\end{remark}

The adaptation of Hitchin's generalised tangent bundle construction now reads as follows:
each object $F$ in $\widetilde{\ECA}{}^\scD(c, \hat{\cU} \to \cU)$ is descent data for a smooth family of exact Courant algebroids on $M$:
over each patch $\hat{U}_a$ of the covering $\hat{\cU}$ of $c {\times} M$ consider the vector bundle $T^v\hat{U}_a \oplus T^{v,\vee}\hat{U}_a$, i.e.~the direct sum of the \textit{vertical} tangent and cotangent bundles on the patch $\hat{U}_a$.
For each $x \in \hat{U}_a$, let $(x_0, x_1) \in c \times M$ be its image under the inclusion $\hat{U}_a \hookrightarrow c \times M$.
This canonically identifies the vertical tangent space \smash{$T^v_{|x} \hat{U}_a$} with \smash{$T_{|x_1} M$}, and similarly the vertical cotangent space \smash{$T^{v,\vee}_{|x} \hat{U}_a$} with \smash{$T^\vee_{|x_1} M$}.
Given a point $y \in \hat{U}_b$ with $y_0 = x_0 \in c$ and $y_1 = x_1 \in M$, we identify each pair \smash{$X + \xi \in T^v_{|x} \hat{U}_a \oplus T^{v,\vee}_{|x} \hat{U}_a$} with the pair
\begin{equation}
	X + \xi + \iota_X F_{ab}
	\quad
	\in T^v_{|y} \hat{U}_b \oplus T^{v,\vee}_{|y} \hat{U}_b\,.
\end{equation}
This defines a vector bundle $E(F)$ on $c {\times} M$.
It sits in a short exact sequence of vector bundles
\begin{equation}
\begin{tikzcd}
	0 \ar[r]
	& T^{v,\vee}M \ar[r]
	& E(F) \ar[r]
	& T^vM \ar[r]
	& 0
\end{tikzcd}
\end{equation}
on $c {\times} M$.
Via the canonical pairing of 1-forms with tangent vectors and the standard Courant bracket (using that the vertical vector fields on $c{\times} M$ and $\hat{U}_a$ are closed under the Lie bracket on vector fields on $c {\times} M$ and $\hat{U}_a$, respectively), we obtain the structure of an exact Courant algebroid on each vector bundle \smash{$E(F)_{|(x_0,-)}$} on $M$.
This depends smoothly on $x_0 \in M$ in the sense that the pairing and the bracket of two smooth sections of $E(F)$ over $c {\times} M$ are a smooth function on $c {\times} M$ and a smooth section of $E(F)$, respectively.
Analogously, each morphism $b \colon F_0 \to F_1$ in $\widetilde{\ECA}{}^\scD(c, \hat{\cU} \to \cU)$ induces a vector bundle isomorphism $E(F_0) \to E(F_1)$ over $c {\times} M$ via the so-called \textit{$B$-field transformation} $e^b$;
with respect to the covering $\hat{\cU}$ it reads as
\begin{equation}
	X + \xi \longmapsto X + \xi + \iota_X b_a\,.
\end{equation}
This is even an isomorphism of exact Courant algebroids (on each fibre over $c$).

We have thus described a morphism
\begin{equation}
	\widetilde{\ECA} \longrightarrow \pi_M^*\CalECA
\end{equation}
of pseudo-functors $\GCov^\scD \to \Gpd$ at the level of what it assigns to each object $(c, \hat{\cU} \to \cU)$ of $\GCov^\scD$.
The extension to morphisms is straightforward (using pullbacks of vector bundles along smooth maps).
Consider an object of $\GCov^\scD$ of the form $(\RR^0, \hat{\cU} \to \RR^0)$; in other words, $\hat{\cU}$ is simply a good open covering of $M$.

\begin{remark}
The map $\check{C} \hat{\cU} \to \varGamma \rmb \Omega^{2,M}_\cl = \rmB\Omega^{2,M}_\cl$ is, equivalently, transition data for a principal bundle (or torsor) on $M$ with structure group the smooth group $\Omega^{2,M}_\cl$.
The bundle $E$ constructed above is then the associated bundle for the action of $\Omega^{2,M}_\cl$ on $\Gamma(-;T^vM) \oplus \Omega^{1,M}$, the sheaf of vector fields plus 1-forms on $M$.
\qen
\end{remark}

In that case, it is known that the above construction produces all exact Courant algebroids on $M \cong \RR^0 \times M$ up to fibre-preserving isomorphism (i.e.~up to $B$-field transformation; see, for instance,~\cite[Thm.2.19]{GFS:Gen_Ricci_flow}).
For convenience, we present a non-standard argument for this fact:
let $\CalECA(\RR^0)$ denote the groupoid of exact Courant algebroids and isomorphisms of exact Courant algebroids on $M$.

\begin{lemma}
The construction of the generalised tangent bundle
\begin{equation}
	\ECA^\scD(\RR^0) \simeq \widetilde{\ECA}{}^\scD(\RR^0, \hat{\cU} \to \RR^0) \longrightarrow \CalECA(\RR^0)
\end{equation}
is an equivalence of groupoids.
\end{lemma}

\begin{proof}
This follows from the computation of the \v{C}ech complex of $\Omega^{2,v}_\cl[-1]$ associated to a good open covering $(\hat{U} \to \cU)$ of $(c {\times} M \to M)$; it has the same explicit description as in~\eqref{eq:Cech cplx of Omega^2v_cl}.
In particular, its first homology agrees with the group of closed 2-forms on $M$, which is also canonically isomorphic to the automorphisms of any exact Courant algebroid on $M$ (see, for instance,~\cite[Prop.~2.15]{GFS:Gen_Ricci_flow}).
By the proof of Lemma~\ref{st:Omega^(2,v)_cl satisfies tau-fam-descent}, we also see that the first homology of the \v{C}ech complex of $\Omega^{2,v}_\cl[-1]$ agrees with $\rmH^3(M;\RR)$, which is canonically isomorphic to the set of isomorphism classes of exact Courant algebroids on $M$ by \v{S}evera's classification of exact Courant algebroids on $M$ in terms of their \v{S}evera class (see, for instance,~\cite[Thm.~2.19]{GFS:Gen_Ricci_flow}; note that this also applies to the case where the diffeomorphism of the base manifold is taken to be trivial).
Thus, (the generalisation of) the generalised tangent bundle construction is fully faithful and essentially surjective.
\end{proof}

\begin{remark}
Isomorphism classes of exact Courant algebroids on a manifold $M$ are in canonical bijection with $\rmH^3(M; \RR)$.
The underlying space (in the sense of Definition~\ref{def:underlying space of a presheaf}) of the simplicial presheaf $\ECA^M = Ne_M^* \ECA^\scD \colon \Cart^\opp \to \sSet$ of exact Courant algebroids on a manifold $M$ and their isomorphisms is contractible by Lemma~\ref{st:spl smooth R-mods have trivial HoType}.
This is in stark contrast with the situation for $\Grb^{1,M}$ (see Theorem~\ref{st:spaces of n-gerbes}).
One way to understand this difference is that gerbes are classified by \textit{integer} cohomology, whose coefficient group has discrete topology, whereas for real cohomology, which classifies exact Courant algebroids, the coefficient group is the real line, which is smoothly contractible.
A slightly different, but related perspective arises when one thinks of gerbes as classified by cohomology with coefficients in $\U(1)$; then, the cohomology groups also inherit a smooth structure, but at the same time they inherit a non-trivial topology from $\U(1)$, which prevents contractibility in this case.
\qen
\end{remark}

The description of exact Courant algebroid in terms of $\Omega^{2,v}_\cl[-1]$ has the following pleasant structural property:
by \v{S}evera's classification, isomorphism classes of exact Courant algebroids on $M$ are in bijection with the elements of $\rmH^3(M;\RR)$, which carries a canonical abelian group structure.
It is thus a natural question to ask whether this abelian group structure is reminiscent of a corresponding structure on exact Courant algebroids themselves.
That is, is there is a braided monoidal structure on the groupoid of exact Courant algebroids on $M$ which \textit{categorifies} the abelian group structure on $\rmH^3(M; \RR)$?
The following observation is a consequence of the construction of \smash{$\widetilde{\ECA}{}^\scD$} from a complex of presheaves of $\ul{\RR}$-modules by means of the Dold-Kan construction; it answers our question in the affirmative:

\begin{lemma}
For each object $(c, \hat{\cU} \to \cU) \in \GCov^\scD$, the groupoid $\widetilde{\ECA}{}^\scD(c, \hat{\cU} \to \cU)$ is strictly symmetric monoidal under the sum of vertical differential forms.
In particular, the nerve of $\widetilde{\ECA}{}^\scD(c, \hat{\cU} \to \cU)$ is a simplicial abelian group.
\end{lemma}

\begin{remark}
If we restrict to the groupoid of exact Courant algebroids defined with respect to a good open covering $\{U_a\}_{a \in \Lambda}$ of $M$, the symmetric monoidal structure takes the following form:
let $\{F_{ab} \in \Omega^2_\cl(U_{ab})\}_{a,b \in \Lambda}$ and $\{F'_{ab} \in \Omega^2_\cl(U_{ab})\}_{a,b \in \Lambda}$ be two \v{C}ech 1-cocycles on $M$ with values in $\Omega^2$, representing two exact Courant algebroids on $M$.
The monoidal product of these exact Courant algebroids is presented by the cocycle $\{F_{ab} + F'_{ab} \in \Omega^2_\cl(U_{ab})\}_{a,b \in \Lambda}$.
Similarly, let $\{\omega_a \in \Omega^1_\cl(U_a)\}_{a \in \Lambda}$ and $\{\omega'_a \in \Omega^1_\cl(U_a)\}_{a \in \Lambda}$ be collections of local 2-forms defining B-field transformations $e^\omega$ and $e^{\omega'}$ between exact Courant algebroids as above.
The monoidal product of these morphisms is $\{\omega_a + \omega'_a \in \Omega^1_\cl(U_a)\}_{a \in \Lambda}$, presenting the B-field transformation $e^{\omega + \omega'}$.
\qen
\end{remark}

\begin{theorem}
\label{st:AtCA_nabla and categorification of AbGrp H^3}
The morphism
\begin{equation}
	\AtCA^\scD \colon \Grb^{1, \scD}_{\nabla|1}
	\longrightarrow \ECA^\scD
\end{equation}
of simplicial presheaves on $\rmB\scD$ is compatible with the group structures; that is, it is even a morphism of presheaves of simplicial abelian groups on $\rmB\scD$.
Moreover, restricting to those morphisms which cover the identity on $M$, the construction categorifies the inclusion of $\rmH^3(M;\ZZ)$ into $\rmH^3(M;\RR)$:
we have a commutative square of abelian groups
\begin{equation}
\begin{tikzcd}[column sep=2cm, row sep=1cm]
	\pi_0 \big( \Grb^{1, M}_{\nabla|1}(\RR^0) \big) \ar[r, "\pi_0 (\AtCA)_{|\RR^0}"] \ar[d, "\cong", "\DD"']
	& \pi_0 \big( \ECA^M(\RR^0) \big) \ar[d, "\cong", "\SC"']
	\\
	\rmH^3(M;\ZZ) \ar[r, hookrightarrow]
	& \rmH^3(M;\RR)
\end{tikzcd}
\end{equation}
Here, the left-hand vertical morphism takes the Dixmier-Douady class of a gerbe with connective structure (note that this is independent of the connective structure), and the right-hand vertical morphism takes the \v{S}evera class of an exact Courant algebroid.
\end{theorem}

\begin{proof}
One checks that the formation of the colimits in the definition of $\Grb^{1, \scD}_{\nabla|1}$ and $\ECA^\scD$ is compatible with the abelian group structure (filtered colimits preserve finite products).
The second claim then follows from the definition of the Dixmier-Douady and \v{S}evera classes of 1-gerbes and exact Courant algebroids, respectively.
\end{proof}

%%%%%%%%%%%%%%%%%%%%%%%%%%%%%%%%%%%%%%%%%%%%%%%%%%%%%%%%%%%%%%%%%%%%%%%%%%%%

\subsection{Smooth families of isotropic splittings and generalised metrics}
\label{sec:Isoslpits and GenMet}

%%%%%%%%%%%%%%%%%%%%%%%%%%%%%%%%%%%%%%%%%%%%%%%%%%%%%%%%%%%%%%%%%%%%%%%%%%%%

We now include isotropic splittings and generalised metrics into our local description of smooth families of exact Courant algebroids on $M$.

\begin{definition}
The simplicial presheaf of \textit{exact Courant algebroids with isotropic splitting},
\begin{equation}
	\ECA_\nabla \colon \Cartfam \longrightarrow \sSet\,,
\end{equation}
is the Dold-Kan construction of the presheaf of chain complexes
\begin{equation}
	\big( \Omega^{2,v}  \hookleftarrow \Omega^{2,v}_\cl \big)
	\colon \Cartfam \longrightarrow \Ch_{\geq 0}\,,
\end{equation}
where $\Omega^{2,v}$ sits in degree zero.
\end{definition}

\begin{remark}
Observe that the chain complex $\Omega^{2,v}_\cl \hookrightarrow \Omega^{2,v}$ arises as the homotopy pullback
\begin{equation}
\begin{tikzcd}
	\big( \Omega^{2,v} \hookleftarrow \Omega^{2,v}_\cl \big) \ar[r] \ar[d]
	& 0 \ar[d]
	\\
	\rmb \Omega^{2,v}_\cl \ar[r, hookrightarrow]
	& \rmb \Omega^{2,v}
\end{tikzcd}
\end{equation}
Each of the vertices of the cospan underlying this homotopy pullback square satisfies homotopy descent with respect to $\tau_\rmfam$-coverings (by arguments analogous to those in the proof of Lemma~\ref{st:Omega^(2,v)_cl satisfies tau-fam-descent}).
Therefore, the presheaf $(\Omega^{2,v} \hookleftarrow \Omega^{2,v}_\cl)$ of chain complexes, and thus also $\ECA_\nabla$, satisfies homotopy descent with respect to $\tau_\rmfam$-coverings as well.
\qen
\end{remark}

We thus obtain a simplicial presheaf
\begin{equation}
	\widetilde{\ECA}{}^\scD_\nabla \colon \GCov^\scD(M)^\opp \longrightarrow \sSet
\end{equation}
associated to $\ECA_\nabla$ by applying the Dold-Kan construction objectwise and then applying Construction~\ref{cstr:sPShs on GCov^D from sPShs on Cart_fam}; this again restricts to essentially constant functors \smash{$\GCov^\scD(M)_{|c}^\opp \to \sSet$} on each fibre of the projection $\varpi^\scD \colon \GCov^\scD(M) \to \rmB\scD$ (by Lemma~\ref{st:wtG ess const on fibres of varpi^D}).
Thus, we obtain a well-behaved left Kan extension
\begin{equation}
	\ECA_\nabla^\scD \colon \rmB\scD^\opp \longrightarrow \sSet\,.
\end{equation}

\begin{example}
A vertex in $\ECA_\nabla(c)$, for $c \in \Cart$, is presented by a good open covering $(\hat{\cU} \to \cU)$ of $c {\times} M \to c$, a 2-form \smash{$F \in \Omega^{2,v}_\cl(\v{C}_1\hat{\cU})$} and a 2-form \smash{$B \in \Omega^{2,v}(\v{C}_0\hat{\cU})$} satisfying
\begin{equation}
	\delta B = F
	\qquad
	(\text{as well as} \quad \delta F = 0 \qandq \dd^v F = 0)\,.
\end{equation}
A 1-simplex $(F_0,B_0) \to (F_1, B_1)$ in $Ne_M^* \ECA_\nabla(c)$ is (after possibly passing to a refinement of good open coverings) a 2-form {$b \in \Omega^{2,v}_\cl(\v{C}_0\hat{\cU})$} satisfying
\begin{equation}
	b = B_1 - B_0\,.
\end{equation}
Geometrically, $b$ is a b-field transformation which preserves isotropic splittings of exact Courant algebroids.
All higher simplices are trivial; that is, $Ne_M^*\ECA_\nabla(c)$ is actually a \textit{set} which has been promoted to a simplicial set (observe that in $\ECA_\nabla(c)$ there are still non-trivial 1-simplices which stem from the action of diffeomorphisms of the base manifold $M$).
\qen
\end{example}

There is a canonical commutative diagram
\begin{equation}
\begin{tikzcd}[column sep=2cm, row sep=0.75cm]
	\rmb_\nabla^2 \U(1) \ar[r] \ar[d]
	& \big( \Omega^{2,v} \hookleftarrow \Omega^{2,v}_\cl \big) \ar[d]
	\\
	\rmb \rmb_\nabla \U(1) \ar[r]
	& \rmb \Omega^{2,v}_\cl
\end{tikzcd}
\end{equation}
at the level of presheaves of chain complexes on $\Cartfam$.
The top morphism reads as
\begin{equation}
\begin{tikzcd}
	\rmb \rmb_\nabla \U(1) \ar[r, equal] \ar[d]
	& \big( \Omega^{2,v} \ar[d, shift left=0.075cm, "\id"]
	& \Omega^{1,v} \ar[l] \ar[d, "\dd^v"]
	& \U(1) \ar[l, "\dd^v \log"'] \ar[d]
	& 0 \ar[d, shift left=-0.075cm] \ar[l] \big)
	\\
	\rmb \Omega^{2,v}_\cl \ar[r, equal]
	& \big( \Omega^{2,v}
	& \Omega^{2,v}_\cl \ar[l, hookrightarrow]
	& 0 \ar[l]
	& 0 \big) \ar[l]
\end{tikzcd}
\end{equation}

This induces a commutative square
\begin{equation}
\label{eq:Grb-ECA square forg conns}
\begin{tikzcd}[column sep=2cm, row sep=0.75cm]
	\Grb^{1,\scD}_\nabla \ar[r, "\AtCA_\nabla"] \ar[d]
	& \ECA_\nabla^\scD \ar[d]
	\\
	\Grb^{1,\scD}_{\nabla|1} \ar[r, "\AtCA"']
	& \ECA^\scD
\end{tikzcd}
\end{equation}
in $\Fun(\rmB\scD^\opp, \sSet)$, where the vertical morphisms forget (part of) the connection data.

\begin{lemma}
\label{st:Grb-ECA square is (ho)cartesian}
The square~\eqref{eq:Grb-ECA square forg conns} is cartesian and homotopy cartesian in $\Fun(\rmB\scD^\opp, \sSet)$.
\end{lemma}

\begin{proof}
Recall that we can compute the values of the simplicial presheaves involved as strict filtered colimits over refinements of good open coverings of $c {\times} M \to c$ (Proposition~\ref{st:Lan_varpi}) and that they send each refinement to a weak equivalence (Lemma~\ref{st:wtG ess const on fibres of varpi^D}).
The fact that the square is cartesian can be seen directly by evaluating on a covering $(\hat{\cU} \to \cU) \in \GCov^\scD(M)$.
The square is also homotopy cartesian:
first, over each covering, the vertices of the underlying cospan are Kan complexes and the right-hand vertical morphism is a Kan fibration (it is the image under $\varGamma$ of a projective fibration in $\Fun(\GCov^\scD, \Ch_{\geq 0})$).
Therefore, with respect to each good open covering, the square is homotopy cartesian.
Finally, Kan fibrations are stable under filtered colimits, so that also the full square~\eqref{eq:Grb-ECA square forg conns} is a homotopy pullback.
\end{proof}

\begin{remark}
The fact that the square~\eqref{eq:Grb-ECA square forg conns} is cartesian implies, in particular, the well-known bijection between curvings on a gerbe with connective structure on the one hand, and isotropic splittings of the exact Courant algebroid associated to the gerbe on the other hand (see, for instance, the end of the proof of~\cite[Prop.~1]{Hitchin:Brackets_forms_and_functionals}).
\qen
\end{remark}

\begin{definition}
The simplicial presheaf of \textit{exact Courant algebroids with generalised metrics} is the product
\begin{equation}
	\GMet^\scD \coloneqq \Met_{0,d}^\scD \times \ECA_\nabla^\scD
	\qquad
	\in \Fun(\rmB\scD^\opp, \sSet)\,.
\end{equation}
\end{definition}

\begin{remark}
Here we use that there is a bijection between generalised metrics on an exact Courant algebroid on $M$ and pairs of a Riemannian metric on $M$ and an isotropic splitting of the exact Courant algebroid~\cite[Prop.~2.40]{GFS:Gen_Ricci_flow}.
\qen
\end{remark}

\begin{remark}
Consider the choices
\begin{equation}
	\Conf^\scD \coloneqq \GMet^\scD\,,
	\qquad
	\Fix^\scD \coloneqq \ECA^\scD
\end{equation}
and the projection morphism $p \colon \GMet^\scD \to \ECA^\scD$.
For $\cE \in \ECA^M(\RR^0)$ an exact Courant algebroid on $M$, we obtain a smooth---but, in this case, not higher---symmetry group, given by a left fibration $\SYM^\rev(\cE) \to N\Cart^\opp$.
This comes with a canonical action on the presheaf $\Cart^\opp \to \Set$ which encodes smooth families of generalised metrics on $\cE$; this is the functor-of-points version of the set-up for which Rubio and Tipler prove a Slice Theorem in~\cite{RT:Courant_Aut_and_Mdl_of_GenMet}.
\qen
\end{remark}

%%%%%%%%%%%%%%%%%%%%%%%%%%%%%%%%%%%%%%%%%%%%%%%%%%%%%%%%%%%%%%%%%%%%%%%%%%%%

\subsection{Comparing moduli stacks of integral NSNS supergravity solutions}
\label{sec:comparing GRic moduli}

%%%%%%%%%%%%%%%%%%%%%%%%%%%%%%%%%%%%%%%%%%%%%%%%%%%%%%%%%%%%%%%%%%%%%%%%%%%%

Lemma~\ref{st:Grb-ECA square is (ho)cartesian} is a manifestation of the following (well-known~\cite{Hitchin:Brackets_forms_and_functionals}) fact:
consider a gerbe with connective structure $(\cG, \cA^{(1)})$ on a manifold $M$, and let $\AtCA (\cG, \cA^{(1)})$ denote its associated exact Courant algebroid on $M$.
Then, the isotropic splittings of $\AtCA (\cG, \cA^{(1)})$ are in canonical bijection with the curvings of $(\cG, \cA^{(1)})$.
Neither the curvings on $(\cG, \cA^{(1)})$, nor the isotropic splittings on $\AtCA (\cG, \cA^{(1)})$ possess higher structure; both form sets.
It may, therefore, appear that the discussions of curvings on gerbes with connective structure and isotropic splittings of exact Courant algebroids are entirely equivalent, so long as both exist (i.e.~so long as the \v{S}evera class of the exact Courant algebroid lies in the image of $\rmH^3(M;\ZZ)$ in $\rmH^3(M;\RR)$).

While this is true at the level of \textit{configurations}, it fails at the level of \textit{moduli}:
in particular, the automorphisms of $\AtCA (\cG, \cA^{(1)})$ in $\ECA^M$ form a \textit{group}, whereas the automorphisms of $(\cG, \cA^{(1)})$ in \smash{$\Grb^{1,M}_{\nabla|1}$} form a \textit{2-group}.
As a result of this discrepancy at the level of automorphisms, the derived quotients of isotropic splittings modulo automorphisms of $\AtCA (\cG, \cA^{(1)})$ and curvings modulo automorphisms of $(\cG, \cA^{(1)})$ are genuinely different.
For instance, on the level of isomorphism classes of objects the former identifies isotropic splittings which differ by a closed 2-form, whereas the latter identifies curvings which differ by a closed 2-form with integer periods.
In the first case, these are all morphisms in the homotopy quotient, whilst in the second case there are also 2-morphisms stemming from morphisms of $\U(1)$-bundles with connection.
Finally, this difference carries over from the quotients by automorphisms to quotients by more general symmetry groups.

When it comes to the discussion of \textit{moduli}, rather than configurations, it thus makes a crucial difference whether we investigate moduli of curvings on a gerbe with connective structure, or isotropic splittings on a Courant algebroid, even though the underlying \textit{configurations} are in canonical bijection.

From the perspective of string theory and functorial field theory it is particularly desirable to consider the integer case, i.e.~the case where the \v{S}evera class lies in the image of the map $\rmH^3(M;\ZZ) \to \rmH^3(M;\RR)$~\cite{Kapustin:D-branes_in_non-triv_B-fields, BW:OCFFTs, BW:Transgression_of_D-branes, GFS:Gen_Ricci_flow}.
This integrality condition ensures that the Wess-Zumino-Witten term in the $\sigma$-model action functional is well-defined (see also~\cite[p.~17, 18]{GFS:Gen_Ricci_flow}).
Further, supergravity can be seen as a low-energy limit of string theory, where charges (in particular those of the B-field) are quantised (see, for instance,~\cite{Freed:Dirac_charge_quantisation, DFM:Spin_strs_and_superstrings, FMS:Uncertainty_of_fluxes, Szabo:Quant_of_Higher_Ab_GT, ABEHSN:SymTFTs_from_String_thy}).
Consequently, the integrality condition persists in this limit.

In the following we concentrate on NSNS supergravity.
In fact, there are three candidate geometric descriptions of the B-field in NSNS supergravity:
\begin{enumerate}
\item[(1)] a connection on a gerbe $\cG$,

\item[(2a)] a connection on a gerbe $(\cG, \cA^{(1)})$ with connective structure, i.e.~a curving,

\item[(2b)] an isotropic splitting of an exact Courant algebroid, which we may assume to the of the form $\AtCA (\cG, \cA^{(1)})$ by charge quantisation.
\end{enumerate}

\begin{remark}
Usually, in setting (2b) the supergravity B-field is combined with the Riemannian metric into a \textit{generalised metric} on the Courant algebroid.
This is advantageous, for instance, to describe and compute T-duality in supergravity and string theory~\cite[Ch.~10]{GFS:Gen_Ricci_flow}.
However, generalised metrics on an exact Courant algebroid are in canonical bijection with pairs $(g, \sigma)$ of a Riemannian metric $g$ on $M$ and an isotropic splitting $\sigma$~\cite[Props.~2.38, 2.40]{GFS:Gen_Ricci_flow}.
\qen
\end{remark}

By the above discussion, the configurations in models (2a) and (2b) are in canonical bijection.
However, the symmetries---and therefore the resulting moduli stacks of these configurations---depend on which underlying geometry we consider.
Moreover, since connections on a given gerbe form a genuine groupoid (Section~\ref{sec:Con_k on fixed n-gerbe})---more precisely, a 1-truncated simplicial presheaf given as a fibre of the projection $\Grb^{1,M}_\nabla \to \Grb^{1,M}$---the configurations in model (1) cannot be equivalent to those in (2a) or (2b).
This poses a genuine problem:

\begin{problem}
\label{prob:mouli of NSNS B-fields}
What is the correct moduli stack of NSNS supergravity B-fields for a fixed cohomology class in $\rmH^3(M;\ZZ)$?
\qen
\end{problem}

Since the B-field is a constituent of the field content of NSNS supergravity, this directly leads to:

\begin{problem}
\label{prob:moduli of NSNS solutions}
What is the correct moduli stack of NSNS supergravity solutions for a fixed cohomology class in $\rmH^3(M;\ZZ)$?
\qen
\end{problem}

The exact Courant algebroid $\AtCA (\cG, \cA^{(1)})$ associated to a gerbe $(\cG, \cA^{(1)})$ with connective structure has an important geometric interpretation:
one can view $(\cG, \cA^{(1)})$ as a categorified principal bundle in $\scP(N\Cart)$ with base given by $M$ and structure group (presented by) the group object $\rmB_\nabla \U(1) \in \Grp(\scP(N\Cart))$.
Then, the Courant algebroid $\AtCA (\cG, \cA^{(1)})$ is the categorified \textit{Atiyah algebroid} of this principal $\rmB_\nabla \U(1)$-bundle~\cite{Collier:Inf_Syms_of_DD_Gerbes}.

Given this background, forming the quotient of isotropic splittings of the exact Courant algebroid $\AtCA(\cG, \cA^{(1)})$ by symmetries \textit{of the exact Courant algebroid} is analogous to taking splittings (as vector bundles) of the Atiyah algebroid $\mathrm{At}(P)$ of a principal bundle $P$ modulo symmetries \textit{of the Atiyah algebroid}.
This is not the quotient at the heart of gauge theory:
while the (vector bundle) splittings of the Atiyah algebroid $\mathrm{At}(P)$ are in one-to-one correspondence with connections on the principal bundle $P$, the automorphisms of $\mathrm{At}(P)$ are different from the gauge transformations of $P$.
This is precisely analogous to the fact that in the set-ups (2a) and (2b) configurations are in bijection, but the symmetries---and thus also the moduli stacks---are different.

By this analogy with classical gauge theory, we take the perspective that in situations (2a) and (2b) the gerbe with connective structure is the fundamental object, and the relevant symmetries are those of the gerbe with connective structure, rather than those of its associated exact Courant algebroid%
\footnote{It is a separate, interesting geometric problem to study the latter quotient, but we will not address this question in the present paper.}.
However, this still leaves us with two possible answers to Problem~\ref{prob:mouli of NSNS B-fields}:

\begin{enumerate}
\item Let $\cG$ be a gerbe presenting the desired class in $\rmH^3(M;\ZZ)$.
The moduli stack of B-fields for this class is that of connections on $\cG$ modulo the symmetries of $\cG$.

\item Let $(\cG, \cA^{(1)})$ be a gerbe with connective structure presenting the desired class in $\rmH^3(M;\ZZ)$.
The moduli stack of B-fields for this class is that of curvings on $(\cG, \cA^{(1)})$ (equivalently, the isotropic splittings of $\AtCA(\cG, \cA^{(1)})$) modulo the symmetries of $(\cG, \cA^{(1)})$.
\end{enumerate}

Note that the configurations and symmetries in both set-ups are not equivalent to each other.
More formally, in the setting of Section~\ref{sec:moduli oo-prestacks of hgeo strs on M} we are thus asking for the moduli $\infty$-stacks arising from the choices
\begin{equation}
	\Sol_B^\scD = \Conf_B^\scD \coloneqq \Grb^{1,\scD}_\nabla\,,
	\qquad
	\Fix_0^\scD \coloneqq \Grb^{1,\scD}_{\nabla|1}\,,
	\qquad
	\Fix_1^\scD \coloneqq \Grb^{1,\scD}\,.
\end{equation}
The morphism $\Fix_0^\scD \to \Fix_1^\scD$ is the canonical projection which forgets the smooth family of connective structures on a smooth family of 1-gerbes on $M$.
We denote the resulting moduli $\infty$-stacks of B-fields by
\begin{equation}
	\scMdl_B(\cG)
	\qquad \text{and} \qquad
	\scMdl_B(\cG, \cA^{(1)})\,,
\end{equation}
respectively.
Now, as a direct consequence of Corollary~\ref{st:Grb conn comps} and Theorem~\ref{st:equiv result for Mdl oo-stacks}, we obtain the following solution to Problem~\ref{prob:mouli of NSNS B-fields}:

\begin{theorem}
\label{st:equiv of B-field moduli}
There is a canonical equivalence of moduli $\infty$-stacks,
\begin{equation}
	\scMdl_B(\cG)
	\simeq \scMdl_B(\cG, \cA^{(1)})\,.
\end{equation}
\end{theorem}

Next, we extend these arguments to the study of NSNS supergravity solutions and the resolution of Problem~\ref{prob:moduli of NSNS solutions}:
we consider the configuration presheaf
\begin{equation}
	\Conf^\scD_{NS} = \Grb^{1,\scD}_\nabla \times \Met_{0,d}^\scD \times \Omega^{1,v,\scD}_\cl
	\qquad
	\in \Fun(\rmB\scD^\opp, \sSet)\,.
\end{equation}

\begin{definition}
\label{def:NSNS SuGra solutions}
An \textit{integral NSNS supergravity solution} on $M$ is a triple $((\cG, \cA), g, \varphi)$ of a 1-gerbe with connection $(\cG, \cA)$ on $M$, a Riemannian metric $g$ on $M$ and a closed 1-form $\varphi$ on $M$, satisfying the equations
\begin{align}
\label{eq:NSNS field eqns}
	\Ric^g + \nabla^g \varphi - \frac{1}{4} \curv(\cG, \cA) \bullet_g \curv(\cG, \cA) &= 0\,,
	\\
	\dd^*_g \curv(\cG, \cA) + \iota^g_\varphi \curv(\cG, \cA) &= 0\,,
	\\
	\dd^*_g \varphi + |\varphi|^2_g - \big| \curv(\cG, \cA) \big|^2_g
	&= \lambda\,,
\end{align}
where $\lambda \in \RR$ is the level of the solution.
\end{definition}

\begin{remark}
For each 1-gerbe with connective structure $(\cG, \cA^{(1)})$ on $M$, triples $(A_2, g, \varphi)$ of a curving $A_2$, a Riemannian metric $g$ on $M$ and a closed 1-form $\varphi$ on $M$ are in canonical bijection with pairs of a generalised metrics on the exact Courant algebroid $\AtCA(\cG,\cA^{(1)})$ and a closed divergence operator for this generalised metric (see~\cite[Prop.~2.53]{GFS:Gen_Ricci_flow}, the paragraph following that proposition, and~\cite[Def.~3.40]{GFS:Gen_Ricci_flow}).
Further restricting to exact 1-forms $\varphi$ recovers the standard bosonic sector of the low-energy effective action in type-II string theory (see~\cite[Def.~3.49]{GFS:Gen_Ricci_flow} and the discussion following that definition).
\qen
\end{remark}

\begin{lemma}
\label{st:Sol_NS^scD is solution subpresheaf}
The set of equations~\eqref{eq:NSNS field eqns} defines a solution subpresheaf (cf.~Definition~\ref{def:solution presheaf})
\begin{equation}
	\Sol_{NS}^\scD \hookrightarrow \Conf_{NS}^\scD\,.
\end{equation}
\end{lemma}

\begin{proof}
Let $c \in \Cart$.
Any pair of vertices of the simplicial set $\Conf_{NS}^\scD(c)$ which lie in the same connected component has the following property:
the smooth families of metrics, the smooth families closed 1-forms $\varphi$, and the smooth families of (fibrewise) curvature 3-forms associated to the underlying families of 1-gerbes on $M$ with connection, differ only by the pullback along a smooth family of diffeomorphisms on $M$.
However, a triple $(H, g, \varphi)$ of a smooth $c$-paramterised family of 3-forms $H$, Riemannian metrics $g$ and closed 1-forms $\varphi$ on $M$ satisfies the equations~\eqref{eq:NSNS field eqns} at level $\lambda \in \RR$ if and only if, for each smooth family $f \colon c \to \Diff(M)$ of diffeomorphisms of $M$, the triple $(f^*H, f^*g, f^*\varphi)$ satisfies equations~\eqref{eq:NSNS field eqns} at level $\lambda \in \RR$.

Thus, for any $c \in \Cart$ and any 1-simplex in $\Conf_{NS}^\scD(c)$, its source is a smooth family of NSNS supergravity solutions on $M$ if and only if its target is so.
Consequently, the inclusion $\Sol_{NS}^\scD(c) \hookrightarrow \Conf_{NS}^\scD(c)$ is an inclusion of a disjoint union of connected components.
\end{proof}

Following our discussion at the beginning of this subsection, in Definition~\ref{def:NSNS SuGra solutions} we have taken the NSNS supergravity B-field to live on a 1-gerbe.
Before addressing Problem~\ref{prob:moduli of NSNS solutions} we briefly comment on the relation of this set-up to NSNS supergravity solutions on exact Courant algebroids (in place of gerbes).

\begin{definition}
\label{def:GRic soliton}
A \textit{NSNS supergravity solution} on $M$ is an exact Courant algebroid $\cE$ with a generalised metric $\cG$ and a closed 1-form $\varphi$ on $M$ satisfying the equations
\begin{equation}
\label{eq:GRic soliton}
	\Ric^g + \nabla^g \varphi - \frac{1}{4} H \bullet_g H = 0\,,
	\qquad
	\dd^*_g H + \iota_\varphi H = 0\,,
	\qquad
	\dd^*_g \varphi + |\varphi|^2_g
	= |H|^2_g + \lambda\, ,
\end{equation}
where $\lambda \in \RR$ is the level of the solution, $g$ is the Riemannian metric canonically determined by $\cG$ and $H$ is is the preferred representative of the \v{S}evera class of $\cE$ defined by the splitting canonically determined by $\cG$ \cite[Prop.~2.40]{GFS:Gen_Ricci_flow}.
\end{definition}

\begin{remark}
The equations of motion defining NSNS supergravity on a gerbe can be entirely expressed in terms of \emph{generalised} curvature operators \cite{Coimbra:2011nw,Garcia-Fernandez:2013gja,GFS:Gen_Ricci_flow}. For our purposes here however the previous formulation is sufficient.
\qen
\end{remark}

We thus set
\begin{equation}
	\Conf_{GRic}^\scD \coloneqq \GMet_{0,d}^\scD \times \Omega^{1,v,\scD}_\cl
	= \Met_{0,d}^\scD \times \ECA_\nabla^\scD \times \Omega^{1,v,\scD}_\cl\,,
\end{equation}
and let $\Sol_{GRic}^\scD \subset \Conf_{GRic}^\scD$ denote the full simplicial subpresheaf on those vertices which are smooth families of NSNS supergravity solutions on $M$ in the sense of Definition~\ref{def:GRic soliton}.
This is a solution subpresheaf by arguments analogous to the proof of Lemma~\ref{st:Sol_NS^scD is solution subpresheaf}.
The close relation between integral NSNS supergravity solutions on $M$ in the sense of Definition~\ref{def:NSNS SuGra solutions} and NSNS supergravity solutions becomes evident in the following statement, which is a direct consequence of Lemma~\ref{st:Grb-ECA square is (ho)cartesian}:

\begin{proposition}
There is a commutative diagram of simplicial presheaves on $\rmB\scD$:
\begin{equation}
\begin{tikzcd}[column sep=1.75cm, row sep=0.75cm]
	\Sol_{NS}^\scD \ar[r, "\AtCA_\nabla \times \id"] \ar[d, hookrightarrow]
	& \Sol_{GRic}^\scD \ar[d, hookrightarrow]
	\\
	\Conf_{NS}^\scD \ar[r, "\AtCA_\nabla \times \id"'] \ar[d]
	& \Conf_{GRic}^\scD \ar[d]
	\\
	\Grb^{1,\scD}_{\nabla|1} \ar[r, "\AtCA"']
	& \ECA^\scD
\end{tikzcd}
\end{equation}
(by a slight abuse of notation, the top morphism employs the restriction of the morphism $\AtCA_\nabla$ in~\eqref{eq:Grb-ECA square forg conns} to the simplicial subpresheaf $\Sol_{NS}^\scD$).
Each square in this diagram is cartesian and homotopy cartesian.
\end{proposition}

By our above discussion there are two valid ways of formulating moduli problems for NSNS supergravity solutions.
These correspond precisely to the following two choices of fixed data:
we set
\begin{equation}
	\Fix_{NS,0}^\scD \coloneqq \Grb^{1,\scD}_{\nabla|1}\,,
	\qquad \text{and} \qquad
	\Fix_{NS,1}^\scD \coloneqq \Grb^{1,\scD}_\nabla\,.
\end{equation}
Let \smash{$(\cG, \cA^{[1]}) \in \Grb^{1,M}_{\nabla|1}(\RR^0)$}.
For each left fibration $Q \to N\Cart^\opp$ with reduced fibres and each morphism $\phi \colon Q \to N\rmB\scD[\cG]^\opp$, we thus obtain two moduli $\infty$-stacks (see Definition~\ref{def:scMdl_phi(cG)}),
\begin{equation}
	\scMdl_{NS,0; \Phi}(\cG, \cA^{(1)})
	\qquad \text{and} \qquad
	\scMdl_{NS,1; \Phi}(\cG)
	\qquad
	\in \scFun(N\Cart^\opp, \scS)
\end{equation}
from the morphisms
\begin{equation}
	\phi^* \textint \Sol_{NS}^\scD \longrightarrow \phi^* \textint \Fix_{NS,0}^\scD
	\qquad \text{and} \qquad
	\phi^* \textint \Sol_{NS}^\scD \longrightarrow \phi^* \textint \Fix_{NS,1}^\scD\,,
\end{equation}
respectively.
The moduli $\infty$-prestack $\scMdl_{NS,0; \Phi}(\cG, \cA^{(1)})$ describes NSNS supergravity solutions on a fixed pair $(\cG, \cA^{(1)})$ of a gerbe with 1-connection on $M$ modulo the action of the smooth $\infty$-group of symmetries of $(\cG, \cA^{(1)})$ lifting the action on $M$ described by the left fibration $Q \to N\Cart^\opp$.
The moduli $\infty$-stack $\scMdl_{NS,1; \Phi}(\cG)$ describes integral NSNS supergravity solutions on a fixed gerbe $\cG$ on $M$ modulo the analogous symmetries of $\cG$.
Equivalently, it describes NSNS supergravity solutiopns on the particular exact Courant algebroid $\AtCA(\cG, \cA^{(1)})$, but modulo symmetries \textit{of the underlying gerbe with connective structure $(\cG, \cA^{(1)})$} rather than symmetries of the exact Courant algebroid.
As an application of our main theorem (Theorem~\ref{st:equiv result for Mdl oo-stacks}), we now arrive at the following equivalence result which, in particular, resolves Problem~\ref{prob:moduli of NSNS solutions}:

\begin{theorem}
\label{st:equiv of NSNS SuGra moduli}
For each gerbe with connective structure $(\cG, \cA^{(1)}) \in \Grb^{1,M}_{\nabla|1}(\RR^0)$ on $M$, there is a canonical equivalence of moduli $\infty$-prestacks
\begin{equation}
	\scMdl_{NS,0; \Phi}(\cG, \cA^{(1)})
	\simeq \scMdl_{NS,1; \Phi}(\cG)\,.
\end{equation}
Moreover, if $H \colon \Cart^\opp \to \Set$ is a sheaf of groups and \smash{$\phi' \colon H \to \Diff_{[\cG]}(H)$} is a morphism of sheaves of groups, let $\phi \colon N \textint \rmB H \longrightarrow \rmB\scD[\cG]$ be the induced morphism in $\sSet_{/N \Cart^\opp}$.
Then, both $\scMdl_{NS,0; \Phi}(\cG, \cA^{(1)})$ and $\scMdl_{NS,1; \Phi}(\cG)$ are $\infty$-stacks on $\Cart$; that is, they satisfy descent with respect to good open coverings of cartesian spaces.
\end{theorem}

\begin{proof}
Theorem~\ref{st:equiv result for Mdl oo-stacks} applies because the canonical projection $\Fix_{NS,1} \to \Fix_{NS,0}$ is a projective fibration which induces a bijection of connected components over each cartesian space (by Proposition~\ref{st:Grb conn comps}).
This proves the first claim.
The second claim follows readily by Theorem~\ref{st:Descent result for moduli oo-prestacks}.
\end{proof}

That is, we can study NSNS supergravity solutions for integer \v{S}evera class equivalently in terms of connections on a fixed 1-gerbe, or in terms of isotropic splittings of the exact Courant algebroid associated to the same 1-gerbe with any choice of connective structure, so long as we divide out by symmetries of the underlying 1-gerbe (respectively with 1-connection).

%%%%%%%%%%%%%%%%%%%%%%%%%%%%%%%%%%%%%%%%%%%%%%%%%%%%%%%%%%%%%%%%%%%%%%%%%%%%

\part{Appendices}
\label{part:Appendices}

%%%%%%%%%%%%%%%%%%%%%%%%%%%%%%%%%%%%%%%%%%%%%%%%%%%%%%%%%%%%%%%%%%%%%%%%%%%%

\begin{appendix}

%%%%%%%%%%%%%%%%%%%%%%%%%%%%%%%%%%%%%%%%%%%%%%%%%%%%%%%%%%%%%%%%%%%%%%%%%%%%

\section{Rectification of simplicial presheaves}
\label{app:rectification}

%%%%%%%%%%%%%%%%%%%%%%%%%%%%%%%%%%%%%%%%%%%%%%%%%%%%%%%%%%%%%%%%%%%%%%%%%%%%

Let $\scC$ be a small category.
In this appendix we recall some results from~\cite{Bunk:Localisation_of_sSet, HM:Left_fibs_and_hocolims_I} which allow us to pass between 1-categorical functors $F \colon \scC \to \sSet$, left fibrations $X \to N\scC$ and $\infty$-functors $\bbF \colon N\scC \to \scS$ from the nerve of $\scC$ to the $\infty$-category $\scS$ of spaces (or, equivalently, of $\infty$-groupoids).
These three objects can be viewed as incarnations of the same mathematical structure, in light of the following well-known result.
Note that versions of these results for functors out of simplicial categories have also been proven in~\cite{Lurie:HTT}.

\begin{theorem}
\label{st:three perspectives on scFun(NC,scS)}
\emph{\cite[Thm.~7.8.9, Cor.~7.9.9]{Cisinski:HiC_and_HoA}}
There are canonical equivalences between the following $\infty$-categories:
\begin{enumerate}
\item the $\infty$-category $\scFun(N\scC, \scS)$ of $\infty$-functors $N\scC \to \scS$,

\item the $\infty$-categorical localisations of the functor category $\Fun(\scC, \sSet)$ at the objectwise weak homotopy equivalences, and

\item the $\infty$-categorical localisations of the slice category $\sSet_{/N\scC}$ at the covariant weak equivalences.
\end{enumerate}
\end{theorem}

In this paper we make frequent use of each of these three models for $\scS$-valued $\infty$-functors.
It is thus both necessary and very helpful to be able to pass between these models.
Here we recall functors which facilitate this, together with some of their main properties.
We do not present their constructions in detail, for it is only the existence of the functors and their properties that we use in the main text.
For details of the constructions of these functors, we refer to~\cite{Cisinski:HiC_and_HoA, Bunk:Localisation_of_sSet, Lurie:HTT} (see below for more detailed references).
The homotopically meaningful passage from functors $\scC \to \sSet$ to simplicial sets over the nerve $N\scC$ is facilitated by the following theorem:

\begin{theorem}
\label{st:r_C^* as Quillen equivalence}
\emph{\cite[Thm.~C]{HM:Left_fibs_and_hocolims_I}}
There is a Quillen equivalence
\begin{equation}
\begin{tikzcd}
	r_{\scC!} : \sSet_{/N\scC} \ar[r, shift left=0.125cm, "\perp"' yshift=0.05cm]
	& \Fun(\scC, \sSet) : r_\scC^*\ar[l, shift left=0.125cm]
\end{tikzcd}
\end{equation}
between the covariant model structure on the over-category $\sSet_{/N\scC}$ and the projective model structure on the functor category $\Fun(\scC, \sSet)$.
\end{theorem}

\begin{definition}
\label{def:r_C^*}
We call the right adjoint, $r_\scC^* \colon \Fun(\scC, \sSet) \longrightarrow \sSet_{/N\scC}$, the \textit{rectification functor} (for the category $\scC$).
\end{definition}

The rectification functor $r_\scC^*$ has pleasant technical properties:

\begin{lemma}
\label{st:r_C^* preserves injective cofibrations}
\emph{\cite[Lemma~2.22]{Bunk:Localisation_of_sSet}}
The functor $r_\scC^*$ sends injective cofibrations in $\Fun(\scC, \sSet)$ (i.e.~objectwise monomorphisms of simplicial sets) to covariant cofibrations in $\sSet_{/N\scC}$ (i.e.~monomorphisms of simplicial sets over $N\scC$).
\end{lemma}

Since $r_\scC^*$ can be constructed by means of finite limits, we also have:

\begin{lemma}
\label{st:r_C^* and filtered colimits}
\emph{\cite[Cor.~2.21]{Bunk:Localisation_of_sSet}}
The rectification functor $r_\scC^*$ commutes with filtered colimits.
\end{lemma}

Rectification is compatible with changing the indexing category $\scC$ in the following way:
this was observed already in~\cite[Rmk.~3.2.5.7]{Lurie:HTT} (for a proof in the present context, see~\cite[Lemma~2.23]{Bunk:Localisation_of_sSet}).

\begin{lemma}
\label{st:r_C^* and pullbacks}
For each functor $\psi \colon \scD \to \scC$ of small categories and each functor $F \in \Fun(\scC, \sSet)$ there is a canonical cartesian square of simplicial sets
\begin{equation}
\begin{tikzcd}
	r_\scD^*(\psi^*F) \ar[r] \ar[d]
	& r_\scC^* F \ar[d]
	\\
	N\scD \ar[r, "N\psi"']
	& N\scC
\end{tikzcd}
\end{equation}
This induces a natural isomorphism of functors%
\footnote{This holds true for any given choice of pullbacks along $N\psi$, as necessary to make the operation of pulling back along this morphism functorial (in the same way as one makes choices to make the formation of (co)limits functorial).}
$\Fun(\scC, \sSet) \to \sSet_{/N\scD}$,
\begin{equation}
	r_\scC^* \circ \psi^*
	\cong (N\psi)^* \circ r_\scC^*\,.
\end{equation}
\end{lemma}

Let $\Kan \subset \sSet$ denote the full subcategory on the Kan complexes.
By combining the functors
\begin{equation}
	r_{[n] \times \scC}^* \colon \Fun \big( [n] {\times} \scC, \sSet \big)
	\longrightarrow \sSet_{/(\Delta^n \times N\scC)}\,,
\end{equation}
for $n \in \NN_0$, and using the model for the $\infty$-category $\scS$ of spaces presented in~\cite[Ch.~5]{Cisinski:HiC_and_HoA}, one can construct an explicit $\infty$-functor~\cite[Thm.~3.11]{Bunk:Localisation_of_sSet}
\begin{equation}
	\gamma_\scC \colon N\Fun(\scC, \Kan) \longrightarrow \scFun(N\scC, \scS)
\end{equation}
which provides a concrete way of passing from strict functors $\scC \to \sSet$ to $\infty$-functors $N\scC \to \scS$.
The following theorem makes manifest the homotopy-theoretic significance of this construction:

\begin{theorem}
\label{st:r_C^* and oo-categorical localisation}
\emph{\cite[Thm.~4.8]{Bunk:Localisation_of_sSet}}
The functor $\gamma_\scC \colon N\Fun(\scC, \Kan) \longrightarrow \scFun(N\scC, \scS)$ exhibits $\scFun(N\scC, \scS)$ as the $\infty$-categorical localisation of $N\Fun(\scC, \Kan)$ at the objectwise weak homotopy equivalences.
\end{theorem}

Finally, the relation between the rectification functor $r_\scC^*$ and the localisation functor $\gamma_\scC$ is elucidated by the following observation:

\begin{theorem}
\emph{\cite[Cor.~3.13]{Bunk:Localisation_of_sSet}}
Let $F \in \Fun(\scC, \Kan)$.
The left fibration classified by the $\infty$-functor $\gamma_\scC(F) \colon N\scC \to \scS$ agrees with the rectification $r_\scC^*F \to N\scC$.
\end{theorem}

This makes precise that the (1-)functor $F \colon \scC \to \Kan$, the left fibration $r_\scC^*F \to N\scC$, and the homotopy coherent $\infty$-functor $\gamma_\scC^*F \colon N\scC \to \scS$ all correspond to each other under the equivalences in Theorem~\ref{st:three perspectives on scFun(NC,scS)}.

%%%%%%%%%%%%%%%%%%%%%%%%%%%%%%%%%%%%%%%%%%%%%%%%%%%%%%%%%%%%%%%%%%%%%%%%%%%%

\section{The Dold-Kan correspondence and $n$-gerbes with $k$-connection}
\label{app:Higher U(1)-bundles via sPShs}

%%%%%%%%%%%%%%%%%%%%%%%%%%%%%%%%%%%%%%%%%%%%%%%%%%%%%%%%%%%%%%%%%%%%%%%%%%%%

%%%%%%%%%%%%%%%%%%%%%%%%%%%%%%%%%%%%%%%%%%%%%%%%%%%%%%%%%%%%%%%%%%%%%%%%%%%%

\subsection{The Dold-Kan correspondence}
\label{app:DK correspondence}

%%%%%%%%%%%%%%%%%%%%%%%%%%%%%%%%%%%%%%%%%%%%%%%%%%%%%%%%%%%%%%%%%%%%%%%%%%%%

In this appendix we summarise, for the reader's convenience, the standard construction of $(n{+}1)$-groupoids of $n$-gerbes with $k$-connections on manifolds.
This proceeds via \v{C}ech resolutions of the Deligne complexes (see, for instance,~\cite{FSS:Cech_for_diff_classes, Schreiber:DCCT_v2} for more details).
In the main text, we enhance this construction to describe smooth families of $n$-gerbes with $k$-connections on a fixed manifold $M$, subject to the smooth action of the diffeomorphism  group $\Diff(M)$ via pullback.

Let $\Mfd$ denote the category of smooth manifolds and smooth maps.
Let $\Cart \subset \Mfd$ denote the full subcategory on the subset of objects $\{c \in \Mfd\, | \, \exists n \in \NN_0: c \cong \RR^n\}$, i.e.~on those manifolds which are diffeomorphic to $\RR^n$ for any $n \in \NN_0$.
We consider the Deligne complex of sheaves of abelian groups%
\footnote{In~\cite{Gajer:Geo_of_Deligne_coho, Brylinski:LSp_and_Geo_Quan} the Deligne complex is instead defined as the quasi-isomorphic chain complex $\Omega^k \leftarrow \cdots \leftarrow \Omega^1 \leftarrow \Omega^0 \hookleftarrow \ZZ$; we follow the conventions of, for instance,~\cite{FSS:Cech_for_diff_classes}.}
on $\Cart$, which is the chain complex given by
\begin{equation}
\rmb^k_\nabla \U(1) = 
\begin{tikzcd}[ampersand replacement=\&]
	\big( \Omega^k
	\& \cdots \ar[l, "\dd"']
	\& \Omega^1 \ar[l, "\dd"']
	\& \U(1) \ar[l, "\dd \log"'] \big)\,,
\end{tikzcd}
\end{equation}
where $\Omega^k$ is situated in degree zero.
We also define the shifted complexes
\begin{equation}
\label{eq:b^l b_nabla^k U(1)}
	\rmb^l \rmb^k_\nabla \U(1) = \rmb^k_\nabla \U(1)[-l]\,.
\end{equation}
In particular, we have $\rmb^l(\rmb^k_\nabla \U(1))_l = \Omega^k$.

We recall that the category $\Ch_{\geq 0}$ of non-negatively graded chain complexes of abelian groups carries a model structure; its weak equivalences are the quasi-isomorphisms, and its fibrations are the morphisms of chain complexes which are surjective in each \textit{positive} degree (see, for instance,~\cite[p.~157]{GJ:Simplicial_HoThy}).
The category $\Ab_\Delta$ of simplicial abelian groups carries a model structure which is right transferred along the forgetful functor $\Ab_\Delta \to \sSet$.
Here, $\sSet$ is the category of simplicial sets endowed with the Kan-Quillen model structure, and the forgetful functor $\Ab_\Delta \to \sSet$ is a right Quillen functor.
We recall the following theorem~\cite[Sec.~II.4]{Quillen:Homotopical_Algebra}:

\begin{theorem}
\label{st:Dold-Kan corr}
(Dold-Kan correspondence)
There are Quillen equivalences
\begin{equation}
\begin{tikzcd}
	N : \Ab_\Delta \ar[r, shift left=0.15cm, "\perp"' yshift=0.05cm]
	& \Ch_{\geq 0} : \varGamma \ar[l, shift left=0.15cm]
	& \text{and}
	& \varGamma : \Ch_{\geq 0} \ar[r, shift left=0.15cm, "\perp"' yshift=0.05cm]
	& \Ab_\Delta : N\,. \ar[l, shift left=0.15cm]
\end{tikzcd}
\end{equation}
Both $N$ and $\varGamma$ are homotopical, i.e.~preserve weak equivalences.
\end{theorem}

For the reader's convenience, we describe the functor in more detail (for a full treatment, we refer to~\cite[Sec.~III.2]{GJ:Simplicial_HoThy}, \cite[Sec.~8.4]{Weibel:Intro_to_HomAlg}; here, we follow the conventions in~\cite[Sec.~1.2.3]{Lurie:HA}):
given a chain complex $C \in \Ch_{\geq 0}$ and $t \in \NN_0$, set
\begin{equation}
	(\varGamma C)_r
	\coloneqq \bigoplus_{\varphi \colon [r] \twoheadrightarrow [t]} C_t\,,
\end{equation}
where the direct sum is over all surjections $\varphi \colon [r] \twoheadrightarrow [t]$ in $\bbDelta$.
To make the collection $\{ (\varGamma C)_r \, | \, r \in \NN_0\}$ into a simplicial set, consider a morphism $\theta \colon [s] \to [r]$ in $\bbDelta$.
One associates to this the map
\begin{equation}
	(\varGamma C)(\theta) \colon (\varGamma C)_r \to (\varGamma C)_s
\end{equation}
described as follows:
let $\varphi \colon [r] \twoheadrightarrow [t]$ be a surjection in $\bbDelta$.
Then, we are left to provide a morphism
\begin{equation}
	(\varGamma C)(\theta; \varphi) \colon C_t
	\longrightarrow \bigoplus_{\psi \colon [s] \twoheadrightarrow [u]} C_u = (\varGamma C)_s\,.
\end{equation}
Let
\begin{equation}
\begin{tikzcd}
	{[s]} \ar[r, "\theta"] \ar[d, twoheadrightarrow, "\pi"']
	& {[r]} \ar[d, twoheadrightarrow, "\varphi"]
	\\
	{[v]} \ar[r, hookrightarrow, "\iota"']
	& {[t]}
\end{tikzcd}
\end{equation}
be the unique factorisation of the composite $\varphi \circ \theta$ into an epimorphism $\pi$ followed by a monomorphism $\iota$.
We now distinguish three cases:
\begin{enumerate}
\item If $\iota$ is the identity (i.e.~if $[v] = [t]$), we let $(\varGamma C)(\theta; \varphi)$ be the canonical inclusion
\begin{equation}
	C_t \hookrightarrow \bigoplus_{\psi \colon [s] \twoheadrightarrow [u]} C_u 
\end{equation}
as the summand labelled by $\pi = \varphi \circ \theta \colon [s] \twoheadrightarrow [v] = [t]$.

\item If $\iota = \partial_0 \colon [t-1] \hookrightarrow [t]$ is the standard cosimplicial coface map (hitting each $i \in [t]$ apart from $i = 0$), then the map $(\varGamma C)(\theta; \varphi)$ is the composition
\begin{equation}
\begin{tikzcd}
	C_t \ar[r, "\partial_C"]
	& C_{t-1} \ar[r]
	& \displaystyle{\bigoplus_{\psi \colon [s] \twoheadrightarrow [u]}} C_u\,,
\end{tikzcd}
\end{equation}
where the first map is the differential of the chain complex $C$, and the second map is the inclusion of $C_t$ as the summand labelled by $\pi \colon [s] \twoheadrightarrow [t-1]$.

\item If $\iota \colon [v] \hookrightarrow [t]$ is any other map, then the map
\begin{equation}
	(\varGamma C)(\theta; \varphi) \colon C_t \longrightarrow \bigoplus_{\psi \colon [s] \twoheadrightarrow [u]} C_u
\end{equation}
is the zero map.
\end{enumerate}
Finally, if two summands in the domain of $(\varGamma C)(\theta)$ map to the same summand in its codomain, we add their contributions using the abelian group structure on the codomain.

\begin{remark}
Since we follow the conventions in~\cite[Sec.~1.2.3]{Lurie:HA}, the simplicial set $\varGamma C$ we associate to a chain complex here is the \textit{opposite} simplicial set of that produced in~\cite{Weibel:Intro_to_HomAlg}.
Note that the auto-equivalence of $\sSet$, $K \mapsto K^\opp$ preserves Kan fibrations, cofibrations and weak homotopy equivalences, so introducing this opposite has no effect on the fact that the Dold-Kan correspondence is a two-sided Quillen equivalence.
\qen
\end{remark}

\begin{example}
\label{eg:DK explicit}
Let $C \in \Ch_{\geq 0}$ be a non-negatively graded chain complex of abelian groups.
We have the following identifications:
\begin{align}
	(\varGamma C)_0
	&= \underbrace{C_0}_{\id_{[0]}}\,,
	\\
	(\varGamma C)_1
	&= \underbrace{C_0}_{[1] \twoheadrightarrow [0]}
	\oplus \underbrace{C_1}_{\id_{[1]}}\,,
	\\
	(\varGamma C)_2
	&= \underbrace{C_0}_{[2] \twoheadrightarrow [0]}
	\oplus \underbrace{C_1}_{\sigma_1 \colon [2] \twoheadrightarrow [1]}
	\oplus \underbrace{C_1}_{\sigma_2 \colon [2] \twoheadrightarrow [1]}
	\oplus \underbrace{C_2}_{\id_{[2]}}\,,
	\quad \text{etc.}
\end{align}
One can now check explicitly that
\begin{alignat}{3}
	d_0 \colon (\varGamma C)_1 &\to (\varGamma C)_0\,,
	&\hspace{2cm}
	d_1(x_0, x_{01}) && &= x_0
	\\
	d_0 \colon (\varGamma C)_1 &\to (\varGamma C)_0\,,
	&\hspace{2cm}
	d_0(x_0, x_{01}) && &= x_0 + \partial_C x_{01}\,,
	\\
	s_0 \colon (\varGamma C)_0 &\to (\varGamma C)_1\,,
	&\hspace{2cm}
	s_0(x_0) && &= (x_0, 0)\,,
	\\
	d_2 \colon (\varGamma C)_2 &\to (\varGamma C)_1\,,
	&\qquad
	d_2(x_0, x_{01}, x_{12}', x_{012}) && &= (x_0, x_{01})\,,
	\\
	d_1 \colon (\varGamma C)_2 &\to (\varGamma C)_1\,,
	&\qquad
	d_1(x_0, x_{01}, x_{12}', x_{012}) && &= (x_0, x_{01} + x_{12}')\,,
	\\
	d_0 \colon (\varGamma C)_2 &\to (\varGamma C)_1\,,
	&\qquad
	d_0(x_0, x_{01}, x_{12}', x_{012}) && &= (x_0 + \partial_C x_{01}, x_{12}' + \partial_C x_{012})\,.
\end{alignat}
The elements in $C_0$ thus become the vertices of $\varGamma C$.
A 1-simplex is a pair $(x_0, x_{01})$ of $x \in C_0$ and $x_{01} \in C_1$; viewing the pair $(x_0, x_{01})$ as an edge in $\varGamma C$, its source is $x_0$, and its target is $x_0 + \partial_C x_{01}$.
With this understanding of vertices and 1-simplices, we can similarly understand a 2-simplex $(x_0, x_{01}, x_{12}', x_{012})$ in $\varGamma C$, given the above computation of its faces.
\qen
\end{example}

%%%%%%%%%%%%%%%%%%%%%%%%%%%%%%%%%%%%%%%%%%%%%%%%%%%%%%%%%%%%%%%%%%%%%%%%%%%%

\subsection{Higher $\U(1)$-bundles with connections via simplicial presheaves}
\label{app:n-gerbes w k-conn via DK}

%%%%%%%%%%%%%%%%%%%%%%%%%%%%%%%%%%%%%%%%%%%%%%%%%%%%%%%%%%%%%%%%%%%%%%%%%%%%

We can now combine the Dold-Kan correspondence with \v{C}ech resolutions of the Deligne complexes in order to define Kan complexes of $n$-gerbes with $k$-connections on manifolds.
The functor categories $\Fun(\Cart^\opp, \Ch_{\geq 0})$ and $\scH \coloneqq \Fun(\Cart^\opp, \sSet)$ carry projective model structures~\cite{Barwick:Localistaions} (since $\Ch_{\geq 0}$ is combinatorial by~\cite[Thm.~2.3.11]{Hovey:MoCats}).
By~\cite[Prop.~5.6.2]{Gillam:Simplicial_methods}, the model category $\Ab_\Delta$ is cofibrantly generated.
Thus, also $\Fun(\Cart^\opp, \Ab_\Delta)$ carries a projective model structure.
By a slight abuse of notation we have:

\begin{corollary}
Applying the functors $N$ and $\varGamma$ objectwise produces Quillen equivalences
\begin{equation}
\begin{tikzcd}[row sep=0.2cm]
	N : \Fun(\Cart^\opp, \Ab_\Delta) \ar[r, shift left=0.15cm, "\perp"' yshift=0.05cm]
	& \Fun(\Cart^\opp, \Ch_{\geq 0}) : \varGamma\,, \ar[l, shift left=0.15cm]
	\\
	\varGamma : \Fun(\Cart^\opp, \Ch_{\geq 0}) \ar[r, shift left=0.15cm, "\perp"' yshift=0.05cm]
	& \Fun(\Cart^\opp, \Ab_\Delta) : N\,. \ar[l, shift left=0.15cm]
\end{tikzcd}
\end{equation}
Both $N$ and $\varGamma$ are homotopical, i.e.~preserve weak equivalences.
\end{corollary}

Similarly, the free-forgetful adjunction induces a Quillen adjunction
\begin{equation}
\begin{tikzcd}[row sep=0.2cm]
	\scH \ar[r, shift left=0.15cm, "\perp"' yshift=0.05cm]
	& \Fun(\Cart^\opp, \Ab_\Delta)\,. \ar[l, shift left=0.15cm]
\end{tikzcd}
\end{equation}
Recall the presheaf of chain complexes $\rmb^l \rmb^k_\nabla \U(1)$ from~\eqref{eq:b^l b_nabla^k U(1)}.

\begin{definition}
\label{def:B^l B^k_nabla U(1) and B^l B^k_nabla Z}
For each $l, k \in NN_0$, we define a simplicial presheaf on $\Cart$ by setting
\begin{equation}
	\rmB^l \rmB^k_\nabla \U(1) \coloneqq \varGamma \big( \rmb^l \rmb^k_\nabla \U(1) \big)
\end{equation}
(where we have not displayed the forgetful functor $\Fun(\Cart^\opp, \Ab_\Delta) \to \scH$).
\end{definition}

The category $\Cart$ carries a Grothendieck coverage:
it is induced by the good open coverings, i.e.~open coverings such that each finite intersection of patches is empty or again a cartesian space.
We consider the \textit{local} projective model structure $\scH^{loc}$ on $\scH$, which is the left Bousfield localisation of $\scH$ at the \v{C}ech nerves of good open coverings of cartesian spaces.
The simplicial presheaves from Definition~\ref{def:B^l B^k_nabla U(1) and B^l B^k_nabla Z} are fibrant objects in $\scH^{loc}$.
That is, they are projectively fibrant and satisfy homotopy descent with respect to good open coverings.
The projective fibrancy is by construction (simplicial groups are Kan complexes~\cite[Lemma~I.3.4]{GJ:Simplicial_HoThy}), and the homotopy descent property can be seen via sheaf-chomology arguments (see, for instance,~\cite{FSS:Cech_for_diff_classes}, or~\cite[Prop.~2.50, Cor.~2.51]{Bunk:Gerbes_Review} for a review).

\begin{definition}
A \textit{good open covering} of an arbitrary manifold $M$ is an open covering $\cU = \{U_a\}_{a \in \Lambda}$ of $M$ such that each finite intersection of patches is empty or a cartesian space.
\end{definition}

Given a good open covering $\cU$ of $M$, we view its \v{C}ech nerve $\check{C} \cU$ as an object $\check{C} \cU \in \scH$ (by forming the associated representable presheaves in each simplicial degree).
Since it is a coproduct of representable presheaves in each degree, it is cofibrant in $\scH$ and hence also in $\scH^{loc}$.
Both these model categories are simplicial.
We denote the simplicial hom functors in a simplicially enriched category $\scC$ by $\ul{\scC}(-,-)$.

\begin{definition}
\label{def:Grb^n_(nabla|k)(X)}
Given a cofibrant object $X \in \scH$ and $n \in \NN_0$, $k \in \{0, \ldots, n+1\}$, we make the following definitions:
\begin{enumerate}
\item We set \begin{equation}
	\Grb^n_{\nabla|k}(X) \coloneqq \ul{\scH} \big( X, \rmB^{n+1-k} \rmB^k_\nabla \U(1) \big)
\end{equation}
and refer to this as the \textit{Kan complex of $n$-gerbes with $k$-connection on $X$}.

\item In the case where $k = n+1$, we also write $\Grb^n_\nabla(X)$ instead of $\Grb^n_{\nabla|n+1}(X)$ and refer to this as the \textit{Kan complex of $n$-gerbes on $X$ with (full) connection}.

\item In the case where $k = 0$, we also write $\Grb^n(X)$ instead of $\Grb^n_{\nabla|0}(X)$ and refer to this as the \textit{Kan complex of $n$-gerbes on $X$} (without connection).
\end{enumerate}
\end{definition}

\begin{remark}
\label{rmk:Grb^n_(nabla|k) and cof rep}
If $Y \in \scH^{loc}$ is not cofibrant, we can choose a cofibrant replacement $X \to Y$ (in $\scH^{loc}$) and understand $\Grb^n_{\nabla|k}(X)$ as the Kan complex of $n$-gerbes with $k$-connection on $Y$, in the sense of Definition~\ref{def:Grb^n_(nabla|k)(X)}.
However, note that this Kan complex is defined only up to weak equivalence; the Kan complexes arising from different choices of cofibrant replacements of $Y$ will, in general, only be related by a zig-zag of weak equivalences in $\sSet$.
\qen
\end{remark}

\begin{remark}
This applies, in particular, to manifolds:
given $M \in \Mfd$, we consider the (simplicially constant) simplicial presheaf
\begin{equation}
	\ul{M} \colon \Cart^\opp \to \sSet\,,
	\qquad
	c \mapsto \Mfd(c,M)\,.
\end{equation}
This construction presents a fully faithful functor
\begin{equation}
	\Mfd \hookrightarrow \scH\,.
\end{equation}
Given a good open covering $\cU$ of $M$, the canonical morphism $\v{C}\cU \to \ul{M}$ is a cofibrant replacement of $\ul{M}$ in the model category $\scH^{loc}$ (not so in $\scH$).
Thus, by Definition~\ref{def:Grb^n_(nabla|k)(X)} and Remark~\ref{rmk:Grb^n_(nabla|k) and cof rep} we obtain the Kan complex of $n$-gerbes with $k$-connection on $M$ as
\begin{equation}
	\Grb^n_{\nabla|k}(\check{C} \cU) = \ul{\scH} \big( \check{C} \cU, \rmB^{n+1-k} \rmB^k_\nabla \U(1) \big)\,.
\end{equation}
Any two choices of good open covering $\cU$ of $M$ produce Kan complexes which are isomorphic in the homotopy category $\Ho(\sSet)$.
\qen
\end{remark}

Finally, for $\cU$ a good open covering of $M$, we compute the Kan complex \smash{$\Grb^n_{\nabla|k}(\check{C} \cU)$} explicitly.
To that end, we include the following (well-known) lemma:
let $\sfc_{\bbDelta^\opp} \colon \Set \to \sSet$ denote the functor which assigns to a set its simplicially constant simplicial set.
This induces a functor $\sfc_{\bbDelta^\opp} \colon \Fun(\Cart^\opp, \Set) \to \scH$.

\begin{lemma}
Let $X \in \scH$ be such that, for each $m \in \NN_0$, the simplicially constant simplicial presheaf $\sfc_{\bbDelta^\opp} X_m$ is cofibrant in $\scH$.
We obtain from $X$ a diagram $\sfc_{\bbDelta^\opp} X \colon \bbDelta^\opp \to \scH$, $[m] \mapsto \sfc_{\bbDelta^\opp} X_m$.
There is a canonical isomorphism in $\Ho \scH$,
\begin{equation}
	\hocolim_{\bbDelta^\opp} (\sfc_{\bbDelta^\opp} X) \cong X\,.
\end{equation}
\end{lemma}

We include a short proof for the reader's convenience.

\begin{proof}
It suffices to show this for any model for the homotopy colimit.
Since $\sfc_{\bbDelta^\opp} \colon \bbDelta^\opp \to \scH$ is valued in cofibrant objects, we have that~\cite[Cor.~5.1.3]{Riehl:Cat_HoThy}
\begin{equation}
	\hocolim_{\bbDelta^\opp} X_\cdot
	\simeq B(*, \bbDelta^\opp, \sfc_{\bbDelta^\opp} X)\,,
\end{equation}
where $B(-,-,-)$ denotes the bar construction~\cite[Sec.~4.2]{Riehl:Cat_HoThy}.
Since colimits and limits are computed pointwise in $\scH = \Fun(\Cart^\opp, \sSet)$ and the bar construction commutes with evaluation on any $c \in \Cart$, we have a natural weak equivalence
\begin{equation}
	(\hocolim_{\bbDelta^\opp} \sfc_{\bbDelta^\opp} X)(c)
	\simeq \hocolim_{\bbDelta^\opp} \big( \sfc_{\bbDelta^\opp} X(c) \big)\,.
\end{equation}
The homotopy colimit on the right-hand side is now taken in $\sSet$, where it is computed by the diagonal~\cite[Cor.~18.7.7]{Hirschhorn:MoCats_and_localisations}.
\end{proof}

Given a good open covering $\cU$ of a manifold $M$, we can now compute \smash{$\Grb^n_{\nabla|k}(\check{C}\cU)$} up to canonical weak equivalences as follows:
\begin{align}
	\Grb^n_{\nabla|k}(\check{C}\cU) &= \ul{\scH} \big( \check{C}\cU, \rmB^{n+1-k}\rmB^k_\nabla \U(1) \big)
	\\
	&\simeq \ul{\scH} \big( \hocolim_{\bbDelta^\opp} \sfc_{\bbDelta^\opp} \check{C} \cU, \rmB^{n+1-k}\rmB^k_\nabla \U(1) \big)
	\\
	&\cong \holim_\bbDelta \ul{\scH} \big( \sfc_{\bbDelta^\opp} \check{C} \cU, \rmB^{n+1-k}\rmB^k_\nabla \U(1) \big)
	\\
	&\cong \holim_{[i] \in \bbDelta} \big( \rmB^{n+1-k}\rmB^k_\nabla \U(1) \big) (\check{C}_i \cU)\,.
\end{align}
Here we have used the short-hand notation
\begin{equation}
	\big(\rmB^{n+1-k}\rmB^k_\nabla\U(1) \big) (\check{C}_i \cU)
	\coloneqq \prod_{a_0, \ldots, a_i \in \Lambda} \big( \rmB^{n+1-k}\rmB^k_\nabla \U(1) \big) (U_{a_0 \cdots a_i})\,,
\end{equation}
where $U_{a_0 \cdots a_i} = U_{a_0} \cap \cdots \cap U_{a_i}$ for $a_0, \ldots, a_i \in \Lambda$.
We further compute
\begin{align}
	\holim_{[i] \in \bbDelta} \big( \rmB^{n+1-k}\rmB^k_\nabla \U(1) \big) (\check{C}_i \cU)
	&= \holim_{[i] \in \bbDelta} \big( \varGamma \rmb^{n+1-k} \rmb^k_\nabla \U(1) \big) (\check{C}_i \cU)
	\\
	&\cong \varGamma \holim_{[i] \in \bbDelta} \big( \rmb^{n+1-k} \rmb^k_\nabla\U(1) \big) (\check{C}_i \cU)\,,
\end{align}
where we have used that $\varGamma$ is a right Quillen functor, and where the homotopy limit under $\varGamma$ is now computed in $\Ch_{\geq 0}$.
Finally, an explicit formula for this homotopy limit is given, for instance, in~\cite[Sec.~B.1]{BSS:Hocolims_and_global_observables}.
We obtain
\begin{equation}
\label{eq:Cech-holim for cochain complexes}
	\holim_{[i] \in \bbDelta} \big( \rmB^{n+1-k}\rmB^k_\nabla \U(1) \big) (\check{C}_i \cU)
	\simeq \tau_{\geq 0} \Tot^\times \big( \big( \rmb^{n+1-k} \rmb^k_\nabla \U(1) \big) (\check{C} \cU) \big)\,,
\end{equation}
the truncation to non-negative degrees of the total (product) complex of $(\rmb^{n+1-k} \rmb^k_\nabla \U(1)) (\check{C} \cU)$.
Note that in the bigrading on the double complex induced by the simplicial structure of $\check{C}\cU$, we count the simplicial degree $i$ in the \v{C}ech nerve $\check{C} \cU$ negatively, i.e.~giving degree $-i$.
Therefore, the resulting total complex has non-zero entries in both positive and negative degrees, but after applying the truncation $\tau_{\geq 0}$ we again obtain a non-negatively graded chain complex.

\begin{example}
\label{eg:transition data for n-grbs w K-conn}
Consider a good open covering $\cU$ of a manifold $M$.
By Example~\ref{eg:DK explicit} a vertex in \smash{$\Grb^n_{\nabla|k}(\check{C} \cU)$} is the same as tuple $(g, A_1, \ldots, A_k)$, where $g \colon \check{C}_{n+1}\cU \to \U(1)$, and \smash{$A_i \in \Omega^i(\check{C}_{n+1-i}\cU)$} for $i \in \{ 1, \ldots, k\}$, satisfying
\begin{align}
	\delta g &= 0\,,
	\\
	(-1)^{n+1} \dd \log g + \delta A_1 &= 0\,,
	\\
	(-1)^n \dd A_1 + \delta A_2 &= 0\,,
	\\
	&\vdots
	\\
	(-1)^{n-k+2} \dd A_{k-1} + \delta A_k &= 0\,.
\end{align}
Similarly, a 1-simplex is a pair $((g, A_1, \ldots, A_k), (h, C_1, \ldots, C_{k-1}))$, where $(g, A_1, \ldots, A_k)$ is a vertex as above, $h \colon \v{C}_n \cU \to \U(1)$ is any smooth function, and $C_i \in \Omega^i(\v{C}_{n-i} \cU)$ is any $i$-form, for $i = 1, \ldots, k-1$.
The source vertex of this 1-simplex is the tuple $(g, A_1, \ldots, A_k)$, and by Example~\ref{eg:DK explicit} its target vertex reads as
\begin{align}
	( &g \cdot \delta h,\,
	\\
	&A_1 + (-1)^n \dd \log h + \delta C_1,\,
	\\
	&A_2 + (-1)^{n-1} \dd C_1 + \delta C_2,\,
	\\
	&\vdots
	\\
	&A_{k-1} + (-1)^{n-k+2} \dd C_{k-2} + \delta C_{k-1},\,
	\\
	&A_k + (-1)^{n-k+1} C_{k-1} )\,.
\end{align}
Higher vertices arise in an analogous way from the explicit formula~\eqref{eq:Cech-holim for cochain complexes} together with computations as in Example~\ref{eg:DK explicit}.
\qen
\end{example}

\end{appendix}

%\newpage
%\renewcommand{\leftmark}{\MakeUppercase{Bibliography}}
\phantomsection
\bibliographystyle{JHEP}
%\bibliographystyle{plain}
%\newpage  

% % % % % % % % % % % % % % % % % % % % % % % % % % % % % % % % % % % % % % 
% % % % % % % % % % % % % % % % % % % % % % % % % % % % % % % % % % % % % %

\end{document}